\documentclass[11pt,leqno]{article}

\usepackage{amsthm,amsfonts,amssymb,amsmath}
\usepackage{fullpage}
\usepackage{graphicx}
\usepackage{mathrsfs}
\usepackage{xcolor}
\numberwithin{equation}{section}
\usepackage{enumitem}
\usepackage{hyperref}
\usepackage{subcaption}
\usepackage[nocompress]{cite}

\makeatletter
\def\namedlabel#1#2{\begingroup
    #2%
    \def\@currentlabel{#2}%
    \phantomsection\label{#1}\endgroup
}
\makeatother

\newtheorem{Theorem}{Theorem}[section]
\newtheorem{Lemma}[Theorem]{Lemma}
\newtheorem{Proposition}[Theorem]{Proposition}	
\newtheorem{Corollary}[Theorem]{Corollary}	
\theoremstyle{definition}
\newtheorem{Remark}[Theorem]{Remark}	 
\newtheorem{Definition}[Theorem]{Definition} 
\numberwithin{equation}{section}

\newcommand{\C}{\mathbb{C}} 
\newcommand{\R}{\mathbb{R}} 
\newcommand{\Z}{\mathbb{Z}} 
\newcommand{\N}{\mathbb{N}} 
\newcommand{\per}{\textup{per}}
\renewcommand{\Re}{\operatorname{Re}}
\renewcommand{\Im}{\operatorname{Im}}
\newcommand{\iu}{\mathrm{i}}
\newcommand{\eu}{\mathrm{e}}
\newcommand{\eps}{\varepsilon}
\newcommand{\at}{\mathsf{a}}
\newcommand{\ub}{\mathbf{u}}
\newcommand{\El}{\mathcal{L}}
\newcommand{\de}{\mathrm{d}}
\newcommand{\A}{\mathcal{A}}
\newcommand{\F}{\mathbb{F}}

\newcommand{\unu}{\underline{u}}
\newcommand{\unub}{\underline{\mathbf{u}}}
\newcommand{\Non}{\mathcal{N}}

\DeclareMathOperator{\sech}{sech}

\allowdisplaybreaks[3]

\usepackage{caption}

\title{Multiple front and pulse solutions in spatially periodic systems}
\author{Lukas Bengel$^*$ and Bj\"orn de Rijk\thanks{Department of Mathematics, Karlsruhe Institute of Technology, Englerstra\ss e 2, 76131 Karlsruhe, Germany; \texttt{lukas.bengel@kit.edu}, \texttt{bjoern.de-rijk@kit.edu}}}

\begin{document}
\maketitle

\begin{abstract} 
In this paper, we develop a comprehensive mathematical toolbox for the construction and spectral stability analysis of stationary multiple front and pulse solutions to general semilinear evolution problems on the real line with spatially periodic coefficients. Starting from a collection of $N$ nondegenerate primary front solutions with matching periodic end states, we realize multifront solutions near a formal concatenation of these $N$ primary fronts, provided the distances between the front interfaces is sufficiently large. Moreover, we prove that nondegenerate primary pulses are accompanied by periodic pulse solutions of large spatial period. We show that spectral (in)stability properties of the underlying primary fronts or pulses are inherited by the bifurcating multifronts or periodic pulse solutions. The existence and spectral analyses rely on contraction-mapping arguments and Evans-function techniques, leveraging exponential dichotomies to characterize invertibility and Fredholm properties. To demonstrate the applicability of our methods, we analyze the existence and stability of multifronts and periodic pulse solutions in some benchmark models, such as the Gross-Pitaevskii equation with periodic potential and a Klausmeier reaction-diffusion-advection system, thereby identifying novel classes of (stable) solutions. In particular, our methods yield the first spectral and orbital stability result of periodic waves in the Gross-Pitaevskii equation with periodic potential, as well as new instability criteria for multipulse solutions to this equation.


\end{abstract}


\section{Introduction} \label{sec:intro}

Let $k,m \in \N$, $T > 0$, and $\F \in \{\R,\C\}$. This paper focuses on stationary front and pulse solutions in general semilinear evolution systems on the real line of the form
\begin{align} \label{e:sys_intro}
\begin{split}
\partial_t u &= \alpha_k(x) \partial_x^k u + \ldots + \alpha_1(x) \partial_x u + \Non(u,x), \qquad u(x,t) \in \F^m, \, x \in \R, \, t \geq 0,\end{split}
\end{align}
with continuous $T$-periodic coefficient functions $\alpha_i \colon \R \to \F^{m \times m}$, and continuous nonlinearity $\Non \colon \F^m \times \R \to \F^m$ which is $C^2$ in its first argument and $T$-periodic in its second argument. We further assume that $\alpha_k(x)$ is invertible for each $x \in \R$. 

Spatially periodic systems of the form~\eqref{e:sys_intro} arise in a wide range of contexts. For instance, they appear as a mean-field approximations in the study of Bose-Einstein condensates in periodic trapping lattices~\cite{Pelinovsky2011}, as models in nonlinear optics with periodic potential~\cite{Bengel_Pinning_2024} or periodic forcing~\cite{Gasmi_Bandwith_2023}, as hydrodynamic bifurcation problems over oscillating domains~\cite{Eckhaus_1997,Schielen_1998}, as ecological models for the dynamics of vegetation patterns on periodic topographies~\cite{Bastiaansen2020}, as equations describing the vertical infiltration of water through periodically layered unsaturated soils~\cite{Fennemore_1998}, or as reaction-diffusion systems in population biology in which the environment consists of favorable and unfavorable patches that are arranged periodically~\cite{Kinezaki_2003,Shigesada_1986}. The solutions of interest in these problems are typically pulse or front solutions: stationary gap solitons in Bose-Einstein condensates~\cite{Pelinovsky2011}, optical signals composed of (a sequence of) standing, highly localized pulses in~\cite{Bengel_Pinning_2024,Gasmi_Bandwith_2023}, stationary patterns consisting of localized patches of vegetation in~\cite{Bastiaansen2020}, so-called wetting fronts describing water infiltration in~\cite{Fennemore_1998}, and invasion fronts mediating population spreading in~\cite{Kinezaki_2003,Shigesada_1986}.

Due to the significance to applications, mathematical studies on the existence, stability, and dynamics of front and pulse solutions have been extensively conducted across a wide variety of spatially periodic systems, see, for instance, the survey paper~\cite{Xin_Front_2000}, the memoirs~\cite{Zelik_2009}, and references therein.
We note that \emph{traveling} front and pulse solutions in such systems are typically \emph{modulated}, i.e., they are time-periodic in the co-moving frame, see Remark~\ref{rem:TW}. 

In this paper, we focus on \emph{stationary} front and pulse solutions to general spatially periodic systems of the form~\eqref{e:sys_intro}. We show that any collection of nondegenerate front (or pulse) solutions with matching asymptotic end states is accompanied by a family of multiple front (or pulse) solutions arising through concatenation or periodic extension. Moreover, we prove that their spectral (in)stability properties are determined by the comprising primary front (or pulse) solutions. 

While stationary multifront and multipulse solutions have been studied in specific model problems, such as the Allen-Cahn equation with spatial  inhomogeneity~\cite{bastiaansen2025multifront} and the Gross-Pitaevskii equation with periodic potential~\cite{Ackerman2019Unstable,Alfimov2013,Pelinovsky2011}, a systematic framework for their construction and spectral stability analysis in general spatially periodic systems of the form~\eqref{e:sys_intro} appears to be novel.\footnote{We note that the existence and spectral analysis of multifronts in~\cite{bastiaansen2025multifront} depend on the smallness rather than the periodicity of the inhomogeneity, in contrast to our results and those in~\cite{Ackerman2019Unstable,Alfimov2013,Pelinovsky2011}.} Furthermore, the spectral stability of large-wavelength \emph{periodic} pulse solutions accompanying a primary pulse has, to the authors' best knowledge, not yet been rigorously addressed in any spatially periodic system prior to this paper.

\begin{Remark} \label{rem:TW}
The existence problem for \emph{traveling-wave solutions} of the form $u(x,t) = \phi(x-ct)$ to the constant-coefficient system~\eqref{e:sys_intro_cc} with wavespeed $c \in \R$ is the same as~\eqref{e:intro_stat} with the sole difference that the coefficient $\alpha_1$ is replaced by $\alpha_1 + c$. However, this no longer holds when the coefficients are spatially periodic. In a frame $\xi = x-ct$ moving with speed $c \in \R$, the coefficients of~\eqref{e:sys_intro} read $\alpha_j(\xi + ct)$. That is, they are periodic in time and space. Consequently, traveling wave solutions to~\eqref{e:sys_intro} of the form $u(x,t) = \phi(x-ct)$ propagating with nonzero speed $c$ while maintaining a fixed shape, cannot generally be expected. Instead, one typically finds modulated traveling waves of the form $u(x,t) = \phi(x-ct,x)$, where $\phi$ is $T$-periodic in its second component, cf.~\cite{Ding2015,Giannoulis_Interaction_2008,Gurevich_2018,Pelinovsky_Moving_2008,Uecker_Stable_2001}.
\end{Remark}

\subsection{The case of constant coefficients}

The construction of multiple front (or pulse) solutions through concatenation or periodic extension is well-documented in general semilinear problems with constant coefficients. If the coefficients $\alpha_i$ and nonlinearity $\Non(u)$ do not depend on $x$, then system~\eqref{e:sys_intro} reads
\begin{align} \label{e:sys_intro_cc}
\begin{split}
\partial_t u &= \alpha_k \partial_x^k u + \ldots + \alpha_1 \partial_x u + \Non(u). \end{split}
\end{align}
Stationary solutions of~\eqref{e:sys_intro_cc} obey the autonomous ordinary differential equation
\begin{align} \label{e:intro_stat}
\alpha_k \partial_x^k u + \ldots + \alpha_1 \partial_x u + \Non(u) = 0,
\end{align}
which can be written as a dynamical system $U' = F(U)$ in $U = (u,\partial_x u,\ldots,\partial_x^{k-1} u)$. Pulse and front solutions can be identified with homoclinic and heteroclinic connections in $U' = F(U)$. The existence of periodic, homoclinic or heteroclinic orbits near a nondegenerate homoclinic connection or a heteroclinic chain follows from homoclinic or heteroclinic bifurcation theory, relying on techniques such as Lin's method, Shil'nikov variables, and homoclinic center manifolds. The bifurcating orbits correspond to periodic pulse solutions with large spatial periods, as well as multifront (or multipulse) solutions with well-separated interfaces. An overview of homoclinic and heteroclinic bifurcation theory can be found in the survey paper~\cite{Homburg2010Homoclinic}. 

Spectral stability of the bifurcating periodic pulse and multifront solutions to~\eqref{e:sys_intro_cc} has been studied in~\cite{Alexander_Topological_1990, Gardner1997Spectral,Sandstede_Stability_1998,Sandstede_Stability_2001} using Lin's method and Evans-function techniques. One finds that there are precisely $M$ eigenvalues of the linearization about an $M$-pulse solution bifurcating from each simple isolated eigenvalue of the underlying primary pulse~\cite{Alexander_Topological_1990,Sandstede_Stability_1998}. Moreover, periodic pulse solutions have continua of eigenvalues in a neighborhood of each isolated eigenvalue of the primary pulse~\cite{Gardner1997Spectral}. Since~\eqref{e:sys_intro_cc} is translational invariant, $0$ must be an eigenvalue of each primary pulse or front. Hence, there are $M$ eigenvalues of the linearization of~\eqref{e:sys_intro_cc} about an $M$-front or $M$-pulse converging to the $0$ as the distance between interfaces tends to infinity. In addition, the linearization of~\eqref{e:sys_intro_cc} about a periodic pulse solution, posed on a space of localized perturbations, features a spectral curve converging to $0$ as the period tends to infinity. Leading-order control on the spectrum in a neighborhood of the origin, established in~\cite{Sandstede_Stability_1998,Sandstede_Stability_2001}, shows that the bifurcating periodic pulse and multifront solutions to~\eqref{e:sys_intro_cc} can be unstable, even if all the underlying primary front or pulse solutions are spectrally stable.

\subsection{Main results}

Our existence and spectral stability analysis of stationary multiple front and pulse solutions in spatially periodic systems differs fundamentally from the constant-coefficient case. On the one hand, there seems to be no natural way to formulate the existence problem as an autonomous dynamical system that would facilitate the application of homoclinic or heteroclinic bifurcation theory. Furthermore, stationary pulse and front solutions to~\eqref{e:sys_intro} generally converge to spatially periodic end states rather than to constant states. As a result, it appears that there is no obvious method for augmenting the eigenvalue problem and transforming it into an autonomous system that would enable the use of geometric dynamical systems techniques for analyzing spectral stability as in~\cite{Alexander_Topological_1990, Gardner1997Spectral}. 

On the other hand, the spatially periodic coefficients of~\eqref{e:sys_intro} break the translational invariance, so that the linearization about a stationary front or pulse solution is generically invertible. Unlike the constant-coefficient case, we can (and do) leverage this property in the existence analysis of periodic pulse and multifront solutions. If system~\eqref{e:sys_intro} is dissipative, then front and pulse solutions can be strongly spectrally stable, meaning that the spectrum of the linearization about the front or pulse is confined to the open left-half plane. This significantly reduces the complexity of the spectral analysis compared to the constant-coefficient case where an eigenvalue must reside at $0$ due to translational invariance. However, we emphasize that our spectral techniques are also useful if system~\eqref{e:sys_intro} is conservative, which naturally precludes strong spectral stability of solutions. We will illustrate this by establishing spectral stability and instability for periodic pulses and multipulses in the Gross-Pitaevskii equation with periodic potential. 

Our existence result may informally be stated as follows, see also Figure~\ref{fig:intro_pic}.

\begin{Theorem}[Informal existence result] \label{t:informalexistence}
Let $M \in \N$. Let $Z_1(x),\ldots,Z_M(x)$ be $M$ stationary front solutions to system~\eqref{e:sys_intro} converging to $T$-periodic end states $v_{1,\pm}(x), \ldots,v_{M,\pm}(x)$ as $x \to \pm \infty$. Assume that $v_{j,+} = v_{j+1,-}$ for $j = 1,\ldots,M-1$. Take a stationary pulse solution $Z_0(x)$ to~\eqref{e:sys_intro} converging to a $T$-periodic end state $v_0(x)$ as $x \to \pm \infty$. Assume that $Z_j$ is nondegenerate in the sense that the linearization of~\eqref{e:sys_intro} about $Z_j$ is invertible for $j = 0,\ldots,M$. Then, there exists $N \in \N$ such that for all $n \in \N$ with $n \geq N$ the following assertions hold.
\begin{itemize}
    \item[1.] There exists a stationary $M$-front solution $u_n(x)$ of~\eqref{e:sys_intro}, which converges uniformly to the formal concatenation
    \begin{align*}
    w_n(x) = \begin{cases} 
    v_{1,-}(x), & x \leq \tfrac12 nT,\\
    Z_j(x - j nT), & x \in \left[(j-\tfrac12) nT, (j+\tfrac12) nT\right], \qquad j = 1,\ldots,M,\\
    v_{M,+}(x), & x \geq (M+\tfrac12) nT,
    \end{cases}
    \end{align*}
    as $n \to \infty$.
    \item[2.] There exists a stationary $nT$-periodic pulse solution $\tilde{u}_n(x)$ of~\eqref{e:sys_intro}, which converges uniformly to the formal periodic extension $\tilde{w}_n(x)$ given by
    \begin{align*}
    \tilde{w}_n(x) = Z_0(x - jnT), \qquad x \in \left[(j-\tfrac12) nT, (j+\tfrac12) nT\right], \, j \in \Z,
    \end{align*}
    as $n \to \infty$.
\end{itemize}
\end{Theorem}

\begin{figure}[t]
    \centering
    \includegraphics[width=0.99\textwidth]{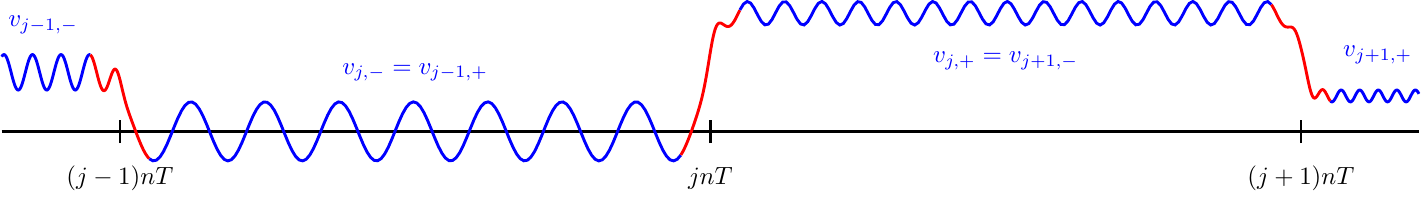}
    \vspace{1em}\\
    \includegraphics[width=0.99\textwidth]{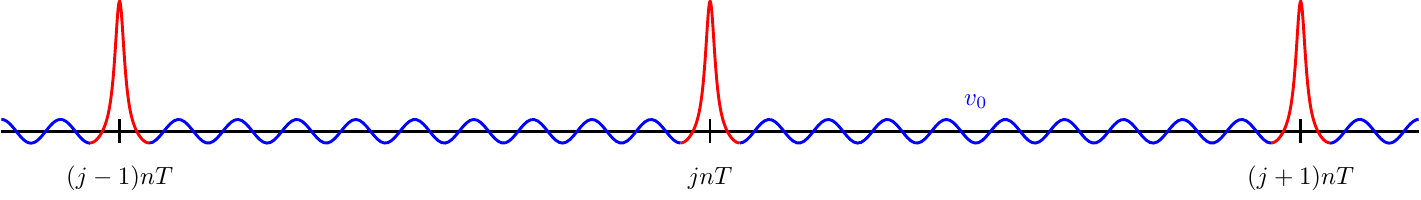}
    \caption{Illustration of the multifront solution $u_n(x)$ (top) and periodic pulse solution $\tilde{u}_n(x)$ (bottom) as established in Theorem~\ref{t:informalexistence}. Both are depicted on the interval $[(j-1)nT,(j+1)nT]$ for some $j \in \{2,\ldots,M-1\}$. The multifront $u_n(x)$ transitions between the $T$-periodic states $v_{\ell,\pm}$ (in blue), with front interfaces (in red) located at $x = \ell nT$, $\ell=1,\dots,M$. The periodic solution $\tilde{u}_n(x)$ consists of a series of localized pulses (in red) centered at $x = jnT$, $j \in \Z$, superimposed on the $T$-periodic background state $v_0$ (in blue).
    }
    \label{fig:intro_pic}
\end{figure}

For the precise statements, we refer to Theorems~\ref{t:existence_multifront} and~\ref{t:existence_periodic}. The proof of the existence of the multifronts and periodic pulse solutions relies on a contraction-mapping argument in the function spaces $H^k(\R)$ and $H^k_{\mathrm{per}}(0,nT)$, respectively. The key idea is to insert the formal multifront $w_n$ or periodic pulse solution $\tilde{w}_n$ into system~\eqref{e:sys_intro} and derive an equation for the resulting error. We then convert the error equation into a fix-point problem by showing that the linearization about the formal solution is invertible. Showing the invertibility is the main technical challenge, which follows from the nondegeneracy of the primary fronts or pulses with the aid of exponential dichotomies. 

We emphasize that this procedure sharply contrasts with the constant-coefficient case, where the lack of invertibility of the linearization about the primary front or pulses necessitates a Lyapunov-Schmidt reduction argument. The reduced problem is typically solved by introducing an additional degree of freedom in the form of a bifurcation parameter or by exploiting additional structure such as reversible symmetry, see~\cite{Homburg2010Homoclinic} and references therein. 

We note that the distances between the interfaces of the multifront solutions in Theorem~\ref{t:informalexistence}, as well as the wavelength of the periodic pulse solutions, are multiples of the period $T$, see Figure~\ref{fig:intro_pic}. In fact, the possible locations of front interfaces and pulse peaks are restricted by the spatial periodicity of~\eqref{e:sys_intro}. This phenomenon, known as \emph{trapping} or \emph{pinning}, cf.~\cite{bastiaansen2025multifront,Dirr_pinning_2006} and references therein, precludes translational invariance and contributes to the enhanced stability properties compared to the constant-coefficient case, where the interfaces are typically free to occupy a continuum of positions. For example, if the constant-coefficient system~\eqref{e:sys_intro_cc} admits a reversible symmetry, then stationary periodic pulse solutions exist for any sufficiently large wavelength~\cite{Vanderbauwhede1992Homoclinic}.

The main outcomes of our spectral analysis may informally be summarized as follows.

\begin{Theorem}[Informal spectral result] \label{t:informalstability}
Let $\mathcal{K} \subset \C$ be compact. Let $M$, $Z_j$, $N$, $u_n$ and $\tilde{u}_n$ be as in Theorem~\ref{t:informalexistence}. Assume that the $L^2$-spectrum in $\mathcal{K}$ of the linearization $\El(Z_j)$ of~\eqref{e:sys_intro} about $Z_j$ consists of isolated eigenvalues of finite algebraic multiplicity only for $j = 0,\ldots,M$. Then, there exists $N_1 \in \N$ with $N \geq N_1$ such that for all $n \in \N$ with $n \geq N_1$ the following holds.
\begin{itemize}
    \item[1.] The $L^2$-spectrum of the linearization $\El(u_n)$ of~\eqref{e:sys_intro} about $u_n$ in the compact set $\mathcal{K}$ consists of isolated eigenvalues only and converges in Hausdorff distance to the union
    \begin{align*}
    \bigcup_{j = 1}^M \sigma(\El(Z_j)) \cap \mathcal{K}
    \end{align*}
    as $n \to \infty$. The total algebraic multiplicity of the eigenvalues of $\El(u_n)$ in $\mathcal{K}$ equals the sum of the total algebraic multiplicies of the eigenvalues of $\El(Z_1),\ldots,\El(Z_M)$ in $\mathcal{K}$.
    \item[2.] The spectrum in $\mathcal{K}$ of the linearization $\El_\per(\tilde{u}_n)$ of~\eqref{e:sys_intro} about $\tilde{u}_n$ on $L^2_{\mathrm{per}}(0,nT)$ consists of isolated eigenvalues only and converges in Hausdorff distance to
    \begin{align*}
    \sigma(\El(Z_0)) \cap \mathcal{K}
    \end{align*}
    as $n \to \infty$. The total algebraic multiplicity of the eigenvalues of $\El_\per(\tilde{u}_n)$ in $\mathcal{K}$ equals  the total algebraic multiplicity of the eigenvalues of $\El(Z_0)$ in $\mathcal{K}$.    
\end{itemize}
\end{Theorem}

For the precise statements and further extensions of Theorem~\ref{t:informalstability}, we refer to Theorems~\ref{thm:instability_multifront} and~\ref{thm:instability_periodic} and Corollaries~\ref{cor:stability_multifront},~\ref{cor:stability_multifront_extension} and~\ref{cor:stability_periodic}. 

The spectral analysis in this paper is divided into two parts. In the first part, we follow an approach similar to the one used in the existence analysis. Specifically, we utilize exponential dichotomies to characterize invertibility and demonstrate that, if the resolvent problem associated with the linearization about the primary fronts or pulses is uniquely solvable on a compact set in $\C$, then the same holds for the resolvent problem corresponding to the linearization about the multifront or periodic pulse solution. 

In the second part, we identify eigenvalues with zeros of the analytic Evans function, see~\cite{KapitulaPromislow2013, Sandstede2002Stability} and references therein, to show that isolated eigenvalues of finite algebraic multiplicity of the linearization about the primary front or pulse perturb continuously into eigenvalues of the linearization about the multifront or periodic pulse posed on $L^2(\R)$ or $L^2_{\mathrm{per}}(0,nT)$, respectively, thereby preserving the total algebraic multiplicity. 

As mentioned earlier, similar results have been obtained for the constant-coefficient case, cf.~\cite{Alexander_Topological_1990,Gardner1997Spectral}. Unlike our existence analysis, which fails in the constant-coefficient setting due to the nondegeneracy condition, our spectral analysis does apply to multifronts and periodic pulse solutions to constant-coefficient systems. Yet, our method differs significantly from the geometric dynamical systems approach employed in~\cite{Alexander_Topological_1990, Gardner1997Spectral}. Instead, it is inspired by the Evans-function analyses in~\cite{RijkDoelmanRademacher2016,Sandstede2000Absolute}. It employs exponential weights and relies on roughness and analyticity properties of exponential dichotomies.

Theorem~\ref{t:informalstability} shows that, if a point $\lambda \in \C$ with $\Re(\lambda) > 0$ is an eigenvalue of finite algebraic multiplicity of the linearizations about some of the primary fronts or pulses and lies in the resolvent set of the linearizations about the other primary fronts or pulses, then bifurcating multifronts and periodic pulse solutions are spectrally unstable. However, if spectral instability of the primary fronts is induced by unstable \emph{essential spectrum}, then the associated multifront may still be spectrally stable. This phenomenon is well-documented in systems with constant coefficients~\cite{Romeo_Stability_2000,Sandstede_2000_gluing}. In~\S\ref{sec:stability_multifront}, we present an extension of Theorem~\ref{t:informalstability}, providing control on the spectrum of the multifront outside the so-called \emph{absolute spectrum}, cf.~\cite{Sandstede2000Absolute,Sandstede_2000_gluing}. This result can be employed to establish strong spectral stability of multifronts, even when the constituting primary fronts are spectrally unstable.

Many dissipative systems admit a-priori bounds that preclude spectrum with nonnegative real part and large modulus. By combining such a-priori bounds with Theorem~\ref{t:informalstability}, one finds that strong spectral stability of the primary fronts or pulses is carried over to the bifurcating multifronts and periodic pulse solutions. This contrasts sharply with the constant-coefficient case where multifronts or periodic pulse solutions can be spectrally unstable, even if the constituting primary fronts are all spectrally stable, cf.~\cite{Sandstede_Stability_1998,Sandstede_Stability_2001}

\subsection{Application to benchmark models}

To illustrate the applicability of our methods, we construct multifronts and periodic pulse solutions in several prototypical models, analyze their spectral stability, and corroborate our findings with numerical simulations performed with the MATLAB package \texttt{pde2path}~\cite{pde2path}. Specifically, we examine multifronts in a scalar reaction-diffusion toy model with a periodic potential and consider multipulses and periodic pulses in an extended Klausmeier model, which describes the dynamics of vegetation patterns on periodic topographies~\cite{Bastiaansen2020}. We demonstrate that the spectral (in)stability of the multifronts, multipulses and periodic pulses is inherited from the comprising primary fronts and pulses. In particular, our analysis shows that the Klausmeier model supports stable periodic (multi)pulses, which has not been identified in the previous work~\cite{Bastiaansen2020}. 

Additionally, we consider the Gross-Pitaevskii equation with a general periodic potential, which arises in the study of Bose-Einstein condensates in optical lattices~\cite{Pelinovsky2011}. Our methods lead to multifront, multipulse and periodic pulse solutions. By combining our spectral results with Krein index counting theory~\cite{KapitulaPromislow2013,KapitulaKevrekidis2004, AddendumKapitulaKevrekidis2004}, we obtain novel spectral instability and stability results for the constructed multipulse and periodic pulse solutions. Due to the conservative nature of the Gross-Piteavskii equation, spectral stability entails that the spectrum of the linearization is confined to the imaginary axis. Notably, the preservation of algebraic multiplicities, as stated in Theorem~\ref{t:informalstability}, is instrumental for the effective application of Krein index counting theory. The spectral stability analysis of the periodic pulse solutions yields that they are orbitally stable. To the best of the authors' knowledge, this is the first orbital stability result of periodic waves in the Gross-Pitaevskii equation with periodic potential. 

\paragraph*{Outline of paper.} In~\S\ref{sec:Abstract_ex_stab} we introduce the necessary notation and formulate the existence and (weighted) eigenvalue problems associated with stationary solutions to~\eqref{e:sys_intro}. The existence analysis of multifronts and periodic pulse solutions is presented in~\S\ref{sec:m-front} and~\S\ref{sec:periodic_pulse}, respectively. In~\S\ref{sec:stability_1fronts} we introduce the necessary concepts for the spectral analysis of fronts solutions with periodic tails. Sections~\ref{sec:stability_multifront} and~\S\ref{sec:stability_periodic} are devoted to the spectral analysis of multifronts and periodic pulse solutions, respectively. We demonstrate the applicability of our methods in several benchmark models in~\S\ref{sec:applications} and corroborate our findings with numerical simulations. Finally, the Appendices~\ref{app:projections},~\ref{app:exp_dich} and~\ref{app:compact} contain several auxiliary results on projections, exponential dichotomies, and multiplication operators, respectively. 

\paragraph*{Acknowledgments.}  This project is funded by the Deutsche Forschungsgemeinschaft (DFG, German Research Foundation) -- Project-ID 258734477 -- SFB 1173.

\section{Notation and set-up}\label{sec:Abstract_ex_stab}

Let $k,m \in \N$, $T>0$ and $\F \in \{\R,\C\}$. This paper focuses on the existence and spectral analysis of stationary front and pulse solutions to the general semilinear evolution system~\eqref{e:sys_intro} of $m$ components on the real line. For notational convenience, we abbreviate the $k$-th order linear differential operator in~\eqref{e:sys_intro} as
\begin{align*}
A u = \alpha_k(x) \partial_x^k u + \alpha_{k-1}(x) \partial_x^k u + \ldots + \alpha_1(x) \partial_x u.
\end{align*}
We recall that $A$ has $T$-periodic coefficients $\alpha_1,\ldots,\alpha_k \in C(\R,\F^{m\times m})$, where $\alpha_k(x)$ invertible for all $x \in \R$. Moreover, the nonlinearity $\Non \in C\big(\F^m \times \R,\F^m\big)$ in~\eqref{e:sys_intro} is twice continuously differentiable in its first argument and $T$-periodic in its second argument.

\subsection{Formulation of the existence and eigenvalue problems}

Stationary solutions to~\eqref{e:sys_intro} obey
\begin{align} \label{existence_problem}
Au + \Non(u,\cdot) = 0.
\end{align}
The ordinary differential equation~\eqref{existence_problem} is the main object of study in the existence analysis of stationary front and pulse solutions to~\eqref{e:sys_intro}. 

For $\unu \in L^\infty(\R)$ we define the linear differential operator $\El(\unu) \colon D(\El(\unu)) \subset L^2(\R) \to L^2(\R)$ by
\begin{align*}
\El(\unu) u = A u + \partial_u \Non(\unu,\cdot) u.
\end{align*}
Clearly, $\El(\unu)$ is a closed operator and has dense domain $D(\El(\unu)) = H^k(\R)$. If $\unu$ is a solution of~\eqref{existence_problem}, then $\El(\unu)$ corresponds to the linearization of~\eqref{e:sys_intro} about $\unu$. The associated eigenvalue problem reads
\begin{align} \label{spectral_problem}
\left(\El(\unu) - \lambda\right) u = 0.
\end{align}
The linear ordinary differential equation~\eqref{spectral_problem} is the main object of study in the spectral analysis of the stationary front and pulse solutions to~\eqref{e:sys_intro}. We adopt the following notions of nondegeneracy and spectral stability.

\begin{Definition}\label{def:spectral_stability}
Let $\unu \in L^\infty(\R)$.
\begin{itemize}
    \item[(i)] We call $\unu$ \emph{nondegenerate} if $\El(\unu)$ is invertible.
    \item[(ii)] We say that $\unu$ is \emph{spectrally stable} if
    $$
        \sigma(\El(\unu)) \subset \{\lambda \in \C : \Re(\lambda)\leq 0\}.   
    $$
    \item[(iii)] We say that $\unu$ is \emph{spectrally stable with simple eigenvalue $\lambda = 0$} if there exists $\varrho>0$ such that
    $$
        \sigma(\El(\unu)) \subset \{\lambda \in \C : \Re(\lambda)\leq -\varrho\} \cup \{0\},
    $$
    and the algebraic multiplicity of $\lambda = 0$ is one.
    \item[(iv)]
    We say that $\unu$ is \emph{strongly spectrally stable} if there exists $\varrho>0$ such that
    $$
        \sigma(\El(\unu)) \subset \{\lambda \in \C : \Re(\lambda)\leq -\varrho\}.
    $$
    \item[(v)] We call $\unu$ \emph{spectrally unstable} it there exists $\lambda \in \sigma(\El(\unu))$ with $\Re(\lambda)>0$.
\end{itemize}
\end{Definition}

As explained in the introduction, the concept of nondegeneracy plays a key role in the construction of multiple front and pulse solutions. 

Spectrally stable front or pulse solutions with a simple eigenvalue at $\lambda = 0$ arise in systems with translational invariance. Such solutions serve as a basis for bifurcation arguments in the models explored in the application section~\S\ref{sec:applications}. 

\subsection{First-order formulation} Because the coefficient matrix $\alpha_k(x)$ is invertible for all $x \in \R$ and the nonlinearity $\Non$ is twice continuously differentiable in its first argument, the eigenvalue problem~\eqref{spectral_problem} can be written as a first-order system
\begin{align} \label{first-order_formulation}
U' = \mathcal{A}(x,\unu(x);\lambda) U,
\end{align}
by setting $U = (u,\partial_x u, \ldots, \partial_x^{k-1} u)^\top$, where the coefficient matrix $\A \colon \R \times \F^m \times \C \to \C^{km}$ is continuous and $T$-periodic in its first argument, continuously differentiable in its second argument, and analytic in its third argument. The formulation~\eqref{first-order_formulation} of the eigenvalue problem as a linear nonautonomous first-order system is essential for applying the theory of exponential dichotomies.

\subsection{Exponentially weighted linearization operator}

Let $\unu \in L^\infty(\R)$ and $\eta_\pm \in \R$. For the spectral analysis of stationary front solutions to~\eqref{e:sys_intro}, it is convenient to consider the exponentially weighted linearization operator $\El_{\eta_-,\eta_+}(\unu) \colon D(\El_{\eta_-,\eta_+}(\unu)) \subset L^2(\R) \to L^2(\R)$ with dense domain $D(\El_{\eta_-,\eta_+}(\unu)) = H^k(\R)$ given by
\begin{align*}
\El_{\eta_-,\eta_+}(\unu) u = \eu^{-\omega_{\eta_-,\eta_+}} \El(\unu) \left[\eu^{\omega_{\eta_-,\eta_+}} u\right]
\end{align*}
where $\omega_{\eta_-,\eta_+} \colon \R \to \R$ is a smooth weight function whose derivative satisfies
\begin{align*}
    \omega_{\eta_-, \eta_+}'(x) = \begin{cases}
    \eta_-, &x \leq -1, \\
    \eta_+, &x \geq 1. 
    \end{cases}
\end{align*}
The associated eigenvalue problem
\begin{align*}
\left(\El_{\eta_-,\eta_+}(\unu) - \lambda\right) u = 0
\end{align*}
can be written as the first-order system
\begin{align*}
U' = \left(\mathcal{A}(x,\unu(x);\lambda) - \omega_{\eta_-,\eta_+}'(x)\right) U.
\end{align*}
In case $\eta_+ = \eta_- = \eta \in \R$, we take $\omega_{\eta_-,\eta_+}(x) = \eta x$ and adopt the notation $\El_{\eta_-,\eta_+}(\unu) = \El_{\eta}(\unu)$. Consequently, it holds $\El_{0,0}(\unu) = \El_0(\unu) = \El(\unu)$.

\subsection{Periodic differential operators} \label{sec:periodic_diff_operators}

Let $n \in \N$ and let $\unu \in C(\R)$ be an $nT$-periodic function. Then, $\El(\unu)$ is a differential operator with $nT$-periodic coefficients. We collect some basic properties of periodic differential operators and their Bloch transforms, which are essential for our spectral analysis. We refer to~\cite{Gardner93,KapitulaPromislow2013,ReedSimon_1978,Scarpellini_1999} for further background. 

Since $\El(\unu)$ has $nT$-periodic coefficients, we can study its action on the space $L^2_\per(0,nT)$, which is convenient for analyzing the spectral stability of $\unu$ against co-periodic perturbations. Thus, we define the operator $\El_{\per}(\unu) \colon D(\El_{\per}(\unu)) \subset L_\per^2(0,nT) \to L_\per^2(0,nT)$ by
\begin{align*}
\El_{\per}(\unu) u = A u + \partial_u \Non(\unu,\cdot)u.
\end{align*}
The operator $\El_{\per}(\unu)$ is closed and has dense domain $D(\El_{\per}(\unu)) = H_\per^k(0,nT)$. Due to the compact embedding $H^k_{\per}(0,nT) \hookrightarrow L^2_{\per}(0,nT)$, it has compact resolvent and its spectrum consists of isolated eigenvalues of finite algebraic multiplicity only. Hence, a point $\lambda \in \C$ lies in the spectrum $\sigma(\El_\per(\unu))$ if and only if the first-order eigenvalue problem~\eqref{first-order_formulation} admits a nontrivial $nT$-periodic solution.

On the other hand, the spectrum of the $nT$-periodic differential operator $\El(\unu)$ on $L^2(\R)$ is purely essential. It is characterized by family of Bloch operators $\El_{\xi,\per}(\unu) \colon D(\El_{\xi,\per}(\unu)) \subset L^2_\per(0,nT) \to L^2_\per(0,nT)$ with dense domain $D(\El_{\xi,\per}(\unu)) = H^k_\per(0,nT)$ given by
\begin{align*}
\El_{\xi,\per}(\unu) u = M_\xi^{-1} \El(\unu) M_\xi u, \qquad \xi \in \left[\frac{\pi}{nT},\frac{\pi}{nT}\right),
\end{align*}
where $M_\xi \colon H^\ell(\R) \to H^\ell(\R)$ is the invertible multiplication operator defined by $(M_\xi u)(x) = \eu^{\iu \xi x} u(x)$ for $\ell \in \N_0$. Since the Bloch operators $\El_{\xi,\per}(\unu)$ have compact resolvent, their spectrum consists of isolated eigenvalues of finite algebraic multiplicity only. Therefore, a point $\lambda \in \C$ lies in the spectrum of $\El_{\xi,\per}(\unu)$ if and only if the first-order system~\eqref{first-order_formulation} admits a nontrivial solution $U(x)$ obeying the boundary condition $U(-\frac{nT}{2}) = \eu^{\iu \xi nT} U(\frac{nT}{2})$. The spectrum of $\El(\unu)$ is then given by the union
\begin{align*}
    \sigma(\El(\unu)) = \bigcup_{\xi \in \left[\frac{\pi}{nT},\frac{\pi}{nT}\right)} \sigma(\El_{\xi,\per}(u_n)),
\end{align*}
which implies
\begin{align}
\label{Bloch_inclusion} \sigma(\El_\per(\unu)) \subset \sigma(\El(\unu)).
\end{align}
Hence, $\lambda \in \C$ lies in the spectrum of $\El(\unu)$ if and only if there exist a point $\gamma$ on the unit circle $S^1 \subset \C$ and a nontrivial solution $U(x)$ of~\eqref{first-order_formulation} obeying $U(-\frac{nT}{2}) = \gamma U(\frac{nT}{2})$. 

\section{Existence of multifront solutions}\label{sec:m-front}

Let $M \in \mathbb{N}_{\geq 2}$. In this section, we construct an $M$-front solution to~\eqref{existence_problem} by concatenating $M$ nondegenerate front solutions with matching periodic limit states. Specifically, we impose the following assumptions:
\begin{itemize}
    \item[\namedlabel{assH1}{\textbf{(H1)}}] There exist $M$ fronts $Z_1,\ldots,Z_M \in L^\infty(\R)$ with associated end states $v_{1,\pm},\ldots,v_{M,\pm} \in H_\per^k(0,T)$. It holds $\chi_\pm\left(Z_j - v_{j,\pm}\right) \in H^k(\R)$ for $j = 1,\ldots,M$, where $\chi_\pm \colon \R \to [0,1]$ is a smooth partition of unity such that $\chi_+$ is supported on $(-1,\infty)$ and $\chi_-$ is supported on $(-\infty,1)$.
    \item[\namedlabel{assH2}{\textbf{(H2)}}] The matching condition $v_{j,+} = v_{j+1,-}$ holds for $j = 1,\ldots,M-1$. 
    \item[\namedlabel{assH3}{\textbf{(H3)}}] The front $Z_j$ is a nondegenerate solution of~\eqref{existence_problem} for $j = 1,\ldots,M$.
\end{itemize}

We realize the $M$-front close to the formal concatenation of the $M$ primary fronts $Z_1,\ldots,Z_M$, see Figure~\ref{fig:intro_pic}. Thus, our ansatz for the multifront solution to~\eqref{existence_problem} reads
\begin{align} \label{multifront_ansatz}
u_n = a_n + w_n, \qquad w_n := \sum_{j = 1}^M \chi_{j,n}Z_j(\cdot - jnT)
\end{align}
with $\chi_{j,n} \colon \R \to [0,1]$, $j = 1,\ldots,M$ a smooth partition of unity satisfying $\|\chi_{j,n}\|_{W^{k,\infty}} \leq 1$,
where $\chi_{1,n}$ is supported on $(-\infty,\frac{3}{2}nT+1)$, $\chi_{M,n}$ is supported on $((M-\frac{1}{2}) nT - 1,\infty)$, and $\chi_{j,n}$ is supported on $((j-\frac{1}{2})nT-1,(j+\frac12) nT + 1)$ for $j = 2,\ldots,M-1$ (only in case $M > 2$). Moreover, $a_n \in H^k(\R)$ is an error term that accounts for the fact that the formal concatenation $w_n$ of the $M$ fronts is not an actual solution to~\eqref{existence_problem}. Our main result of this section confirms that there exists a small $a_n \in H^k(\R)$ such that $u_n$ is indeed a solution to~\eqref{existence_problem}, provided $n \in \N$ is sufficiently large.

\begin{Theorem} \label{t:existence_multifront}
Assume~{\upshape \ref{assH1}},~{\upshape \ref{assH2}} and~{\upshape \ref{assH3}}. Then, there exist $C > 0$ and $N \in \N$ such that for each $n \in \N$ with $n \geq N$ there exists an $M$-front solution to~\eqref{existence_problem} given by~\eqref{multifront_ansatz} with $a_n \in H^k(\R)$ satisfying 
\begin{align*}
\|a_n\|_{H^k} \leq C \sum_{j = 1}^M \Big( \|\chi_-(Z_j-v_{j,-})\|_{H^k(\R \setminus [-\tfrac{n}{2}T+1,\tfrac{n}{2}T-1])} + \|\chi_+(Z_j-v_{j,+})\|_{H^k(\R \setminus [-\tfrac{n}{2}T+1,\tfrac{n}{2}T-1])} \Big).
\end{align*}
In particular, $\|a_n\|_{H^k}$ converges to $0$ as $n \to \infty$.
\end{Theorem}

\begin{Remark}
One could prove a more general result where the front interfaces are located on positions $n_1T,\ldots,n_mT$ as long as the distances $n_{i+1}-n_i$ are sufficiently large for each $i = 1,\ldots,M-1$. More precisely, there exist $C > 0$ and $N \in \N$ such that for each vector $n = (n_1,\ldots,n_m) \in \N^M$  with $\kappa := \min\{n_{i+1} - n_i : i \in \{1,\ldots,M-1\}\} \geq N$, there exists an $M$-front solution
\begin{align*}
    u_n = a_n + \sum_{j = 1}^M \tilde\chi_{j,n} Z_j(\cdot - jn_jT)
\end{align*}
of~\eqref{existence_problem}, where $\tilde\chi_{j,n} \colon \R \to [0,1]$, $j = 1,\ldots,M$ is a smooth partition of unity satisfying $\|\tilde\chi_{j,n}\|_{W^{k,\infty}} \leq 1$ with $\tilde\chi_{1,n}$ supported on $(-\infty,\frac{1}{2}(n_1+n_2)T+1)$, $\tilde\chi_{M,n}$ supported on $(\frac{1}{2} (n_{M-1}+n_M) T - 1,\infty)$, and $\tilde\chi_{j,n}$ supported on $(\frac12 (n_{j-1}+n_j)T-1,\frac12 (n_j + n_{j+1})T + 1)$ for $j = 2,\ldots,M-1$ (only in case $M > 2$). Moreover, the error $a_n \in H^k(\R)$ converges to $0$ as $\kappa \to \infty$, The proof of this statement proceeds along the lines of Theorem~\ref{t:existence_multifront}, but with considerably more involved notation, and is therefore omitted.
\end{Remark} 

The proof of Theorem~\ref{t:existence_multifront} relies on a contraction-mapping argument. Inserting the ansatz~\eqref{multifront_ansatz} into~\eqref{existence_problem}, one arrives at an equation for the error $a_n$, whose linear part is given by $\El(w_n) a_n$. Here, $\El(w_n)$ represents the linearization of~\eqref{e:sys_intro} about the formal concatenation $w_n$ of the $M$ fronts, with $\El(\cdot)$ defined in~\S\ref{sec:Abstract_ex_stab}. To solve for the error, we recast the equation as a fixed-point problem in $H^k(\R)$ by inverting the linear operator $\El(w_n)$.

The invertibility of $\El(w_n)$ is established by transferring the nondegeneracy of the primary fronts $Z_1,\ldots,Z_M$ to the concatenation $w_n$. This is achieved by characterizing invertibility through exponential dichotomies~\cite{MasseraSchaefer1966}. Specifically, $\El(\unu) - \lambda$ can be inverted if and only if the first-order formulation~\eqref{first-order_formulation} of the eigenvalue problem admits an exponential dichotomy on $\R$. By applying pasting and roughness techniques to the exponential dichotomies arising through the nondegeneracy of the individual fronts, we construct an exponential dichotomy on $\R$ for the first-order formulation of the eigenvalue problem associated with $\El(w_n)$, thereby establishing its invertibility.

The invertibility result is formalized in the following lemma, which is stated in a slightly more general form. This generalization also plays a central role in the subsequent spectral analysis of the multifront.

\begin{Lemma} \label{lem:invertibility_multifront}
Assume~{\upshape \ref{assH1}} and~{\upshape \ref{assH2}}. Let $\mathcal{K} \subset \C$ be a compact set.  Moreover, let $\{a_n\}_n$ be a sequence in $H^k(\R)$ with $\|a_n\|_{H^k} \to 0$ as $n \to \infty$. 

Suppose that the linear operator $\El(Z_j) - \lambda$ is invertible for each $\lambda \in \mathcal{K}$ and $j = 1,\ldots,M$. Then, there exist $C > 0$ and $N \in \N$ such that for each $n \in \N$ with $n \geq N$ the resolvent set of the operator
\begin{align*}
\El(w_n + a_n), \qquad w_n := \sum_{j = 1}^M \chi_{j,n} Z_j(\cdot - jnT)
\end{align*}
contains $\mathcal{K}$ and the resolvent obeys the bound
\begin{align} \label{inversebound}
\left\|(\El(w_n+a_n)-\lambda)^{-1}\right\|_{L^2 \to H^k} \leq C
\end{align}
for $\lambda \in \mathcal{K}$.
\end{Lemma}
\begin{proof}
Lemma~\ref{lem:expdi1} yields that the first-order system
\begin{align} \label{variationalsys1}
U' = \A\left(x,Z_j(x);\lambda \right) U
\end{align}
admits an exponential dichotomy on $\R$ for each $\lambda \in \mathcal{K}$ and for $j = 1,\ldots,M$. By continuity of $\A$ and roughness of exponential dichotomies, cf.~\cite[Proposition~4.1]{Coppel1978}, there exists for each $\lambda_0 \in \mathcal{K}$ an open disk $B_{\lambda_0} \subset \C$ with $\lambda_0 \in B_{\lambda_0}$ and constants $K_{\lambda_0},\mu_{\lambda_0} > 0$ such that~\eqref{variationalsys1} has an exponential dichotomy on $\R$ for each $\lambda \in B_{\lambda_0}$ with constants $K_{\lambda_0},\mu_{\lambda_0} > 0$. By compactness of $\mathcal{K}$ the open cover $\{B_{\lambda_0} : \lambda_0 \in \mathcal{K}\}$ has a finite subcover. It follows that~\eqref{variationalsys1} has for each $\lambda \in \mathcal{K}$ an exponential dichotomy on $\R$ with $\lambda$-independent constants.

Clearly, system
\begin{align} \label{variationalsys2}
U' = \A\left(x,Z_j(x - j nT) ;\lambda \right) U
\end{align}
is for each $n \in \N$ and $j = 1,\ldots,M$ a $jnT$-translation of system~\eqref{variationalsys1}. So,~\eqref{variationalsys2} possesses for each $n \in \N$, $j = 1,\ldots,M$ and $\lambda \in \mathcal{K}$ an exponential dichotomy on $\R$ with $\lambda$- and $n$-independent constants.

We use roughness techniques to transfer the exponential dichotomy of system~\eqref{variationalsys2} to an exponential dichotomy of system
\begin{align} \label{variationalsysfull}
    U' = \A(x,w_n(x) + a_n(x);\lambda) U
\end{align}
on an interval $I_j$ for $j = 1,\ldots,M$, where we denote $I_1 = (-\infty,\frac{5}{3} nT]$, $I_j = [(j-\frac{2}{3}) nT, (j+\frac{2}{3}) nT]$ for $j = 2,\ldots,M-1$ (only in case $M > 2$), and $I_M = [(M-\frac{2}{3}) nT,\infty)$. Since $\partial_u \A$ is continuous, $\mathcal{K}$ is compact and we have $Z_1,\ldots,Z_M \in L^\infty(\R)$ and $a_n \in H^k(\R) \hookrightarrow L^\infty(\R)$, we obtain by the mean value theorem a $\lambda$- and $n$-independent constant $K_0 > 0$ such that
\begin{align} \label{existest1}
\begin{split}
    &\left\|\A(x,w_n(x) + a_n(x);\lambda) - \A\left(x,Z_j(x-jnT);\lambda \right)\right\|\\ 
    &\qquad \leq K_0\Bigg(\|a_n\|_{L^\infty} + \sum_{\ell \in \{1,\ldots,M\}} \Big(\|\chi_-\left(Z_\ell-v_{\ell,-}\right)\|_{L^\infty(\R \setminus (-\frac{1}{3}nT,\frac{1}{3}nT))} \\
    & \qquad\qquad\qquad
    + \|\chi_+\left(Z_\ell-v_{\ell,+}\right)\|_{L^\infty(\R \setminus (-\frac{1}{3}nT,\frac{1}{3}nT))}\Big)\Bigg)
\end{split}
\end{align}
for $x \in I_j$, $n \in \N$, $\lambda \in \mathcal{K}$ and $j = 1,\ldots,M$.

It is readily seen by approximation with simple functions that for each $g \in L^2(\R)$ it holds $\|g\|_{L^2(\R \setminus [-R,R])} \to 0$ as $R \to \infty$. Thus, $\|\chi_\pm(Z_j-v_{j,\pm})\|_{H^1(\R \setminus [-R,R])}$ converges to $0$ as $R \to \infty$ for $j = 1,\ldots,M$. Hence, noting that $H^1(I)$ continuously embeds into $L^\infty(I)$ for each interval $I \subset \R$, the right-hand side of the estimate~\eqref{existest1} converges to $0$ uniformly on $I_j$ as $n \to \infty$ for $j = 1,\ldots,M$. Therefore, using that~\eqref{variationalsys2} has an exponential dichotomy on $\R$ with $\lambda$- and $n$-independent constants, we establish, provided $n \in \N$ is sufficiently large, by~\cite[Proposition~4.1]{Coppel1978} an exponential dichotomy for~\eqref{variationalsysfull} on $I_j$ with $\lambda$-, $j$- and $n$-independent constants $K_1,\mu_1 > 0$ and projections $P_{j,n}(x;\lambda)$ for each $\lambda \in \mathcal{K}$ and $j = 1,\ldots,M$.

For $j = 1,\ldots,M-1$ we iteratively paste the exponential dichotomies for~\eqref{variationalsysfull} on the intervals $(-\infty,(j+\frac{2}{3}) nT]$ and $I_{j+1}$ together at the point $x = (j + \frac12) n T$ to obtain an exponential dichotomy on $\R$. Let $j \in \{1,\ldots,M-1\}$. Given an exponential dichotomy for~\eqref{variationalsysfull} on $(-\infty,(j+\frac{2}{3}) nT]$ with $\lambda$- and $n$-independent constants $K_j,\mu_j > 0$ and projections $Q_{j,n}(x;\lambda)$, we employ Lemma~\ref{l:projest} and arrive, provided $n \in \N$ is sufficiently large, at
\begin{align*}
\left\|Q_{j,n}\left((j+\tfrac12) n T;\lambda\right) - P_{j+1,n}\left((j+\tfrac12) n T;\lambda\right)\right\| \leq 2 K_1 K_j \eu^{-(\mu_1 + \mu_j) \frac{n}{6} T} < 1,
\end{align*}
for each $\lambda \in \mathcal{K}$. Hence, the subspaces $\ker(Q_{j,n}((j+\tfrac12) n T;\lambda))$ and $\mathrm{ran}(P_{j+1,n}((j+\tfrac12) n T;\lambda))$ are complementary by Lemma~\ref{l:projest2} and the associated projection $P_{\circ,n}(\lambda)$ onto $\mathrm{ran}(P_{j+1,n}((j+\tfrac12) n T;\lambda))$ along $\ker(Q_{j,n}((j+\tfrac12) n T;\lambda))$ is well-defined and satisfies
\begin{align*} \left\|P_{\circ,n}(\lambda)\right\| \leq \frac{K_j}{1 - 2 K_1K_j \eu^{-(\mu_1 + \mu_j) \frac{n}{6} T}},\end{align*}
for each $\lambda \in \mathcal{K}$. Hence, provided $n \in \N$ is sufficiently large, Lemma~\ref{l:pastingexpdi} yields an exponential dichotomy for system~\eqref{variationalsysfull} on $(-\infty,(j+\frac{2}{3}) nT] \cup I_{j+1}$ with $\lambda$- and $n$-independent constants for each $\lambda \in \mathcal{K}$. Thus, iteratively repeating the above procedure for $j = 1,\ldots,M-1$, we establish, provided $n \in \N$ is sufficiently large, an exponential dichotomy of~\eqref{variationalsysfull} on $(-\infty,(M-\frac{1}{3}) nT] \cup I_{M} = \R$ with $\lambda$- and $n$-independent constants for each $\lambda \in \mathcal{K}$.

Using the compactness of $\mathcal{K}$ and the continuity of $\A$, it follows that $\|\A(x,w_n(x) + a_n(x);\lambda)\|_{L^\infty}$ can be bounded by a $\lambda$- and $n$-independent constant for each $\lambda \in \mathcal{K}$ and $n \in \N$. So, provided $n \in \N$ is sufficiently large, Lemma~\ref{l:inhomexpdi} yields a $\lambda$- and $n$-independent constant $C > 0$ such that for each $g \in H^1(\R) \hookrightarrow C(\R)$ and $\lambda \in \mathcal{K}$ the inhomogeneous linear problem
\begin{align*}
U' = \A(x,w_n(x) + a_n(x);\lambda)U + \psi
\end{align*}
with inhomogeneity $\psi = (0,\ldots,0,g)^\top \in H^1(\R)$ has a solution $U \in H^1(\R)$ satisfying
\begin{align*}
\|U\|_{H^1} \leq C \|\psi\|_{L^2} = C \|g\|_{L^2}.
\end{align*}
Using that we have $U_i' = U_{i+1} \in H^1(\R)$ for $i = 1,\ldots,k-1$, we readily observe that $u = U_1 \in H^k(\R)$ solves the resolvent problem
\begin{align} \label{resolvent_multifront}
\left(\El(w_n + a_n) - \lambda\right) u = g,
\end{align}
and satisfies
\begin{align} \label{resolvent_multifront_bound}
\|u\|_{H^k} \leq \|U\|_{H^1} \leq C\|g\|_{L^2}.
\end{align}
Since the operator $\El(w_n+a_n)$ is closed, it follows by the density of $H^1(\R)$ in $L^2(\R)$ that, provided $n \in \N$ is sufficiently large, the resolvent problem~\eqref{resolvent_multifront} possesses for each $g \in L^2(\R)$ and $\lambda \in \mathcal{K}$ a solution $u \in H^k(\R)$ satisfying~\eqref{resolvent_multifront_bound}.

Finally, if $u \in H^k(\R)$ lies in the kernel of $\El(w_n+a_n) - \lambda$, then $U = (u,\partial_x u, \ldots,\partial_x^{k-1} U)^\top \in H^1(\R)$ is a localized solution of the first-order variational problem~\eqref{variationalsysfull}. Since~\eqref{variationalsysfull} has an exponential dichotomy on $\R$, $U$ must be the trivial solution and, thus, we find $u = 0$. 

So, we have established that, provided $n \in \N$ is sufficiently large, $\El(w_n+a_n) - \lambda$ is bounded invertible and satisfies~\eqref{inversebound} for each $\lambda \in \mathcal{K}$.
\end{proof}

With the aid of Lemma~\ref{lem:invertibility_multifront}, we now prove Theorem~\ref{t:existence_multifront} using a contraction-mapping argument.

\begin{proof}[Proof of Theorem~\ref{t:existence_multifront}]
First, Lemma~\ref{lem:invertibility_multifront} implies that the linear operator 
$\El\left(w_n\right)$ is invertible and there exists an $n$-independent constant $K > 0$ such that
\begin{align} \label{linest}
\|\El(w_n)^{-1}\|_{L^2 \to H^k} \leq K.
\end{align}

Inserting the ansatz $u = w_n + a$ with correction term $a \in H^k(\R)$ into~\eqref{existence_problem} yields the equation
\begin{align} \label{perteq1}
a = \widetilde{\Non}(a) + R,
\end{align}
where the nonlinear map $\widetilde{\Non} \colon H^k(\R) \to H^k(\R)$ is given by
\begin{align*}
\widetilde{\Non}(a) &= \El(w_n)^{-1} \left( \Non\left(w_n,\cdot\right) + \partial_u \Non\left(w_n,\cdot\right)a - \Non\left(w_n + a,\cdot\right)\right)
\end{align*}
and the residual $R \in H^k(\R)$ is given by
\begin{align*}
R = -\El(w_n)^{-1} \left(\Non(w_n,\cdot) + A(w_n)\right).
\end{align*}

Using the continuous embedding $H^1(\R) \hookrightarrow L^\infty(\R)$ and the fact that the nonlinearity $\Non$ is twice continuously differentiable in its first argument, it follows by Taylor's theorem and estimate~\eqref{linest} that $\smash{\widetilde{\Non}} \colon H^k(\R) \to H^k(\R)$ is well-defined and for all $\rho >0 $ there exists an $n$-independent constant $C_1 > 0$ such that
\begin{align} \label{nonlLipschitz}
\left\|\widetilde{\Non}(a_0) - \widetilde{\Non}(a_1)\right\|_{H^k} \leq C_1\left(\|a_0\|_{H^k} + \|a_1\|_{H^k}\right)\left\|a_0 - a_1\right\|_{H^k}.
\end{align}
for $a_0,a_1 \in H^k(\R)$ with $\|a_0\|_{L^\infty},\|a_1\|_{L^\infty} \leq \rho$.

Next, the fact that $Z_j(\cdot - jnT)$ is a solution of~\eqref{existence_problem} implies that
\begin{align*}
R &= \El(w_n)^{-1}\left(\Non\left(Z_j(\cdot - jnT),\cdot\right) - \Non(w_n,\cdot) - A\left(w_n - Z_j(\cdot-jnT)\right)\right)
\end{align*}
for $j = 1,\ldots,M$. We partition $\R = I_1 \cup \ldots \cup I_M$ with $I_1 = (-\infty,\frac{3}{2} nT]$, $I_j = ((j-\frac{1}{2}) nT, (j+\frac{1}{2}) nT]$ for $j = 2,\ldots,M-1$ (only in case $M > 2$), and $I_M = ((M-\frac{1}{2}) nT,\infty)$. Since the nonlinearity $\Non$ is continuously differentiable in its first argument and it holds $\|\chi_{j,n}\|_{W^{k,\infty}} \leq 1$ for $j = 1,\ldots,M$, the mean value theorem and estimate~\eqref{linest} yield an $n$-independent constant $C_0 > 0$ such that
\begin{align} \label{inhomMbound}
\begin{split}
\|R\|_{H^k} &\leq K \sum_{j = 1}^M \left\|\Non\left(Z_j(\cdot - jnT),\cdot\right) - \Non(w_n,\cdot) - A\left(w_n - Z_j(\cdot-jnT)\right)\right\|_{L^2(I_j)} \leq C_0 \delta_n,
\end{split}
\end{align}
where we denote
\begin{align*}
\delta_n = \sum_{j = 1}^M \Big( \|\chi_-(Z_j-v_{j,-})\|_{H^k(\R \setminus [-\tfrac{n}{2}T+1,\tfrac{n}{2}T-1])} + \|\chi_+(Z_j-v_{j,+})\|_{H^k(\R \setminus [-\tfrac{n}{2}T+1,\tfrac{n}{2}T-1])} \Big).
\end{align*}
We observe that $\delta_n$ converges to $0$ as $n \to \infty$. 

Motivated by the estimate~\eqref{inhomMbound}, we introduce the rescaled variable
\begin{align} \label{rescaling}
\at = \delta_{n}^{-1} a,
\end{align}
in which~\eqref{perteq1} reads
\begin{align} \label{perteq2}
\begin{split}
\at &= \delta_{n}^{-1} \widetilde{\Non}(\delta_n \at) + \delta_{n}^{-1} R.
\end{split}
\end{align}
We regard~\eqref{perteq2} as an abstract fixed point problem
\begin{align} \at = \mathcal{F}_{n}(\at)\label{contrmap}
\end{align}
and show that $\mathcal{F}_{n} \colon B_0(2C_0) \to B_0(2C_0)$ is a well-defined contracting mapping on the ball $B_0(2C_0)$ of radius $2C_0$ in $H^k(\R)$ centered at the origin, where $C_0 > 0$ is the $n$-independent constant appearing in the bound~\eqref{inhomMbound} on $R$. Combining the estimates~\eqref{nonlLipschitz} and~\eqref{inhomMbound} and noting $\smash{\widetilde{\Non}(0)} = 0$, yields an $n$-independent constant $K_0 > 0$ such that
\begin{align*}
\begin{split}
\left\|\mathcal{F}_{n}(g)\right\|_{H^k} &\leq C_0 + K_0 \delta_{n},\\
\left\|\mathcal{F}_{n}(g) - \mathcal{F}_{n}(h)\right\|_{H^2} &\leq K_0 \delta_{n} \|g-h\|_{H^k},
\end{split}
\end{align*}
for $g,h \in B_0(2C_0)$. Therefore, using that $\delta_n \to 0 $ as $n \to \infty$, $\mathcal{F}_{n}$ is a well-defined contraction mapping on $B_0(2C_0)$, provided $n \in \N$ is sufficiently large. By the Banach fixed point theorem there exists a unique solution $\at_n \in B_0(2C_0)$ to~\eqref{contrmap}. Undoing the rescaling~\eqref{rescaling}, we found a solution $a_n \in H^k(\R)$ of equation~\eqref{perteq1} with
\begin{align*} 
\|a_n\|_{H^k} \leq 2C_0 \delta_{n},
\end{align*}
provided $n \in \N$ is sufficiently large, which yields the result.
\end{proof}

\section{Existence of periodic pulse solutions}\label{sec:periodic_pulse}

In this section, we construct a periodic solution to~\eqref{existence_problem} by periodically extending a nondegenerate pulse solution, see Figure~\ref{fig:intro_pic}. Thus, we impose the following assumptions: 

\begin{itemize}
    \item[\namedlabel{assH4}{\textbf{(H4)}}] There exist $v \in H_\per^k(0,T)$ and $z \in H^k(\R)$ such that $v$ and $z + v$ are solutions to~\eqref{existence_problem}.
    \item[\namedlabel{assH5}{\textbf{(H5)}}] The pulse $z+v$ is nondegenerate.
\end{itemize}

The main result of this section reads as follows.

\begin{Theorem} \label{t:existence_periodic}
Assume~{\upshape \ref{assH4}} and~{\upshape \ref{assH5}}. Then, there exist $C > 0$ and $N \in \N$ such that for each $n \in \N$ with $n \geq N$ there exists an $nT$-periodic pulse solution $u_n \in H_\per^k(0,nT)$ of~\eqref{existence_problem} given by
\begin{align} \label{ansatz_periodic}
u_n(x) = \chi_n(x) z(x) + v(x) + a_n(x), \qquad x \in \left[-\tfrac{n}{2} T, \tfrac{n}{2}T\right),
\end{align}
 with $a_n \in H_\per^k(0,nT)$ satisfying 
\begin{align*}
\|a_n\|_{H_\per^k(0,nT)} \leq C \|z\|_{H^k(\R \setminus (-\frac{n}{6}T,\frac{n}{6} T))},
\end{align*}
and where $\chi_n \colon \R \to [0,1]$ is a smooth $nT$-periodic cut-off function satisfying $\|\chi_n\|_{W^{k,\infty}} \leq 1$, $\chi_n(x) = 1$ for $x \in [-\frac{n}{6}T,\frac{n}{6} T]$ and $\chi_n(x) = 0$ for $x \in [-\frac{n}{2} T,-\frac{n}{3} T] \cup [\frac{n}{3} T,\frac{n}{2} T]$. In particular, $\|a_n\|_{H_\per^k(0,nT)}$ converges to $0$ as $n \to \infty$.
\end{Theorem}

We prove Theorem~\ref{t:existence_periodic} using a similar strategy as for Theorem~\ref{t:existence_multifront}. Specifically, we arrive at an equation for the error $a_n$ by substituting the ansatz $u_n = w_n + a_n$ into~\eqref{existence_problem}, where $w_n$ is the $nT$-periodic extension of the primary pulse $\chi_n z + v$ on $[-\frac{n}{2} nT,\frac{n}{2} nT)$. We convert this equation into a fixed-point problem in $H^k_{\mathrm{per}}(0,nT)$ by inverting its linear component, given by the linearization $\El_{\per}(w_n)$ of~\eqref{e:intro_stat} about the formal periodic extension $w_n$, with $\El_{\per}(\cdot)$ defined in~\S\ref{sec:Abstract_ex_stab}. The proof is then completed by applying a contraction-mapping argument.

The invertibility of $\El_{\per}(w_n)$ is ensured by the following lemma, which serves as the periodic counterpart of Lemma~\ref{lem:invertibility_multifront} and also plays an essential role in the subsequent spectral analysis of the periodic pulse solution. Its proof proceeds along the lines of the proof of Lemma~\ref{lem:invertibility_multifront}, now relying on a periodic extension result for exponential dichotomies~\cite{Palmer1987}.

\begin{Lemma} \label{lem:invertibility_periodic}
Let $\mathcal{K} \subset \C$ be a compact. Let $v \in H_\per^k(0,T)$ and $z \in H^k(\R)$. Moreover, let $\{a_n\}_n$ be a sequence with $a_n \in H_\per^k(0,nT)$ satisfying $\|a_n\|_{H_\per^k(0,nT)} \to 0$ as $n \to \infty$. Finally, let $z_n \in H_\per^k(0,nT)$ be the $nT$-periodic extension of the function $\chi_n z$ on $[-\tfrac{n}{2} T, \tfrac{n}{2}T)$, where $\chi_n$ is the cut-off function from Theorem~\ref{t:existence_periodic}. 

Suppose that the linear operator $\El(z + v) - \lambda$ is invertible for each $\lambda \in \mathcal{K}$. Then, there exist $C > 0$ and $N \in \N$ such that for each $\lambda \in \mathcal{K}$ and each $n \in \N$ with $n \geq N$ the operators $\El(z_n + v + a_n) - \lambda$ and $\El_{\per}(z_n + v + a_n) - \lambda$ are invertible with
\begin{align} \label{inversebound2}
\left\|\left(\El(z_n + v + a_n) - \lambda\right)^{-1}\right\|_{L^2 \to H^k} \leq C,
\end{align}
and
\begin{align} \label{inverseboundper}
\left\|\left(\El_{\per}(z_n + v + a_n) - \lambda\right)^{-1}\right\|_{L_\per^2(0,nT) \to H_\per^k(0,nT)} \leq C.
\end{align}
\end{Lemma}
\begin{proof}
As in the proof of Lemma~\ref{lem:invertibility_multifront}, we obtain that for each $\lambda \in \mathcal{K}$ the first-order system
\begin{align} \label{variationalsysper}
U' = \A\left(x,z(x) + v(x);\lambda \right) U
\end{align}
admits an exponential dichotomy on $\R$ with $\lambda$-independent constants and projections $P(x;\lambda)$. So, we find that the $nT$-translated system
\begin{align} \label{variationalsysper2}
U' = \A\left(x,z(x - nT) + v(x);\lambda \right) U
\end{align}
has for each $n \in \N$ and $\lambda \in \mathcal{K}$ also an exponential dichotomy on $\R$ with $\lambda$- and $n$-independent constants.

We apply roughness techniques to carry over the exponential dichotomies of~\eqref{variationalsysper} and~\eqref{variationalsysper2} to the system
\begin{align} \label{variationalsysfullper}
U' = \A\left(x,z_n(x) + v(x) + a_n(x);\lambda \right) U.
\end{align} 
Since $\partial_u \A$ is continuous, $\mathcal{K}$ is compact, it holds $z \in H^1(\R)$ and $v,a_n \in H^1_\per(0,nT)$, and $H^1(\R)$ and $H^1_\per(0,nT)$ embed continuously into $L^\infty(\R)$ with $n$-independent constant, we obtain by the mean value theorem a $\lambda$- and $n$-independent constant $K_0 > 0$ such that
\begin{align} \label{existestper1}
\begin{split}
&\left\|\A\left(x,z_n(x) + v(x) + a_n(x);\lambda \right) - \A\left(x,z(x) + v(x);\lambda \right)\right\| 
\\ &\qquad \leq K_0\begin{cases} \left(\|(1-\chi_n(x)) z(x)\| + \|a_n(x)\|\right), & x \in \left[-\tfrac{n}{2} T,\frac{n}{2} T\right],\\
\left(\|z(x)\| + \|z_n(x)\| + \|a_n(x)\|\right), & x \in \left[\tfrac{n}{2} T,\frac{5n}{6} T\right],\end{cases}
\end{split}
\end{align}
for $x \in [-\tfrac{n}{2} T,\frac{5n}{6} T]$ and $\lambda \in \mathcal{K}$, and
\begin{align} \label{existestper2}
\begin{split}
&\left\|\A\left(x,z_n(x) + v(x) + a_n(x);\lambda \right) - \A\left(x,z(x-nT) + v(x);\lambda \right)\right\| 
\\ &\qquad \leq K_0 \left(\|(1-\chi_n(x-nT)) z(x-nT)\| + \|a_n(x)\|\right)
\end{split}
\end{align}
for $x \in [\tfrac{n}{2} T,\frac{3n}{2} T]$ and $\lambda \in \mathcal{K}$. Since $H^1(I)$ embeds continuously into $L^\infty(I)$ for each interval $I \subset \R$ and $\|z\|_{H^1(\R \setminus [-R,R])}$ converges to $0$ as $R \to \infty$, we find that the right-hand sides of the estimates~\eqref{existestper1} and~\eqref{existestper2} converge to $0$ uniformly on $[-\tfrac{n}{2} T,\frac{5n}{6} T]$ and on $[\tfrac{n}{2} T,\frac{3n}{2} T]$, respectively, as $n \to \infty$. Combing the latter with the fact that~\eqref{variationalsysper} admits an exponential dichotomy on $\R$ with $\lambda$- and $n$-independent constants, we infer thanks to~\cite[Proposition~5.1]{Coppel1978} that, provided $n \in \N$ is sufficiently large,~\eqref{variationalsysfullper} has an exponential dichotomy on $[-\frac{n}{2} T,\frac{5n}{6} T]$ with $\lambda$- and $n$-independent constants $K_1,\mu_1 > 0$ and projections $P_{1,n}(x;\lambda)$ for each $\lambda \in \mathcal{K}$.  Similarly, provided $n \in \N$ is sufficiently large, the exponential dichotomy of~\eqref{variationalsysper2} on $\R$ yields an exponential dichotomy of~\eqref{variationalsysfullper} on $[\frac{n}{2} T,\frac{3n}{2} T]$ with $\lambda$- and $n$-independent constants $K_2,\mu_2 > 0$ and projections $P_{2,n}(x;\lambda)$ for each $\lambda \in \mathcal{K}$.

We glue the exponential dichotomies of~\eqref{variationalsysfullper} on $[-\tfrac{n}{2} T,\frac{5n}{6} T]$ and on $[\tfrac{n}{2} T,\frac{3n}{2} T]$ together at the point $x = \frac{2n}{3} T$. First, Lemma~\ref{l:projest} yields, provided $n \in \N$ is sufficiently large, the bound
\begin{align*}
\left\|P_{1,n}\left(\tfrac{2n}{3} T;\lambda\right) - P_{2,n}\left(\tfrac{2n}{3} T;\lambda\right)\right\| \leq 2 K_1K_2 \eu^{-(\mu_1 + \mu_2) \frac{n}{6} T} < 1,
\end{align*}
for each $\lambda \in \mathcal{K}$. So, the subspaces $\ker(P_{1,n}(\tfrac{2n}{3} T;\lambda))$ and $\mathrm{ran}(P_{2,n}(\tfrac{2n}{3} T;\lambda))$ are complementary by Lemma~\ref{l:projest2}. We infer that the projection $P_{\circ,n}(\lambda)$ onto $\mathrm{ran}(P_{2,n}(\tfrac{2n}{3} T;\lambda))$ along $\ker(P_{1,n}(\tfrac{2n}{3} T;\lambda))$ is well-defined and satisfies
\begin{align*} \left\|P_{\circ,n}(\lambda)\right\| \leq \frac{K_2}{1 - 2 K_1K_2 \eu^{-(\mu_1 + \mu_2) \frac{n}{6} T}},\end{align*}
for each $\lambda \in \mathcal{K}$. Therefore, provided $n \in \N$ is sufficiently large, Lemma~\ref{l:pastingexpdi} establishes an exponential dichotomy for system~\eqref{variationalsysfullper} on the interval $[-\frac{n}{2} T,\frac{3n}{2} T]$ of length $2nT$ with $\lambda$- and $n$-independent constants for each $\lambda \in \mathcal{K}$. Moreover, using the continuity of $\A$ and the compactness of $\mathcal{K}$, it follows that $\|\A\left(\cdot,z_n(\cdot) + v(\cdot) + a_n(\cdot);\lambda \right)\|_{L^\infty}$ can be bounded by a $\lambda$- and $n$-independent constant for each $\lambda \in \mathcal{K}$. Combining the last two sentences with the fact that system~\eqref{variationalsysfullper} has $nT$-periodic coefficients,~\cite[Theorem~1]{Palmer1987} yields, provided $n \in \N$ is sufficiently large, an exponential dichotomy of system~\eqref{variationalsysfullper} on $\R$ with $\lambda$- and $n$-independent constants $K_0,\mu_0 > 0$ for each $\lambda \in \mathcal{K}$. 

In the following, we denote by $\mathcal{H}^l$ either the space $H^l(\R)$ or the space $H_\per^l(0,nT)$ and $l \in \N_0$. Since $\|\A\left(\cdot,z_n(\cdot) + v(\cdot) + a_n(\cdot);\lambda \right)\|_{L^\infty}$ can be bounded by a $\lambda$- and $n$-independent constant and the $nT$-periodic system~\eqref{variationalsysfullper} has an exponential dichotomy on $\R$ with $\lambda$- and $n$-independent constants, Lemma~\ref{l:inhomexpdi} yields, provided $n \in \N$ is sufficiently large, a $\lambda$- and $n$-independent constant $C > 0$ such that for each $g \in \mathcal{H}^1 \hookrightarrow C(\R)$ and $\lambda \in \mathcal{K}$ the inhomogeneous problem
\begin{align*}
U' = \A\left(x,z_n(x) + v(x) + a_n(x);\lambda \right) U + \psi
\end{align*}
with $\psi = (0,\ldots,0,g)^\top \in \mathcal{H}^1$ has a solution $U \in \mathcal{H}^1$ satisfying
\begin{align*}
\|U\|_{\mathcal{H}^1} \leq C \|\psi\|_{\mathcal{H}^0} = C \|g\|_{\mathcal{H}^0}.
\end{align*}
Using that we have $U_j' = U_{j+1} \in \mathcal{H}^1$ for $j = 1,\ldots,k-1$, we readily observe that $u = U_1 \in \mathcal{H}^k$ solves the resolvent problem
\begin{align} \label{resolvent_periodic}
\left(\El(z_n + v + a_n) - \lambda\right) u = g,
\end{align}
in case $\mathcal{H}^l = H^l(\R)$, and
\begin{align} \label{resolvent_periodic2}
\left(\El_{\per}(z_n + v + a_n) - \lambda\right) u = g,
\end{align}
in case $\mathcal{H}^l = H_\per^l(0,nT)$. Moreover, it obeys the estimate
\begin{align} \label{resolvent_bound_periodic}
\|u\|_{\mathcal{H}^k} \leq \|U\|_{\mathcal{H}^1} \leq C\|g\|_{\mathcal{H}^0}.
\end{align}
Since $\El(z_n + v + a_n)$ and $\El_\per(z_n+v+a_n)$ are closed operators, it follows by the density of $\mathcal{H}^1$ in $\mathcal{H}^0$ that the resolvent problem~\eqref{resolvent_periodic}, in case $\mathcal{H}^l = H^l(\R)$, and the resolvent problem~\eqref{resolvent_periodic2}, in case $\mathcal{H}^l = H_\per^l(0,nT)$, possesses for each $g \in \mathcal{H}^0$ and $\lambda \in \mathcal{K}$ a solution $u \in \mathcal{H}^k$ satisfying~\eqref{resolvent_bound_periodic}.

On the other hand, an element $u \in \ker(\El(z_n + v + a_n) - \lambda)$ or $u \in \ker(\El_{\per}(z_n + v + a_n) - \lambda)$ yields a bounded solution $U = (u,\partial_x u,\ldots,\partial_x^{k-1} u)$ of system~\eqref{variationalsysfullper}, which must be $0$, because~\eqref{variationalsysfullper} has an exponential dichotomy on $\R$ for each $\lambda \in \mathcal{K}$. 

We conclude that, provided $n \in \N$ is sufficiently large, $\El(z_n + v + a_n) - \lambda$ and $\El_{\per}(z_n + v + a_n) - \lambda$ are bounded invertible and obey~\eqref{inversebound2} and~\eqref{inverseboundper}, respectively, for each $\lambda \in \mathcal{K}$.
\end{proof}

With the aid of Lemma~\ref{lem:invertibility_periodic}, we are now able to establish the main result of this section using a contraction-mapping argument.

\begin{proof}[Proof of Theorem~\ref{t:existence_periodic}]
Let $z_n \in H_\per^k(0,nT)$ be the $nT$-periodic extension of the function $\chi_n z$ on $[-\tfrac{n}{2} T,\tfrac{n}{2} T)$. By Lemma~\ref{lem:invertibility_periodic} the linear operator $\El_\circ := \El_{\per}(z_n + v)$ is invertible and there exists an $n$-independent constant $K > 0$ such that
\begin{align} \label{linestper}
\|\El_\circ^{-1}\|_{L_\per^2(0,nT) \to H_\per^k(0,nT)} \leq K.
\end{align}

We substitute the ansatz $u = \chi_n z + v + a$ with correction term $a \in H_\per^k(0,nT)$ into~\eqref{existence_problem} and arrive at the equation
\begin{align} \label{perteq1per}
a = \widetilde{\Non}(a) + R,
\end{align}
with nonlinearity $\widetilde{\Non} \colon H_\per^k(0,nT) \to H_\per^k(0,nT)$ given by
\begin{align*}
\widetilde{\Non}(a) &= \El_\circ^{-1} \left(\Non(z_n + v,\cdot) + \partial_u \Non(z_n  + v,\cdot)a - \Non(z_n + v + a,\cdot)\right)
\end{align*}
and residual $R \in H_\per^k(0,nT)$ given by
\begin{align*}
R = -\El_\circ^{-1} \left(\Non(z_n +v,\cdot) + A(z_n+v)\right).
\end{align*}

Since $\Non$ is twice continuously differentiable in its first argument and $H_\per^1(0,nT)$ embeds continuously into $L^\infty(\R)$ with $n$-independent constant, Taylor's theorem and estimate~\eqref{linestper} imply that $\smash{\widetilde{\Non}}$ is well-defined and for all $\rho >0 $ there exists an $n$-independent constant $C_1 > 0$ such that
\begin{align*}
\left\|\widetilde{\Non}(a_0) - \widetilde{\Non}(a_1)\right\|_{H_\per^k(0,nT)} \leq C_1\left(\|a_0\|_{H_\per^k(0,nT)} + \|a_1\|_{H_\per^k(0,nT)}\right)\left\|a_0 - a_1\right\|_{H_\per^k(0,nT)},
\end{align*}
for $a_0,a_1 \in H_\per^k(0,nT)$ with $\|a_0\|_{L^\infty},\|a_1\|_{L^\infty} \leq \rho$.

We proceed with bounding the residual. First, as $z + v$ and $v$ are solutions of~\eqref{existence_problem}, we have
\begin{align*}
\Non(z_n(x) +v(x),x) + Az_n(x) + Av(x) = 0,
\end{align*}
for $x \in [-\frac{n}{6} T, \frac{n}{6} T]$ and
\begin{align*}
\Non(z_n(x) +v(x),x) + Az_n(x) + Av(x) = \Non(z_n(x) +v(x),x) - \Non(v(x),x) + Az_n(x),
\end{align*}
for $x \in [-\frac{n}{2}T,\frac{n}{2} T]$. Hence, since $H_\per^1(0,ntT)$ embeds continuously into $L^\infty(\R)$ with $n$-independent constant, $\Non$ is continuously differentiable and it holds $\|\chi_n\|_{W^{k,\infty}} \leq 1$, there exists by the mean value theorem and estimate~\eqref{linestper} an $n$-independent constant $C_0 > 0$ such that
\begin{align} \label{inhomMboundper}
 \begin{split}
\|R\|_{H_\per^k(0,nT)} &\leq C_0 \delta_n,
\end{split}
\end{align}
where we denote
\begin{align*}
\delta_n := \|z\|_{H^k(\R \setminus (-\frac{n}{6}T,\frac{n}{6} T))}.
\end{align*}
We observe that $\delta_n$ converges to $0$ as $n \to \infty$. 

We introduce the rescaled variable
\begin{align} \label{rescalingper}
\at = \delta_{n}^{-1} a,
\end{align}
in which~\eqref{perteq1per} reads
\begin{align} \label{perteq2per}
\begin{split}
\at &= \delta_{n}^{-1} \widetilde{\Non}(\delta_n \at) + \delta_{n}^{-1} R.
\end{split}
\end{align}
Regard~\eqref{perteq2per} as an abstract fixed point problem
\begin{align} \at = \mathcal{F}_{n}(\at). \label{contrmapper}
\end{align}
Analogous to the proof of Theorem~\ref{t:existence_multifront}, one establishes, provided $n \in \N$ is sufficiently large, that $\mathcal{F}_n \colon B_0(2C_0) \to B_0(2C_0)$ is a well-defined contracting mapping on the ball $B_0(2C_0)$ of radius $2C_0$ in $H_\per^k(0,nT)$ centered at the origin, where $C_0 > 0$ is the $n$-independent constant appearing in the bound~\eqref{inhomMboundper} on the residual $R$. Then, an application of the Banach fixed point theorem yields a unique solution $\at_n \in B_0(2C_0)$ to~\eqref{contrmapper}. Undoing the rescaling~\eqref{rescalingper}, we obtain a solution $a_n \in H_\per^k(0,nT)$ of equation~\eqref{perteq1per} with $\|a_n\|_{H_\per^k(0,nT)} \leq 2C_0 \delta_{n}$, provided $n \in \N$ is sufficiently large, which yields the result.
\end{proof}

\section{Spectral analysis of front solutions with periodic tails} \label{sec:stability_1fronts}

In this section, we collect some background material on the spectral stability of front solutions connecting periodic end states. Specifically, we impose the following assumption:

\begin{itemize}
    \item[\namedlabel{assH6}{\textbf{(H6)}}] There exists a front $Z \in L^{\infty}(\R)$ with associated end states $v_\pm \in H_\per^k(0,T)$. We have $\chi_\pm(Z-v_\pm) \in H^k(\R)$, where $\chi_\pm \colon \R \to [0,1]$ is a smooth partition of unity such that $\chi_+$ is supported on $(-1,\infty)$ and $\chi_-$ is supported on $(-\infty,1)$.
\end{itemize}

We adopt the following distinction between essential and point spectrum, cf.~\cite{Sandstede2002Stability,KapitulaPromislow2013}. 

\begin{Definition}
Let $L \colon D(L) \subset L^2(\R) \to L^2(\R)$ be a linear operator with domain $D(L)=H^k(\R)$. The \emph{essential spectrum} of $L$ is defined as the set of all $\lambda \in \C$ for which the operator $L-\lambda$ is not Fredholm of index zero. The \emph{point spectrum} is defined as the complement $\sigma(L) \setminus \sigma_\text{ess}(L)$.
\end{Definition}

Assume~\ref{assH6}. Let $\eta_{\pm} \in \R$. The Fredholm properties of the exponentially weighted linearization operator $\El_{\eta_-,\eta_+}(Z)$, see~\S\ref{sec:Abstract_ex_stab}, are determined by the periodic end states $v_\pm$. By leveraging Lemma~\ref{l:compact} and Weyl's theorem, cf.~\cite[Theorem~VI.5.26]{Kato1995}, we infer that, for each $\lambda \in \C$, the operator $\El_{\eta_-,\eta_+}(\chi_-v_-+\chi_+v_+) - \lambda$ is Fredholm of index $j \in \Z$ if and only if $\El_{\eta_-,\eta_+}(Z)-\lambda$ is Fredholm of index $j$. Consequently, the essential spectra of $\El_{\eta_-,\eta_+}(Z)$ and $\El_{\eta_-,\eta_+}(\chi_-v_-+\chi_+v_+)$ coincide. 

The results of Palmer,~\cite[Lemma~4.2]{Palmer1984},~\cite[Theorem~1]{Palmer1987} and~\cite[Theorem~1]{Palmer1988}, imply that $\El_{\eta_-,\eta_+}(\chi_-v_-+\chi_+v_+) - \lambda$ is Fredholm if and only if both of the associated first-order eigenvalue problems
\begin{align}\label{variationalbackground}
    U' = \left(\A(x,v_\pm(x);\lambda) - \eta_\pm\right) U
\end{align}
admit an exponential dichotomy on $\R$. Its Fredholm index is then given by
$$
    \text{ind}(\El_{\eta_-,\eta_+}(Z)-\lambda) = \text{ind}(\El_{\eta_-,\eta_+}(\chi_-v_-+\chi_+v_+) - \lambda) =     l_{-}(\lambda) - l_{+}(\lambda),
$$
where the \emph{Morse index} $l_\pm(\lambda)$ is the rank of the projection associated with the exponential dichotomy of~\eqref{variationalbackground} on $\R$. We note that~\eqref{variationalbackground} possesses an exponential dichotomy on $\R$ if and only if the operator $\El_{\eta_\pm}(v_\pm)-\lambda$ is invertible, cf.~Lemmas~\ref{l:inhomexpdi} and~\ref{lem:expdi1}. Furthermore, Floquet's theorem,~\cite[Theorem~2.1.27]{KapitulaPromislow2013}, yields that the $T$-periodic system~\eqref{variationalbackground} has an exponential dichotomy on $\R$ if and only if it possesses no Floquet exponents $\nu \in \C$ on the imaginary axis~\footnote{The Floquet exponents are the principal logarithms of the eigenvalues of the monodromy matrix $\mathcal{T}_\pm(T,0;\lambda)$, where $\mathcal{T}_\pm(x,y;\lambda)$ is the evolution of system~\eqref{variationalbackground}. We refer to~\cite{Brown2013,KapitulaPromislow2013} for further background on Floquet theory.}. The Morse index is then equal to the number of Floquet exponents $\nu \in \C$ with negative real part (counted with algebraic multiplicity). We summarize these observations in the following proposition.

\begin{Proposition} \label{prop:essential_spec1front}    
Assume~{\upshape \ref{assH6}}. Let $\eta_\pm \in \R$. Then, the following assertions hold true.
\begin{itemize}
\item[1.] A point $\lambda \in \C$ lies in the spectrum of $\El_{\eta_\pm}(v_\pm)$ if and only if the $T$-periodic system~\eqref{variationalbackground} possesses purely imaginary Floquet exponents.
    \item[2.] We have
\begin{align*}
\sigma_{\mathrm{ess}}\left( \El_{\eta_-,\eta_+}(Z)\right) &= \sigma_{\mathrm{ess}}\left(\El_{\eta_-,\eta_+}(\chi_- v_{-} + \chi_+ v_{+})\right)\\
&= \left\{\lambda \in \C : l_-(\lambda) \neq l_+(\lambda)\right\} \cup \sigma(\El_{\eta_-}(v_-)) \cup \sigma(\El_{\eta_+}(v_+)),
\end{align*}
where $l_\pm(\lambda)$ is the number of Floquet exponenents $\nu \in \C$ (counted with algebraic multiplicity) of~\eqref{variationalbackground} with negative real part. 
\end{itemize}
\end{Proposition}

In the following proposition, we introduce the \emph{Evans function}, a well-known tool to locate point spectrum in the stability analysis of nonlinear waves, see~\cite{Alexander_Topological_1990,Evans1976,KapitulaPromislow2013, Pego1992,Sandstede2002Stability} and references therein. The Evans function is an analytic determinantal function measuring the alignment or mismatch between the subspace of solutions decaying as $x \to \infty$ and the subspace of solutions decaying as $x \to -\infty$. Consequently, its zeros correspond to eigenvalues, including their multiplicities.

\begin{Proposition} \label{prop:Evans1front}
Assume~{\upshape \ref{assH6}}. Let $\eta_\pm \in \R$. Let $\Omega$ be a connected component of 
$$\C \setminus \big(\sigma_\textup{ess}(\El_{\eta_-}(v_-)) \cup \sigma_\textup{ess}(\El_{\eta_+}(v_+))\big).$$ 
Then, the following assertions hold true.
\begin{itemize}
\item[1.] The number $l_{0,\pm} \in \{0,\ldots,km\}$ of Floquet exponents $\nu \in \C$ (counted with algebraic multiplicity) of~\eqref{variationalbackground} with negative real part is constant for each $\lambda \in \Omega$. 
\item[2.] System~\eqref{variationalbackground} has for each $\lambda \in \Omega$ an exponential dichotomy on $\R$ with projections $Q_\pm(x;\lambda)$ of rank $l_{0,\pm}$. Here, $Q_\pm(\cdot;\lambda) \colon \R \to \C^{km \times km}$ is $T$-periodic for each $\lambda \in \Omega$ and $Q_\pm(x;\cdot) \colon \Omega \to \C^{km \times km}$ is analytic for each $x \in \R$.
\item[3.] System 
\begin{align} \label{variationalsys}
U' = \left(\A\left(x,Z(x);\lambda \right) - \omega_{\eta_-,\eta_+}'(x)\right) U
\end{align}
possesses for each $\lambda \in \Omega$ exponential dichotomies on $\R_\pm$ with projections $P_\pm(\pm x;\lambda), x \geq 0$ of fixed rank $l_{0,\pm}$, where $\omega_{\eta_-,\eta_+}$ is the weight function defined in~\S\ref{sec:Abstract_ex_stab}. Moreover, there exist analytic functions $B_{\mathrm{s}} \colon \Omega \to \C^{km \times l_{0,+}}$ and $B_{\mathrm{u}}  \colon \Omega \to \C^{km \times (km-l_{0,-})}$ such that $B_{\mathrm{s}}(\lambda)$ is a basis of $\mathrm{ran}(P_+(0;\lambda))$ and $B_{\mathrm{u}}(\lambda)$ is a basis of $\ker(P_-(0;\lambda))$ for each $\lambda \in \Omega$. 
\item[4.] Assume further $l_{0,-} = l_{0,+}$. Then, there is no essential spectrum of $\El_{\eta_-,\eta_+}(Z)$ in $\Omega$. Moreover, a point $\lambda_0 \in \Omega$ lies in the point spectrum of $\El_{\eta_-,\eta_+}(Z)$ if and only if $\lambda_0$ is a root of the analytic Evans function $\mathcal{E} \colon \Omega \to \C$ given by
\begin{align*}
\mathcal{E}(\lambda) = \det(B_{\mathrm{u}}(\lambda) \mid B_{\mathrm{s}}(\lambda)).
\end{align*}
The geometric multiplicity $m_g(\lambda_0)$ of an eigenvalue $\lambda_0 \in \Omega$ of the operator $\El_{\eta_-,\eta_+}(Z)$ is equal to $\dim(\ker(P_-(0;\lambda_0)) \cap \mathrm{ran}(P_+(0;\lambda_0)))$. Moreover, if $\mathcal{E}$ does not vanish identically on $\Omega$, then the algebraic multiplicity $m_a(\lambda_0)$ of an eigenvalue $\lambda_0 \in \Omega$ of $\El_{\eta_-,\eta_+}(Z)$ is equal to the multiplicity of $\lambda_0$ as a root of $\mathcal{E}$. In particular, the roots of $\mathcal{E}$ and their multiplicities are independent of the choice of bases.
\item[5.] Let $\mathcal{K} \subset \Omega$ be compact. Then, there exist constants $K_0,\mu_0,\tau_0 > 0$ such that system~\eqref{variationalsys} admits for each $\lambda \in \mathcal{K}$ exponential dichotomies on $\R_\pm$ with constants $K_0,\mu_0 > 0$ and projections $\tilde P_\pm(\pm x;\lambda), x \geq 0$ satisfying
\begin{align} \label{projest2}
\begin{split}
\left\|\tilde{P}_\pm(\pm x;\lambda) - Q_\pm(\pm x;\lambda)\right\| &\leq K_0\left(\eu^{-\mu_0 x} + 
\sup_{y \in [x,\infty)} \left\|Z(\pm y)-v_\pm(\pm y)\right\|\right),
\end{split}
\end{align}
 for each $x \geq \tau_0$. 
\end{itemize}
\end{Proposition}
\begin{proof}
The first assertion is an immediate consequence of Proposition~\ref{prop:essential_spec1front} and the fact that the Floquet exponents of~\eqref{variationalbackground} depend continuously on $\lambda$.

It follows from Floquet's theorem, cf.~\cite[Theorem~2.1.27]{KapitulaPromislow2013}, Proposition~\ref{prop:essential_spec1front} and Lemma~\ref{l:projper} that the $T$-periodic system~\eqref{variationalbackground} possesses for each $\lambda \in \Omega$ an exponential dichotomy on $\R$ with projections $Q_\pm(x;\lambda)$, which have rank $l_{0,\pm}$ and are $T$-periodic in their first component. Thanks to the uniqueness of exponential dichotomies on $\R$, cf.~\cite[p.~19]{Coppel1978}, and the fact that~\eqref{variationalbackground} depends analytically on $\lambda$, it follows from~\cite[Lemma~A.2]{DasLatushkin2011} that $Q_\pm(x;\cdot)$ is analytic on $\Omega$ for each $x \in \R$. This proves the second assertion.

Next, we observe that~\cite[Lemma~3.4]{Palmer1984},~\ref{assH6} and the continuous embedding $H^1(\R) \hookrightarrow L^\infty(\R)$ yield for each $\lambda \in \Omega$ exponential dichotomies for system~\eqref{variationalsys} on $\R_\pm$ with projections $P_\pm(\pm x;\lambda), x \geq 0$ of rank $l_{0,\pm}$. Using that system~\eqref{variationalsys} depends analytically on $\lambda$ and the subspaces $\ker(P_-(-x;\lambda))$ and $\mathrm{ran}(P_+(x;\lambda))$ are by~\cite[p.~19]{Coppel1978} uniquely determined for $x \geq 0$ and $\lambda \in \Omega$,~\cite[Lemma~A.2]{DasLatushkin2011} and~\cite[Section~II.4.2]{Kato1995} provide analytic functions $B_{\mathrm{s}} \colon \Omega \to \C^{km \times l_0}$ and $B_{\mathrm{u}}  \colon \Omega \to \C^{km \times (km-l_0)}$ such that $B_{\mathrm{s}}(\lambda)$ is a basis of $\mathrm{ran}(P_+(0;\lambda))$ and $B_{\mathrm{u}}(\lambda)$ is a basis of $\ker(P_-(0;\lambda))$ for each $\lambda \in \Omega$. Thus, we have established the third assertion.

Assume $l_{0,+} = l_{0,-}$. Then, there is no essential spectrum of $\El_{\eta_-,\eta_+}(Z)$ in $\Omega$ by Proposition~\ref{prop:essential_spec1front}. Set $l_0 = l_{0,\pm}$ and take $\lambda_0 \in \Omega$. As $\El_{\eta_-,\eta_+}(Z) - \lambda_0$ is Fredholm of index $0$, it is invertible if and only if $\lambda_0$ is not an eigenvalue of $\El_{\eta_-,\eta_+}(Z)$. Since~\eqref{variationalsys} is the first-order formulation of the eigenvalue problem $\El_{\eta_-,\eta_+}(Z) u = \lambda u$, there is a one-to-one correspondence between elements $u_0 \in \ker(\El_{\eta_-,\eta_+}(Z)-\lambda_0)$ and $H^1$-solutions $U_0 = (u_0,\partial_x u_0,\ldots,\partial_x^{k-1} u_0) \in H^1(\R)$ of~\eqref{variationalsys} at $\lambda = \lambda_0$. The exponential dichotomies of~\eqref{variationalsys} on $\R_\pm$ yield that $U_0$ is an $H^1$-solution of~\eqref{variationalsys} at $\lambda = \lambda_0$ if and only if $U_0(0) \in \ker(P_-(0;\lambda_0)) \cap \mathrm{ran}(P_+(0;\lambda_0))$. Therefore, $\lambda_0 \in \Omega$ is an eigenvalue of $\El_{\eta_-,\eta_+}(Z)$ if and only if $\mathcal{E}(\lambda_0) = 0$. In addition, the geometric multiplicity $m_g(\lambda_0)$ of $\lambda_0$ equals $\dim(\ker(P_-(0;\lambda_0)) \cap \mathrm{ran}(P_+(0;\lambda_0)))$. Finally,~\cite[Theorem 2.9]{Kapitula2004} asserts that, if $\mathcal{E}$ is not identically $0$, then the algebraic multiplicity $m_a(\lambda_0)$ of $\lambda_0$ is equal to the multiplicity of $\lambda_0$ as a root of $\mathcal{E}$. This completes the proof of the fourth assertion.

Finally, we recall that system~\eqref{variationalbackground} possesses an exponential dichotomy on $\R$ for each $\lambda$ in the compact set $\mathcal{K}$ with projections $Q_\pm(x;\lambda)$. Thus, as in the proof of Lemma~\ref{lem:invertibility_multifront}, we find that~\eqref{variationalbackground} has for each $\lambda \in \mathcal{K}$ an exponential dichotomy on $\R$ with $\lambda$-independent constants and, by uniqueness, projections $Q_\pm(x;\lambda)$. Thus, the fifth assertion follows from~\cite[Lemma~3.4]{Palmer1984} and its proof.
\end{proof}

\section{Spectral analysis of multifront solutions}\label{sec:stability_multifront}

This section is devoted to the spectral stability analysis of stationary multifront solutions to~\eqref{e:sys_intro}. We consider a multifront $u_n$ of the form~\eqref{multifront_ansatz}, where $w_n$ is a formal concatenation of $M$ primary fronts $Z_1,\ldots,Z_M$ with matching periodic end states, and $a_n$ is an error term converging to $0$ in $H^k(\R)$ as $n \to \infty$. Our goal is to transfer spectral properties of the linearizations $\El(Z_1),\ldots,\El(Z_M)$ of~\eqref{e:sys_intro} about the primary front solutions to the linearization $\El(u_n)$ about the multifront. 

A first key observation, provided by Lemma~\ref{lem:invertibility_multifront}, is that, if $\El(Z_1) - \lambda,\ldots,\El(Z_M)-\lambda$ are invertible for each $\lambda$ in a compact set $\mathcal{K} \subset \C$, then so is $\El(u_n)- \lambda$ for $n \in \N$ sufficiently large. In other words, if $\mathcal{K}$ is a compact subset of the resolvent set $\rho(\El(Z_j))$ for each $j = 1,\ldots,M$, then $\mathcal{K}$ is also contained in $\rho(\El(u_n))$, provided $n \in \N$ is sufficiently large.

In dissipative systems, such as the reaction-diffusion models discussed in~\S\ref{sec:applications}, the spectral stability analysis can often be reduced to an $n$-independent compact set with the aid of a-priori bounds that preclude spectrum with nonnegative real part and large modulus. As a result, Lemma~\ref{lem:invertibility_multifront} yields the following corollary, asserting that strong spectral stability of the $M$ primary fronts $Z_1,\ldots,Z_M$ is inherited by the $M$-front $u_n$.

\begin{Corollary}\label{cor:stability_multifront}
Assume~{\upshape \ref{assH1}},~{\upshape \ref{assH2}} and~{\upshape \ref{assH3}}. Suppose that the front solutions $Z_j$ are strongly spectrally stable for $j = 1,\ldots,M$. Moreover, assume that there exist a compact set $\mathcal{K} \subset \C$ and $N \in \N$ such that
\begin{align*}
    \sigma(\El(u_n)) \cap \{z \in \C : \Re(z) \geq 0\} \subset \mathcal{K}
\end{align*}
for $n \geq N$, where $u_n$ is the multifront solution established in Theorem~\ref{t:existence_multifront}. Then, there exists $N_1 \in \N$ with $N_1 \geq N$ such that for all $n \in \N$ with $n \geq N_1$ the multifront $u_n$ is strongly spectrally stable.
\end{Corollary}

In the remainder of this section, we study the spectra associated with stationary multifront solutions to~\eqref{e:sys_intro} in more detail. The obtained results are particularly useful in scenarios where one of the primary fronts is either spectrally unstable or only marginally stable. We assume that~\ref{assH1} and~\ref{assH2} hold and consider a multifront $u_n$ of the form~\eqref{multifront_ansatz}, where $a_n$ is an error term converging to $0$ in $H^k(\R)$ as $n \to \infty$. 

By Proposition~\ref{prop:essential_spec1front}, the essential spectrum of $\El(u_n)$ is determined solely by its periodic end states $v_{1,-}$ and $v_{M,+}$, making it independent of $n$. Specifically, it is given by
\begin{align*}
\sigma_{\mathrm{ess}}(\El(u_n)) &= \sigma_{\mathrm{ess}}\left(\El(\chi_- v_{1,-} + \chi_+ v_{M,+})\right) = \left\{\lambda \in \C : l_{-}(\lambda) \neq l_{+}(\lambda)\right\} \cup \sigma(\El(v_{1,-})) \cup \sigma(\El(v_{M,+})),
\end{align*}
where $l_-(\lambda)$ and $l_+(\lambda)$ are the Morse indices, corresponding to the number of Floquet exponents $\nu\in\C$ of negative real part (counted with algebraic multiplicity), of the asymptotic systems 
\begin{align} \label{variationalbackground2}
    U' = \A(x,v_{1,-}(x);\lambda) U, \qquad U' = \A(x,v_{M,+}(x);\lambda) U,
\end{align}
respectively. 

The main result of this section concerns the approximation of the point spectrum of $\El(u_n)$. For each connected component $\Omega$ of $\C \setminus \sigma_{\mathrm{ess}}(\El(u_n))$, we establish such an approximation in the subset
\begin{align*}
\rho_{\mathrm{abs},\Omega} &:= \left\{\lambda \in \Omega : \text{for each } j \in \{1,\ldots,M-1\} \text{ there exists } \eta_j \in \R \text{ such that } \ell_{\eta_j,j}(\lambda) = l_+(\lambda)\right\},
\end{align*}
where $\ell_{\eta,j}(\lambda) \in \N_0$ denotes the number of Floquet exponents $\nu \in \C$ with negative real part (counted with algebraic multiplicity) of the $T$-periodic system
\begin{align*}
    U' = \left(\A(x,v_{j,+}(x);\lambda) - \eta\right) U.
\end{align*}
In accordance with~\cite{Sandstede_2000_gluing}, we refer to the complement $\sigma_{\mathrm{abs},\Omega} := \Omega \setminus \rho_{\mathrm{abs},\Omega}$ as the \emph{absolute spectrum of $\El(u_n)$ in $\Omega$}, see also Remark~\ref{rem:absolute}. We observe that a point $\lambda \in \Omega$ lies in the absolute spectrum $\sigma_{\mathrm{abs},\Omega}$ if and only if there is a $j \in \{1,\ldots,M-1\}$ such that there is no $\eta \in \R$ separating the Floquet exponents of the asymptotic system
\begin{align} \label{variationalbackground22}
    U' = \A(x,v_{j,+}(x);\lambda) U,
\end{align}
into $l_+(\lambda)$ exponents in the half plane $\{\nu \in \C : \Re(\nu) < \eta\}$ and $km - l_+(\lambda)$ exponents in its complement (counting algebraic multiplicities). So, ordering the Floquet exponents $\nu_{1,j}(\lambda),\ldots,\nu_{km,j}(\lambda)$ of the system~\eqref{variationalbackground22} by their real parts,
$$\Re(\nu_{1,j}(\lambda)) \leq \ldots \leq \Re(\nu_{km,j}(\lambda)),$$ 
a point $\lambda \in \Omega$ lies in the absolute spectrum $\sigma_{\mathrm{abs},\Omega}$ if and only if we have $\Re(\nu_{l_+(\lambda),j}(\lambda)) = \Re(\nu_{l_+(\lambda) + 1,j}(\lambda))$ for some $j \in \{1,\ldots,M-1\}$. Since the Floquet exponents $\nu_{i,j}(\lambda)$ depend continuously on $\lambda$, and $\Omega$ is open, we infer that $\rho_{\mathrm{abs},\Omega}$ is also open. We note that the imaginary axis enforces the desired splitting of Floquet exponents if $\lambda \in \Omega$ lies outside the essential spectra of the linearizations about the primary fronts, see Proposition~\ref{prop:essential_spec1front}. So, we have the inclusion
\begin{align*}
\sigma_{\mathrm{abs},\Omega} \subset \bigcup_{j = 1}^M \sigma_{\mathrm{ess}}(\El(Z_j)).
\end{align*}

Let $\Omega$ be a connected component of $\C \setminus \sigma_{\mathrm{ess}}(\El(u_n))$ and let $\lambda_0 \in \rho_{\mathrm{abs},\Omega}$. Set $\eta_0 = 0$ and $\eta_M = 0$. Then, by continuous dependence of the Floquet exponents on $\lambda$, there exist an open neighborhood $\mathcal{U} \subset \rho_{\mathrm{abs},\Omega}$ of $\lambda_0$ and $\eta_1,\ldots,\eta_{M-1} \in \R$ such that the $T$-periodic asymptotic system
\begin{align}\label{variationalbackground3}
    U' = \left(\A(x,v_{j,+}(x);\lambda) - \eta_j\right) U
\end{align}
has $l_+(\lambda_0)$ Floquet exponents $\nu \in \C$ in the open left-half plane and $km - l_+(\lambda_0)$ Floquet exponents in the open right-half plane for all $\lambda \in \mathcal{U}$ and $j = 1,\ldots,M$ (counting algebraic multiplicities), cf.~Figure~\ref{fig:Stability_multifront_proof}(top). The main result of this section establishes that the point spectrum in $\mathcal{U}$ of the linearization $\El(u_n)$ about the multifront converges, as $n \to \infty$, to the union of the point spectra in $\mathcal{U}$ of the (exponentially weighted) linearizations $\El_{\eta_0,\eta_1}(Z_1), \ldots, \El_{\eta_{M-1},\eta_M}(Z_M)$ about the primary fronts, thereby preserving the total algebraic multiplicity of eigenvalues. 

\begin{Theorem} \label{thm:instability_multifront}
Let $M \in \N_{\geq 2}$. Assume~{\upshape \ref{assH1}} and~{\upshape \ref{assH2}}. Suppose that there exists $N \in \mathbb{N}$ such that, for each $n \in \N$ with $n \geq N$, there exists an $M$-front $u_n$ of the form~\eqref{multifront_ansatz}, where $\{a_n\}_n$ is a sequence in $H^k(\R)$ converging to $0$.

Let $\Omega$ be a connected component of $\C \setminus \sigma_{\mathrm{ess}}\left(\El(\chi_- v_{1,-} + \chi_+ v_{M,+})\right)$. Let $\lambda_0 \in \rho_{\mathrm{abs},\Omega}$. Take an open neighborhood $\mathcal{U} \subset \rho_{\mathrm{abs},\Omega}$ of $\lambda_0$ and real numbers $\eta_j \in \R$ such that for each $j = 1,\ldots,M-1$ and all $\lambda \in \mathcal{U}$ systems~\eqref{variationalbackground2} and~\eqref{variationalbackground3} have the same number of Floquet exponents in both the open left-half plane and the open right-half plane (counting algebraic multiplicities). Set $\eta_0 = 0$ and $\eta_M = 0$. Let $\mathcal{E}_j \colon \mathcal{U} \to \C$ be the Evans function of the first-order system
\begin{align} \label{variationalsys3}
U' = \left(\A\left(x,Z_j(x);\lambda \right) - \omega_{\eta_{j-1},\eta_j}'(x)\right) U
\end{align}
for $j = 1,\ldots,M$, as constructed in Proposition~\ref{prop:Evans1front}.

Suppose $\lambda_0$ is a root of $\mathcal{E} := \mathcal{E}_1 \cdot \ldots \cdot \mathcal{E}_M$ of multiplicity $m_0 \in \N_0$. Then, there exists $\varrho_0 > 0$ such that for each $\varrho \in (0,\varrho_0)$ there exists $N_\varrho \in \N$ with $N_\varrho \geq N$ such that for all $n \in \N$ with $n \geq N_\varrho$ the following assertions hold true.
\begin{enumerate}
    \item[1.] An Evans function $E_n \colon B_{\lambda_0}(\varrho) \to \C$ associated with~\eqref{variationalsysfull} has precisely $m_0$ roots in $B_{\lambda_0}(\varrho)$ (counting multiplicities).
    \item[2.] The operator $\El(u_n)$ has point spectrum in $B_{\lambda_0}(\varrho)$ if and only if $m_0 \neq 0$. The total algebraic multiplicity of the eigenvalues of $\El(u_n)$ in $B_{\lambda_0}(\varrho)$ is $m_0$.  
\end{enumerate}
\end{Theorem}

\begin{Remark}
The integer $m_0 \in \N_0$ in Theorem~\ref{thm:instability_multifront} equals the number of eigenvalues (counting algebraic multiplicities) of $\El(u_n)$ converging to $\lambda_0$ as $n \to \infty$. In particular, $m_0$ is independent of the choice of neighborhood $\mathcal{U} \subset \rho_{\mathrm{abs},\Omega}$ of $\lambda_0$ and the choice of reals $\eta_j \in \R$ in Theorem~\ref{thm:instability_multifront}.
\end{Remark}

\begin{Remark}
If $\lambda_0$ lies in the complement $\C \setminus \bigcup_{j = 1}^M \sigma_{\mathrm{ess}}(\El(Z_j))$, then we may take $\eta_1,\ldots,\eta_{M-1} = 0$ in Theorem~\ref{thm:instability_multifront}. The result then asserts that there exists an $n$-independent neighborhood $\mathcal{U}$ of $\lambda_0$ such that the point spectrum in $\mathcal{U}$ of $\El(u_n)$ converges, as $n \to \infty$, to the union of the point spectra of the linearizations $\El(Z_1),\ldots,\El(Z_M)$ about the primary fronts, while preserving the total algebraic multiplicity of eigenvalues. This observation proves the first assertion in Theorem~\ref{t:informalstability}. 
\end{Remark}

Theorem~\ref{thm:instability_multifront} can be applied to show that, if one of the (weighted) linearizations about the primary fronts possesses unstable point spectrum, then so does the linearization about the multifront. However, it also serves for the purpose of counting eigenvalues, which is for instance useful for spectral (in)stability arguments in Hamiltonian systems based on Krein index theory~\cite{KapitulaKevrekidis2004,AddendumKapitulaKevrekidis2004,KapitulaPromislow2013}. We illustrate this in~\S\ref{sec:applications} by proving instability results for stationary multipulse solutions to the Gross-Pitaevskii equation with periodic potential. Notably, the control over algebraic multiplicities provided by Theorem~\ref{thm:instability_multifront} is essential for applying Krein index counting theory effectively. Finally, Theorem~\ref{thm:instability_multifront} can be used to establish strong spectral stability of the multifront in cases where Corollary~\ref{cor:stability_multifront} does not apply. More precisely, if both the essential spectrum $\sigma_{\mathrm{ess}}(\El(u_n))$ and the absolute spectrum $\sigma_{\mathrm{abs},\Omega}$ in the right-most connected component of $\C \setminus \sigma_{\mathrm{ess}}(\El(u_n))$ are confined to the open left-half plane, then we can employ Theorem~\ref{thm:instability_multifront} to preclude spectrum of $\El(u_n)$ in $n$-independent compact subsets of the closed right-half plane. This leads to the following extension of Corollary~\ref{cor:stability_multifront}.

\begin{Corollary}\label{cor:stability_multifront_extension}
Let $M \in \N_{\geq 2}$. Assume~{\upshape \ref{assH1}} and~{\upshape \ref{assH2}}. Assume further that the following conditions hold:
\begin{itemize}
    \item[1.] The essential spectrum $\sigma_{\mathrm{ess}}(\El(\chi_- v_{1,-} + \chi_+ v_{M,+}))$ and the absolute spectrum $\sigma_{\mathrm{abs},\Omega}$ in the right-most connected component $\Omega$ of $\C \setminus  \sigma_{\mathrm{ess}}(\El(\chi_- v_{1,-} + \chi_+ v_{M,+}))$ are confined to the open-left half plane. 
    \item[2.] The Evans function $\mathcal{E} \colon \Omega \to \C$ in Theorem~\ref{thm:instability_multifront} has no zeros in the closed right-half plane.
    \item[3.] There exist $N \in \mathbb{N}$ and a compact set $\mathcal{K} \subset \C$ such that, for each $n \in \N$ with $n \geq N$, there exists a stationary $M$-front solution $u_n$ to~\eqref{e:sys_intro} of the form~\eqref{multifront_ansatz}, where $\{a_n\}_n$ is a sequence in $H^k(\R)$ converging to $0$, such that
    \begin{align*}
    \sigma(\El(u_n)) \cap \{z \in \C : \Re(z) \geq 0\} \subset \mathcal{K}
    \end{align*}
    for $n \geq N$.
\end{itemize}
Then, there exists $N_1 \in \N$ with $N_1 \geq N$ such that for all $n \in \N$ with $n \geq N_1$ the multifront $u_n$ is strongly spectrally stable.
\end{Corollary}

\begin{Remark}
The conditions in Corollary~\ref{cor:stability_multifront_extension} can be satisfied even if the linearization $\El(Z_j)$ about one of the primary fronts $Z_j$ has unstable essential spectrum, see Figure~\ref{fig:Stability_multifront_proof}(top). Thus, spectrally unstable primary fronts may produce strongly spectrally stable multifronts. 

The observation that multifronts composed of unstable primary fronts can still be stable is not new: the phenomenon is well-studied in constant-coefficient systems~\cite{Sandstede_2000_gluing, Romeo_Stability_2000}. An advantage of the spatially periodic setting considered here is that it breaks the translational symmetry. As a result, front solutions can be \emph{strongly} spectrally stable, eliminating the need to track eigenvalues of $\El(u_n)$ that converge to $0$ as $n \to \infty$, cf.~\cite{Sandstede_Stability_1998}.
\end{Remark}

\begin{Remark} \label{rem:absolute}
In systems with constant coefficients, it was shown in~\cite{Sandstede_2000_gluing} that eigenvalues of the linearization about a multifront accumulate onto each point of the absolute spectrum as the distances between interfaces tend to infinity, see also~\cite{Nii_Accumulation_2000,Sandstede2000Absolute}. We anticipate that, using the techniques developed in~\cite{Sandstede_2000_gluing,Sandstede2000Absolute}, a similar result can be established for the spatially periodic setting considered here.
\end{Remark}

\begin{Remark}
Theorem~\ref{thm:instability_multifront} does not require that the primary fronts constituting the multifront are nondegenerate. Therefore, the theorem applies even when the linearization about a primary front is not invertible due to additional symmetries, such as translational or rotational symmetries. In~\S\ref{sec:applications}, we will apply Theorem~\ref{thm:instability_multifront} to prove spectral instability of multipulses in a spatially periodic Gross-Pitaevskii equation, which exhibits a rotational symmetry. 
\end{Remark}

The proof of Theorem~\ref{thm:instability_multifront} hinges on a delicate factorization procedure of the Evans function $E_n$. The procedure is inductive and employs the well-known identity
\begin{align} \label{detblockid}
\det\left(CA^{-1} B + D\right) = \frac{(-1)^\ell}{\det(A)} \det\begin{pmatrix} -A & B \\ C & D \end{pmatrix}  =  \det(A^{-1})\det\begin{pmatrix} B & -A \\ D & C \end{pmatrix}
\end{align}
in each induction step, where $A,B,C,D \in \C^{\ell \times \ell}$ are block matrices with $A$ invertible and $\ell$ a natural number. 

By Proposition~\ref{prop:Evans1front}, system~\eqref{variationalsys3} possesses exponential dichotomies on both half-lines for $j = 1,\ldots,M$. By applying roughness and pasting techniques, these transfer to exponential dichotomies of the first-order system
\begin{align} \label{variationalsysfull3}
U' = \left(\mathcal{A}(x,u_n(x);\lambda) - \omega'(x)\right) U,
\end{align}
associated with the eigenvalue problem $(\El(u_n) - \lambda) u = 0$ about the multifront $u_n$, on the intervals $(-\infty,nT]$, $[MnT,\infty)$, and $[jnT,(j+1)nT]$ for $j = 1,\ldots,M-1$. Here, $\omega \colon \R \to \R$ is a suitably chosen concatenation of the weight functions $\omega_{\eta_{j-1},\eta_j}$ for $j = 1,\ldots,M$, see Figure~\ref{fig:Stability_multifront_proof}(bottom).

Given an exponential dichotomy of~\eqref{variationalsysfull3} on an interval $\mathcal{I}$, the key idea is to use Lemma~\ref{l:projlem} to write the associated projection
$P(x;\lambda)$ as
\begin{align*}
P(x;\lambda) = \mathcal{B}(x;\lambda) \left(\Theta(x;\lambda)^\top \mathcal{B}(x;\lambda)\right)^{-1} \Theta(x;\lambda)^\top
\end{align*}
where $\mathcal{B}(x;\lambda)$ and $\Theta(x;\lambda)$ are matrices whose columns constitute a basis of $\mathrm{ran}(P(x;\lambda))$ and $\mathrm{ran}(P(x;\lambda)^\top)$, respectively. We demonstrate that
\begin{align*}
\Pi(x,y;\lambda) := \left(\Theta(x;\lambda)^\top \mathcal{B}(x;\lambda)\right)^{-1} \Theta(x;\lambda)^\top \mathcal{T}(x,y;\lambda) \mathcal{B}(y;\lambda) \left(\Theta(y;\lambda)^\top \mathcal{B}(y;\lambda)\right)^{-1}
\end{align*}
is invertible for each $x,y \in \mathcal{I}$, where $\mathcal{T}(x,y;\lambda)$ is the evolution of system~\eqref{variationalsysfull3}. Thus, given matrices $\mathcal{X},\mathcal{Y}$, we can write
\begin{align} \label{detform}
\begin{split}
\mathcal{X} \mathcal{T}(x,y;\lambda) \mathcal{Y} &= \mathcal{X} P(x;\lambda) \mathcal{T}(x,y;\lambda) P(y;\lambda) \mathcal{Y} + \mathcal{X} \mathcal{T}(x,y;\lambda) (I-P(y;\lambda)) \mathcal{Y}\\ 
&= \mathcal{X} \mathcal{B}(x;\lambda) \Pi(x,y;\lambda) \Theta(y;\lambda)^\top \mathcal{Y} + \mathcal{X}  \mathcal{T}(x,y;\lambda) (I-P(y;\lambda)) \mathcal{Y}
\end{split}
\end{align}
for $x,y \in \mathcal{I}$. We then apply the formula~\eqref{detblockid} to the determinant of expressions of the form~\eqref{detform} with $A = \Pi(x,y;\lambda)^{-1}$, $B = \Theta(y;\lambda)^\top \mathcal{Y}$, $C = \mathcal{X} \mathcal{B}(x;\lambda)$ and $D = \mathcal{X}  \mathcal{T}(x,y;\lambda) (I-P(y;\lambda)) \mathcal{Y}$. 

Specifically, we show that the Evans function $E_n$ can be expressed, up to an invertible analytic factor, as $\det(\mathcal{X}_0(\lambda) \mathcal{T}(nT,MnT;\lambda)\mathcal{Y}_0(\lambda))$, where $\mathcal{X}_0(\lambda)$ and $\mathcal{Y}_0(\lambda)$ are appropriately chosen matrices. By applying~\eqref{detblockid} inductively, we obtain that $E_n$ is, up to a nonzero analytic factor, equal to the determinant of a matrix that becomes an upper triangular block matrix as $n \to \infty$. The determinants of the diagonal blocks can be identified with the Evans functions $\mathcal{E}_1,\ldots,\mathcal{E}_M$. An application of Rouch\'e's theorem then yields the desired approximation of the zeros of $E_n$ by those of the product $\mathcal{E} = \mathcal{E}_1\cdot\ldots\cdot\mathcal{E}_M$.

\begin{proof}[Proof of Theorem~\ref{thm:instability_multifront}] This proof is structured as follows. We begin with establishing exponential dichotomies for the weighted eigenvalue problems~\eqref{variationalsys3} along the primary fronts. Then, we define a suitable weight function $\omega$ and transfer these exponential dichotomies to the unweighted and weighted eigenvalue problems~\eqref{variationalsysfull} and~\eqref{variationalsysfull3} along the multifront. Subsequently, we define an Evans function $E_n$ associated with~\eqref{variationalsysfull}. Next, we inductively establish a leading-order factorization of $E_n$,  relating it to the product $\mathcal{E} = \mathcal{E}_1\cdot\ldots\cdot\mathcal{E}_M$. Finally, the result follows by an application of Rouch\'e's theorem. 

\paragraph*{Exponential dichotomies along the primary fronts.} We list some consequences of Propositions~\ref{prop:essential_spec1front} and~\ref{prop:Evans1front}. First of all, system~\eqref{variationalsys3} has for all $\lambda \in \mathcal{U}$ and each $j = 1,\ldots,M$ exponential dichotomies on $\R_\pm$ with projections $P_{j,\pm}(\pm x;\lambda), x \geq 0$ of rank $l_0$, where $l_0$ is the $\lambda$- and $j$-independent number of Floquet exponents in the open left-half plane (counted with algebraic multiplicity) of the systems~\eqref{variationalbackground2} and~\eqref{variationalbackground3}. Moreover, there exist analytic functions $B_{j,\mathrm{s}} \colon \mathcal{U} \to \C^{km \times l_0}$ and $B_{j,\mathrm{u}} \colon \mathcal{U} \to \C^{km \times (km-l_0)}$ such that $B_{j,\mathrm{s}}(\lambda)$ is a basis of $\mathrm{ran}(P_{j,+}(0;\lambda))$ and $B_{j,\mathrm{u}}(\lambda)$ is a basis of $\ker(P_{j,-}(0;\lambda))$ for all $\lambda \in \mathcal{U}$ and $j = 1,\ldots,M$. The associated Evans function $\mathcal{E}_j \colon \mathcal{U} \to \C$ is given by
\begin{align*}\mathcal{E}_j(\lambda) = \det(B_{j,\mathrm{u}}(\lambda) \mid B_{j,\mathrm{s}}(\lambda))\end{align*}
for $j = 1,\ldots,M$. Because $\mathcal{E} = \mathcal{E}_1 \cdot\ldots \cdot\mathcal{E}_M$ is analytic, there exists a closed disk $\smash{\overline{B}_{\lambda_0}}(\varrho_2) \subset \mathcal{U}$ of some radius $\varrho_2 > 0$ such that $\lambda_0$ is the only root of $\mathcal{E}$ in $\smash{\overline{B}_{\lambda_0}}(\varrho_2)$. Finally, there exist constants $K_0,\mu_0,\tau_0 > 0$ such that~\eqref{variationalsys3} possesses for $\lambda \in \smash{\overline{B}_{\lambda_0}}(\varrho_2)$ and $j = 1,\ldots,M$ exponential dichotomies on $\R_\pm$ with constants $K_0,\mu_0 > 0$ and projections $\smash{\tilde P_{j,\pm}}(\pm x;\lambda), x \geq 0$ satisfying~\eqref{projest2} for each $x \geq \tau_0$. By uniqueness of the range of $P_{j,+}(0;\lambda)$ and the kernel of $P_{j,-}(0;\lambda)$, cf.~\cite[p.~19]{Coppel1978}, $B_{j,\mathrm{s}}(\lambda)$ is a basis of $\mathrm{ran}(P_{j,+}(0;\lambda)) = \mathrm{ran}(\smash{\tilde P_{j,+}}(0;\lambda))$ and $B_{j,\mathrm{u}}(\lambda)$ is a basis of $\ker(P_{j,-}(0;\lambda)) = \ker(\smash{\tilde P_{j,-}}(0;\lambda))$ for each $\lambda \in \smash{\overline{B}_{\lambda_0}}(\varrho_2)$ and $j = 1,\ldots,M$.

Since $\mathcal{A}$ is $T$-periodic in $x$, system
\begin{align} \label{variationalsys32}
U' = \left(\A\left(x,Z_j(x-jnT);\lambda \right) - \omega_{\eta_{j-1},\eta_j}'(x-jnT) \right) U
\end{align}
is, for each $n \in \N$, a $jnT$-translation of system~\eqref{variationalsys3} for $j = 1,\ldots,M$. So, it admits for each $\lambda \in \smash{\overline{B}_{\lambda_0}}(\varrho_2)$ and $j = 1,\ldots,M$ exponential dichotomies on the half lines $(-\infty,jnT]$ and $[jnT,\infty)$ with constants $K_0,\mu_0$ and projections $\smash{\check{P}_{j,\pm}(jnT \pm x;\lambda) = \tilde{P}_{j,\pm}(\pm x;\lambda)}, x \geq 0$. 

\paragraph*{Weighted eigenvalue problem.} Let $\tilde{\omega} \colon \R \to \R$ be given by
\begin{align*}
\tilde{\omega}(x) = \begin{cases} \omega_{\eta_0,\eta_1}'(x-nT), & x \in \left(-\infty,\tfrac{3}{2}nT\right],\\
\omega_{\eta_{j-1},\eta_j}'(x - jnT), & x \in \left((j-\tfrac12)nT,(j+\tfrac12)nT\right], \, j = 2,\ldots,M-1,\\
\omega_{\eta_{M-1},\eta_M}'(x - MnT), & x \in \left((M-\tfrac12)nT,\infty\right),\end{cases}
\end{align*}
see Figure~\ref{fig:Stability_multifront_proof}(bottom). By definition of the exponential weights $\omega_{\eta_{j-1},\eta_j}, j = 1,\ldots,M$, see~\S\ref{sec:Abstract_ex_stab}, the function $\tilde{\omega}$ is smooth and has support within the interval $(nT-1,MnT+1)$. Therefore, its primitive $\omega \colon \R \to \R$, given by
$$\omega(x) = \int_0^x \tilde{\omega}(y) \de y,$$
is smooth and bounded. 

Denote by $T_n(x,y;\lambda)$ and $\mathcal{T}_n(x,y;\lambda)$ the evolutions of system~\eqref{variationalsysfull} and~\eqref{variationalsysfull3}, respectively. Since $\omega$ is bounded and it holds
\begin{align*}
T_n(x,y;\lambda) = \eu^{\omega(x) - \omega(y)} \mathcal{T}_n(x,y;\lambda)
\end{align*}
for $x,y \in \R$ and $\lambda \in \C$, system~\eqref{variationalsysfull} possesses an exponential dichotomy on an interval $\mathcal{I} \subset \R$ with projections $P_n(x;\lambda)$ if and only if system~\eqref{variationalsysfull3} does.

\begin{figure}[t]
    \centering
    \includegraphics[width=0.99\textwidth]{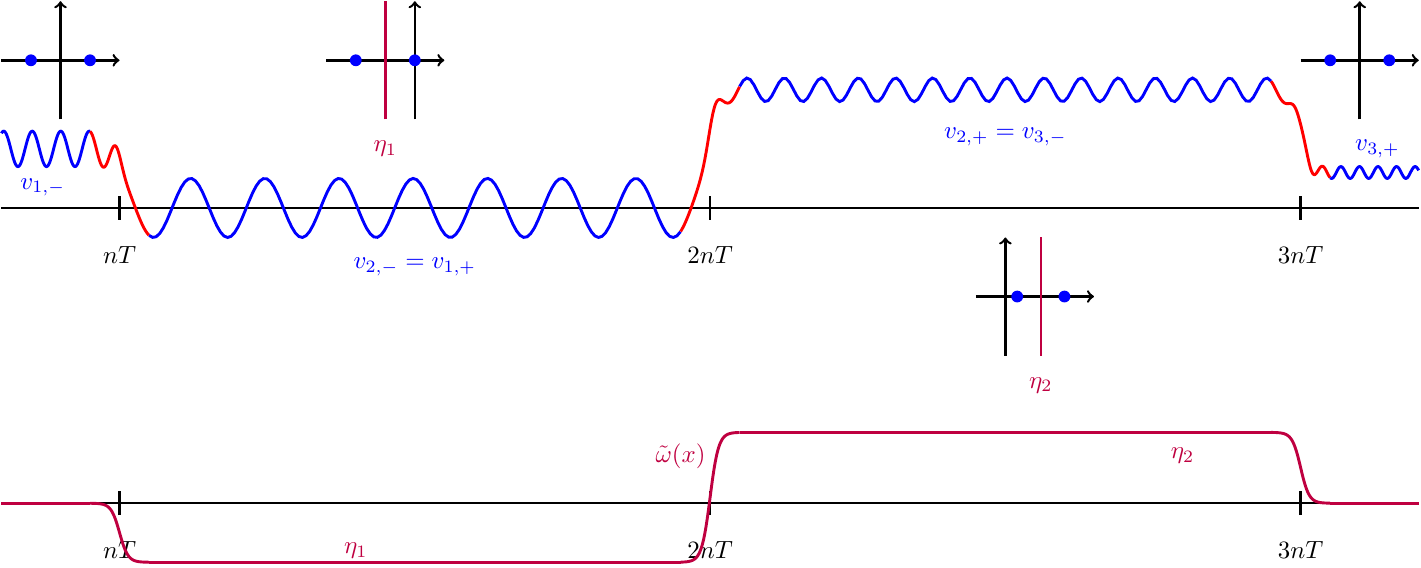}
    \caption{Illustration of a stationary 3-front solution (top), with insets showing the Floquet exponents (blue dots) associated with the periodic end states $v_{j,\pm}$ for $j=1,2,3$, depicted in blue. The derivative of the corresponding weight $\omega$, denoted by $\tilde{\omega}$, is shown in purple (bottom).}
    \label{fig:Stability_multifront_proof}
\end{figure}

\paragraph*{Exponential dichotomies along the multifront.} Our next step is to use roughness and pasting techniques to transfer the exponential dichotomies of system~\eqref{variationalsys32} to exponential dichotomies for system~\eqref{variationalsysfull3} (and thus for system~\eqref{variationalsysfull}) on the intervals $(-\infty,\frac{3}{2}nT]$, $[(M-\frac12)nT,\infty)$ and $[(j-\frac12) nT,(j+\frac12)nT]$ for $j = 2,\ldots,M-1$ (only in case $M > 2$). 

Let $\tilde\chi_{j,n} \colon \R \to [0,1]$ be a family of smooth cut-off functions, where $\tilde\chi_{1,n}$ is supported on $(-\infty,\frac32 nT + 2)$ and equal to $1$ on $(-\infty,\frac32 nT+1]$, $\tilde\chi_{M,n}$ is supported on $((M - \frac12)nT - 2, \infty)$ and equal to $1$ on $[(M-\frac12)nT-1,\infty)$, and $\tilde\chi_{j,n}$ is supported on $((j-\frac12) nT - 2, (j+\frac12) nT + 2)$ and equal to $1$ on $[(j-\frac12)nT-1,(j+\frac12)nT+1]$ for $j = 2,\ldots,M-1$ (only in case $M > 2$). We define
\begin{align*}
Z_{j,n} &= a_n + (1-\tilde\chi_{j,n}) Z_j(\cdot-jnT) + \tilde\chi_{j,n} w_n
\end{align*}
for $j = 1,\ldots,M$, where we recall $w_n$ is the formal concatenation of the $M$ primary fronts, defined in~\eqref{multifront_ansatz}. Using that $\partial_u \mathcal{A}$ is continuous, $\smash{\overline{B}_{\lambda_0}}(\varrho_2)$ is compact and we have $a_n \in H^1(\R) \hookrightarrow L^\infty(\R)$ and $Z_\ell \in L^\infty(\R)$ for $\ell = 1,\ldots,M$, the mean value theorem yields a $\lambda$- and $n$-independent constant $R > 0$ such that the estimate
\begin{align} \label{existestsys3}
\left\|\A\left(x,Z_{j,n}(x);\lambda \right) - \A\left(x,Z_j(x-jnT);\lambda \right)\right\| \leq R \delta_n
\end{align}
holds for $x \in \R$, $\lambda \in \smash{\overline{B}_{\lambda_0}}(\varrho_2)$ and $j = 1,\ldots,M$, where we denote
\begin{align*}
\delta_n = \|a_n\|_{L^\infty} + \sum_{j = 1}^M \left( \|\chi_-(Z_j-v_{j,-})\|_{L^\infty\left((-\infty,-\frac12 nT + 2]\right)} + \|\chi_+(Z_j-v_{j,+})\|_{L^\infty\left([\frac12 nT-2,\infty)\right)}\right).
\end{align*}

Using the embedding $H^1(\R) \hookrightarrow L^\infty(\R)$, one readily observes that $\delta_n$ converges to $0$ as $n \to \infty$. Thus, since for all $\lambda \in \smash{\overline{B}_{\lambda_0}}(\varrho_2)$  system~\eqref{variationalsys32} has exponential dichotomies on $(-\infty,jnT]$ and $[jnT,\infty)$ with $\lambda$- and $n$-independent constants $K_0,\mu_0$ and  projections $\smash{\check{P}_{j,\pm}(jnT\pm x;\lambda) = \tilde{P}_{j,\pm}(\pm x;\lambda)}, x \geq 0$, the estimate~\eqref{existestsys3} and Lemma~\ref{l:projanarough} give rise to $\lambda$- and $n$-independent constants $M_1,\mu_1 > 0$ and $\varrho_1 \in (0,\varrho_2)$ such that, provided $n \in \mathbb{N}$ is sufficiently large, the following two statements hold for $j = 1,\ldots,M$. First, the system
\begin{align} \label{variationalcutoff2}
     U' = \left(\mathcal{A}\left(x,Z_{j,n}(x);\lambda\right) - \omega_{\eta_{j-1},\eta_j}'(x)\right)  U
\end{align}
has exponential dichotomies on $(-\infty,jnT]$ and $[jnT,\infty)$ with $\lambda$- and $n$-independent constants and projections $\mathcal{Q}_{j,\pm,n}(jnT \pm x;\lambda), x \geq 0$ for each $\lambda \in \smash{\overline{B}_{\lambda_0}}(\varrho_1)$. Second, the map $\mathcal{Q}_{j,\pm,n}(jnT \pm x;\cdot) \colon B_{\lambda_0}(\varrho_1) \to \C^{km \times km}$ is analytic for each $x \geq 0$ and the estimates
\begin{align} \label{projest3multi}
\begin{split}
\|\mathcal{Q}_{j,-,n}((j-\tfrac12) n T;\lambda) - \tilde{P}_{j,-}(-\tfrac{n}{2} T;\lambda)\| &\leq M_1\left(\delta_n + \eu^{-\mu_1 n T}\right),\\
\|\mathcal{Q}_{j,+,n}((j+\tfrac12) nT;\lambda)-\tilde{P}_{j,+}(\tfrac{n}{2} T;\lambda)\|  &\leq M_1\left(\delta_n + \eu^{-\mu_1 n T}\right),\\
\left\|\mathcal{Q}_{j,\pm,n}(jnT;\lambda) - Q_{j,\pm}(\lambda)\right\| &\leq M_1\delta_n
\end{split}
\end{align}
hold for all $\lambda \in \overline{B}_{\lambda_0}(\varrho_1)$, where we denote by $Q_{j,+}(\lambda)$ the projection onto $\mathrm{ran}(P_{j,+}(0;\lambda))$ along $\mathrm{ran}(P_{j,+}(0;\lambda_0))^\perp$ and $Q_{j,-}(\lambda)$ is the projection onto $\mathrm{ker}(P_{j,-}(0;\lambda_0))^\perp$ along $\mathrm{ker}(P_{j,-}(0;\lambda))$. By Lemma~\ref{l:proj_ana}, the maps $Q_{j,\pm} \colon B_{\lambda_0}(\varrho_1) \to \C^{km \times km}$ are analytic for $j = 1,\ldots,M$. 

Take $j \in \{1,\ldots,M\}$. Owing to~\cite[Section~II.4.2]{Kato1995} there exist analytic maps $\Phi_j \colon B_{\lambda_0}(\varrho_1) \to \C^{km \times l_0}$ and $\Psi_j \colon B_{\lambda_0}(\varrho_1) \to \C^{km \times (km - l_0)}$ such that $\Phi_j(\lambda)$ is a basis of $\smash{\mathrm{ran}(Q_{j,-}(\lambda)^\top)}$ and $\Psi_j(\lambda)$ is a basis of $\smash{\mathrm{ran}(I - Q_{j,-}(\lambda)^\top) = \ker(Q_{j,-}(\lambda)^\top)}$ for $\lambda \in B_{\lambda_0}(\varrho_1)$. Since $B_{j,\mathrm{u}}(\lambda)$ is a basis of $\ker(Q_{j,-}(\lambda))$, we have $\smash{\Phi_j(\lambda)^\top} B_{j,\mathrm{u}}(\lambda) = 0$ for $\lambda \in B_{\lambda_0}(\varrho_1)$.
Therefore, we arrive at
\begin{align*}
\det\left(\left(\Psi_j(\lambda) \mid \Phi_j(\lambda)\right)^\top\right) \mathcal{E}_j(\lambda) = \det\left(\Psi_j(\lambda)^\top B_{j,\mathrm{u}}(\lambda)\right)\det\left(\Phi_j(\lambda)^\top B_{j,\mathrm{s}}(\lambda)\right) 
\end{align*}
for $\lambda \in B_{\lambda_0}(\varrho_1)$. Clearly, the matrix $\left(\Psi_j(\lambda) \mid \Phi_j(\lambda)\right)$ is invertible for all $\lambda \in B_{\lambda_0}(\varrho_1)$, and so is $\Psi_j(\lambda)^\top B_{j,\mathrm{u}}(\lambda)$ by Lemma~\ref{l:projlem}. We conclude that $\mathcal{E}_j$ has the same zeros (including multiplicities) in $B_{\lambda_0}(\varrho_1)$ as the analytic function $\smash{\tilde{\mathcal{E}}_j} \colon B_{\lambda_0}(\varrho_1) \to \C$ given by
\begin{align*}
\tilde{\mathcal{E}}_j(\lambda) = \det\left(\Phi_j(\lambda)^\top B_{j,\mathrm{s}}(\lambda)\right).
\end{align*}

Take $j \in \{1,\ldots,M-1\}$. We invoke Lemma~\ref{l:projest2}, use the $T$-periodicity of $Q_\pm(\cdot;\lambda)$ and employ estimates~\eqref{projest2} and~\eqref{projest3multi}, to conclude that, provided $n \in \mathbb{N}$ is sufficiently large, the subspaces $\ker(\mathcal{Q}_{j,+,n}((j+\frac12) n T;\lambda))$ and $\mathrm{ran}(\mathcal{Q}_{j+1,-,n}((j+\frac12) nT;\lambda))$ are complementary and there exists a $\lambda$- and $n$-independent constant $M_2 > 0$ such that the projection $\smash{\check{\mathcal{P}}_{j,n}(\lambda)}$ onto $\mathrm{ran}(\mathcal{Q}_{j+1,-,n}((j+\frac12) n T;\lambda))$ along $\ker(\mathcal{Q}_{j,+,n}((j+\frac12) n T;\lambda))$ is well-defined and enjoys the bound
\begin{align} \label{projest4multi}
\left\|\check{\mathcal{P}}_{j,n}(\lambda)\right\| \leq M_2
\end{align}
for $\lambda \in \overline{B}_{\lambda_0}(\varrho_1)$. Since $\mathcal{Q}_{j,+,n}((j+\tfrac12) n T;\cdot), \mathcal{Q}_{j+1,-,n}((j+\tfrac12) n T;\cdot) \colon B_{\lambda_0}(\varrho_1) \to \C^{km \times km}$ are analytic, Lemma~\ref{l:proj_ana} affords analytic maps $\mathcal{B}_{j,\pm,n} \colon B_{\lambda_0}(\varrho_1) \to \C^{km \times l_0}$ such that $\mathrm{ran}(\mathcal{Q}_{j+1,-,n}((j+\frac12) n T;\lambda)) = \mathrm{ran}(\mathcal{B}_{j,-,n}(\lambda))$ and $\ker(\mathcal{Q}_{j,+,n}((j+\frac12) n T;\lambda)) = \{u \in \C^{km} : z^\top u = 0 \text{ for all } z \in \mathrm{ran}(\mathcal{B}_{j,+,n}(\lambda))\}$. Therefore, Lemma~\ref{l:proj_ana} implies that $\smash{\check{\mathcal{P}}_{j,n}} \colon B_{\lambda_0}(\varrho_1) \to \C^{km \times km}$ is analytic. 

By construction of the cut-off functions $\tilde\chi_{j,n}$ and the weight $\omega$, system~\eqref{variationalsysfull3} coincides with~\eqref{variationalcutoff2} on $(-\infty,\frac{3}{2} n T]$ for $j = 1$, on $[(M-\frac12) nT,\infty)$ for $j = M$ and on $[(j-\frac{1}{2}) nT,(j+\frac{1}{2}) nT]$ for $j = 2,\ldots,M-1$ (only in case $M > 2$). Hence, it follows, thanks to Lemma~\ref{l:pastingexpdi} and estimate~\eqref{projest4multi}, that system~\eqref{variationalsysfull3} admits for all $\lambda \in \smash{\overline{B}_{\lambda_0}}(\varrho_1)$ an exponential dichotomy on $(-\infty,nT]$ with $\lambda$- and $n$-independent constants and projections $\mathcal{Q}_{1,-,n}(x;\lambda)$, an exponential dichotomy on $[MnT,\infty)$ with $\lambda$- and $n$-independent constants and projections $\mathcal{Q}_{M,+,n}(x;\lambda)$, and an exponential dichotomy on $[jnT,(j+1) nT]$ with $\lambda$- and $n$-independent constants and projections 
\begin{align*}
\mathcal{Q}_{j,\circ,n}(x;\lambda) = \mathcal{T}_n(x,(j+\tfrac12) nT;\lambda) \smash{\check{\mathcal{P}}_{j,n}}(\lambda) \mathcal{T}_n((j+\tfrac12) n T,x;\lambda), 
\end{align*}
for $j = 1,\ldots,M-1$. In addition, there exist $\lambda$- and $n$-independent constants $M_3,\mu_3 > 0$ such that, provided $n \in \N$ is sufficiently large, the projections obey
\begin{align*}
\left\|\mathcal{Q}_{j,\circ,n}(jnT;\lambda) - \mathcal{Q}_{j,+,n}(jnT;\lambda)\right\| &\leq M_3 \eu^{-\mu_3 nT},\\
\left\|\mathcal{Q}_{j,\circ,n}((j+1)nT;\lambda) - \mathcal{Q}_{j+1,-,n}((j+1)nT;\lambda)\right\| &\leq M_3 \eu^{-\mu_3 nT}
\end{align*}
for $\lambda \in \overline{B}_{\lambda_0}(\varrho_1)$ and $j = 1,\ldots,M-1$. Combining the latter with~\eqref{projest3multi}, we arrive at
\begin{align} \label{projest5multi}
\left\|\mathcal{Q}_{j,\circ,n}(jnT;\lambda) - Q_{j,+}(\lambda)\right\|, \left\|\mathcal{Q}_{j,\circ,n}((j+1)nT;\lambda) - Q_{j+1,-}(\lambda)\right\| \leq M_3 \eu^{-\mu_3 nT} + M_1 \delta_n
\end{align}
for $\lambda \in \overline{B}_{\lambda_0}(\varrho_1)$ and $j = 1,\ldots,M-1$. Finally, we observe that $\mathcal{Q}_{j,\circ,n}(x;\cdot) \colon B_{\lambda_0}(\varrho_1) \to \C^{km \times km}$ is analytic, since $\smash{\check{\mathcal{P}}_{j,n}}$ and $\mathcal{T}_n(x,y;\cdot)$ are analytic for $x,y \in [jnT,(j+1)nT]$ and $j = 1,\ldots,M-1$, cf.~\cite[Lemma~2.1.4]{KapitulaPromislow2013}.

We define the analytic maps $\Xi_{0,n}, \Xi_{j,n}, \mathcal{B}_{j,n}, \mathcal{B}_{M,n} \colon B_{\lambda_0}(\varrho_1) \to \C^{km \times l_0}$ by
\begin{align*}
\Xi_{0,n}(\lambda) = \mathcal{Q}_{1,-,n}(nT;\lambda)^\top \Phi_{1}(\lambda),
\qquad
\Xi_{j,n}(\lambda) = \mathcal{Q}_{j,\circ,n}((j+1)nT;\lambda)^\top \Phi_{j+1}(\lambda)
\end{align*}
and
\begin{align*}
\mathcal{B}_{j,n}(\lambda) = \mathcal{Q}_{j,\circ,n}(jnT;\lambda) B_{j,\mathrm{s}}(\lambda), \qquad
\mathcal{B}_{M,n}(\lambda) = \mathcal{Q}_{M,+,n}(MnT;\lambda) B_{M,\mathrm{s}}(\lambda)
\end{align*}
for $j = 1,\ldots,M-1$. Let $\varrho_0 \in (0,\varrho_1)$. Employing estimates~\eqref{projest3multi} and~\eqref{projest5multi}, using the compactness of $\overline{B}_{\lambda_0}(\varrho_0)$ and recalling that $\Phi_{j}$ and $B_{j,\mathrm{s}}$ are analytic on $B_{\lambda_0}(\varrho_1)$, we obtain a $\lambda$- and $n$-independent constant $M_4 > 0$ such that, provided $n \in \N$ is sufficiently large, we have
\begin{align} \label{projestmulti}
\begin{split}
\left\|\Xi_{j-1,n}(\lambda) - \Phi_j(\lambda)\right\|, \left\|\mathcal{B}_{j,n}(\lambda) - B_{j,\mathrm{s}}(\lambda)\right\| &\leq M_4 \left(\eu^{-\mu_4 nT} + \delta_n\right), \\
\left\|\Xi_{j-1,n}(\lambda)\right\|, \left\|\mathcal{B}_{j,n}(\lambda) \right\| &\leq M_4,
\end{split}
\end{align}
for all $\lambda \in \overline{B}_{\lambda_0}(\varrho_0)$ and $j = 1,\ldots,M$.  Because $\Phi_{j}(\lambda)$ and $B_{j,\mathrm{s}}(\lambda)$ are bases, the estimate~\eqref{projestmulti} implies that, provided $n \in \N$ is sufficiently large, $\Xi_{0,n}(\lambda)$ forms a basis of $\smash{\mathrm{ran}(\mathcal{Q}_{1,-,n}(nT;\lambda)^\top)}$, $\Xi_{j,n}(\lambda)$ is a basis of $\smash{\mathrm{ran}(\mathcal{Q}_{j,\circ,n}((j+1)T;\lambda)^\top)}$, $\mathcal{B}_{j,n}(\lambda)$ is a basis of $\mathrm{ran}(\mathcal{Q}_{j,\circ,n}(jnT;\lambda))$, and $\mathcal{B}_{M,n}(\lambda)$ is a basis of $\mathrm{ran}(\mathcal{Q}_{M,+,n}(MnT;\lambda))$ for $j = 1,\ldots,M-1$ and $\lambda \in \overline{B}_{\lambda_0}(\varrho_0)$. 

By~\cite[Section~II.4.2]{Kato1995} there exist analytic maps $\tilde{\Xi}_{j,n}, \tilde{\mathcal{B}}_{j,n} \colon B_{\lambda_0}(\varrho_1) \to \C^{km \times l_0}$ such that $\tilde{\Xi}_{j,n}(\lambda)$ forms a basis of $\mathrm{ran}(\mathcal{Q}_{j,\circ,n}(jnT;\lambda)^\top)$ and $\tilde{\mathcal{B}}_{j,n}(\lambda)$ is a basis of $\mathrm{ran}(\mathcal{Q}_{j,\circ,n}((j+1)nT;\lambda))$ for $\lambda \in B_{\lambda_0}(\varrho_1)$ and $j = 1,\ldots,M-1$. By Lemma~\ref{l:projlem}, the matrices $\tilde{\Xi}_{j,n}(\lambda)^\top \mathcal{B}_{j,n}(\lambda)$ and $\Xi_{j,n}(\lambda)^\top \tilde{\mathcal{B}}_{j,n}(\lambda)$ are invertible for $\lambda \in \overline{B}_{\lambda_0}(\varrho_0)$ and $j = 1,\ldots,M-1$. Hence, one readily verifies that it must hold
\begin{align} \label{e:projids}
\begin{split}
\mathcal{Q}_{j,\circ,n}(jnT;\lambda) &= \mathcal{B}_{j,n}(\lambda)\left(\tilde{\Xi}_{j,n}(\lambda)^\top \mathcal{B}_{j,n}(\lambda)\right)^{-1}\tilde{\Xi}_{j,n}(\lambda)^\top, \\ \mathcal{Q}_{j,\circ,n}((j+1)nT;\lambda) &= \tilde{\mathcal{B}}_{j,n}(\lambda)\left(\Xi_{j,n}(\lambda)^\top \tilde{\mathcal{B}}_{j,n}(\lambda)\right)^{-1}\Xi_{j,n}(\lambda)^\top 
\end{split}
\end{align}
for $\lambda \in \overline{B}_{\lambda_0}(\varrho_0)$ and $j = 1,\ldots,M-1$. Therefore, defining $\Pi_{j,n}, \Theta_{j,n} \colon B_{\lambda_0}(\varrho_1) \to \C^{l_0 \times l_0}$ by
\begin{align*}
\Pi_{j,n}(\lambda) = \left(\tilde{\Xi}_{j,n}(\lambda)^\top \mathcal{B}_{j,n}(\lambda)\right)^{-1} \tilde{\Xi}_{j,n}(\lambda)^\top \mathcal{T}_n(jnT,(j+1)nT;\lambda) \tilde{\mathcal{B}}_{j,n}(\lambda) \left(\Xi_{j,n}(\lambda)^\top \tilde{\mathcal{B}}_{j,n}(\lambda)\right)^{-1}
\end{align*}
and
\begin{align*}
\Theta_{j,n}(\lambda) = \Xi_{j,n}(\lambda)^\top \mathcal{T}_n((j+1)nT,jnT;\lambda)) \mathcal{B}_{j,n}(\lambda),
\end{align*}
we deduce
\begin{align*}
 \Theta_{j,n}(\lambda) \Pi_{j,n}(\lambda) &= \Xi_{j,n}(\lambda)^\top \mathcal{Q}_{j,\circ,n}((j+1)nT;\lambda) \mathcal{T}_n((j+1)nT,jnT;\lambda) \\ 
&\qquad \cdot \mathcal{T}_n(jnT,(j+1)nT;\lambda) \tilde{\mathcal{B}}_{j,n}(\lambda) \left(\Xi_{j,n}(\lambda)^\top \tilde{\mathcal{B}}_{j,n}(\lambda)\right)^{-1} = I
\end{align*}
and, similarly,
\begin{align*}
&\Pi_{j,n}(\lambda) \Theta_{j,n}(\lambda) = I
\end{align*}
for $\lambda \in \overline{B}_{\lambda_0}(\varrho_0)$ and $j = 1,\ldots,M-1$. We conclude that $\Theta_{j,n}(\lambda)$ is invertible with inverse $\Pi_{j,n}(\lambda)$ for $\lambda \in \overline{B}_{\lambda_0}(\varrho_0)$ and $j = 1,\ldots,M-1$. Moreover, $\Phi_{j,n}$ and $\Theta_{j,n}$ are analytic for $j = 1,\ldots,M-1$, since $\Xi_{j,n}$, $\mathcal{B}_{j,n}$, $\smash{\tilde{\Xi}}_{j,n}$, $\smash{\tilde{\mathcal{B}}}_{j,n}$ and $\mathcal{T}_n(x,y;\cdot)$ are analytic for $x,y \in \R$ by~\cite[Lemma~2.1.4]{KapitulaPromislow2013}.

By~\cite[p.~13]{Coppel1978}, we can extend the exponential dichotomy of system~\eqref{variationalsysfull3} on $[MnT,\infty)$ to an exponential dichotomy on $[nT,\infty)$ with projections $\mathcal{Q}_{M,+,n}(x;\lambda), x \geq nT$ by setting
\begin{align*}
\mathcal{Q}_{M,+,n}(x;\lambda) = \mathcal{T}_n(x,MnT;\lambda) \mathcal{Q}_{M,+,n}(MnT;\lambda)\mathcal{T}_n(MnT,x;\lambda)\end{align*}  
for each $x \in [nT,MnT]$ and $\lambda \in B_{\lambda_0}(\varrho_1)$. As noted before, this implies that the unweighted system~\eqref{variationalsysfull} also admits exponential dichotomies on the half lines $(-\infty,nT]$ and $[nT,\infty)$ with projections $\mathcal{Q}_{1,-,n}(x;\lambda), x \leq nT$ and $\mathcal{Q}_{M,+,n}(x;\lambda), x \geq nT$, respectively.

\paragraph*{The Evans function for the multifront.} By~\cite[Section~II.4.2]{Kato1995}, there exist holomorphic functions $ \mathcal{U}_n, \Upsilon_n \colon B_{\lambda_0}(\varrho_1) \to \C^{km \times (km-l_0)}$ such that $\mathcal{U}_n(\lambda)$ forms a basis of $\ker(\mathcal{Q}_{1,-,n}(nT;\lambda)) = \mathrm{ran}(I-\mathcal{Q}_{1,-,n}(nT;\lambda))$ and $\Upsilon_n(\lambda)$ is a basis of $\mathrm{ran}(I - \mathcal{Q}_{1,-,n}(nT;\lambda)^\top)$ for each $\lambda \in B_{\lambda_0}(\varrho_1)$. Upon defining $\mathcal{S}_n \colon B_{\lambda_0}(\varrho_1) \to \C^{km \times l_0}$ by $\mathcal{S}_n(\lambda) = \mathcal{T}_n(nT,MnT;\lambda) \mathcal{B}_{M,n}(\lambda)$, we find that $\mathcal{S}_n(\lambda)$ is a basis of $\mathrm{ran}(\mathcal{Q}_{M,+,n}(nT;\lambda))$ for each $\lambda \in \overline{B}_{\lambda_0}(\varrho_0)$, because $\mathcal{B}_{M,n}(\lambda)$ is a basis of $\mathrm{ran}(\mathcal{Q}_{M,+,n}(MnT;\lambda))$. Moreover, $\mathcal{S}_n$ is analytic, since $\mathcal{B}_{M,n}$ and $\mathcal{T}_n(nT,MnT;\cdot)$ are, cf.~\cite[Lemma~2.1.4]{KapitulaPromislow2013}. Thus, an analytic Evans function $E_n \colon B_{\lambda_0}(\varrho_1) \to \C$ for system~\eqref{variationalsysfull} is given by
\begin{align*}
E_n(\lambda) = \det(\mathcal{U}_n(\lambda) \mid \mathcal{S}_n(\lambda)).
\end{align*}

\paragraph*{Leading-order factorization of the Evans function.}
Our next step is to multiply $E_n(\lambda)$ with several nonzero analytic functions in order to relate it to $\mathcal{E} = \mathcal{E}_1 \cdot\ldots \cdot\mathcal{E}_M$, where $\mathcal{E}_j$ is the Evans function associated with system~\eqref{variationalsys3}.

First, noting that $\Xi_{0,n}(\lambda)^\top \mathcal{U}_n(\lambda) = 0$, we compute
\begin{align*}
\det\left(\left(\Upsilon_{n}(\lambda) \mid \Xi_{0,n}(\lambda)\right)^\top\right) E_n(\lambda) = \det\left(\Upsilon_{n}(\lambda)^\top \mathcal{U}_n(\lambda)\right) \det\left(\Xi_{0,n}(\lambda)^\top \mathcal{S}_n(\lambda) \right)
\end{align*}
for $\lambda \in \overline{B}_{\lambda_0}(\varrho_0)$. Since the matrices $\left(\Upsilon_{n}(\lambda) \mid \Xi_{0,n}(\lambda)\right)$ and $\Upsilon_{n}(\lambda)^\top \mathcal{U}_n(\lambda)$ are invertible for all $\lambda \in \overline{B}_{\lambda_0}(\varrho_0)$ by Lemma~\ref{l:projlem}, the Evans function $E_n$ possesses the same roots (including multiplicities) in $B_{\lambda_0}(\varrho_0)$ as the analytic function $\smash{\tilde{E}_n} \colon B_{\lambda_0}(\varrho_1) \to \C$ given by
\begin{align} \label{e:firststep}
\tilde{E}_n(\lambda) = \det\left(\Xi_{0,n}(\lambda)^\top \mathcal{S}_n(\lambda) \right) = \det\left(\Xi_{0,n}(\lambda)^\top \mathcal{T}_n(nT,MnT;\lambda) \mathcal{B}_{M,n}(\lambda)\right).
\end{align}

Denote by $I_{\ell \times \ell}$ the identity matrix in $\C^{\ell \times \ell}$ and by $0_{\ell \times s}$ the zero matrix in $\C^{\ell \times s}$ for $\ell, s \in \N$. We claim that, provided $n \in \N$ is sufficiently large, there exists an analytic function $h_{j,n} \colon B_{\lambda_0}(\varrho_1) \to \C$ such that
\begin{align} \label{inductionclaim}
\tilde{E}_n(\lambda) = h_{j,n}(\lambda) \det\left(\begin{pmatrix} \tilde{B}_{j,n}(\lambda) & -\tilde{A}_j \\ \tilde{D}_{j,n}(\lambda) & \tilde{C}_{j,n}(\lambda)\end{pmatrix} + \tilde{H}_{j,n}(\lambda)\right), 
\end{align}
for $\lambda \in \overline{B}_{\lambda_0}(\varrho_0)$ and $j = 1,\ldots, M-1$. Here, $\tilde{B}_{j,n}(\lambda) \in \C^{2^{j-1} l_0 \times 2^{j-1} l_0}$ is an upper triangular $(l_0\times l_0)$-block matrix, whose blocks above the diagonal are equal to $-I_{l_0 \times l_0}$ or $0_{l_0 \times l_0}$ and whose diagonal contains exactly one copy of each of the blocks $\Xi_{M-1,n}(\lambda)^\top \mathcal{B}_{M,n}(\lambda), \ldots, \Xi_{M-j,n}(\lambda)^\top \mathcal{B}_{M-j+1,n}(\lambda)$ and further only $(l_0 \times l_0)$-identity matrices. Furthermore, $\tilde{A}_j, \tilde{C}_{j,n}(\lambda), \tilde{D}_{j,n}(\lambda) \in \C^{2^{j-1} l_0 \times 2^{j-1} l_0}$ are given by
\begin{align*}
\tilde{A}_j &= \begin{pmatrix} I_{(2^{j-1} - 1)l_0 \times (2^{j-1}-1)l_0} & 0_{(2^{j-1}-1)l_0 \times l_0}\\ 0_{l_0 \times (2^{j-1}-1)l_0} & 0_{l_0 \times l_0}\end{pmatrix}, \\ 
\tilde{C}_{j,n}(\lambda) &= \begin{pmatrix} I_{(2^{j-1} - 1)l_0 \times (2^{j-1}-1)l_0} & 0_{(2^{j-1}-1)l_0 \times l_0}\\ 0_{l_0 \times (2^{j-1}-1)l_0} & \Xi_{0,n}(\lambda)^\top \mathcal{T}_n(nT, (M-j)nT;\lambda) \mathcal{B}_{M-j,n}(\lambda)\end{pmatrix}
\end{align*}
and
\begin{align*}
&\tilde{D}_{j,n}(\lambda) =\\ 
&\begin{pmatrix} 0_{(2^{j-1} - 1)l_0 \times l_0} & \ldots & 0_{(2^{j-1} - 1)l_0 \times l_0} \\ \Xi_{0,n}(\lambda)^\top \mathcal{T}_n(nT, (M-j)nT;\lambda) H_{j,1,n}(\lambda) & \ldots & \Xi_{0,n}(\lambda)^\top \mathcal{T}_n(nT, (M-j)nT;\lambda) H_{j,2^{j-1},n}(\lambda) \end{pmatrix},
\end{align*}
for $j = 1,\ldots,M-1$ and $\lambda \in B_{\lambda_0}(\varrho_1)$. Moreover, $H_{j,\ell,n}(\lambda) \in \C^{l_0 \times l_0}$ and $\tilde{H}_{j,n}(\lambda) \in \C^{2^j l_0 \times 2^j l_0}$ obey the bound
\begin{align} \label{inductionbound}
\left\|H_{j,\ell,n}(\lambda)\right\|, \left\|\tilde{H}_{j,n}(\lambda)\right\| \leq M_5 \eu^{-\mu_5 nT}
\end{align}
for $\ell = 1,\ldots,2^{j-1}$, $j = 1,\ldots,M-1$ and $\lambda \in \overline{B}_{\lambda_0}(\varrho_0)$, where $M_5,\mu_5 > 0$ are $\lambda$- and $n$-independent constants. Finally, $h_{j,n}$ is nonvanishing on $\overline{B}_{\lambda_0}(\varrho_0)$ for $j = 1,\ldots,M-1$.

We prove our claim inductively for $j = 1,\ldots,M-1$. In our proof we rely on the identity~\eqref{detblockid}. We start our induction proof with considering the case $j = 1$. Here, we employ identities~\eqref{detblockid},~\eqref{e:projids} and~\eqref{e:firststep} and use the fact that $\Theta_{M-1,n}(
\lambda)$ is invertible with inverse $\Pi_{M-1,n}(\lambda)$ to derive
\begin{align*}
\tilde{E}_n(\lambda) &= \det\left(\Xi_{0,n}^\top \mathcal{T}_n(nT,(M-1)nT) \mathcal{B}_{M-1,n} \Pi_{M-1,n} \Xi_{M-1,n}^\top \mathcal{B}_{M,n} \right.\\
&\left.\qquad + \, \Xi_{0,n}^\top \mathcal{T}_n(nT,(M-1)nT) \mathcal{T}_n((M-1)nT,MnT) \left(I-\mathcal{Q}_{M-1,\circ,n}(MnT)\right) \mathcal{B}_{M,n}\right)\\
&= \det(\Pi_{M-1,n})\det\left(\begin{pmatrix}  \Xi_{M-1,n}^\top \mathcal{B}_{M,n} & 0_{l_0 \times l_0} \\ \tilde{D}_{1,n} & \Xi_{0,n}^\top \mathcal{T}_n(nT,(M-1)nT) \mathcal{B}_{M-1,n} \end{pmatrix} + \tilde{H}_{1,n}\right) 
\end{align*}
for $\lambda \in \overline{B}_{\lambda_0}(\varrho_0)$, where we suppressed $\lambda$-dependency on the right-hand side and denote
\begin{align*}
\tilde{D}_{1,n}(\lambda) &= \Xi_{0,n}(\lambda)^\top \mathcal{T}_n(nT,(M-1)nT;\lambda)) H_{1,1,n}(\lambda), \\ H_{1,1,n}(\lambda) &= \mathcal{T}_n((M-1)nT,MnT;\lambda) \left(I-\mathcal{Q}_{M-1,\circ,n}(M nT;\lambda)\right) \mathcal{B}_{M,n}(\lambda) ,
\end{align*}
and
\begin{align*}
\tilde{H}_{1,n}(\lambda) = \begin{pmatrix} 0_{l_0 \times l_0} & -\Theta_{M-1,n}(\lambda) \\ 0_{l_0 \times l_0} & 0_{l_0 \times l_0}\end{pmatrix}.
\end{align*}
Using the bound~\eqref{projestmulti} and the fact that $\mathcal{B}_{M-1,n}(\lambda)$ is a basis of $\mathrm{ran}(\mathcal{Q}_{M-1,\circ,n}((M-1)nT;\lambda))$, we obtain, provided $n \in \N$ is sufficiently large, $\lambda$- and $n$-independent constants $M_6,\mu_6 > 0$ such that
\begin{align*}
\left\|H_{1,1,n}(\lambda)\right\|, \left\|\tilde{H}_{1,n}(\lambda)\right\| \leq M_6 \eu^{-\mu_6 nT}
\end{align*}
for $\lambda \in \overline{B}_{\lambda_0}(\varrho_0)$. Finally, we note that the function $h_{1,n} \colon B_{\lambda_0}(\varrho_1) \to \C$ given by $h_{1,n}(\lambda) = \det(\Pi_{M-1,n}(\lambda))$ is analytic and does not vanish on $\overline{B}_{\lambda_0}(\varrho_0)$, since $\Pi_{M-1,n}$ is analytic and $\Pi_{M-1,n}(\lambda)$ is invertible for all $\lambda \in \overline{B}_{\lambda_0}(\varrho_0)$. We conclude that our claim is valid for $j = 1$.

Next, we perform the induction step. That is, we assume that our claim holds for some $j = j_0 \in \{1,\ldots,M-2\}$ and prove that it then also holds for $j = j_0 + 1$. First, we recall that $\Theta_{M-j_0-1,n}(\lambda)$ is invertible with inverse $\Pi_{M-j_0-1,n}(\lambda)$ and use 
\begin{align*}
&\mathcal{T}_n((M-j_0-1)nT,(M-j_0)nT;\lambda)
= \mathcal{B}_{M-j_0-1}(\lambda) \Pi_{M-j_0-1,n}(\lambda) \Xi_{M-j_0-1,n}(\lambda)^\top\\
&\qquad \qquad +\, \mathcal{T}_n((M-j_0-1)nT,(M-j_0)nT;\lambda) \left(I - \mathcal{Q}_{M-j_0-1,\circ,n}((M-j_0)nT;\lambda)\right),
\end{align*}
cf.~\eqref{e:projids}, to express
\begin{align} \label{induction_reexpress}
\begin{pmatrix} \tilde{B}_{j_0,n}(\lambda) & -\tilde{A}_{j_0} \\ \tilde{D}_{j_0,n}(\lambda) & \tilde{C}_{j_0,n}(\lambda)\end{pmatrix} + \tilde{H}_{j_0,n}(\lambda) =  \tilde{C}_{j_0+1,n}(\lambda) A_{j_0+1,n}(\lambda)^{-1} B_{j_0+1,n}(\lambda) + D_{j_0+1,n}(\lambda)
\end{align}
with
\begin{align*}
A_{j_0+1,n}(\lambda) &= \begin{pmatrix} I_{(2^{j_0} - 1)l_0 \times (2^{j_0}-1)l_0} & 0_{(2^{j_0}-1)l_0 \times l_0}\\ 0_{l_0 \times (2^{j_0}-1)l_0} & \Theta_{M-j_0-1,n}(\lambda)\end{pmatrix}, \\ 
B_{j_0+1,n}(\lambda) &= \begin{pmatrix} \tilde{B}_{j_0,n}(\lambda) & -\tilde{A}_{j_0}\\ \hat{D}_{j_0+1,n}(\lambda) & \hat{C}_{j_0+1,n}(\lambda)\end{pmatrix}, \qquad D_{j_0+1,n}(\lambda) = \tilde{D}_{j_0+1,n}(\lambda) + \tilde{H}_{j_0,n}(\lambda)
\end{align*}
and
\begin{align*}
\hat{C}_{j_0+1,n}(\lambda) &= \begin{pmatrix} I_{(2^{j_0-1} - 1)l_0 \times (2^{j_0-1}-1)l_0} & 0_{(2^{j_0-1}-1)l_0 \times l_0}\\ 0_{l_0 \times (2^{j_0-1}-1)l_0} & \Xi_{M-j_0-1,n}(\lambda)^\top \mathcal{B}_{M-j_0,n}(\lambda)\end{pmatrix}, \\ 
\hat{D}_{j_0+1,n}(\lambda) &= \begin{pmatrix} 0_{(2^{j_0-1} - 1)l_0 \times l_0} & \ldots & 0_{(2^{j_0-1} - 1)l_0 \times l_0} \\ \Xi_{M-j_0-1,n}(\lambda)^\top H_{j_0,1,n}(\lambda) & \ldots & \Xi_{M-j_0-1,n}(\lambda)^\top H_{j_0,2^{j_0-1},n}(\lambda)\end{pmatrix}
\end{align*}
for $\lambda \in \overline{B}_{\lambda_0}(\varrho_0)$, where the matrix $\tilde{D}_{j_0+1,n}(\lambda)$ is defined by setting
\begin{align*}
H_{j_0+1,\ell,n}(\lambda) &= \mathcal{T}_n((M-j_0-1)nT,(M-j_0)nT) \left(I-\mathcal{Q}_{M-j_0-1,\circ,n}((M-j_0)nT)\right) H_{j_0,\ell,n},\\
H_{j_0+1,\tilde{\ell},n}(\lambda) &= 0_{l_0 \times l_0},\\
H_{j_0+1,2^{j_0},n}(\lambda) &= \mathcal{T}_n((M-j_0-1)nT,(M-j_0)nT) \left(I-\mathcal{Q}_{M-j_0-1,\circ,n}((M-j_0)nT)\right) \mathcal{B}_{M-j_0,n}
\end{align*}
for $\ell = 1,\ldots,2^{j_0-1}$ and $\tilde{\ell} = 2^{j_0-1} + 1,\ldots, 2^{j_0} - 1$, suppressing $\lambda$-dependency on the right-hand sides. Next, we take determinants in~\eqref{induction_reexpress}, use that~\eqref{inductionclaim} holds for $j = j_0$, and apply~\eqref{detblockid} to arrive at
\begin{align*}
\tilde{E}_n(\lambda) = h_{j_0,n}(\lambda) \det\left(\Pi_{M-j_0-1,n}(\lambda)\right) \det\left(\begin{pmatrix} \tilde{B}_{j_0+1,n}(\lambda) & -\tilde{A}_{j_0+1} \\ \tilde{D}_{j_0+1,n}(\lambda) & \tilde{C}_{j_0+1,n}(\lambda)\end{pmatrix} + \tilde{H}_{j_0+1,n}(\lambda)\right),
\end{align*}
with
\begin{align*}
\tilde{B}_{j_0+1,n}(\lambda) = \begin{pmatrix} \tilde{B}_{j_0,n}(\lambda) & -\tilde{A}_{j_0}\\ 0_{2^{j_0} l_0 \times 2^{j_0} l_0} & \hat{C}_{j_0+1,n}(\lambda)\end{pmatrix}
\end{align*}
and
\begin{align*}
\tilde{H}_{j_0+1,n}(\lambda) = \begin{pmatrix} \hat{H}_{j_0+1,n}(\lambda) & \check{H}_{j_0+1,n}(\lambda) \\ 
\tilde{H}_{j_0,n}(\lambda) & 0_{2^{j_0}l_0 \times 2^{j_0} l_0}\end{pmatrix},
\end{align*}
where we denote
\begin{align*}
\hat{H}_{j_0+1,n}(\lambda) = \begin{pmatrix} 0_{2^{j_0-1}l_0 \times 2^{j_0-1} l_0} & 0_{2^{j_0-1}l_0 \times 2^{j_0-1} l_0} \\ \hat{D}_{j_0+1,n}(\lambda) & 0_{2^{j_0-1}l_0 \times 2^{j_0-1} l_0}\end{pmatrix}
\end{align*}
and
\begin{align*}
\check{H}_{j_0+1,n}(\lambda) = \begin{pmatrix} 0_{(2^{j_0}-1)l_0 \times (2^{j_0}-1) l_0} & 0_{(2^{j_0}-1)l_0 \times l_0} \\ 0_{l_0 \times (2^{j_0}-1) l_0} & -\Theta_{M-j_0-1,n}(\lambda)\end{pmatrix}.
\end{align*}
Moreover, recalling that $\mathcal{B}_{M-j_0-1,n}(\lambda)$ is a basis of $\mathrm{ran}(\mathcal{Q}_{M-j_0-1,\circ,n}((M-j_0-1)nT;\lambda))$, employing the bound~\eqref{projestmulti}, and using that~\eqref{inductionbound} holds for $j = j_0$ and $\ell = 1,\ldots,2^{j_0-1}$, we establish $\lambda$- and $n$-independent constants $M_7,\mu_7 > 0$ such that, provided $n \in \N$ is sufficiently large, we have
\begin{align*}
\left\|H_{j_0+1,\ell,n}(\lambda)\right\|, \left\|\tilde{H}_{j_0+1,n}(\lambda)\right\| \leq M_7 \eu^{-\mu_7 nT}
\end{align*}
for $\ell = 1,\ldots,2^{j_0}$ and $\lambda \in \overline{B}_{\lambda_0}(\varrho_0)$. Finally, the function $h_{j_0 + 1,n} \colon B_{\lambda_0}(\varrho_1) \to \C$ given by $h_{j_0 + 1,n}(\lambda) = h_{j_0,n}(\lambda) \det(\Pi_{M-j_0-1,n}(\lambda))$ is analytic and does not vanish on $\overline{B}_{\lambda_0}(\varrho_0)$, since $h_{j_0,n}$ and $\Pi_{M-j_0-1,n}$ are analytic, $h_{j_0,n}(\lambda)$ is nonzero and $\Pi_{M-j_0-1,n}(\lambda)$ is invertible for all $\lambda \in \overline{B}_{\lambda_0}(\varrho_0)$. Therefore, our claim holds for $j = j_0 + 1$.

Inductively, we have thus established our claim for $j = 1,\ldots,M-1$ as desired. In particular, applying our claim with $j = M-1$ we find, provided $n \in \N$ is sufficiently large, that
\begin{align} \label{inductionoutcome}
\tilde{E}_n(\lambda) = h_{M-1,n}(\lambda) \hat{E}_n(\lambda) 
\end{align}
for $\lambda \in \overline{B}_{\lambda_0}(\varrho_0)$ with $\hat{E}_n \colon \overline{B}_{\lambda_0}(\varrho_0) \to \C$ given by
\begin{align*}
\hat{E}_n(\lambda) = \det\left(A_{\mathrm{main},n}(\lambda) + A_{\mathrm{res},n}(\lambda)\right),
\end{align*}
where we denote
\begin{align*}
A_{\mathrm{res},n}(\lambda) = \begin{pmatrix} 0_{2^{M-2} l_0 \times 2^{M-2} l_0} & 0_{2^{M-2} l_0 \times 2^{M-2} l_0} \\ \tilde{D}_{M-1,n}(\lambda)  & 0_{2^{M-2} l_0 \times 2^{M-2} l_0}\end{pmatrix} + \tilde{H}_{M-1,n}(\lambda)
\end{align*}
and where
\begin{align*}
A_{\mathrm{main},n}(\lambda) = \begin{pmatrix} \tilde{B}_{M-1,n}(\lambda) & -\tilde{A}_{M-1} \\ 0_{2^{M-2} l_0 \times 2^{M-2} l_0} & \tilde{C}_{M-1,n}(\lambda)\end{pmatrix} \in \C^{2^{M-1}l_0 \times 2^{M-1} l_0}
\end{align*}
is an upper triangular $(l_0 \times l_0)$-block matrix, whose blocks above the diagonal are equal to $- I_{l_0 \times l_0}$ or to $0_{l_0 \times l_0}$, and whose diagonal contains $(l_0 \times l_0)$-identity matrices and precisely one copy of each of the blocks $\Xi_{M-1,n}(\lambda)^\top \mathcal{B}_{M,n}(\lambda), \ldots, \Xi_{0,n}(\lambda)^\top \mathcal{B}_{1,n}(\lambda)$. Hence, by estimates~\eqref{projestmulti} and~\eqref{inductionbound} there exist $\lambda$- and $n$-independent constants $M_8,\mu_8 > 0$ such that, provided $n \in \N$ is sufficiently large, we have
\begin{align*}
\left\|A_{\mathrm{main},n}(\lambda) - A_0(\lambda)\right\| \leq M_8 \left(\eu^{-\mu_8 nT} + \delta_n\right), \qquad \left\|A_0(\lambda)\right\| \leq M_8, \qquad \left\|A_{\mathrm{res},n}(\lambda)\right\| \leq M_8 \eu^{-\mu_8 nT}
\end{align*}
for $\lambda \in \overline{B}_{\lambda_0}(\varrho_0)$, where $A_0(\lambda)$ is the upper triangular block matrix arising by substituting the blocks $\Xi_{M-1,n}(\lambda)^\top \mathcal{B}_{M,n}(\lambda), \ldots, \Xi_{0,n}(\lambda)^\top \mathcal{B}_{1,n}(\lambda)$ in $A_{\mathrm{main},n}(\lambda)$ by the blocks $\Phi_{M}^\top B_{M,\mathrm{s}}(\lambda), \ldots, \Phi_{1}^\top B_{1,\mathrm{s}}(\lambda)$, respectively. Therefore, we obtain $\lambda$- and $n$-independent constants $M_9,\mu_9 > 0$ such that, provided $n \in \N$ is sufficiently large, we have
\begin{align} \label{rouche}
\left|\hat{E}_n(\lambda) - \tilde{\mathcal{E}}(\lambda)\right| \leq M_9 \left(\eu^{-\mu_9 n T} + \delta_n\right)
\end{align}
for $\lambda \in \overline{B}_{\lambda_0}(\varrho_0)$, where we denote
\begin{align*}
\tilde{\mathcal{E}} =  \tilde{\mathcal{E}}_1(\lambda)\ldots\tilde{\mathcal{E}}_M(\lambda).
\end{align*}
Moreover, since $h_{M-1,n}$ and $\tilde{E}_n$ are analytic and $h_{M-1,n}$ does not vanish on $\overline{B}_{\lambda_0}(\varrho_0)$, we find by identity~\eqref{inductionoutcome} that $\smash{\hat{E}}_n$ is analytic on $B_{\lambda_0}(\varrho_0)$ and has the same zeros (including multiplicities) in $B_{\lambda_0}(\varrho_0)$ as $\tilde{E}_n$.

\paragraph*{Application of Rouch\'e's theorem.} Let $\varrho \in (0,\varrho_0)$. Since the Evans function $\mathcal{E}_j$ has the same roots (including multiplicities) as $\smash{\tilde{\mathcal{E}}}_j$ in $B_{\lambda_0}(\varrho_1)$ for $j = 1, \ldots,M$ and $\mathcal{E}$ does not vanish on $\partial B_{\lambda_0}(\varrho)$, we find that $\tilde{\mathcal{E}}$ is also nonzero on $\partial B_{\lambda_0}(\varrho)$. So, the bound~\eqref{rouche} yields, provided $n \in \N$ is sufficiently large, that
\begin{align*}
    \left|\hat{E}_n(\lambda) - \tilde{\mathcal{E}}(\lambda)\right| < \left|\tilde{\mathcal{E}}(\lambda)\right|
\end{align*}
for all $\lambda \in \partial B_{\lambda_0}(\varrho)$. Therefore, noting that $\tilde{\mathcal{E}}$ has only one root in $B_{\lambda_0}(\varrho)$ having multiplicity $m_0$, Rouch\'e's theorem implies that $\smash{\hat{E}_n}$ possesses precisely $m_0$ zeros in $B_{\lambda_0}(\varrho)$ (counting multiplicities). Since the zeros of $\smash{\hat{E}_n}$, $\smash{\tilde{E}_n}$ and $E_n$ in $B_{\lambda_0}(\varrho)$ and their multiplicities coincide, the first assertion follows. The second assertion is a direct consequence of the first by Proposition~\ref{prop:Evans1front}. 
\end{proof}

\section{Spectral analysis of periodic pulse solutions}\label{sec:stability_periodic}

In this section, we study the spectral stability of stationary periodic pulse solutions to~\eqref{e:sys_intro}. We consider an $nT$-periodic pulse solution $u_n$ of the form~\eqref{ansatz_periodic}. That is, we have $u_n = w_n + a_n$, where $w_n$ is the formal $nT$-periodic extension of a primary pulse $Z = z+v \in H^k(\R) \oplus H^k_\per(0,T)$, and $a_n$ is an error term converging to $0$ in $H_\per^k(0,nT)$. Our goal is to show that spectral (in)stability properties of the primary pulse $Z$ transfer to the periodic pulse solution $u_n$.

We begin by observing that Lemma~\ref{lem:invertibility_periodic} implies that, if $\mathcal{K}$ is a compact subset of the resolvent set $\rho(\El(Z))$, then $\mathcal{K}$ is also contained in $\rho(\El(u_n))$ for $n \in \N$ sufficiently large. Hence, if a-priori bounds preclude unstable spectrum outside of the compact region $\mathcal{K}$, then this leads to the following analogue of Corollary~\ref{cor:stability_multifront}, which asserts that strong spectral stability of the primary pulse is inherited by the associated periodic pulse solutions.

\begin{Corollary}\label{cor:stability_periodic}
Assume~{\upshape \ref{assH4}} and~{\upshape \ref{assH5}}. Suppose that the pulse solution $z + v$ is strongly spectrally stable. Moreover, assume that there exist a compact set $\mathcal{K} \subset \C$ and $N \in \N$ such that
\begin{align*}
    \sigma(\El_n) \cap \{z \in \C : \Re(z) \geq 0\} \subset \mathcal{K}
\end{align*}
for $n \geq N$, where $u_n$ is the periodic pulse solution established in Theorem~\ref{t:existence_periodic}, and $\El_n$ is the operator $\El(u_n)$ or $\El_\per(u_n)$. Then, there exists $N_1 \in \N$ with $N_1 \geq N$ such that for all $n \in \N$ with $n \geq N_1$ the spectrum of the operator $\El_n$ is confined to the open left-half plane.
\end{Corollary}

In the remainder of this section, we analyze the spectra of the operators $\El(u_n)$ and $\El_{\per}(u_n)$ in more detail. Our approach relies on comparing the Evans function $\mathcal{E}$ associated with $\El(Z)$, as constructed in Proposition~\ref{prop:Evans1front}, with an Evans function for the first-order formulation 
\begin{align} \label{variationalsysfullper2}
U' = \mathcal{A}(x,u_n(x);\lambda) U
\end{align}
of the eigenvalue problem along the periodic pulse solution $u_n$.  If $\mathcal{T}_n(x,y;\lambda)$ is the evolution of system~\eqref{variationalsysfullper2}, then this analytic Evans function $E_{n,\gamma} \colon \C \to \C$ is given by 
$$
    E_{n,\gamma}(\lambda) = \det\left(\mathcal{T}_n(0,-\tfrac{n}{2} T;\lambda) - \gamma \mathcal{T}_n(0,\tfrac{n}{2} T;\lambda)\right)
$$
for $\gamma$ lying in the unit circle $S^1 \subset \C$, cf.~\cite{Gardner93,Sandstede2000Absolute}. Clearly, it holds $E_{n,\gamma}(\lambda_0) = 0$ for some $\lambda_0\in \C$ if and only if  system~\eqref{variationalsysfullper2} possesses a nontrivial solution $U(x)$ satisfying the boundary condition $U(-\tfrac{n}{2}T) = \gamma U(\tfrac{n}{2}T)$. Hence, $\lambda_0$ lies in the spectrum of the Bloch operator $\El_{\xi,\per}(u_n)$ if and only $\lambda_0$ is a zero of $E_{n,\eu^{\iu \xi nT}}$, cf.~\S\ref{sec:periodic_diff_operators}. In fact, it was shown in~\cite{Gardner93}, see also~\cite[Section~8.4]{KapitulaPromislow2013}, that the algebraic multiplicity of $\lambda_0$ as an eigenvalue of $\El_{\xi,\per}(u_n)$ equals the multiplicity of $\lambda_0$ as a root of the Evans function $E_{n,\eu^{\iu \xi nT}}$. We conclude that $\lambda_0$ lies in the spectrum of $\El_\per(u_n)$ if and only if $\lambda_0 \in \C$ is a zero of $E_{n,0}$. Moreover, we have $\lambda_0\in\sigma(\El(u_n))$ if and only there exists $\gamma \in S^1$ such that $E_{n,\gamma}(\lambda_0) = 0$. 

The main result of this section establishes that isolated zeros of the Evans function $\mathcal{E}$ associated with the primary pulse perturb into zeros of the Evans function $E_{n,\gamma}$ of the periodic pulse solution, thereby preserving the total multiplicity. That is, isolated eigenvalues (including their algebraic multiplicities) of the linearization $\El(Z)$ about the primary pulse persist as eigenvalues of the Bloch operator $\El_{\xi,\per}(u_n)$ for each $\xi \in [-\frac{\pi}{nT},\frac{\pi}{nT})$.

\begin{Theorem} \label{thm:instability_periodic}
Let $z+v \in H^k(\R) \oplus H_\per^k(0,T)$. Suppose there exists $N \in \N$ such that, for each $n \in \N$ with $n \geq N$, there exists a periodic pulse solution $u_n \in H_\per^k(0,nT)$ of the form $u_n = z_n + v + a_n$, where $\{a_n\}_n$ is a sequence with $a_n \in H_\per^k(0,nT)$ satisfying $\|a_n\|_{H_\per^k(0,nT)} \to 0$ as $n \to \infty$, and $z_n \in H_\per^k(0,nT)$ is the $nT$-periodic extension of the function $\chi_n z$ on $\left[-\tfrac{n}{2} T, \tfrac{n}{2}T\right)$ and $\chi_n$ is the cut-off function from Theorem~\ref{t:existence_periodic}.

Let $\Omega$ be a connected component of $\C \setminus \sigma_{\mathrm{ess}}(\El(v))$. Let $\mathcal{E} \colon \Omega \to \C$ be the Evans function associated with the first-order system~\eqref{variationalsysper}, as constructed in Proposition~\ref{prop:Evans1front}. 

Suppose that $\lambda_0 \in \Omega$ is a root of $\mathcal{E}$ of multiplicity $m_0 \in \N$. Then, there exists $\varrho_0 > 0$ such that for each $\varrho \in (0,\varrho_0)$ there exists $N_\varrho \in \N$ such that for all $n \in \N$ with $n \geq N_\varrho$ the following assertions hold true.
\begin{itemize}
    \item[1.] For each $\gamma \in S^1$ the Evans function $E_{n,\gamma}$ possesses precisely $m_0$ roots in the disk $B_{\lambda_0}(\varrho)$ (counting multiplicities). 
    \item[2.] For each $\xi \in [-\frac{\pi}{nT},\frac{\pi}{nT})$ the Bloch operator $\El_{\xi,\per}(u_n)$ has precisely $m_0$ eigenvalues in $B_{\lambda_0}(\varrho)$ (counting algebraic multiplicities).
    \item[3.] The operator $\El_\per(u_n)$ has precisely $m_0$ eigenvalues in $B_{\lambda_0}(\varrho)$ (counting algebraic multiplicities).
    \item[4.] The operator $\El(u_n)$ has spectrum in $B_{\lambda_0}(\varrho)$.
\end{itemize}
\end{Theorem}

\begin{Remark}
Theorem~\ref{thm:instability_periodic} establishes the convergence of the point spectrum of the Bloch operators $\El_{\xi,\per}(u_n)$ within $n$-independent compact sets $\mathcal{K} \subset \C\setminus \sigma_{\mathrm{ess}}(\El(Z))$ to the point spectrum of $\El(Z)$ in $\mathcal{K}$ as $n \to \infty$. This naturally raises the question of whether spectrum of $\El_{\xi,\per}(u_n)$ converges to the essential spectrum $\sigma_{\mathrm{ess}}(\El(Z))$ as $n \to \infty$. In the case of constant coefficients, this question has been answered affirmative. Specifically, it was shown in~\cite{Sandstede2000Absolute} that eigenvalues of $\El_{\xi,\per}(u_n)$ accumulate onto each point of the essential spectrum $\sigma_{\mathrm{ess}}(\El(Z))$ as $n \to \infty$. Consequently, on compact subsets, the spectra of both $\El(u_n)$ and $\El_{\per}(u_n)$ converge to $\sigma(\El(Z))$ in Hausdorff distance as $n \to \infty$. We strongly expect that, using the techniques developed in~\cite{Sandstede2000Absolute}, a similar result can be obtained in the spatially periodic setting considered here. 
\end{Remark}

Theorem~\ref{thm:instability_periodic} implies that spectral instability of the primary pulse is inherited by the periodic pulse solution. Furthermore, it serves as an important tool in spectral (in)stability arguments based on Krein index counting theory. In particular, we employ Theorem~\ref{thm:instability_periodic} in~\S\ref{sec:applications} to demonstrate the spectral and orbital stability of periodic pulse solutions to the Gross-Pitaevskii equation with a periodic potential.

As mentioned in~\S\ref{sec:intro}, Theorem~\ref{thm:instability_periodic} was established in the constant-coefficient case in~\cite{Gardner1997Spectral}, using geometric dynamical systems techniques and topological arguments based on Chern numbers. The result was subsequently refined in~\cite{Sandstede2000Absolute} by showing that there exists an $n$-independent constant $\mu > 0$ such that the roots of $E_{n,\gamma}$ in $B_{\lambda_0}(\varrho)$ remain $\mathcal{O}(\eu^{-\mu n})$-close to $\lambda_0$.

Our proof of Theorem~\ref{thm:instability_periodic} builds upon the approach of~\cite[Theorem~2]{Sandstede2000Absolute}. Specifically, we employ roughness techniques to transfer exponential dichotomies on $\R_\pm$ for the eigenvalue problem~\eqref{variationalsysper} along the primary pulse to the system
\begin{align} \label{variationalpercutoff}
     U' = \mathcal{A}(x,\chi_n(x) z(x) + v(x) + a_n(x);\lambda) U.
\end{align}
System~\eqref{variationalpercutoff} coincides with the eigenvalue problem~\eqref{variationalsysfullper2} along the periodic pulse $u_n$ on a single periodicity interval $[-\frac{n}{2} T,\frac{n}{2} T]$. Denoting the projections of the exponential dichotomies of~\eqref{variationalpercutoff} on $\R_\pm$ by $\mathcal{Q}_{\pm,n}(x;\lambda)$, we construct analytic bases of $\ker(\mathcal{Q}_{-,n}(-\frac{n}{2} T;\lambda))$ and $\mathrm{ran}(\mathcal{Q}_{+,n}(\frac{n}{2} T;\lambda))$. Multiplying $E_{n,\gamma}(\lambda)$ with the nonzero determinant of the matrix formed by these basis vectors yields an approximation of the Evans function $\mathcal{E}$ associated with the primary pulse. The conclusion then follows from an application of Rouch\'e's theorem. 

\begin{proof}[Proof of Theorem~\ref{thm:instability_periodic}]
We start by collecting some facts from Proposition~\ref{prop:Evans1front}. First, system~\eqref{variationalsysper} possesses for each $\lambda \in \Omega$ exponential dichotomies on $\R_\pm$ with projections $P_\pm(\pm x;\lambda), x \geq 0$ of some fixed rank $l_0$, which is independent of $\lambda$ and $x$. Moreover, there exist analytic functions $B_{\mathrm{s}} \colon \Omega \to \C^{km \times l_0}$ and $B_{\mathrm{u}}  \colon \Omega \to \C^{km \times (km-l_0)}$ such that $B_{\mathrm{s}}(\lambda)$ is a basis of $\mathrm{ran}(P_+(0;\lambda))$ and $B_{\mathrm{u}}(\lambda)$ is a basis of $\ker(P_-(0;\lambda))$ for each $\lambda \in \Omega$. The associated Evans function $\mathcal{E} \colon \Omega \to \C$ is given by
\begin{align*}\mathcal{E}(\lambda) = \det(B_{\mathrm{u}}(\lambda) \mid B_{\mathrm{s}}(\lambda)).\end{align*}
Because $\mathcal{E}$ is analytic and $\lambda_0$ is a root of $\mathcal{E}$ of finite multiplicity, there exists a closed disk $\smash{\overline{B}_{\lambda_0}(\varrho_1)} \subset \Omega$ of some radius $\varrho_1 > 0$ such that $\lambda_0$ is the only root of $\mathcal{E}$ in $\smash{\overline{B}_{\lambda_0}(\varrho_1)}$. Finally, there exist constants $K_0,\mu_0,\tau_0 > 0$ such that system~\eqref{variationalsysper} admits for each $\lambda \in \smash{\overline{B}_{\lambda_0}(\varrho_1)}$ exponential dichotomies on $\R_\pm$ with constants $K_0,\mu_0 > 0$ and projections $\smash{\tilde P_\pm(\pm x;\lambda)}, x \geq 0$ satisfying~\eqref{projest2} for each $x \geq \tau_0$, where $Q(\cdot;\lambda)$ is $T$-periodic. By uniqueness of exponential dichotomies, cf.~\cite[p.~19]{Coppel1978}, $B_{\mathrm{s}}(\lambda)$ is a basis of $\mathrm{ran}(P_+(0;\lambda)) = \mathrm{ran}(\smash{\tilde P_+}(0;\lambda))$ and $B_{\mathrm{u}}(\lambda)$ is a basis of $\ker(P_-(0;\lambda)) = \ker(\smash{\tilde P_-}(0;\lambda))$ for each $\lambda \in \smash{\overline{B}_{\lambda_0}(\varrho_1)}$.

Since $\partial_u \mathcal{A}$ is continuous, $\overline{B}_{\lambda_0}(\varrho_1)$ is compact, it holds $z \in H^1(\R)$ and $v,a_n \in H_\per^1(0,nT)$, and $H^1(\R)$ and $H^1_\per(0,nT)$ embed continuously into $L^\infty(\R)$ with $n$-independent constant, we obtain by the mean value theorem a $\lambda$- and $n$-independent constant $R > 0$ such that
\begin{align*}
\begin{split}
&\left\|\A\left(x,\chi_n z(x) + v(x) + a_n(x);\lambda \right) - \A\left(x,z(x) + v(x);\lambda \right)\right\| \leq R \delta_n
\end{split}
\end{align*}
for $x \in \R$ and $\lambda \in \overline{B}_{\lambda_0}(\varrho_1)$, where
\begin{align*}
\delta_n := \sup_{x \in \R} \left(\|(1-\chi_n(x)) z(x)\| + \|a_n(x)\|\right)
\end{align*}
converges to $0$ as $n \to \infty$. So, by Lemma~\ref{l:projanarough} there exist constants $M_1,\mu > 0$ and $\varrho_0 \in (0,\varrho_1)$ such that system~\eqref{variationalpercutoff} admits, provided $n \in \mathbb{N}$ is sufficiently large, exponential dichotomies on $\R_\pm$ with $\lambda$- and $n$-independent constants and projections $\mathcal{Q}_{\pm,n}(\pm x;\lambda), x \geq 0$ for each $\lambda \in \overline{B}_{\lambda_0}(\varrho_0)$. Here, the maps $\mathcal{Q}_{\pm,n}(\pm x;\cdot) \colon B_{\lambda_0}(\varrho_0) \to \C^{km \times km}$ are analytic for each $x \geq 0$ and the estimates
\begin{align} \label{projest3}
\begin{split}
    \|\mathcal{Q}_{\pm,n}(\pm \tfrac{n}{2} T;\lambda)-\tilde{P}_\pm(\pm \tfrac{n}{2} T;\lambda)\| &\leq M_1\left(\delta_n + \eu^{-\mu \tfrac{n}{2} T}\right),\\
    \left\|\mathcal{Q}_{\pm,n}(0;\lambda) - Q_\pm(\lambda)\right\| &\leq M_1\delta_n
\end{split}
\end{align}
hold for each $\lambda \in \overline{B}_{\lambda_0}(\varrho_0)$, where $Q_+(\lambda)$ is the projection onto $\mathrm{ran}(P_+(0;\lambda))$ along $\mathrm{ran}(P_+(0;\lambda_0))^\top$ and $Q_-(\lambda)$ is the projection onto $\mathrm{ker}(P_-(0;\lambda_0))^\top$ along $\mathrm{ker}(P_-(0;\lambda))$.  

Now set $\mathcal{B}_{\mathrm{s},n}(\lambda) = \mathcal{Q}_{+,n}(0;\lambda) B_{\mathrm{s}}(\lambda)$ and $\mathcal{B}_{\mathrm{u},n}(\lambda) = (I-\mathcal{Q}_{-,n}(0;\lambda)) B_{\mathrm{u}}(\lambda)$. Then, $\mathcal{B}_{\mathrm{s},n}(\lambda)$ and $\mathcal{B}_{\mathrm{u},n}(\lambda)$ are analytic in $\lambda$ on $B_{\lambda_0}(\varrho_0)$. Moreover, the fact that the analytic maps $B_{\mathrm{s}}$ and $B_{\mathrm{u}}$ are bounded on the compact set $\overline{B}_{\lambda_0}(\varrho_0)$ in combination with estimate~\eqref{projest3} affords a $\lambda$- and $n$-independent constant $M_2 > 0$ such that
\begin{align} \label{projestanb}
\left\|B_{\mathrm{s}}(\lambda) - \mathcal{B}_{\mathrm{s},n}(\lambda)\right\|, \left\|B_{\mathrm{u}}(\lambda) - \mathcal{B}_{\mathrm{u},n}(\lambda)\right\| \leq M_2 \delta_n, \qquad \left\|\mathcal{B}_{\mathrm{s},n}(\lambda)\right\|, \left\|\mathcal{B}_{\mathrm{u},n}(\lambda)\right\| \leq M_2, 
\end{align}
for all $\lambda \in \overline{B}_{\lambda_0}(\varrho_0)$. So, provided $n \in \mathbb{N}$ is sufficiently large, $\mathcal{B}_{\mathrm{s},n}(\lambda)$ is a basis of $\mathrm{ran}(\mathcal{Q}_{+,n}(0;\lambda))$ and $\mathcal{B}_{\mathrm{u},n}(\lambda)$ is a basis of $\ker(\mathcal{Q}_{-,n}(0;\lambda))$ for each $\lambda \in \overline{B}_{\lambda_0}(\varrho_0)$.

Since $Q(\cdot;\lambda)$ is $T$-periodic, estimates~\eqref{projest2} and~\eqref{projest3} and Lemma~\ref{l:projest2} imply that, provided $n \in \mathbb{N}$ is sufficiently large, the subspaces $\mathrm{ran}(\mathcal{Q}_{-,n}(-\frac{n}{2} T;\lambda))$ and $\ker(\mathcal{Q}_{+,n}(\frac{n}{2} T;\lambda))$ are complementary and there exists a $\lambda$- and $n$-independent constant $M_3 > 0$ such that the projection $\smash{\check{\mathcal{P}}_{n}(\lambda)}$ onto $\mathrm{ran}(\mathcal{Q}_{-,n}(-\frac{n}{2} T;\lambda))$ along $\ker(\mathcal{Q}_{+,n}(\frac{n}{2} T;\lambda))$ obeys
\begin{align} \label{projest4}
\left\|\check{\mathcal{P}}_{n}(\lambda)\right\| \leq M_3
\end{align}
for all $\lambda \in \overline{B}_{\lambda_0}(\varrho_0)$. In addition, since the functions $\mathcal{Q}_{\pm,n}(\pm \frac{n}{2} T;\cdot) \colon B_{\lambda_0}(\varrho_0) \to \C^{km \times km}$ are analytic, Lemma~\ref{l:proj_ana} yields analytic maps $\mathcal{B}_{1,n}, \mathcal{B}_{2,n} \colon B_{\lambda_0}(\varrho_0) \to \C^{km \times l_0}$ with the property that $\mathrm{ran}(\mathcal{Q}_{-,n}(-\frac{n}{2} T;\lambda)) = \mathrm{ran}(\mathcal{B}_{1,n}(\lambda))$ and $\ker(\mathcal{Q}_{+,n}(\frac{n}{2} T;\lambda)) = \{u \in \C^{km} : z^\top u = 0 \text{ for all } z \in \mathrm{ran}(\mathcal{B}_{2,n}(\lambda))\}$. So, it follows, again by Lemma~\ref{l:proj_ana}, that $\smash{\check{\mathcal{P}}_{n}} \colon B_{\lambda_0}(\varrho_0) \to \C^{km \times km}$ is analytic. 

Since the evolution $\mathcal{T}_n(x,y;\lambda)$ of system~\eqref{variationalsysfullper2} depends analytically on $\lambda$ by~\cite[Lemma~2.1.4]{KapitulaPromislow2013}, the Evans function $E_{n,\gamma}$ is analytic for each $\gamma \in S^1$. Because system~\eqref{variationalsysfullper2} coincides with system~\eqref{variationalpercutoff} on $[-\frac{n}{2} T,\frac{n}{2} T]$, it holds
\begin{align*}
    E_{n,\gamma}(\lambda) = \det\left(T_n(0,-\tfrac{n}{2} T;\lambda) - \gamma T_n(0,\tfrac{n}{2} T;\lambda)\right)
\end{align*}
for all $\gamma \in S^1$, where $T_n(x,y;\lambda)$ denotes the evolution of~\eqref{variationalpercutoff}, which depends analytically on $\lambda$ by~\cite[Lemma~2.1.4]{KapitulaPromislow2013}. Define $\mathcal{H}_{n,\gamma} \colon \overline{B}_{\lambda_0}(\varrho_0) \to \C^{km \times km}$ by
\begin{align*}
\mathcal{H}_{n,\gamma}(\lambda) = \left(\left(I-\check{\mathcal{P}}_n(\lambda)\right)T_n\left(-\tfrac{n}{2}T,0;\lambda\right)\mathcal{B}_{\mathrm{u},n}(\lambda) \mid -\gamma^{-1} \check{\mathcal{P}}_n(\lambda) T_n\left(\tfrac{n}{2}T,0;\lambda\right)\mathcal{B}_{\mathrm{s},n}(\lambda)\right)
\end{align*}
for $\gamma \in S^1$. We note that $\mathcal{H}_{n,\gamma}$ is analytic on $B_{\lambda_0}(\varrho_0)$. Moreover, $T_n\left(-\tfrac{n}{2}T,0;\lambda\right)\mathcal{B}_{\mathrm{u},n}(\lambda)$ constitutes a basis of $\ker(\mathcal{Q}_{-,n}(-\frac{n}{2}T;\lambda))$, whereas $I-\smash{\check{\mathcal{P}}_n(\lambda)}$ projects along the complementary subspace $\mathrm{ran}(\mathcal{Q}_{-,n}(-\frac{n}{2}T;\lambda))$. Hence, the first $km - l_0$ columns of $\mathcal{H}_{n,\gamma}(\lambda)$ form a basis of $\mathrm{ran}(I-\smash{\check{\mathcal{P}}_n(\lambda)}) = \ker(\mathcal{Q}_{+,n}(\frac{n}{2} T;\lambda))$. Similarly, the last $l_0$ columns of $\mathcal{H}_{n,\gamma}(\lambda)$ constitute a basis of the complementary subspace $\mathrm{ran}(\mathcal{Q}_{-,n}(-\frac{n}{2}T;\lambda))$. Therefore, provided $n \in \mathbb{N}$ is sufficiently large, $\mathcal{H}_{n,\gamma}(\lambda)$ is invertible for each $\lambda \in \smash{\overline{B}_{\lambda_0}(\varrho_0)}$ and $\gamma \in S^1$. 

Recall that~\eqref{variationalpercutoff} possesses exponential dichotomies on $\R_\pm$ with projections $\mathcal{Q}_{\pm, n}(\pm x;\lambda), x \geq 0$ and $\lambda$- and $n$-independent constants, which we denote by $C_1,\mu_1 > 0$. Combining the latter with~\eqref{projestanb} and~\eqref{projest4}, we obtain a $\lambda$- and $n$-independent constant $M_4 > 0$ such that
\begin{align} \begin{split} \left\|T_n(0,\tfrac{n}{2} T;\lambda)\left(I-\check{\mathcal{P}}_n(\lambda)\right)\right\|, \left\|T_n(0,-\tfrac{n}{2} T;\lambda)\check{\mathcal{P}}_n(\lambda)\right\| \leq  M_4 \eu^{-\mu_1 \frac{n}{2} T}, \\
\left\|T_n(-\tfrac{n}{2} T,0;\lambda)\mathcal{B}_{\mathrm{u},n}(\lambda)\right\|, \left\|T_n(\tfrac{n}{2} T,0;\lambda)\mathcal{B}_{\mathrm{s},n}(\lambda)\right\| \leq  M_4 \eu^{-\mu_1 \frac{n}{2} T}, 
\end{split}
\label{basesest}
\end{align}
for each $\lambda \in \overline{B}_{\lambda_0}(\varrho_0)$. Therefore, using the estimates~\eqref{projestanb} and~\eqref{basesest}, we find an $n$-, $\gamma$- and $\lambda$-independent constant $M_5 > 0$ such that, provided $n \in \mathbb{N}$ is sufficiently large, it holds
\begin{align*}
\begin{split}
\left\|\left(T_n(0,-\tfrac{n}{2} T;\lambda) - \gamma T_n(0,\tfrac{n}{2} T;\lambda)\right)\mathcal{H}_{n,\gamma}(\lambda) - \left(B_u(\lambda) \mid B_s(\lambda)\right)\right\| &\leq M_5\left(\delta_n + \eu^{-\mu_1 n T}\right), \\
\left\|\left(T_n(0,-\tfrac{n}{2} T;\lambda) - \gamma T_n(0,\tfrac{n}{2} T;\lambda)\right)\mathcal{H}_{n,\gamma}(\lambda)\right\| &\leq M_5, 
\end{split}
\end{align*}
for each $\lambda \in \overline{B}_{\lambda_0}(\varrho_0)$ and $\gamma \in S^1$. So, taking determinants, we establish an $n$-, $\gamma$- and $\lambda$-independent constant $M_6 > 0$ such that 
\begin{align*} \left|E_{n,\gamma}(\lambda) \det(\mathcal{H}_{n,\gamma}(\lambda)) - \mathcal{E}(\lambda)\right| \leq M_6\left(\delta_n + \eu^{-\mu_1 n T}\right), \end{align*}
for each $\lambda \in \overline{B}_{\lambda_0}(\varrho_0)$ and $\gamma \in S^1$. 

Let $\varrho \in (0,\varrho_0)$. Since $\mathcal{E}$ does not vanish on $\partial B_{\lambda_0}(\varrho)$, the latter estimate yields, provided $n \in \mathbb{N}$ is sufficiently large, that
\begin{align*} \left|E_{n,\gamma}(\lambda) \det(\mathcal{H}_{n,\gamma}(\lambda)) - \mathcal{E}(\lambda)\right| 
< \left|\mathcal{E}(\lambda)\right| \end{align*}
for each $\lambda \in \partial B_{\lambda_0}(\varrho)$ and $\gamma \in S^1$. We recall that $\det(\mathcal{H}_{n,\gamma}(\cdot))$ is nonzero and the functions $\mathcal{E}$, $\det(\mathcal{H}_{n,\gamma}(\cdot))$ and $E_{n,\gamma}$ are analytic on the open disk $B_{\lambda_0}(\varrho_0) \subset \Omega$, which contains $\partial B_{\lambda_0}(\varrho)$. So, applying Rouch\'e's theorem to the latter inequality and noting that $\lambda_0$ is the only root of $\mathcal{E}$ in $B_{\lambda_0}(\varrho_0)$, which has multiplicity $m_0$, we find that the Evans function $E_{n,\gamma}$ has precisely $m_0$ zeros in $B_{\lambda_0}(\varrho)$ (counting with multiplicities) for each $\gamma \in S^1$. This proves the first assertion.

The second assertion immediately follows from the first assertion by taking $\gamma = \eu^{\iu \xi nT}$ and applying~\cite[Proposition~2.5]{Gardner93}, see also~\cite[Lemmas~8.4.1 and~8.4.2]{KapitulaPromislow2013}. Since $\El_{0,\per}(u_n) = \El_\per(u_n)$, the third assertion is a direct consequence of the second. Finally, the third implies the fourth assertion by evoking~\eqref{Bloch_inclusion}.
\end{proof}

\section{Applications}\label{sec:applications}

In this section, we employ our methods to construct multifronts and periodic pulse solutions in specific prototype models and analyze their stability. To illustrate the applicability of our theory in a simple setting, we first consider a reaction-diffusion toy model. We then focus on a Klausmeier reaction-diffusion-advection system which describes the dynamics of vegetation patterns on periodic topographies~\cite{Bastiaansen2020}. Finally, we consider the Gross-Pitaevskii equation with periodic potential, which arises in the study of Bose-Einstein condensates in optical lattices~\cite{Pelinovsky2011}. Our findings are supported by numerical simulations performed with the MATLAB package \texttt{pde2path}~\cite{pde2path}.

\subsection{A reaction-diffusion model problem}

We consider the scalar reaction-diffusion equation
\begin{align}\label{eq:RDE_toy}
    \partial_tu = \partial_{x}^2u + \eps V(x) u - \sin(u), \qquad u(x,t) \in \R, \, x \in \R, \, t \geq 0,
\end{align}
where $V \in C^1(\R)$ is a given real-valued potential of period $T > 0$. Here, the parameter $\eps \geq 0$ measures the strength of the potential $V$ and will serve as a bifurcation parameter.

We are interested in the existence and spectral stability of stationary multifronts and periodic pulse solutions to~\eqref{eq:RDE_toy}. Stationary solutions to~\eqref{eq:RDE_toy} solve the ODE
\begin{align}\label{eq:RDE_toy_stationary}
    \partial_{x}^2u + \eps V(x) u - \sin(u) = 0,
\end{align}
which is of the form~\eqref{existence_problem}. 

Let $\unu \in L^\infty(\R)$. For the upcoming spectral stability analysis, we define the closed differential operator $L_\eps(\unu)\colon D(L_\eps(\unu)) \subset L^2(\R) \to L^2(\R)$ with dense domain $D(L_\eps(\unu)) = H^2(\R)$ by
$$
    L_\eps(\unu)= \partial_x^2 + \eps V - \cos(\unu).
$$
Since the operator $L_\eps(\unu)$ is self-adjoint, its spectrum must be confined to the numerical range, leading to the following spectral a-priori bound.

\begin{Lemma} \label{lem:a_priori_toy}
Let $\varrho > 0$. Then, the spectrum of the operator $L(\unu) \colon D(L(\unu)) \subset L^2(\R) \to L^2(\R)$ with $D(L(\unu)) = H^2(\R)$, given by
\begin{align*}
L(\unu) = \partial_x^2 + \unu,
\end{align*}
satisfies $\sigma(L(\unu)) \subset (-\infty,\varrho]$ for all real-valued $\unu \in L^\infty(\R)$ with $\|\unu\|_{L^\infty} \leq \varrho$.
\end{Lemma}
\begin{proof}
Because $L(\unu)$ is self-adjoint, its spectrum must be contained in the numerical range, which is confined to $(-\infty,\varrho]$. 
\end{proof}

\subsubsection{Existence and spectral stability of fronts for \texorpdfstring{$\eps=0$}{epsilon=0}} \label{sec:RDE_toy_0}

Stationary front solutions of~\eqref{eq:RDE_toy_stationary} for $\eps = 0$ correspond to heteroclinic solutions 
to the autonomous system
\begin{align}\label{eq:RDE_toy_stationary_0}
    \partial_{x}^2u - \sin(u) = 0.
\end{align}
We introduce the coordinates $(u,v)^\top = (u,\partial_x u)^\top$ and write~\eqref{eq:RDE_toy_stationary_0} as the first-order system
\begin{align}\label{eq:RDE_toy_Hamiltonian}
        \partial_x
        \begin{pmatrix}
            u \\ v
        \end{pmatrix} = J \nabla H(u,v), 
        \qquad
        J :=
        \begin{pmatrix}
            0 & 1 \\ -1 & 0
        \end{pmatrix},
\end{align}
with Hamiltonian 
$$H(u,v) = \frac{1}{2} v^2 + \cos(u).$$
Solutions to~\eqref{eq:RDE_toy_Hamiltonian} lie on the level sets of $H$, see Figure~\ref{fig:phase_portrait_RDE}. Thus, we find infinitely many heteroclinics in~\eqref{eq:RDE_toy_Hamiltonian}, connecting the fixed points $(2\pi k,0)^\top$ to $(2\pi (k\pm 1),0)^\top$ for all $k \in \Z$. The associated front solutions to~\eqref{eq:RDE_toy_stationary_0} admit the explicit formula
\begin{align}\label{eq:RDE_front_0}
    u_{0,k,\pm 1} (x) =  4 \arctan\left(\eu^{\pm x} \right) + 2 \pi \min\{k,k\pm 1\}, \qquad k \in \Z.
\end{align}

\begin{figure}[t]
    \centering
    \includegraphics[width=0.35\textwidth]{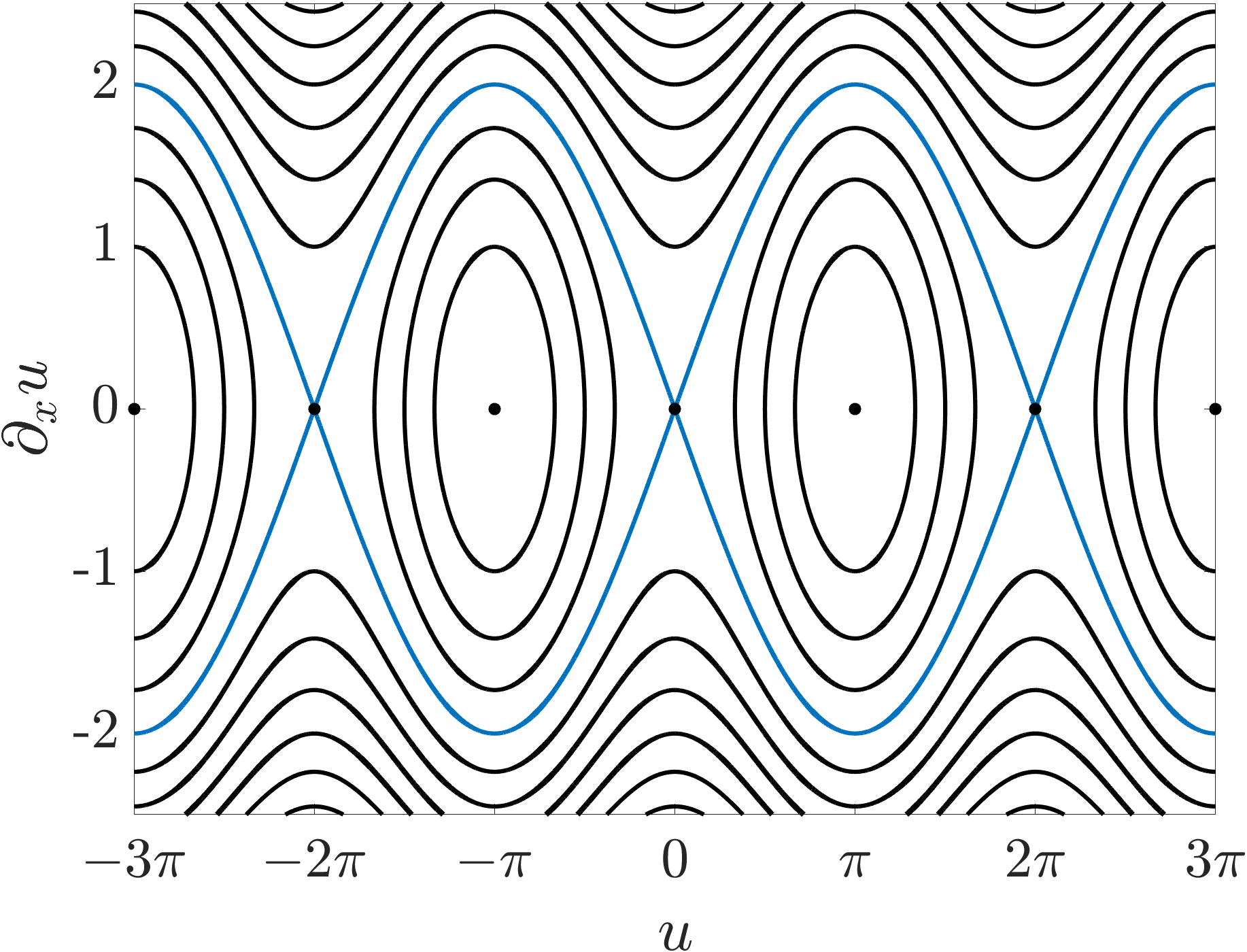} \hspace{2em}
    \includegraphics[width=0.4\textwidth]{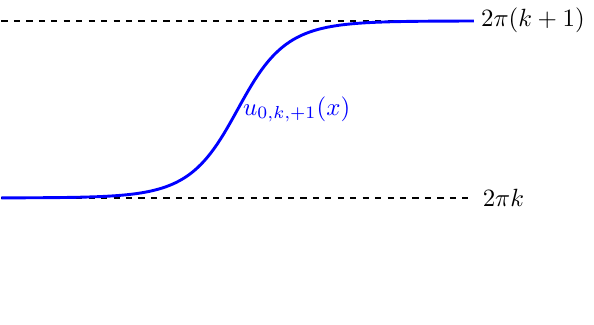}
    \caption{Left: phase portrait of the Hamiltonian system~\eqref{eq:RDE_toy_Hamiltonian}. Blue curves correspond to heteroclinic orbits connecting the fixed points $2\pi k$ to $2\pi (k\pm1)$ for each $k \in \Z$. Black dots correspond to equilibria of the system. Right: plot of the associated front solution $u_{0,k,+1}$ to~\eqref{eq:RDE_toy_stationary_0}.}
    \label{fig:phase_portrait_RDE}
\end{figure}

Fix $k \in \Z$. We examine the spectrum of the linearization $L_0(u_{0,k,\pm 1})$ about the front solution $u_{0,k,\pm 1}$ of~\eqref{eq:RDE_toy} at $\eps = 0$. A simple calculation reveals $\sigma(L_0(2\pi \ell)) = (-\infty,-1]$ for $\ell \in \Z$. Hence, Proposition~\ref{prop:essential_spec1front} yields $\sigma_\text{ess}(L_0(u_{0,k,\pm 1})) = (-\infty,-1]$. Moreover, by translational symmetry of~\eqref{eq:RDE_toy_stationary_0}, $0$ is a simple eigenfunction of $L_0(u_{0,k,\pm 1})$ with eigenfunction $u_{0,k,\pm1}'$. Since $u_{0,k,\pm1}'$ has no zeros, Sturm-Liouville theory, cf.~\cite[Theorem 2.3.3]{KapitulaPromislow2013}, yields that the front $u_{0,k,\pm1}$ is spectrally stable with simple eigenvalue $\lambda=0$, cf.~Definition~\ref{def:spectral_stability}.

\subsubsection{Existence and spectral stability of fronts for \texorpdfstring{$\eps>0$}{epsilon>0}}

In the following, we prove that the front solutions, obtained in~\S\ref{sec:RDE_toy_0}, persist for small values of $\eps>0$ under a generic assumption on the periodic potential $V$, which can be checked analytically or numerically. The fronts connect $T$-periodic end states $v_-(\eps)$ to $v_+(\eps)$, see Figure~\ref{fig:RDE_toy_1fronts}. Moreover, we establish the front's spectral (in)stability and nondegeneracy.

The existence and spectral analysis of the front solutions to~\eqref{eq:RDE_toy_stationary} consists of three steps. First, we construct their periodic end states by bifurcating from the fixed points $2\pi k$, $k \in \Z$ of~\eqref{eq:RDE_toy_stationary_0}. Second, we prove that the shifted front $u_{0,k,\pm1,\varsigma} := u_{0,k,\pm1}(\cdot-\varsigma)$ perturbs into a nondenegerate front solution of~\eqref{eq:RDE_toy_stationary} for small $\eps \neq 0$, provided that $\varsigma_0 \in \R$ is a simple zero of the \emph{effective potential}
\begin{align}\label{effective_potential}
    V_\text{eff}(\varsigma) = \int_\R V(x+\varsigma) u_{0,k,\pm1}(x)u_{0,k,\pm1}'(x) \de x.
\end{align}
In the third step, we derive a stability criterion saying that the fronts are strongly spectrally stable if $\eps V_\text{eff}'(\varsigma_0)>0$ and that they are spectrally unstable if $\eps V_\text{eff}'(\varsigma_0)<0$.

\begin{Theorem}\label{thm:RDE_1front}
Let $k\in \Z$, $l \in \{\pm 1\}$, and $T > 0$. Let $V \in C^1(\R)$ be $T$-periodic. Assume that there exists $\varsigma_0 \in \R$ such that the effective potential $V_\textup{eff} \colon \R \to \R$, given by~\eqref{effective_potential}, has a simple zero at $\varsigma_0$. Then, there exist $C,\eps_0,\varrho>0$ such that for all $\eps\in (-\eps_0,\eps_0) \setminus \{0\}$ there exists a nondegenerate solution $u(\eps)$ to~\eqref{eq:RDE_toy_stationary}, satisfying
\begin{align} \label{bounds_front_toy}
    \|u(\eps)-u_{0,k,l}(\cdot - \varsigma_0)\|_{L^\infty} \leq C \eps, \qquad
    \chi_\pm\left(u(\eps)-v_{\pm}(\eps)\right) \in H^2(\R),
\end{align}
where $\chi_\pm\colon \R \to [0,1]$ is a smooth partition of unity such that $\chi_+$ is supported on $(-1,\infty)$ and $\chi_-$ is supported on $(-\infty,1)$, $u_{0,k,l}$ is the front given by~\eqref{eq:RDE_front_0}, and $v_{\pm}(\eps) \in H_\per^2(0,T)$ are periodic solutions to~\eqref{eq:RDE_toy_stationary} with
\begin{align}\label{eq:toy_RDE_estimates_periodic}
\|v_{-}(\eps)-2\pi k\|_{H_\per^2(0,T)},\|v_{+}(\eps)-2\pi (k+l)\|_{H_\per^2(0,T)} \leq C\eps.
\end{align}
Finally, it holds
$$
    \sigma(L_\eps(u(\eps))) \subset (-\infty,-\varrho] \cup \{\lambda_0(\eps)\}
$$
for $\eps \in (-\eps_0,\eps_0)$, where $\lambda_0(\eps)$ is a real simple eigenvalue of $L_\eps(u(\eps))$ obeying the expansion 
$$\lambda_0(\eps) = -\eps \frac{V_\textup{eff}'(\varsigma_0)}{\|u_{0,k,l}'\|_{L^2}^2} + \mathcal{O}(\eps^2).$$
\end{Theorem}

\begin{figure}[t]
    \centering
    \includegraphics[width=0.3\textwidth]{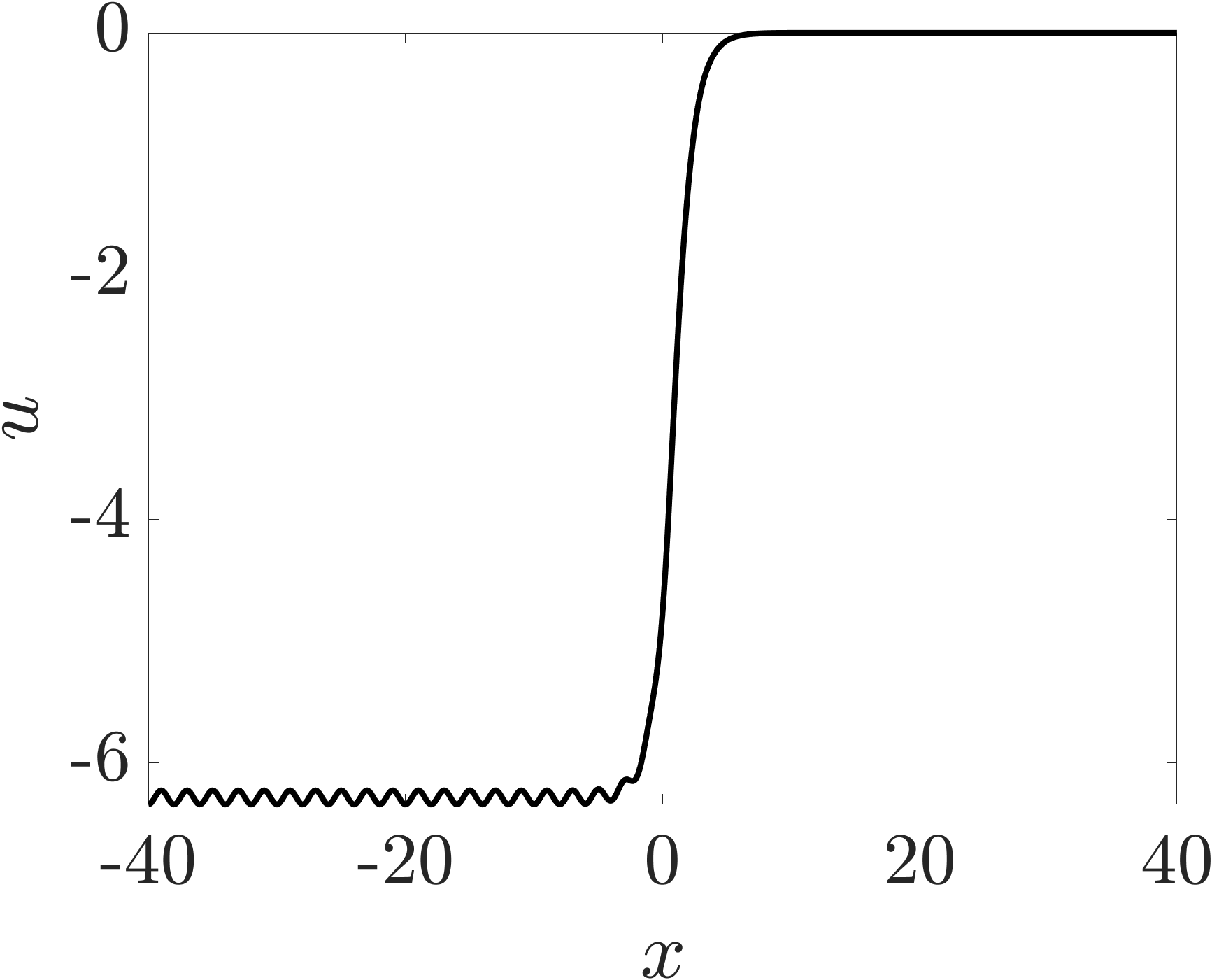} \hspace{1em}
    \includegraphics[width=0.3\textwidth]{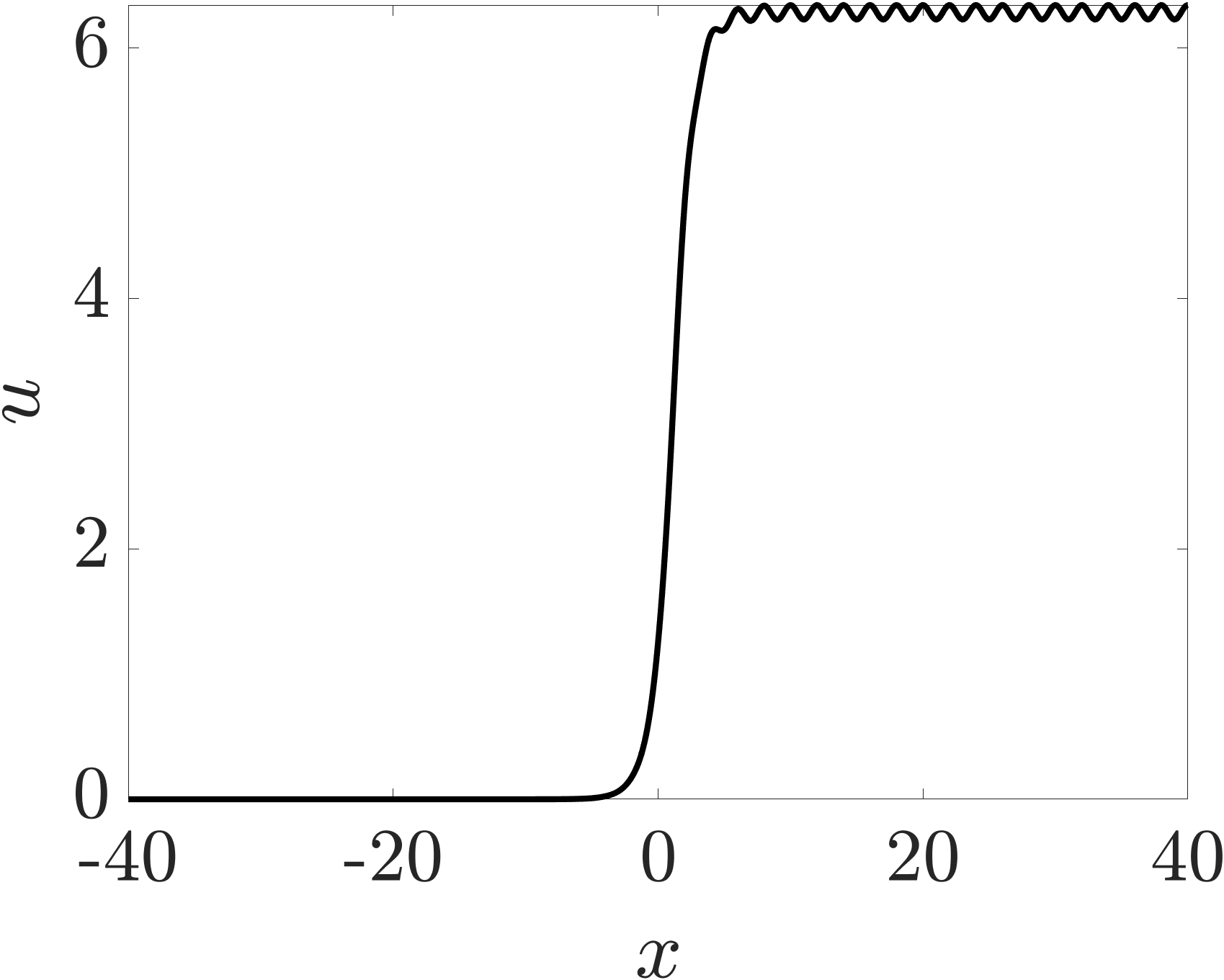}\hspace{1em}
    \includegraphics[width=0.3\textwidth]{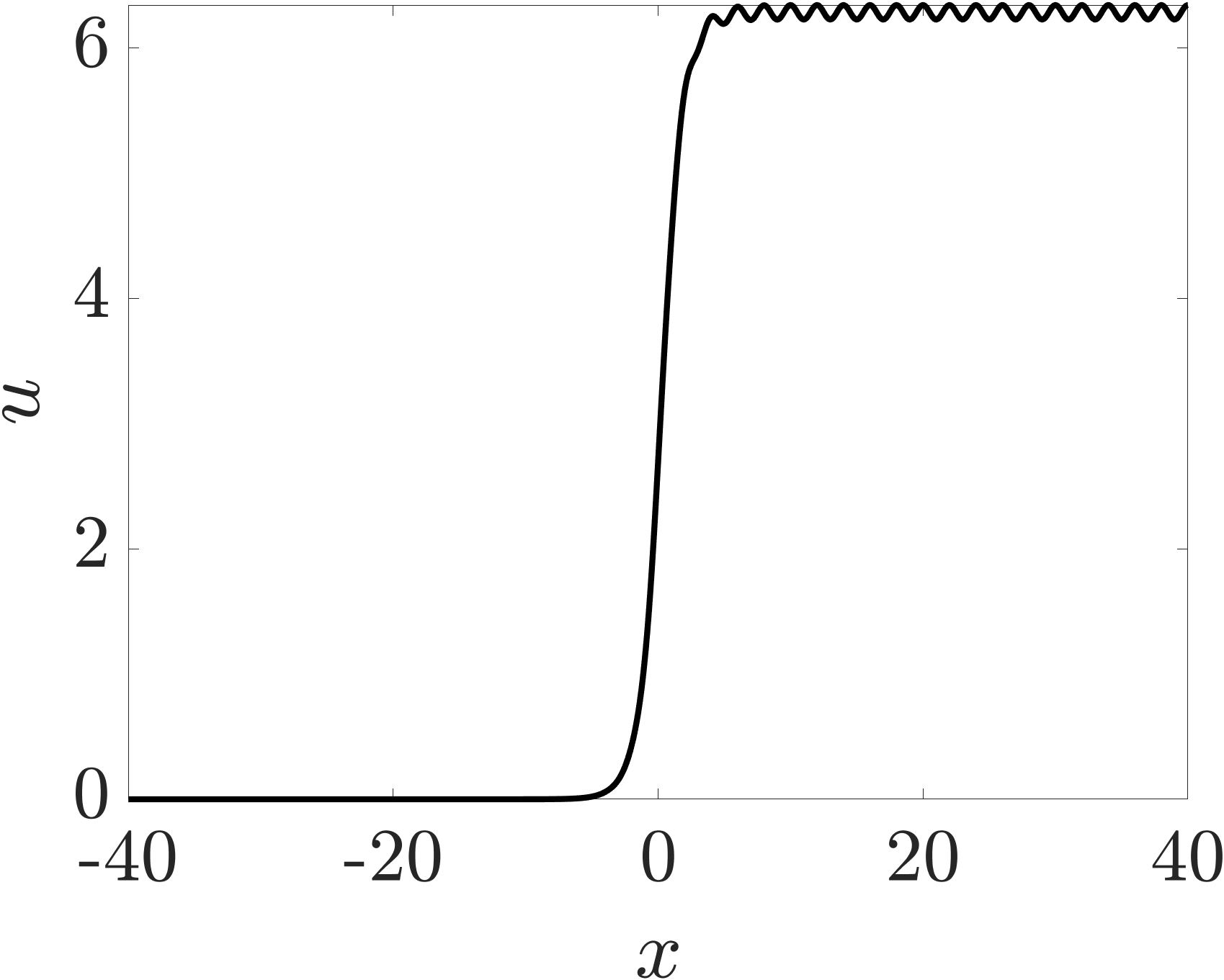}\\\vspace{1em}
    \includegraphics[width=0.3\textwidth]{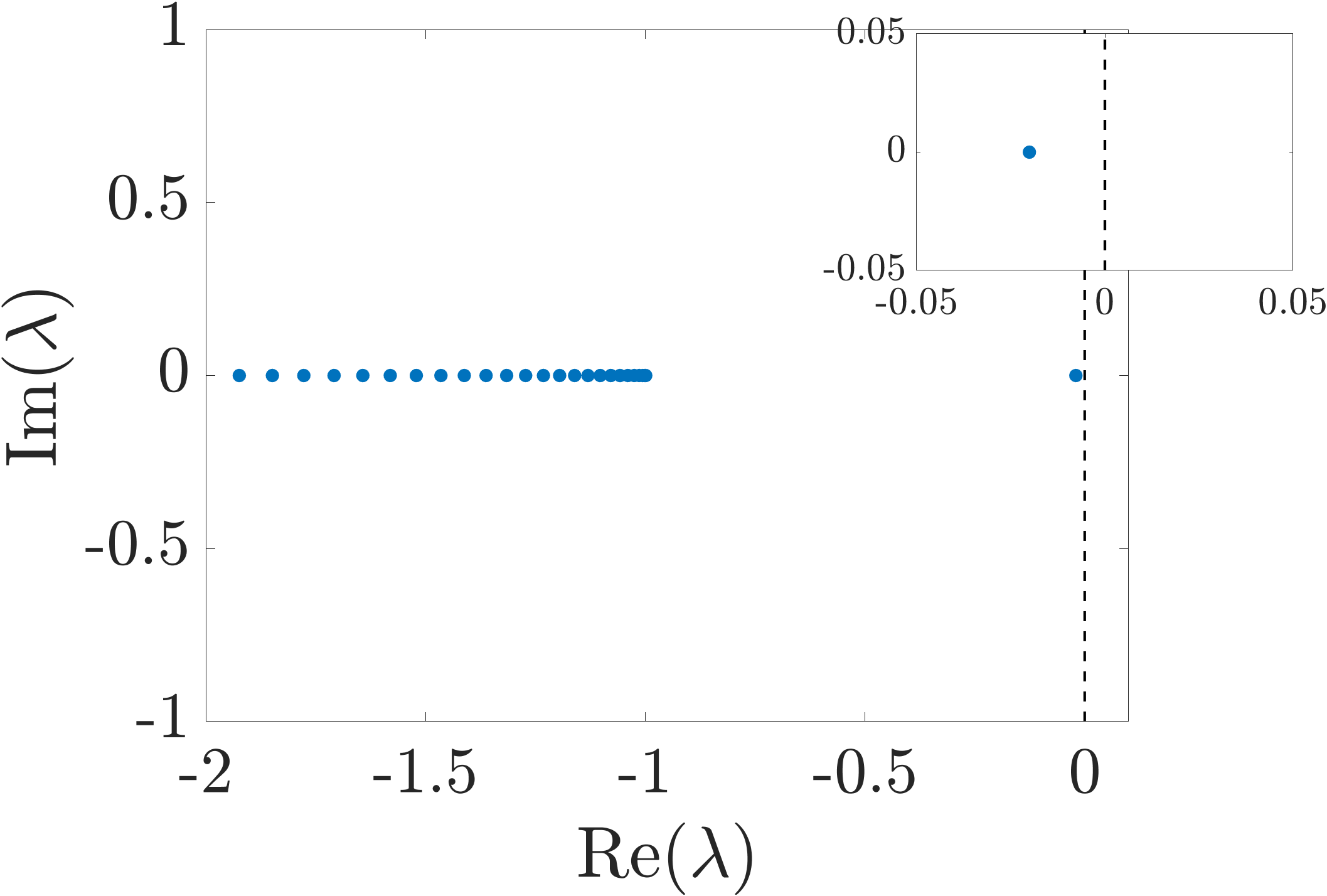} \hspace{1em}
    \includegraphics[width=0.3\textwidth]{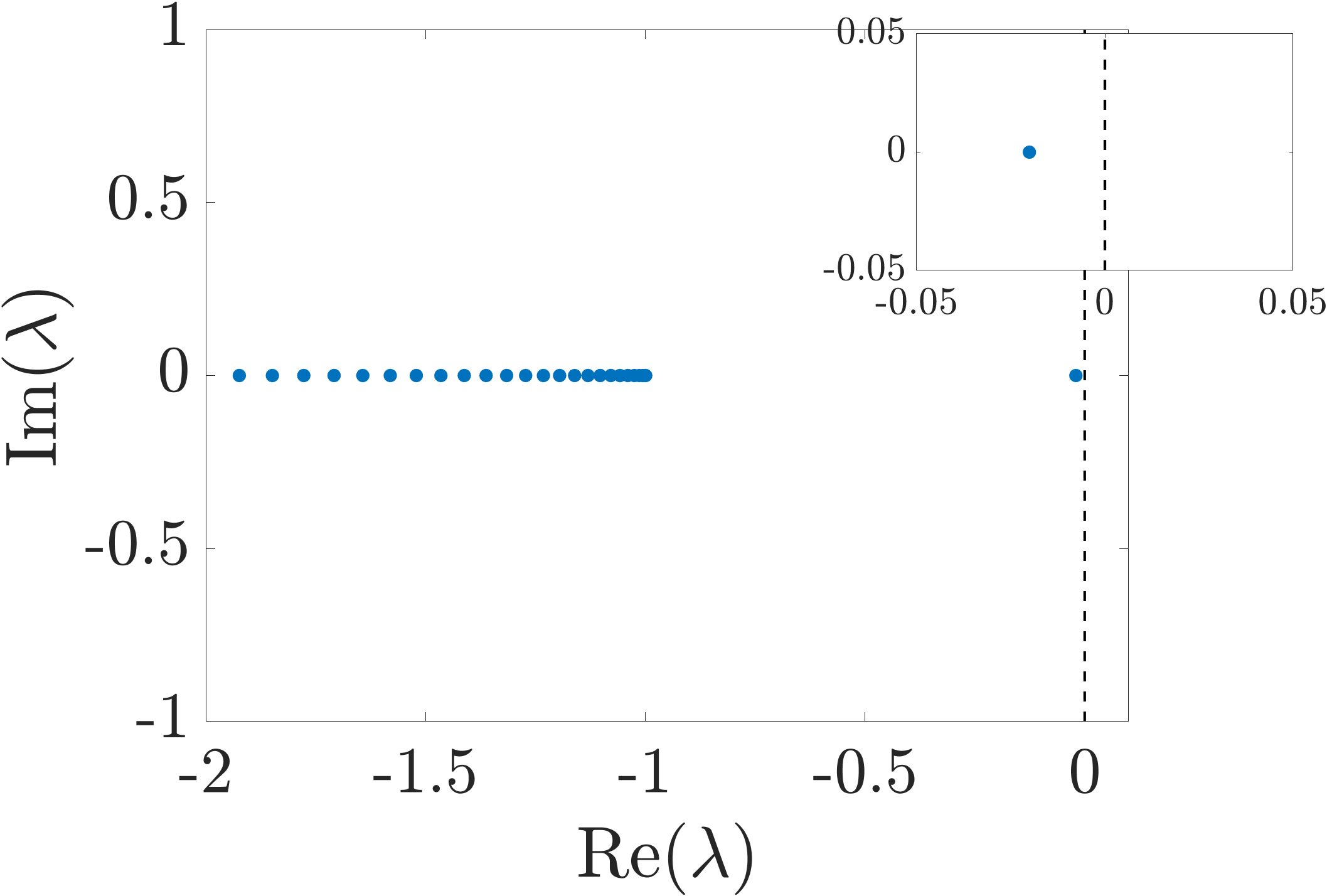}\hspace{1em}
    \includegraphics[width=0.3\textwidth]{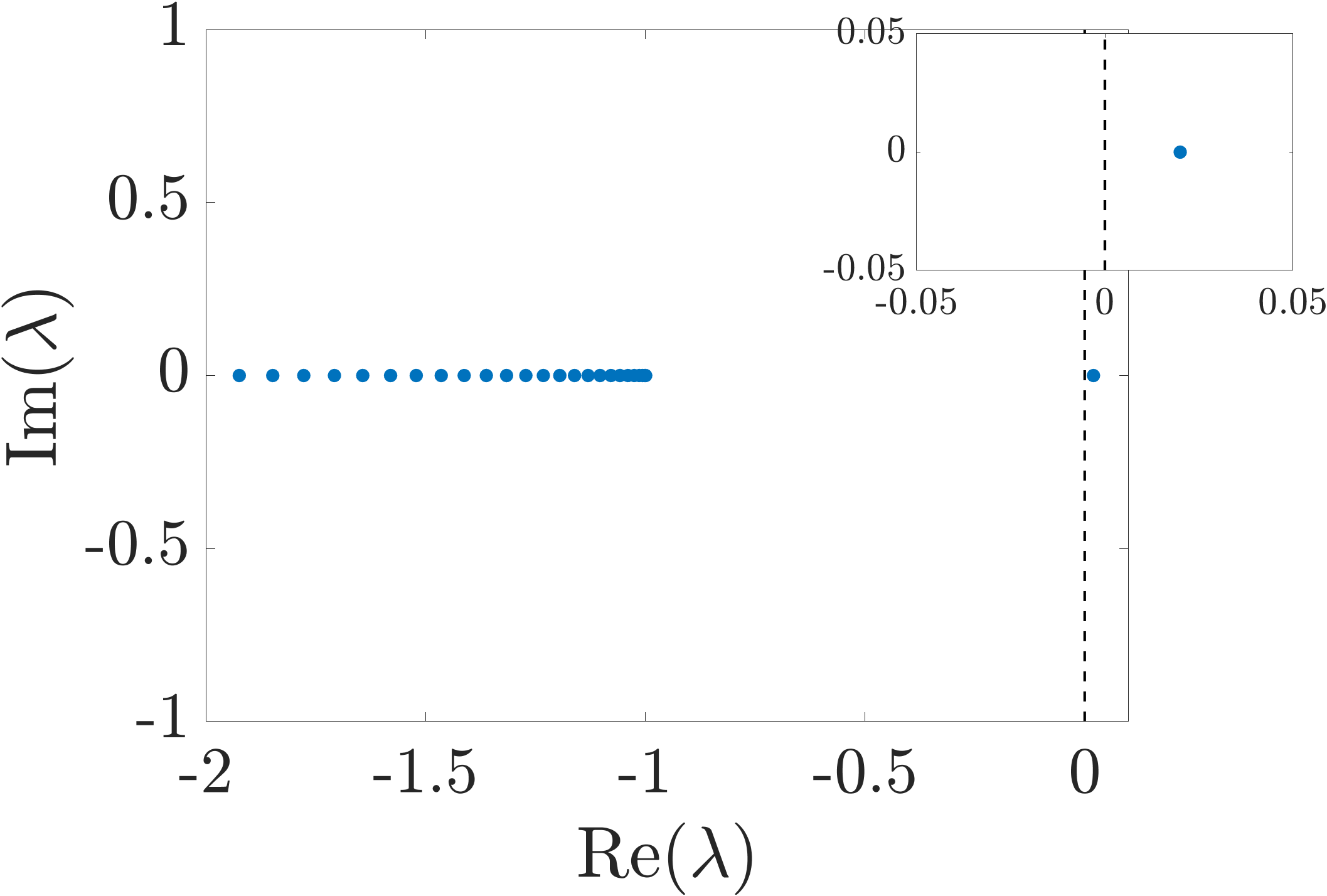}
    \caption{Approximations of stationary 1-front solutions to~\eqref{eq:RDE_toy}, along with their spectra, for system coefficients $\eps = 0.1$ and $V(x) = \cos(\pi x)$. The insets provide a closer view of the small eigenvalues near zero. The left and middle panels depict strongly spectrally stable 1-front solutions that connect the periodic state near $-2\pi$ to $0$ and $0$ to the periodic state near $2\pi$, respectively. The right panel depicts a spectrally unstable front solution connecting $0$ to the periodic state near $2\pi$. The $1$-front solutions are obtained through numerical continuation with the MATLAB package \texttt{pde2path}~\cite{pde2path} by starting from the explicit 1-front solutions $u_{0,k,+1,\varsigma}$ for $k \in \{-1,0\}$ and $\varsigma \in \R$.}
    \label{fig:RDE_toy_1fronts}
\end{figure}

\begin{proof}
We begin with the construction of the period end states $v_\pm(\eps)$. To this end, we consider the nonlinear map $\mathcal{F}_\per \colon H_\per^2(0,T) \times \R \to L_\per^2(0,T)$ given by
$$
    \mathcal{F}_\per(v,\eps) = v'' + \eps V v - \sin(v).
$$
We observe that $\mathcal{F}_\per$ is well-defined and smooth. Fix $j \in \Z$. Then, we have $\mathcal{F}_\per(2\pi j,0) = 0$ and
$$
    \partial_v \mathcal{F}_\per(2\pi j,0)  = \partial_x^2 - 1
$$
is invertible. Therefore, the implicit function theorem implies that there exist $\eps_1 > 0$ and a locally unique smooth map $v \colon (-\eps_1,\eps_1) \to H^2_\per(0,T)$ with
$$
v(0) = 2\pi j, \qquad    \mathcal{F}_\per(v(\eps),\eps) = 0
$$
for all $\eps \in (-\eps_1,\eps_1)$. We denote by $v_+(\eps)$ the locally unique periodic solution bifurcating from $2\pi(k+l)$ and by $v_-(\eps)$ the periodic solution bifurcating from $2\pi k$ for $\eps \in (-\eps_1,\eps_1)$. The bound~\eqref{eq:toy_RDE_estimates_periodic} is a direct consequence of the smoothness of $v_\pm$. 

In the next step, we construct the interface connecting the state $v_-(\eps)$ to $v_+(\eps)$. Accounting for the fact that the potential breaks the translational symmetry of~\eqref{eq:RDE_toy_stationary}, we impose the ansatz
\begin{align} \label{ansatz_toy}
    u = v_-(\eps) \chi_- + v_+(\eps) \chi_+ + u_{0,k,l,\varsigma} - 2\pi k\chi_- -2\pi (k+l)\chi_+ + w
\end{align}
for the desired front solution to~\eqref{eq:RDE_toy_stationary_0}, where $w \in H^2(\R)$ is a small correction term and $u_{0,k,l,\varsigma} = u_{0,k,l}(\cdot-\varsigma)$ is the shifted front solution to~\eqref{eq:RDE_toy_stationary_0}. Abbreviating $A= \partial_x^2$, $\mathcal{N}(u) =-\sin(u)$, we write the existence problem~\eqref{eq:RDE_toy_stationary} as
\begin{align} \label{existence_toy}
    A u + \eps V u + \mathcal{N}(u) = 0. 
\end{align}
Inserting the ansatz~\eqref{ansatz_toy} into~\eqref{existence_toy} leads to an equation for the correction $w$ and the shift parameter $\varsigma$ of the form 
\begin{align}\label{eq:existence_proof_RDE_front}
    \mathcal{F}(w,\eps,\varsigma) =0
\end{align}
with
\begin{align*}
    \mathcal{F}(w,\eps,\varsigma) = L_\eps(u_{0,k,l,\varsigma}+ \chi_- \tilde v_-(\eps)+ \chi_+ \tilde v_+(\eps)) w + R(\eps,\varsigma) + \check{\mathcal{N}}(w,\eps,\varsigma),
\end{align*}
where we denote $\tilde v_-(\eps) = v_-(\eps) -2\pi k $, $\tilde v_+(\eps) = v_+(\eps) -2\pi (k+l)$ and
\begin{align*}
    R(\eps,\varsigma) &= A (u_{0,k,l,\varsigma} + \chi_- \tilde v_-(\eps) + \chi_+ \tilde v_+(\eps)) + \eps V  (u_{0,k,l,\varsigma} + \chi_- \tilde v_-(\eps) + \chi_+ \tilde v_+(\eps)) \\
    &\quad + \mathcal{N} (u_{0,k,l,\varsigma} + \chi_- \tilde v_-(\eps) + \chi_+ \tilde v_+(\eps)) , \\
    \check{\mathcal{N}}(w,\eps,\varsigma) &= \mathcal{N}(u_{0,k,l,\varsigma} + \chi_- \tilde v_-(\eps) + \chi_+ \tilde v_+(\eps)+w)-\mathcal{N}(u_{0,k,l,\varsigma} + \chi_- \tilde v_-(\eps) + \chi_+ \tilde v_+(\eps))\\
    & \quad-\mathcal{N}'(u_{0,k,l,\varsigma} + \chi_- \tilde v_-(\eps) + \chi_+ \tilde v_+(\eps)) w,
\end{align*}
for $\eps \in (-\eps_1,\eps_1)$, $\varsigma\in\R$ and $w \in H^2(\R)$. We have $\mathcal{F}(0,0,\varsigma_0) = 0$ and $\partial_w \mathcal{F}(0,0,\varsigma_0) = L_0(u_{0,k,l,\varsigma_0})$ for all $\varsigma_0 \in \R$. We recall from~\S\ref{sec:RDE_toy_0} that the kernel of $L_0(u_{0,k,l,\varsigma_0})$ is spanned by $u_{0,k,l,\varsigma_0}'$. In particular, $L_0(u_{0,k,l,\varsigma_0})$ is not invertible. To address this, we employ Lyapunov-Schmidt reduction to solve~\eqref{eq:existence_proof_RDE_front}. We note that, since $L_0(u_{0,k,l,\varsigma_0})$ is self-adjoint, its range is given by the orthogonal complement $\{u_{0,k,l,\varsigma_0}\}^\perp$. 

We proceed by obtaining bounds on the residual $R$ and the nonlinearity $\check{\mathcal{N}}$. We employ estimate~\eqref{eq:toy_RDE_estimates_periodic}, rely on the continuous embedding $H^1_\per(0,T) \hookrightarrow L^\infty(\R)$, apply the mean value theorem twice, and use the identities $A(u_{0,k,l,\varsigma}) = -\mathcal{N}(u_{0,k,l,\varsigma})$ and $A v_-(\eps) + \eps V v_-(\eps) = -\mathcal{N}(v_-(\eps))$ to establish $\eps$- and $\varsigma$-independent constants $C_{1,2} > 0$ such that the pointwise estimate 
\begin{align*}
    |R(\eps,\varsigma)(x)| &= \left|\eps V(x)  \left(u_{0,k,l,\varsigma}(x) - 2\pi k\right) + \mathcal{N} \left(u_{0,k,l,\varsigma}(x) + v_-(\eps)(x) - 2\pi k\right) - \mathcal{N}(u_{0,k,l,\varsigma}(x))\right.\\ 
    &\qquad \left.- \, \left(\mathcal{N}(v_-(\eps)(x)) - \mathcal{N}(2\pi k)\right)\right| \\
    &\leq \eps \|V\|_{L^\infty}|u_{0,k,l,\varsigma}(x) - 2\pi k|\\ &\qquad + \, \sup_{|z| \leq \|u_{0,k,l,\varsigma} - 2\pi k\|_{L^\infty}} \left|\mathcal{N}'(v_-(\eps)(x)+z) - \mathcal{N}'(2\pi k + z)\right| \left|u_{0,k,l,\varsigma}(x) - 2 \pi k\right| \\
    & \leq \eps \|V\|_{L^\infty} |u_{0,k,l,\varsigma}(x) - 2\pi k| + C_1 \|v_-(\eps)- 2\pi k\|_{L^\infty} |u_{0,k,l,\varsigma}(x) - 2 \pi k| \\
    &\leq \eps C_2 |u_{0,k,l,\varsigma}(x) - 2 \pi k|
\end{align*}
holds for $x \leq -1$, $\eps \in (-\eps_1,\eps_1)$ and $\varsigma \in \R$. Similarly, we find an $\eps$- and $\varsigma$-independent constant $C_3 > 0$ such that
\begin{align*}
    |R(\eps,\varsigma)(x)| \leq \eps C_3 |u_{0,k,l,\varsigma}(x) - 2 \pi (k+l)|
\end{align*}
holds for $x \geq 1$, $\eps \in (-\eps_1,\eps_1)$ and $\varsigma \in \R$. Moreover, using estimate~\eqref{eq:toy_RDE_estimates_periodic}, the continuous embedding $H^1_\per(0,T) \hookrightarrow L^\infty(\R)$, the mean value theorem, and the identities $A(u_{0,k,l,\varsigma}) = -\mathcal{N}(u_{0,k,l,\varsigma})$ and $A v_\pm(\eps) = - \eps V v_\pm(\eps) -\mathcal{N}(v_\pm(\eps))$, we obtain an $\eps$- and $\varsigma$-independent constant $C_4 > 0$ such that
\begin{align*}
     |R(\eps,\varsigma)(x)| &\leq \|A (\chi_- \tilde v_-(\eps) + \chi_+ \tilde v_+(\eps))\|_{L^\infty} + \eps \|V\|_{L^\infty}  \|u_{0,k,l,\varsigma} + \chi_- \tilde v_-(\eps) + \chi_+ \tilde v_+(\eps)\|_{L^\infty}  \\& \quad\quad+ \|\mathcal{N} (u_{0,k,l,\varsigma} + \chi_- \tilde v_-(\eps) + \chi_+ \tilde v_+(\eps))- \mathcal{N}(u_{0,k,l,\varsigma})\|_{L^\infty}  \leq \eps C_4
\end{align*}
holds for $x \in [-1,1]$, $\eps \in (-\eps_1,\eps_1)$ and $\varsigma \in \R$. Combining the latter three estimates, we establish for each compact $\eps$-independent subset $\mathcal{K} \subset \R$, an $\eps$- and $\varsigma$-independent constant $C_0>0$ such that 
\begin{align} \label{resesttoy}
    \|R(\eps,\varsigma)\|_{L^2} \leq C_0 |\eps|
\end{align}
for $\eps \in (-\eps_1,\eps_1)$ and $\varsigma \in \mathcal{K}$. On the other hand, it follows from Taylor theorem, the estimate~\eqref{eq:toy_RDE_estimates_periodic}, and the continuous embedding $H^1_\per(0,T) \hookrightarrow L^\infty(\R)$, that there exists an $\eps$- and $\varsigma$-independent constant $C> 0$ such that
\begin{align} \label{nonlesttoy}
    \|\check{\mathcal{N}}(w,\eps,\varsigma)\|_{L^2} \leq C \|w\|_{H^2}^2
\end{align}
for $\eps \in (-\eps_1,\eps_1)$, $\varsigma\in \R$, and $w \in H^2(\R)$ with $\|w\|_{H^2}\leq 1$. 

Our next step is to use Lyapunov-Schmidt reduction to solve equation~\eqref{eq:existence_proof_RDE_front}. The reduction relies on a parameter-dependent decomposition of the spaces $H^2(\R)$ and $L^2(\R)$, which is induced by the orthogonal projections $P_\varsigma\colon H^\ell(\R) \to H^\ell(\R)$ and $P_\varsigma^\perp := I -P_\varsigma$ given by
$$ 
    P_\varsigma w =\frac{\langle u_{0,k,l,\varsigma}',w\rangle_{L^2}}{\|u_{0,k,l,\varsigma}'\|_{L^2}^2} u_{0,k,l,\varsigma}'
$$
for $\ell \in \N_0$ and $\varsigma \in \R$. We decompose the problem~\eqref{eq:existence_proof_RDE_front} into a regular and singular part which we complement with a phase condition, which leads to the equivalent problem
\begin{align}
    P_\varsigma^\perp L_\eps(u_{0,k,l,\varsigma}+ \chi_- \tilde v_-(\eps)+ \chi_+ \tilde v_+(\eps))P_\varsigma^\perp w + P_\varsigma^\perp \left( R(\eps,\varsigma) + \check{\mathcal{N}}(w,\eps,\varsigma) \right)&= 0, \label{eq:LS_RDE_1} \\
    P_\varsigma L_\eps(u_{0,k,l,\varsigma}+ \chi_- \tilde v_-(\eps)+ \chi_+ \tilde v_+(\eps))P_\varsigma^\perp w + P_\varsigma\left( R(\eps,\varsigma) + \check{\mathcal{N}}(w,\eps,\varsigma) \right)&= 0,\label{eq:LS_RDE_2}\\
    P_\varsigma w &= 0. \label{eq:LS_RDE_3}
\end{align} 
First, we solve the equations~\eqref{eq:LS_RDE_1} and~\eqref{eq:LS_RDE_3}, corresponding to the regular part of the system. To this end, we define the smooth nonlinear operator $\mathcal{G}\colon H^2(\R)\times (-\eps_1,\eps_1)\times \R \to L^2(\R)$ by
$$
    \mathcal{G}(w,\eps,\varsigma)= P_\varsigma^\perp L_\eps(u_{0,k,l,\varsigma}+ \chi_- \tilde v_-(\eps)+ \chi_+ \tilde v_+(\eps))P_\varsigma^\perp w + P_\varsigma^\perp \left( R(\eps,\varsigma) + \check{\mathcal{N}}(w,\eps,\varsigma) \right)+P_\varsigma w.
$$
We observe that $\mathcal{G}(w,\eps,\varsigma) = 0$ is equivalent to the equations~\eqref{eq:LS_RDE_1} and~\eqref{eq:LS_RDE_3}. We compute
$$
    \mathcal{G}(0,0,\varsigma) = 0, \qquad
    \partial_w\mathcal{G}(0,0,\varsigma) = P_\varsigma^\perp L_0(u_{0,k,l,\varsigma})P_\varsigma^\perp +P_\varsigma
$$
for $\varsigma \in \R$. The derivative $\partial_w\mathcal{G}(0,0,\varsigma)$ is invertible in the space of bounded linear operators from $H^2(\R)$ to $L^2(\R)$ for $\varsigma \in \R$. Therefore, the implicit function theorem yields $\eps_2 \in (0,\eps_1)$, an open neighborhood $W\subset H^2(\R)$ of $0$, and a smooth map $w \colon (-\eps_2,\eps_2) \times \R \to W$ such that the triple $(w,\eps,\varsigma) \in U \times (-\eps_2,\eps_2)\times \R$ solves
$$
    \mathcal{G}(w,\eps,\varsigma) = 0
$$
if and only if $(w,\eps,\varsigma) = (w(\eps,\varsigma),\eps,\varsigma)$ for $(\eps,\varsigma) \in (-\eps_2,\eps_2) \times \R$. Moreover, we have $w(0,\varsigma) = 0$ for $\varsigma \in \R$. We plug the solution of equations~\eqref{eq:LS_RDE_1} and~\eqref{eq:LS_RDE_3} into~\eqref{eq:LS_RDE_2} to arrive at the reduced problem
$$
    g(\eps,\varsigma) = 0,
$$
where $g \colon (-\eps_2,\eps_2) \times \R \to \R$ is given by
\begin{align*}
g(\eps,\varsigma) = \left\langle u_{0,k,l,\varsigma}', L_\eps\left(u_{0,k,l,\varsigma} + \chi_- \tilde v_-(\eps) + \chi_+ \tilde v_+(\eps)\right) P_\varsigma^\perp w(\eps,\varsigma) + R(\eps,\varsigma) + \check{\mathcal{N}}(w(\eps,\varsigma),\eps,\varsigma)\right\rangle_{L^2}.
\end{align*}
We solve the equation by desingularizing the smooth function $g$. To this end, we define the function $\tilde{g} \colon (-\eps_2,\eps_2) \times \R \to \R$ by
$$
    \tilde{g}(\eps,\varsigma)=
    \left\{ 
    \begin{array}{ll}
    \frac{g(\eps,\varsigma)}{\eps}, & \eps \neq 0,  \\
    \partial_\eps g(0,\varsigma), & \eps = 0.
    \end{array}
    \right.
$$
We observe that $\tilde{g}$ is smooth and obeys $\tilde{g}(0,\varsigma_0) = \partial_\eps g(0,\varsigma_0) = V_\text{eff}(\varsigma_0) = 0$ and $\partial_\varsigma\tilde{g}(0,\varsigma_0) = \partial_\varsigma\partial_\eps g(0,\varsigma_0) = V_\text{eff}'(\varsigma_0) \neq 0$. The implicit function theorem yields $\eps_3\in (0,\eps_2)$ and a smooth function $\varsigma \colon (-\eps_3,\eps_3) \to \R$ such that $\varsigma(0) = \varsigma_0$ and 
$$
    \tilde{g}(\eps,\varsigma(\eps)) = 0
$$
for all $\eps \in (-\eps_3,\eps_3)$. Consequently, the triple $(w(\eps,\varsigma(\eps)),\eps,\varsigma(\eps))$ solves the system~\eqref{eq:LS_RDE_1}-\eqref{eq:LS_RDE_3} for all $\eps \in (-\eps_3,\eps_3)$. We conclude that $u \colon (-\eps_3,\eps_3) \to L^\infty(\R)$ given by
$$
    u(\eps) =v_-(\eps) \chi_- + v_+(\eps) \chi_+ + u_{0,k,l,\varsigma(\eps)} - 2\pi k\chi_- -2\pi (k+l)\chi_+ + w(\eps,\varsigma(\eps)),
$$
is the desired front solution to~\eqref{eq:RDE_toy}, which satisfies~\eqref{bounds_front_toy}.

It remains to prove the assertions on the spectrum of the linearization operator $L_\eps(u(\eps))$. First, we recall from~\S\ref{sec:RDE_toy_0} that there exists $\varrho>0$ such that
\begin{align} \label{spectral_bound_toy}
    \sigma(L_0(u(0))) \subset (-\infty,-\varrho) \cup \{0\},
\end{align}
where $0\in \sigma(L_0(u(0)))$ is a simple eigenvalue. On the other hand, estimate~\eqref{bounds_front_toy} implies that $\|u(\eps)\|_{L^\infty}$ is bounded by an $\eps$-independent constant for $\eps \in (-\eps_3,\eps_3)$. Consequently, there exists by Lemma~\ref{lem:a_priori_toy} an $\eps$-independent constant $\varrho_1 > 0$ such that  $\sigma(L_\eps(u(\eps))) \subset (-\infty,\varrho_1]$. Combining the latter with the fact that $L_\eps(u(\eps)) - L_0(u(0))\colon L^2(\R) \to L^2(\R)$ is a bounded operator, we infer by~\cite[Theorem~IV.3.18]{Kato1995} and estimate~\eqref{spectral_bound_toy} that there exists $\eps_4 \in (0,\eps_3)$ such that
$$
\sigma(L_\eps(u(\eps))) \subset (-\infty,-\varrho) \cup \{\lambda_0(\eps)\}
$$
for all $\eps \in (-\eps_4,\eps_4)$, where $\lambda_0(\eps) \in \R$ is again a simple eigenvalue of $L_\eps(u(\eps))$. By~\cite[Proposition~I.7.2]{Kielhoefer2012} there exist $\eps_5\in (0,\eps_4)$ and $C^1$-curves $\lambda_0 \colon (-\eps_5,\eps_5) \to \R$ and $z \colon (-\eps_5,\eps_5) \to H^2(\R)$ with $z(0)=0$ and $\lambda_0(0)=0$ solving the eigenvalue problem
\begin{align*}
    L_\eps(u(\eps))(u_{0,k,l,\varsigma_0}' + z(\eps))
    =\lambda_0(\eps)(u_{0,k,l,\varsigma_0}' + z(\eps))
\end{align*}
for $\eps \in (-\eps_5,\eps_5)$. Taking the derivative on both sides with respect to $\eps$ and evaluating at $\eps=0$ yields
\begin{align}\label{proof_RDE1}
    L_0(u_{0,k,l,\varsigma_0}) \partial_\eps z(0)
    +V u_{0,k,l,\varsigma_0}'
    + \mathcal{N}''(u_{0,k,l,\varsigma_0})[\partial_\eps u(0),u_{0,k,l,\varsigma_0}']
    = \lambda_0'(0) u_{0,k,l,\varsigma_0}'.
\end{align}
On the other hand, differentiating the equation
\begin{align*}
    A u(\eps) + \eps V u(\eps) + \mathcal{N}(u(\eps)) = 0
\end{align*}
with respect to $x$ and $\eps$ and subsequently setting $\eps = 0$, we obtain
\begin{align}\label{proof_RDE2}
    L_0(u_{0,k,l,\varsigma_0}) \partial_\eps\partial_x u(0) + V'u_{0,k,l,\varsigma_0}+ V u_{0,k,l,\varsigma_0}'+\mathcal{N}''(u_{0,k,l,\varsigma_0})[\partial_\eps u(0),u_{0,k,l,\varsigma_0}'] 
    = 0.
\end{align}
Subtracting~\eqref{proof_RDE2} from~\eqref{proof_RDE1}, we arrive at
\begin{align*}
    L_0(u_{0,k,l,\varsigma_0})(\partial_\eps z(0) - \partial_\eps\partial_x u(0)) - V'u_{0,k,l,\varsigma_0} 
    = \lambda_0'(0) u_{0,k,l,\varsigma_0}'.
\end{align*}
Taking the $L^2$-scalar product of the last equation with $u_{0,k,l,\varsigma_0}' \in \ker(L_0(u_{0,k,l,\varsigma_0}))$, we establish
\begin{align*}
    \lambda_0'(0) \|u_{0,k,l,\varsigma_0}'\|_{L^2}^2 
    = -\int_\R V'(x) u_{0,k,l,\varsigma_0}(x) u_{0,k,l,\varsigma_0}'(x) \de x = - V_\text{eff}'(\varsigma_0),
\end{align*}
which finishes the proof.
\end{proof}

We observe that the front solutions, established in Theorem~\ref{thm:RDE_1front}, obey the assumptions~\ref{assH1} and~\ref{assH3}. Therefore, given any collection of $M$ such fronts with matching asymptotic end states, Theorems~\ref{t:existence_multifront} and~\ref{thm:instability_multifront} and Corollary~\ref{cor:stability_multifront} yield the existence and spectral stability of multifront solutions bifurcating from the formal concatenation of these $M$ fronts, see Figure~\ref{fig:RDE_toy_2fronts}.

\begin{Corollary}\label{cor:RDE_toy_existence_multifront}
Let $M \in \N$. Let $\{u_j\}_{j=1}^M$ be a sequence of front solutions to~\eqref{eq:RDE_toy_stationary}, established in Theorem~\ref{thm:RDE_1front}, with end states $v_{j,\pm}(\eps)$. Assume that it holds $v_{j,+}(\eps) = v_{j+1,-}(\eps)$ for $j = 1,\ldots,M-1$. Then, there exists $N \in \N$ such that for any $n \in \N$ with $n \geq N$ there exists a nondegenerate stationary multifront solution $\tilde{u}_n$ to~\eqref{eq:RDE_toy} of the form 
\begin{align*} \tilde{u}_n = a_n + \sum_{j = 1}^M \chi_{j,n} u_j(\cdot - jnT),
\end{align*}
where $\chi_{j,n}, j = 1,\ldots,M$ is the smooth partition of unity defined in~\S\ref{sec:m-front}, and $\{a_n\}_n$ is a sequence in $H^2(\R)$ converging to $0$ as $n \to \infty$. If the fronts $u_j$ are strongly spectrally stable for $j = 1,\dots,M$, then so is the multifront $\tilde{u}_n$. Moreover, if there exists $j_0 \in \{1,\dots,M\}$ such that $u_{j_0}$ is spectrally unstable, then so is $\tilde{u}_n$.
\end{Corollary}
\begin{proof}
Theorems~\ref{t:existence_multifront} and~\ref{thm:RDE_1front} yield the existence of the multifronts $\tilde{u}_n$. Since the primary fronts $u_j$ are nondegenerate for $j = 1 ,\ldots,M$, Theorem~\ref{thm:instability_multifront} yields that the multifront $\tilde{u}_n$ is also nondegenerate. Due to the continuous embedding $H^1(\R) \hookrightarrow L^\infty(\R)$, $\|\tilde{u}_n\|_{L^\infty}$ is bounded by an $n$-independent constant. Therefore, the statements about the spectral (in)stability of the multifront $\tilde{u}_n$ follow from Corollary~\ref{cor:stability_multifront}, Theorem~\ref{thm:instability_multifront}, and Lemma~\ref{lem:a_priori_toy}. 
\end{proof}

Applying Theorem~\ref{t:existence_periodic} to the nondegenerate multifront solutions established in Corollary~\ref{cor:RDE_toy_existence_multifront}, we find that multifronts connecting to the same periodic end state at $\pm \infty$ are accompanied by large wavelength periodic multipulse solutions. Their spectral stability follows from Corollary~\ref{cor:stability_periodic}, Theorem~\ref{thm:instability_periodic}, and Lemma~\ref{lem:a_priori_toy}.

\begin{Corollary}
Let $u$ be a multifront solution to~\eqref{eq:RDE_toy_stationary}, as established in Corollary~\ref{cor:RDE_toy_existence_multifront}. Assume that $u$ connects to the same periodic end state $v \in H_\per^2(0,T)$ as $x \to \pm \infty$. Then, there exists $N \in \N$ such that for all $n \in \N$ with $n \geq N$ there exists a stationary $nT$-periodic solution $u_n$ to~\eqref{eq:RDE_toy} given by
\begin{align*}
u_n(x) = \chi_n(x) u(x) + (1 - \chi_n(x)) v(x) + a_n(x), \qquad x \in \left[-\tfrac{n}{2}T,\tfrac{n}{2} T\right),
\end{align*}
where $\chi_n$ is the cut-off function from Theorem~\ref{t:existence_periodic}, and $\{a_n\}_n$ is a sequence with $a_n \in H_\per^2(0,nT)$ satisfying $\|a_n\|_{H_\per^2(0,nT)} \to 0$ as $n \to \infty$. Moreover, if $u$ is strongly spectrally stable, then so is $u_n$. Finally, if $u$ is spectrally unstable, then so is $u_n$.
\end{Corollary}

\begin{figure}[t]
    \centering
    \includegraphics[width=0.3\textwidth]{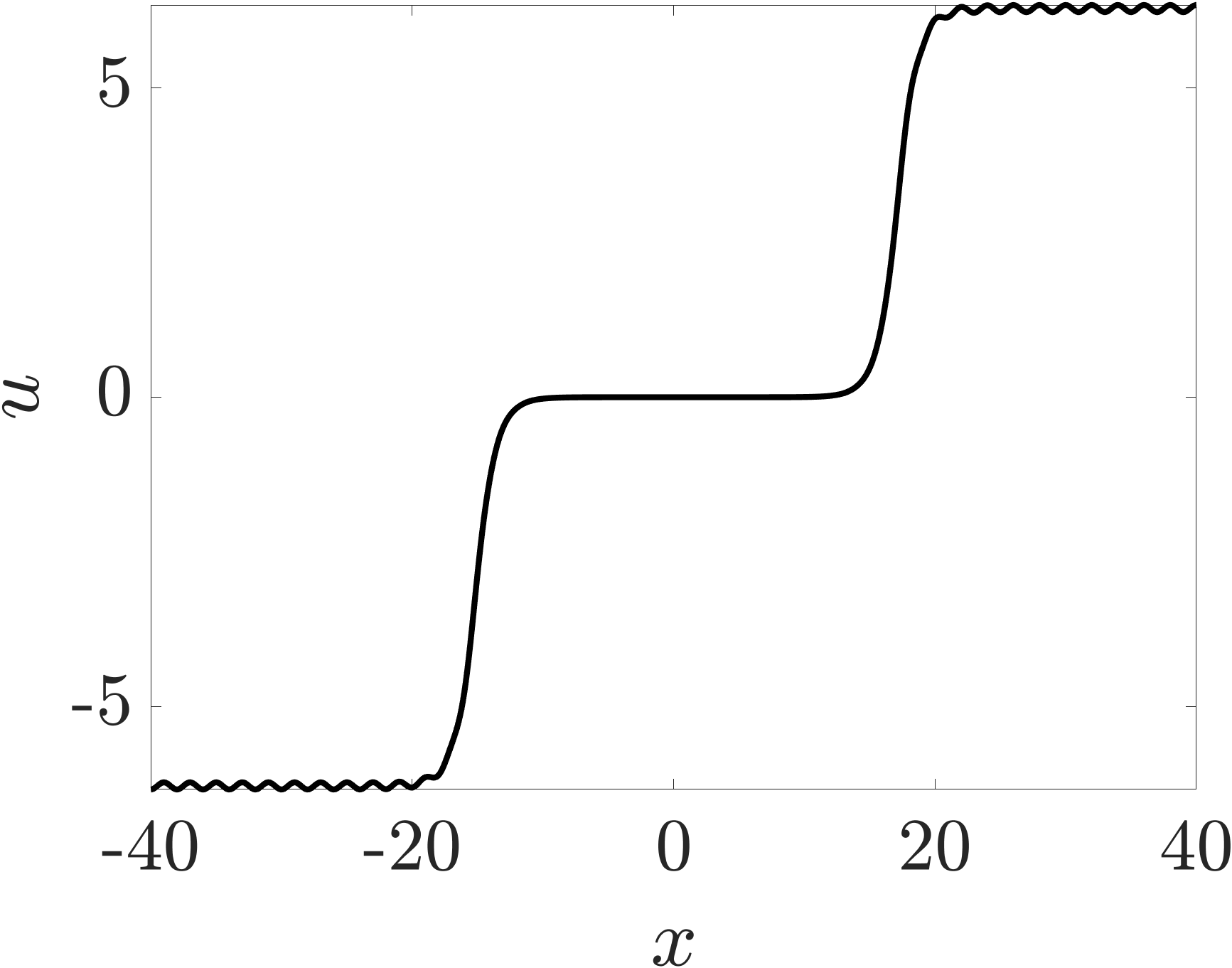} \hspace{1em}
    \includegraphics[width=0.3\textwidth]{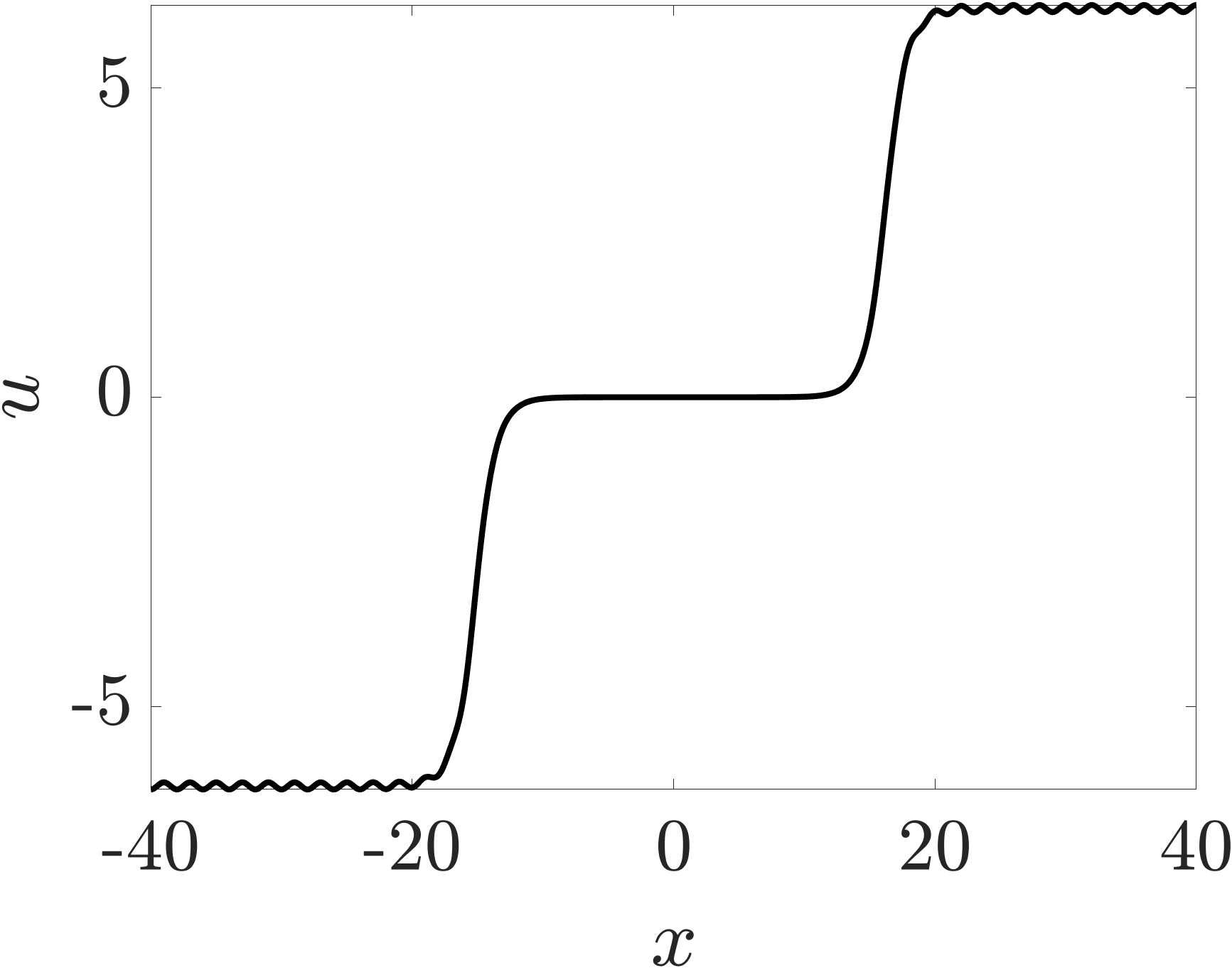}\\\vspace{1em}
    \includegraphics[width=0.3\textwidth]{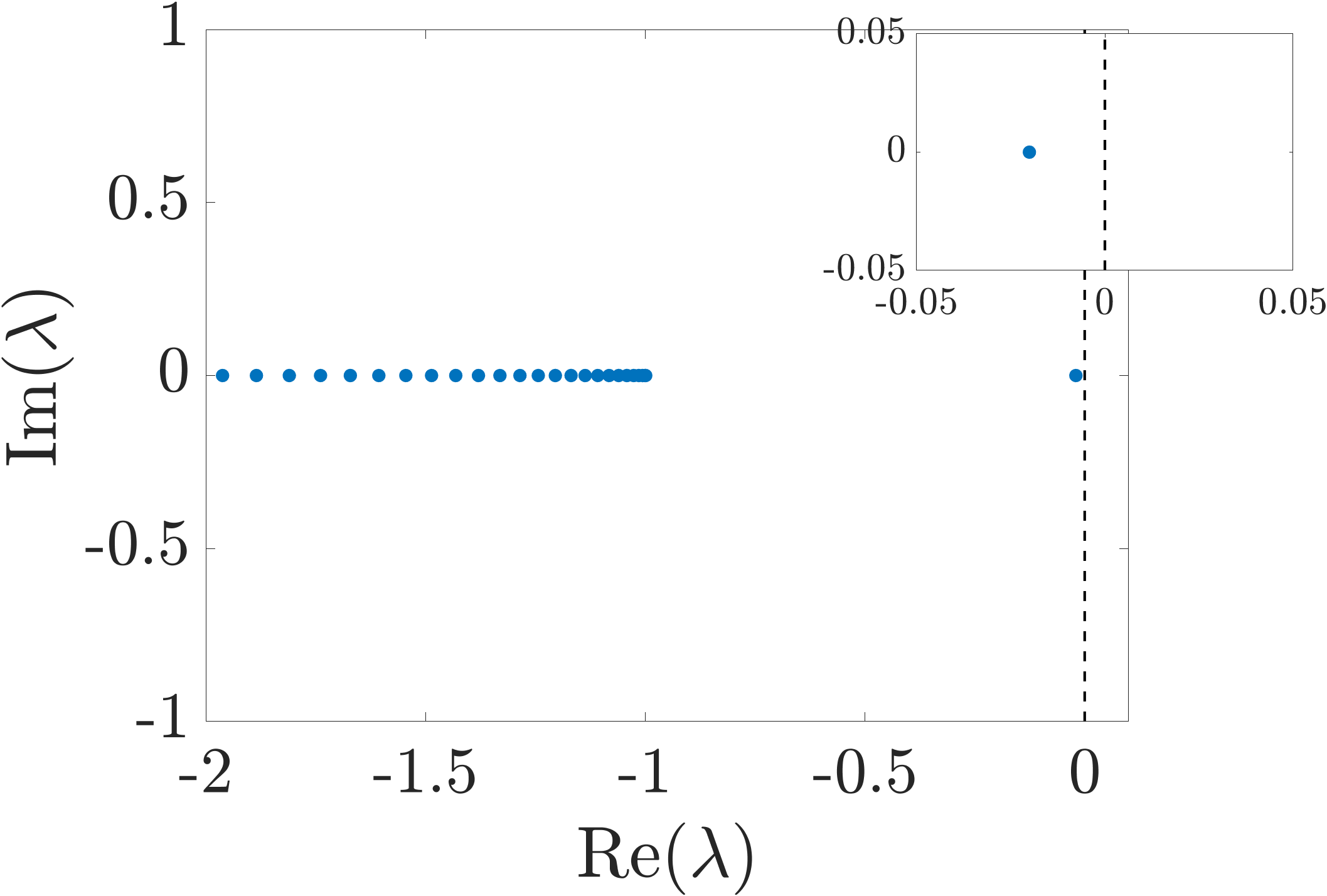} \hspace{1em}
    \includegraphics[width=0.3\textwidth]{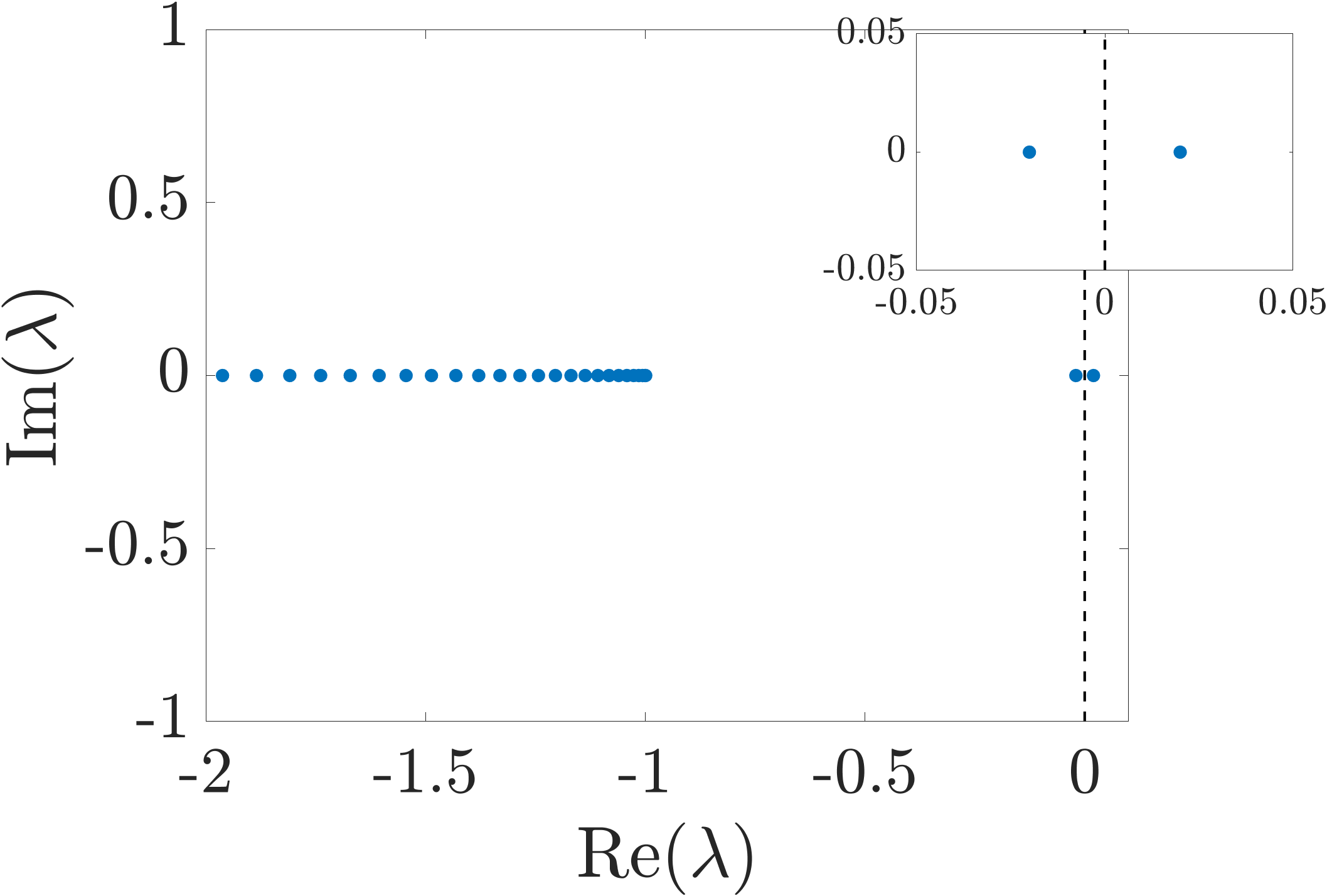}\hspace{1em}
    \caption{Approximations of stationary 2-front solutions to~\eqref{eq:RDE_toy}, along with their spectra, for system coefficients $\eps = 0.1$ and $V(x) = \cos(\pi x)$. The insets provide a closer view of the small eigenvalues near zero. Left: a strongly spectrally stable 2-front solution obtained through numerical continuation by starting from the formal concatenation of the strongly spectrally stable front solutions depicted in the left and middle panels of Figure~\ref{fig:RDE_toy_1fronts}. Right: a spectrally unstable 2-front obtained through numerical continuation by starting from the formal concatenation of a strongly spectrally stable and a spectrally unstable 1-front solution (left and right panels of Figure~\ref{fig:RDE_toy_1fronts}).    
    }
    \label{fig:RDE_toy_2fronts}
\end{figure}

\subsection{A Klausmeier reaction-diffusion-advection system}
We consider a Klausmeier-type model with spatially periodic coefficients as an example for a 2-component reaction-diffusion-advection system to which our theory applies. Using our methods, we rigorously establish the existence of strongly spectrally stable stationary multipulse solutions and corresponding periodic pulse solutions. These results extend the recent findings in~\cite{Bastiaansen2020}, where stationary $1$-pulse solutions to this Klausmeier model were constructed using singular perturbation theory and their stability was analyzed. The system of equations reads
\begin{align*}
    \begin{split}
    \partial_t w & = \partial_x^2w +\eps \big(f(x) \partial_x w + g(x) w \big) - w - w p^2 + a, \\
    \partial_t p &= d^2 \partial_x^2 p - mp + wp^2,
    \end{split} \qquad \begin{pmatrix} w(x,t)\\ p(x,t)\end{pmatrix} \in \R^2, \, x \in \R, \, t \geq 0
\end{align*}
with parameters $d,a,m,\eps>0$ and real-valued functions $f,g \in C^1(\R)$ with period $T>0$. This model is employed in ecology to describe the dynamics of vegetation patterns resulting from the interaction between water $w$ and plants $p$ across a spatially heterogeneous terrain with periodic topography modeled by the functions $f$ and $g$. Here, $d>0$ is a diffusion coefficient, $a$ models the amount of rain fall, $1/m$ is a quantity corresponding to the life time of plants, and $\eps>0$ measures the influence of the terrain on the vegetation dynamics. For more background on the model, including the role of the individual parameters and the functions $f$ and $g$, we refer to~\cite{Bastiaansen2020} and references therein. 

To fit the system in our framework, we set $\ub = (w,p)^\top$ and write the equations as
\begin{align}\label{klausmeier_system}
    \partial_t \ub = A_\eps \ub + \mathcal{N}_\eps(\ub,x)
\end{align}
with
\begin{align*}
    A_\eps =
    \begin{pmatrix}
        \partial_x^2 +\eps f \partial_x & 0 \\
        0 & d^2\partial_x^2 
    \end{pmatrix}, \qquad
    \mathcal{N}_\eps\left(\begin{pmatrix}
        w\\p
    \end{pmatrix},x\right) =
    \begin{pmatrix}
        \eps g(x)w - w - wp^2 + a\\-mp + wp^2 
    \end{pmatrix}.
\end{align*}
The existence problem for stationary solutions is then given by
\begin{align}\label{eq:stationary_Klausmeier}
    A_\eps \ub + \mathcal{N}_\eps(\ub,\cdot) = 0,
\end{align}
which is of the form~\eqref{existence_problem}. The associated linearization operator $L_\eps(\unub) \colon D(L_\eps(\unub)) \subset L^2(\R) \to L^2(\R)$ with dense domain $D(L_\eps(\unub)) = H^2(\R)$ is given by
\begin{align*}
    L_\eps(\unub)= A_\eps + \partial_u \mathcal{N}_\eps(\unub,\cdot)
\end{align*}
for $\unub \in L^\infty(\R)$. Given that the $\eps$-independent principal part $A_0$ of the operator $L_\eps(\unub)$ is sectorial, and the remainder $L_\eps(\unub) - A_0$ is relatively $A_0$-bounded, cf.~\cite[Definition~III.2.1]{EngelNagel2000}, we obtain the following spectral a-priori bound. 

\begin{Lemma} \label{lem:apriori_Klausmeier}
Let $f,g \in L^\infty(\R)$ and $C,d,a,m,\eps_0 > 0$. 
Then, there exists a constant $\varrho > 0$, depending only on $\|f\|_{L^\infty},\|g\|_{L^\infty},C,d,a,m$ and $\eps_0$, such that we have
\begin{align*}
\sigma(L_\eps(\unub)) \cap \{\lambda \in \C : \Re(\lambda) \geq -1\} \subset \overline{B}_0(\varrho)
\end{align*}
for all $\eps \in (-\eps_0,\eps_0)$ and each $\unub \in L^\infty(\R)$ with $\|\unub\|_{L^\infty} \leq C$. Here, $\overline{B}_0(\varrho)$ is the closed ball of radius $\varrho$ centered at the origin. 
\end{Lemma}
\begin{proof}
On the one hand, it is well-known that the diagonal diffusion operator $A_0$ generates a bounded analytic semigroup, cf.~\cite[Example~II.4.10]{EngelNagel2000}. On the other hand, the exposition in~\cite[Example~III.2.2]{EngelNagel2000} demonstrates that the residual $L_\eps(\unub)-A_0$ is relatively $A_0$-bounded with $A_0$-bound $0$. The result then follows directly from~\cite[Lemma~III.2.6]{EngelNagel2000} and its proof.   
\end{proof}

The main goal of this section is to show that~\eqref{klausmeier_system} admits stationary (periodic) multipulse solutions corresponding to (periodic sequences of) localized vegetation patches. The fundamental building blocks of these multiple pulse solutions are nondegenerate $1$-pulse solutions. The existence and spectral stability of stationary $1$-pulse solutions to~\eqref{klausmeier_system} for small $\eps>0$ were established in~\cite{Bastiaansen2020} for a broad class of heterogeneities $f$ and $g$. However, the spectral analysis in~\cite{Bastiaansen2020} assumes that $f$ and $g$ are localized, leaving the spectral stability and nondegeneracy of the $1$-pulse solutions unaddressed for periodic $f$ and $g$.

To bridge this gap, we begin by establishing the existence and spectral stability of nondegenerate stationary $1$-pulse solutions to~\eqref{klausmeier_system} for small $\eps > 0$ and periodic $f$ and $g$. This is achieved by bifurcating from even $1$-pulse solutions to the unperturbed problem
\begin{align}\label{eq:stationary_Klausmeier_eps0}
    A_0 \ub + \mathcal{N}_0(\ub,\cdot) = 0,
\end{align} 
which were constructed in~\cite[Theorem~2.20]{Bastiaansen2020} using geometric singular perturbation theory in the regime $\nu := a/m \ll 1$, $d\nu^2 m^{1/2}, d /\nu^{2} \leq C$ for some $\nu$-independent constant $C>0$. Since these solutions correspond to homoclinics to a hyperbolic equilibrium in~\eqref{eq:stationary_Klausmeier_eps0}, they are exponentially localized, see Figure~\ref{fig:pulses_Klausmeier}. Moreover, it follows from~\cite[Theorem~3.2]{Bastiaansen2020} that they are spectrally stable with simple eigenvalue $\lambda = 0$ as in  Definition~\ref{def:spectral_stability}. With the aid of the implicit function theorem, we prove that these $1$-pulse solutions persist for $\eps > 0$ in case $f$ is odd and $g$ is even. Moreover, we show that their spectral stability is determined by the sign of a Melnikov-type integral.

\begin{Theorem}\label{thm:existence_klausmeier}
Let $T, d,a,m>0$. Let $f \in C^1(\R)$ be $T$-periodic and odd. Let $g \in C^1(\R)$ be $T$-periodic and even. Let $\ub_0 = \mathbf{z}_0+\mathbf{v}_0 \in H^2(\R) \oplus \R^2$ be an even solution to~\eqref{klausmeier_system} for $\eps=0$, which is spectrally stable with simple eigenvalue $\lambda=0$. Assume that the Melnikov integral
\begin{align*}
\mathcal{M} = \int_\R \big( f'(x) \partial_x\ub_{01}(x) + g'(x) \ub_{01}(x) \big) \Psi_{\mathrm{ad}1}(x) \de x
\end{align*}
does not vanish, where $\Psi_\textup{ad} \in H^2(\R)$ spans the kernel of the adjoint operator $L_0(\ub_0)^*$ and satisfies $\langle \partial_x\ub_0,\Psi_\textup{ad} \rangle_{L^2} = 1$.

Then, there exist constants $C,\eps_0,\eta>0$ such that for all $\eps \in(-\eps_0,\eps_0) \setminus \{0\}$ there exists a nondegenerate even solution $\ub(\eps) = \mathbf{z}(\eps) + \mathbf{v}(\eps) \in H^2(\R) \oplus H^2_\per(0,T)$ to~\eqref{klausmeier_system} satisfying
\begin{align} \label{bounds_Klausmeier}
    \|\ub(\eps)-\ub_0\|_{L^\infty} \leq C\eps, \qquad \|\mathbf{v}(\eps)-\mathbf{v}_0\|_{H_\per^2(0,T)} \leq C \eps.
\end{align}
Furthermore, we have
\begin{align*}
    \sigma(L_\eps(\ub(\eps)) \subset \{\lambda \in \C : \Re(\lambda) \leq- \eta\} \cup \{\lambda_0(\eps)\}
\end{align*}
for $\eps \in (-\eps_0,\eps_0)$, where $\lambda_0(\eps) \in \R$ is a real simple eigenvalue of $L_\eps(\ub(\eps))$ obeying the expansion 
$$\lambda_0(\eps) = -\eps \mathcal{M} + \mathcal{O}(\eps^2).$$
\end{Theorem}

\begin{proof}
We start with the construction of the periodic background wave $\mathbf{v}(\eps)$ by perturbing from the rest state $\mathbf{v}_0$ in~\eqref{eq:stationary_Klausmeier_eps0}. The smooth function $\mathcal{F}\colon H_\mathrm{even,per}^2(0,T) \times \R \to L_\mathrm{even,per}^2(0,T)$ given by
\begin{align*}
\mathcal{F}(\mathbf{v},\eps)= A_\eps\mathbf{v} + \mathcal{N}_\eps(\mathbf{v},\cdot)
\end{align*}
obeys $\mathcal{F}(\mathbf{v}_0,0) = 0$ and $\partial_{\mathbf{v}}\mathcal{F}(\mathbf{v}_0,0) = L_0(\mathbf{v}_0)$. Since $\ub_0$ is spectrally stable with simple eigenvalue $\lambda =0$, the essential spectrum of $L_0(\ub_0)$ is confined to the open left-half plane, which, by Proposition~\ref{prop:essential_spec1front}, is given by $\sigma_{\mathrm{ess}}(L_0(\ub_0)) = \sigma(L_0(\mathbf{v}_0))$. Therefore, $\partial_\mathbf{v}\mathcal{F}(\mathbf{v}_0,0)$ is invertible as an operator from $H_{\mathrm{even},\per}^2(0,T)$ into $L_{\mathrm{even},\per}^2(0,T)$. An application of the implicit function theorem yields $\eps_1>0$ and a smooth map $\mathbf{v} \colon (-\eps_1,\eps_1) \to H_\mathrm{even,per}^2(0,T)$ with $\mathbf{v}(0) = \mathbf{v}_0$ such that 
\begin{align*}
    \mathcal{F}(\mathbf{v}(\eps),\eps) = 0 
\end{align*}
for all $\eps \in (-\eps_1,\eps_1)$.

Substituting the ansatz $\ub = \mathbf{z} + \mathbf{v}(\eps)$ with $\mathbf{z} \in H_\mathrm{even}^2(\R)$ into~\eqref{eq:stationary_Klausmeier}, we arrive at
\begin{align*}
0 = A_\eps (\mathbf{z}+\mathbf{v}(\eps)) + \mathcal{N}_\eps(\mathbf{z}+\mathbf{v}(\eps),\cdot) &= A_\eps\mathbf{z} + \mathcal{N}_\eps(\mathbf{z}+\mathbf{v}(\eps),\cdot) - \mathcal{N}_\eps(\mathbf{v}(\eps),\cdot). 
\end{align*}
Clearly, the smooth nonlinear operator $\mathcal{G}\colon H_\mathrm{even}^2(\R) \times (-\eps_1,\eps_1) \to L_\mathrm{even}^2(\R)$ given by
\begin{align*}
    \mathcal{G}(\mathbf{z},\eps) =A_\eps\mathbf{z} + \mathcal{N}_\eps(\mathbf{z}+\mathbf{v}(\eps),\cdot) - \mathcal{N}_\eps(\mathbf{v}(\eps),\cdot)
\end{align*}
is well-defined and satisfies $\mathcal{G}(\mathbf{z}_0,0) = 0$ and $\partial_\mathbf{z} \mathcal{G}(\mathbf{z}_0,0) = L_0(\ub_0)|_{H_\mathrm{even}^2}$, where $L_0(\ub_0)|_{H_\mathrm{even}^2}$ is the restriction of the linear operator $L_0(\ub_0) \colon H^2(\R) \to L^2(\R)$ to the subspace $H_{\mathrm{even}}^2(\R) \subset H^2(\R)$ of even functions. Since $\ub_0$ is spectrally stable with simple eigenvalue $\lambda=0$, the kernel of $L_0(\ub_0)$ is spanned by the odd function $\partial_x \ub_0$. Therefore, $\partial_\mathbf{z} \mathcal{G}(\mathbf{z}_0,0) = L_0(\ub_0)|_{H_\mathrm{even}^2}$ is invertible. Hence, the implicit function theorem affords $\eps_2 \in (0,\eps_1)$ and a smooth map $\mathbf{z} \colon (-\eps_2,\eps_2) \to H_\mathrm{even}^2(\R)$ with $\mathbf{z}(0) = \mathbf{z}_0$ such that
\begin{align*}
    \mathcal{G}(\mathbf{z}(\eps),\eps) = 0
\end{align*}
for all $\eps \in (-\eps_2,\eps_2)$. In particular, we obtain a smooth map $\mathbf{u} \colon (-\eps_2,\eps_2) \to H^2(\R) \oplus H^2_\per(0,T)$ such that $\mathbf{u}(\eps) = \mathbf{z}(\eps) + \mathbf{v}(\eps)$ is an even solution to~\eqref{eq:stationary_Klausmeier}. The bounds~\eqref{bounds_Klausmeier} follows by the smoothness of $\mathbf{u}$ and $\mathbf{v}$ and the continuous embeddings $H^1(\R) \hookrightarrow L^\infty(\R)$ and $H^1_\per(0,T) \hookrightarrow L^\infty(\R)$. 

Next, we establish the nondegeneracy of $\mathbf{u}(\eps)$ as well as its spectral (in)stability. To this end, we begin by tracing the simple eigenvalue of $L_\eps(\ub(\eps))$ converging to $0$ as $\eps \to 0$. Since $L_0(\ub_0)$ is spectrally stable with simple eigenvalue $\lambda = 0$, it follows from~\cite[Proposition I.7.2]{Kielhoefer2012} that there exist $\eps_3 \in (0,\eps_2)$ and $C^1$-curves $\lambda_0 \colon (-\eps_3,\eps_3) \to \R$ and $\mathbf{w} \colon (-\eps_3,\eps_3) \to H^2(\R)$ with $\mathbf{w}(0)=0$ and $\lambda_0(0)=0$ solving the eigenvalue problem
\begin{align*}
    L_\eps(\ub(\eps))(\partial_x\ub_0 + \mathbf{w}(\eps)) 
    =\lambda_0(\eps)(\partial_x\ub_0 + \mathbf{w}(\eps)).
\end{align*}
Taking the derivative on both sides with respect to $\eps$ and evaluating at $\eps=0$ yields
\begin{align}\label{proof_Klausmeier1}
    L_0(\ub_0) \partial_\eps \mathbf{w}(0)
    +(f\partial_x + g) \begin{pmatrix} \partial_x \ub_{01}\\0\end{pmatrix} 
    + \partial_{\ub\ub} \mathcal{N}_0(\ub_0,\cdot)[\partial_\eps \ub(0), \partial_x \ub_0]
    = \lambda_0'(0) \partial_x \ub_0,
\end{align}
where we denote $\ub_0 = (\ub_{01},\ub_{02})^\top$.
On the other hand, differentiating the equation
\begin{align*}
    A_\eps \ub(\eps) + \mathcal{N}_\eps(\ub(\eps),\cdot) = 0
\end{align*}
with respect to $x$ and $\eps$ and subsequently setting $\eps = 0$, we obtain
\begin{align}\label{proof_Klausmeier2}
 L_0(\ub_0) \partial_\eps\partial_x \ub_0  + (f\partial_x + g) \begin{pmatrix} \partial_x \ub_{01}\\0\end{pmatrix} 
    +\partial_{\ub\ub} \mathcal{N}_0(\ub_0,\cdot)[\partial_\eps\ub(0),\partial_x \ub_0] 
    + (f'\partial_x + g') \begin{pmatrix} \ub_{01}\\0\end{pmatrix} = 0.
\end{align}
Subtracting~\eqref{proof_Klausmeier2} from~\eqref{proof_Klausmeier1}, we find
\begin{align*}
    L_0(\ub_0) \big(\partial_\eps \mathbf{w}(0)-  \partial_\eps\partial_x \ub(0)\big) - (f'\partial_x + g') \begin{pmatrix} \ub_{01}\\0\end{pmatrix}   
    = \lambda_0'(0) \partial_x \ub_0.
\end{align*}
Taking the $L^2$-scalar product of the last equation with $\Psi_\mathrm{ad} \in H^2(\R)$, we establish
\begin{align*}
    \lambda_0'(0) = -\langle(f'\partial_x + g') \ub_{01}, \Psi_{\mathrm{ad}1} \rangle_{L^2}
    = -\int_\R \big( f'(x) \partial_x\ub_{01}(x) + g'(x) \ub_{01}(x) \big) \Psi_{\mathrm{ad}1}(x) \de x =-\mathcal{M} \neq 0.
\end{align*}

Finally, we prove that the remaining part of the spectrum of $L_\eps(\ub(\eps))$ lies in the open left-half plane. Since $\ub_0$ is spectrally stable with simple eigenvalue $\lambda = 0$, there exists a constant $\varrho > 0$ such that
\begin{align} \label{spectral_bound_Klausmeier}
\sigma(\El_0(\ub_0)) \cap \{\lambda \in \C : \Re(\lambda) \geq -\varrho\} = \{0\}.
\end{align}
On the other hand, since $\|\ub(\eps)\|_{L^\infty}$ is bounded by an $\eps$-independent constant by estimate~\eqref{bounds_Klausmeier}, there exists by Lemma~\ref{lem:apriori_Klausmeier} an $\eps$-independent constant $\varrho_1 > 0$ such that
\begin{align*}
\sigma(\El_\eps(\ub(\eps))) \cap \{\lambda \in \C : \Re(\lambda) \geq -1\} \subset \overline{B}_0(\varrho_1)
\end{align*}
for $\eps \in (-\eps_3,\eps_3)$. Combining the latter with the fact that $L_\eps(\ub(\eps)) - L_0(\ub_0)$ is a bounded operator on $L^2(\R)$,~\cite[Theorem~IV.3.18]{Kato1995} and~\eqref{spectral_bound_Klausmeier}
yield $\eps_4 \in (0,\eps_3)$ such that
$$\sigma(L_\eps(\ub(\eps)) \cap \left\{\lambda \in \C : \Re(\lambda) \geq -\tfrac{1}{2} \varrho\right\}$$
contains exactly one algebraically simple eigenvalue of $L_\eps(\ub(\eps))$ for $\eps \in (-\eps_4,\eps_4)$, which is then necessarily given by $\lambda(\eps) \in \R$. This completes the proof.
\end{proof}

\begin{Remark}
We emphasize that the Melnikov integral $\mathcal{M}$ in Theorem~\ref{thm:existence_klausmeier} is generically nonvanishing, since the integrand is an even function of $x$. 
\end{Remark}

\begin{figure}[t]
    \centering
    \includegraphics[width=0.3\textwidth]{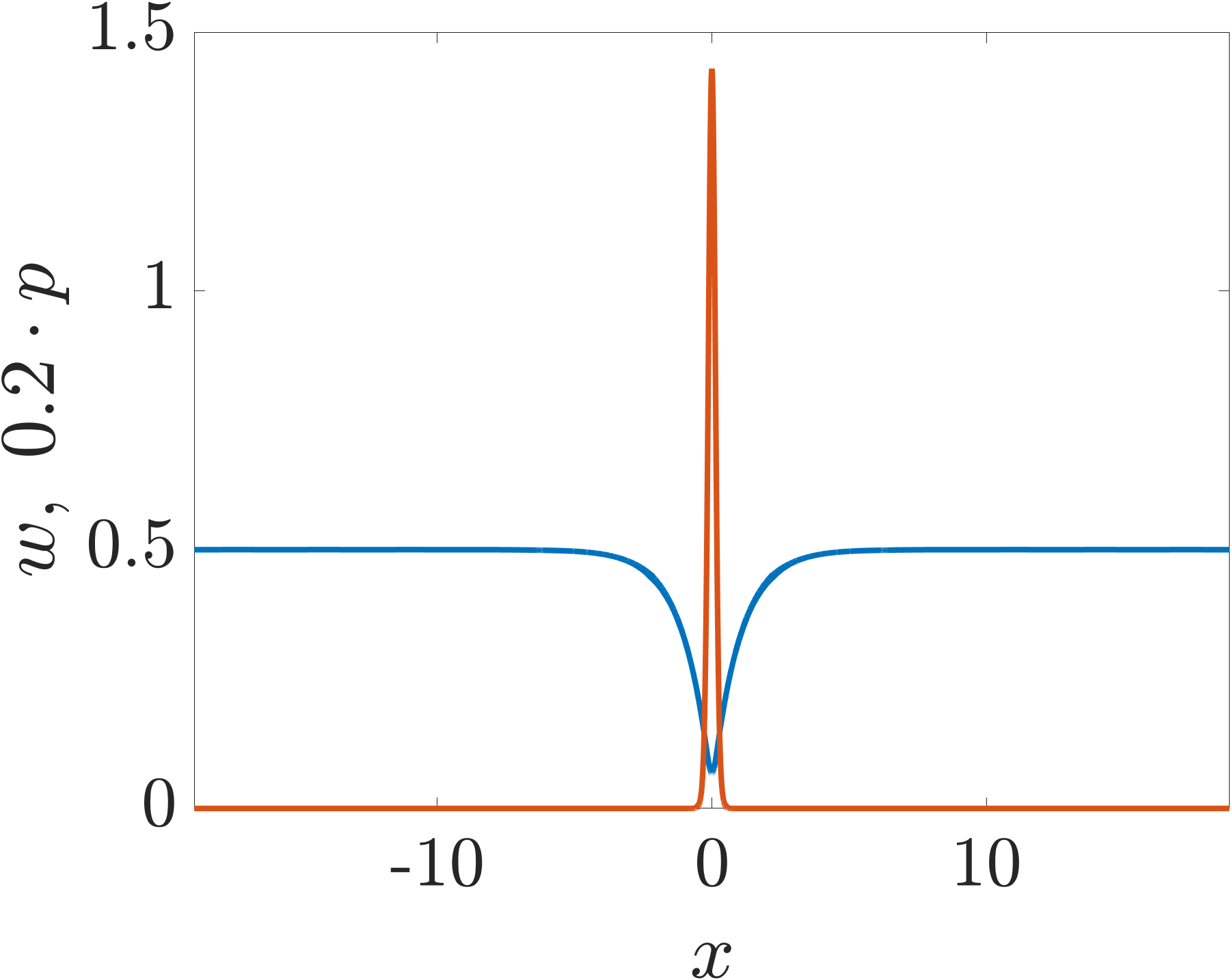} \hspace{1em}
    \includegraphics[width=0.31\textwidth]{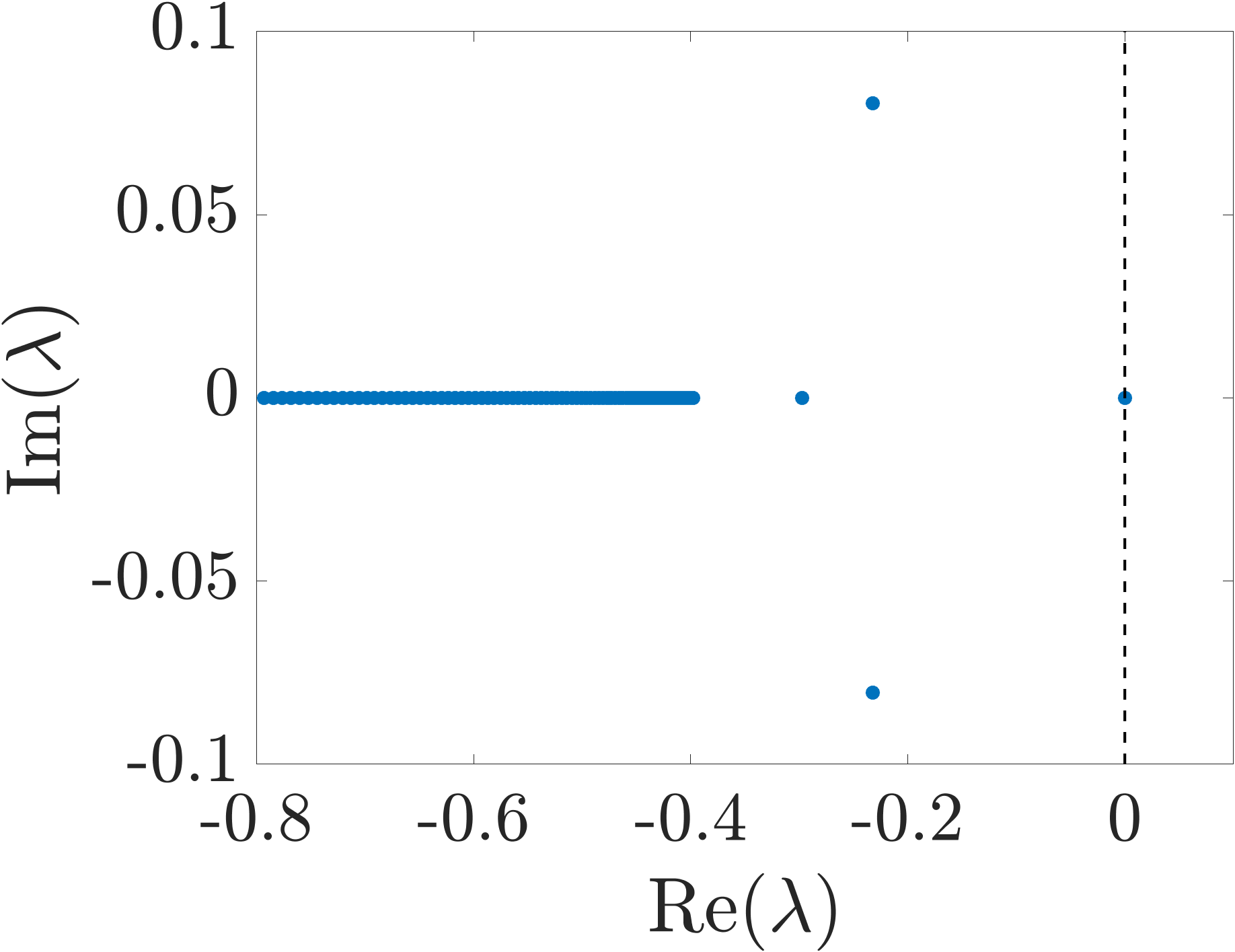}
    \hspace{1em}
    \includegraphics[width=0.3\textwidth]{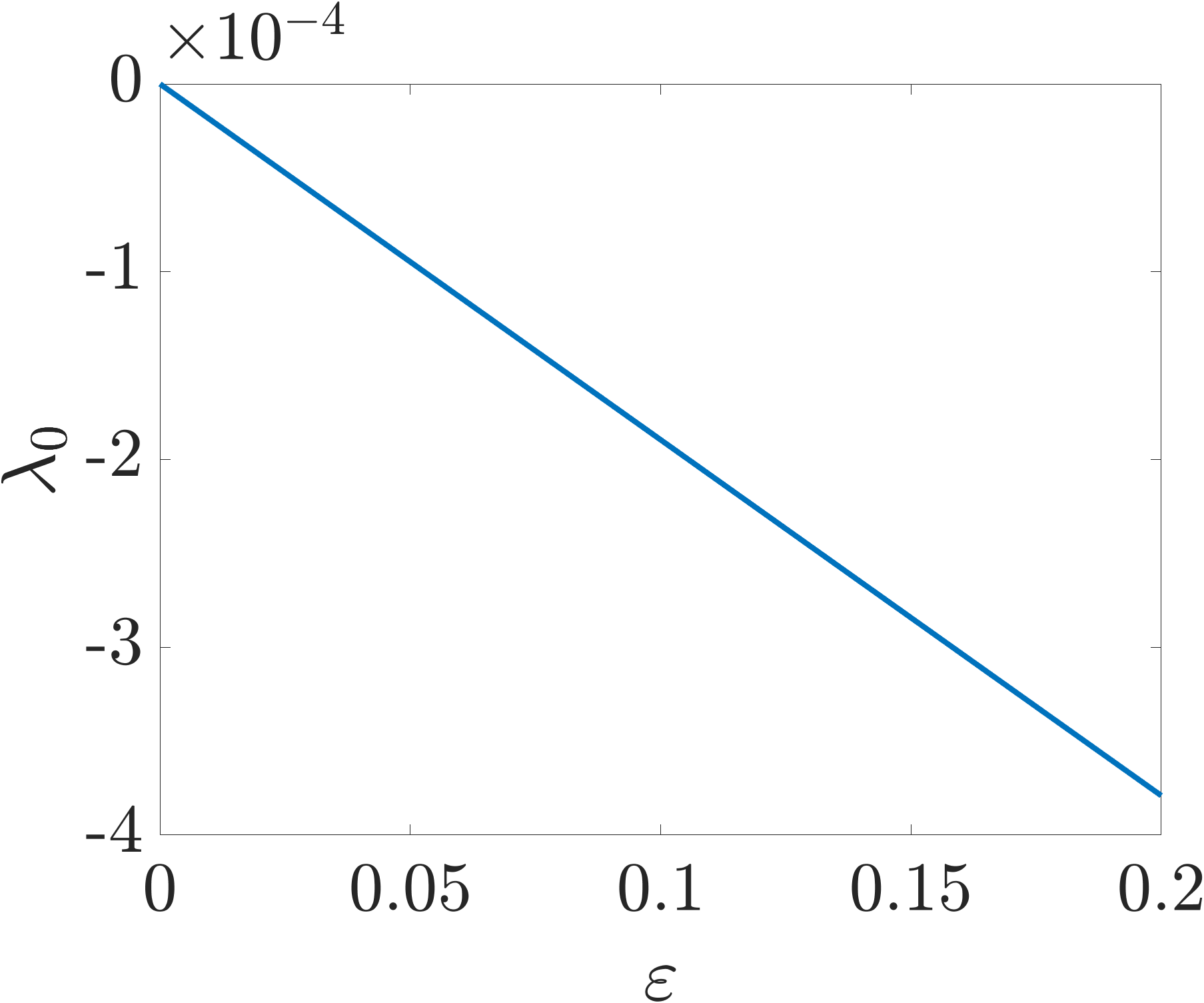} \\\vspace{1em}
    \includegraphics[width=0.3\textwidth]{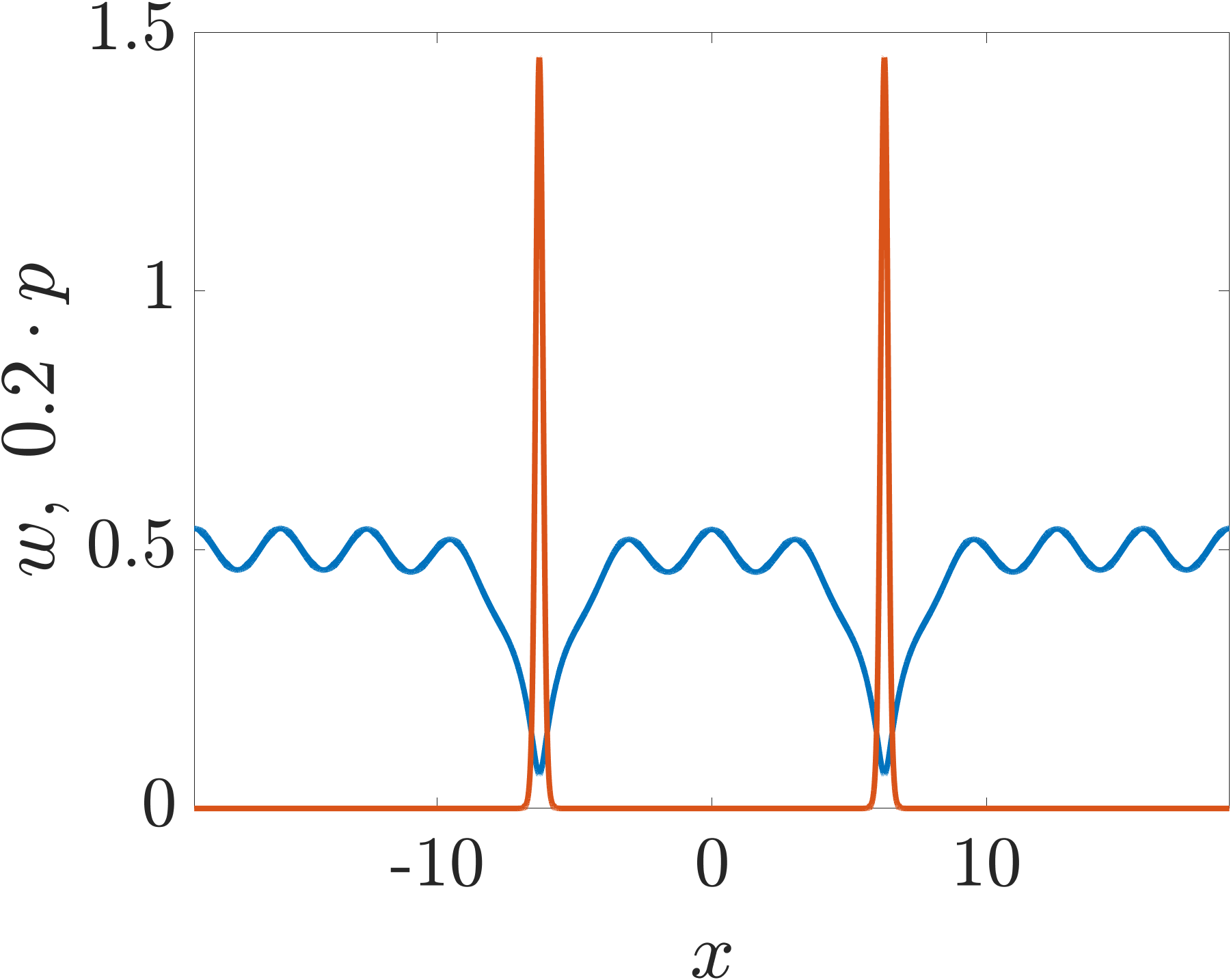} \hspace{1em}
    \includegraphics[width=0.35\textwidth]{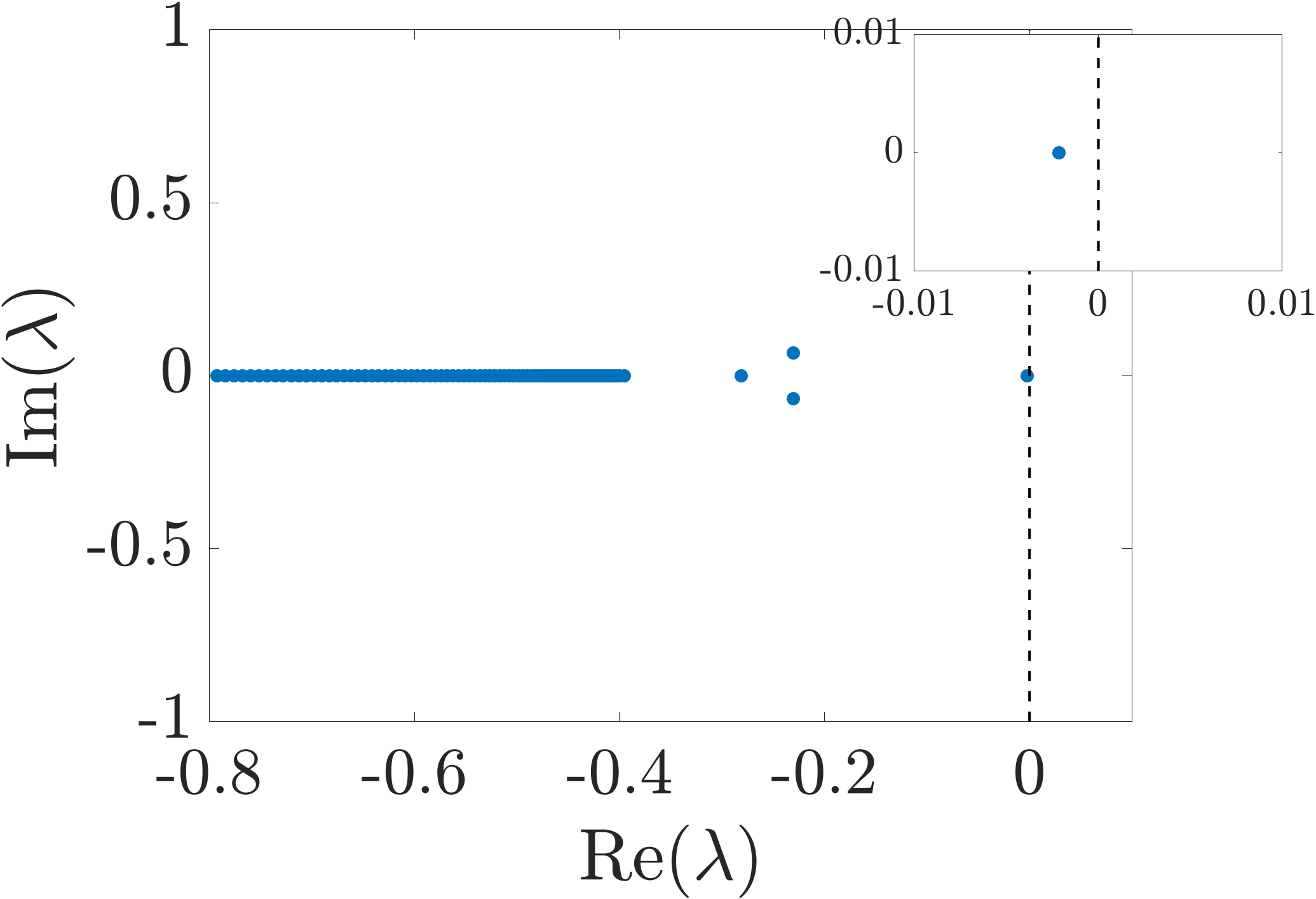}
    \caption{A spectrally stable stationary $1$-pulse solution to~\eqref{klausmeier_system} for $\eps = 0$ (top left), along with its spectrum (top middle) with simple eigenvalue $\lambda=0$. The $p$-component in the left panel is scaled by a factor $0.2$ to improve visibility. We continued this $1$-pulse solution in $\eps$ using the MATLAB package \texttt{pde2path}~\cite{pde2path} and plotted its critical eigenvalue $\lambda_0(\eps)$ as a function of $\eps$ (top right). One observes that the curve $\lambda_0(\eps)$ is to leading order linear, which is in agreement with the expansion of $\lambda_0(\eps)$ provided in Theorem~\ref{thm:existence_klausmeier}. The bottom row depicts a strongly spectrally stable stationary $2$-pulse solution to~\eqref{klausmeier_system} for $\eps = 1$ (bottom left), along with its spectrum (bottom right). The inset provides a closer view of the small eigenvalues near zero. The $2$-pulse is obtained through numerical continuation starting from the superposition of two spectrally stable $1$-pulse solutions to~\eqref{klausmeier_system}. The system coefficients are $d= 0.04, a = 0.5, m = 0.4, f(x) = 0.2 \sin(2x)$ and $g(x) = 0.4 \cos(2x)$.}
    \label{fig:pulses_Klausmeier}
\end{figure}

Since the $1$-pulse solutions, established in Theorem~\ref{thm:existence_klausmeier}, satisfy the assumptions~\ref{assH1}-\ref{assH3}, Theorems~\ref{t:existence_multifront} and~\ref{thm:instability_multifront}, Corollary~\ref{cor:stability_multifront}, and Lemma~\ref{lem:apriori_Klausmeier} yield the existence and spectral stability of bifurcating multipulse solutions.

\begin{Corollary}\label{cor:Klausmeier_existence_multipulse}
Let $M \in \N$. Let $\ub \in H^{2}(\R) \oplus H_\per^2(0,T)$ be a pulse solution to~\eqref{eq:stationary_Klausmeier} as established in Theorem~\ref{thm:existence_klausmeier}. Then, there exists $N \in \N$ such that for any $n \in \N$ with $n \geq N$ there exists a nondegenerate stationary multipulse solution $\tilde\ub_n$ to~\eqref{klausmeier_system} of the form
\begin{align*}
\tilde\ub_n = \mathbf{a}_n + \sum_{j = 1}^M \chi_{j,n} \ub(\cdot - jnT), 
\end{align*}
where $\chi_{j,n}, j = 1,\ldots,M$ is the smooth partition of unity defined in~\S\ref{sec:m-front}, and $\{\mathbf{a}_n\}_n$ is a sequence in $H^2(\R)$ converging to $0$ as $n \to \infty$. If the pulse $\ub$ is strongly spectrally stable, then so is the multipulse $\tilde\ub_n$. Moreover, if $\ub$ is spectrally unstable, then the same holds for $\tilde\ub_n$.
\end{Corollary}
\begin{proof}
The existence of the multipulse $\tilde\ub_n$ is a direct consequence of Theorems~\ref{t:existence_multifront} and~\ref{thm:existence_klausmeier}. Since the primary pulse $\ub$ is nondegenerate, Theorem~\ref{thm:instability_multifront} implies that the multipulse $\tilde\ub_n$ is also nondegenerate. Thanks to the continuous embedding $H^1(\R) \hookrightarrow L^\infty(\R)$, $\|\tilde\ub_n\|_{L^\infty}$ is bounded by an $n$-independent constant. So, the assertions about the spectral (in)stability of $\tilde \ub_n$ follow from Corollary~\ref{cor:stability_multifront}, Theorem~\ref{thm:instability_multifront}, and Lemma~\ref{lem:apriori_Klausmeier}. 
\end{proof}

The nondegenerate multipulse solutions, established in Corollary~\ref{cor:Klausmeier_existence_multipulse}, are accompanied by large wavelength periodic multipulse solutions. Their existence and spectral (in)stability are derived from Theorems~\ref{t:existence_periodic} and~\ref{thm:instability_periodic}, Corollary~\ref{cor:stability_periodic}, and Lemma~\ref{lem:apriori_Klausmeier}.

\begin{Corollary}
Let $\ub = \mathbf{z} + \mathbf{v} \in H^2(\R) \oplus H^2_\per(0,T)$ be a multipulse solution to~\eqref{eq:stationary_Klausmeier}, as established in Corollary~\ref{cor:Klausmeier_existence_multipulse}. Then, there exists $N \in \N$ such that for all $n \in \N$ with $n \geq N$ there exists a stationary $nT$-periodic solution $\ub_n$ to~\eqref{klausmeier_system} given by
\begin{align*}
\ub_n(x) = \chi_n(x) \ub(x) + (1-\chi_n(x)) \mathbf{v}(x) + \mathbf{a}_n(x), \qquad x \in \left[-\tfrac{n}{2}T,\tfrac{n}{2} T\right),
\end{align*}
where $\chi_n$ is the cut-off function from Theorem~\ref{t:existence_periodic}, and $\{\mathbf{a}_n\}_n$ is a sequence with $\mathbf{a}_n \in H_\per^2(0,nT)$ satisfying $\|\mathbf{a}_n\|_{H_\per^2(0,nT)} \to 0$ as $n \to \infty$. If $\ub$ is strongly spectrally stable, then so is $\ub_n$. Moreover, if $\ub$ is spectrally unstable, then the same holds for $\ub_n$.
\end{Corollary}

\subsection{The Gross-Pitaevskii equation}

Let $T > 0$. We consider a nonlinear Schr\"odinger (NLS) equation with an external potential, which is commonly referred to as Gross-Pitaevskii (GP) equation and arises, for instance, as a mean-field approximation in the study of Bose-Einstein condensates in optical lattices, see~\cite{Branzhnyi2004,Kevrekidis2008Basic,Pitaevskii2003} and references therein. The Gross-Pitaevskii equation is given by
\begin{align}\label{GP_periodic_potential}
    \iu \partial_t u = - \partial_x^2 u + \mu V(x) u + \kappa |u|^2 u, \qquad u(x,t) \in \C, \, x \in \R, \, t \geq 0
\end{align}
with parameters $\kappa \in \{\pm1\}$ and $\mu \in \R$, and real-valued potential $V \in C^1(\R)$. In the context of Bose-Einstein condensation, $u(x,t)$ represents a macroscopic wave function, $|u(x,t)|^2$ is its atomic density, $V(x)$ is the external potential created by the optical lattice, $\mu$ measures the strength of the potential, and the sign of $\kappa$ determines whether the nonlinear interaction of the Bose-Einstein condensates is \emph{attractive} ($\kappa = -1$) or \emph{repulse} ($\kappa = 1$). We say that the nonlinearity in~\eqref{GP_periodic_potential} is \emph{defocusing} if $\kappa = 1$ and \emph{focusing} if $\kappa=-1$. Here, we are interested in \emph{periodic} optical trapping lattices formed by the interference of laser beams. Thus, we consider $T$-periodic potentials $V$.

We search for time-harmonic solutions to~\eqref{GP_periodic_potential}. Thus, we insert $u(x,t) = \eu^{\iu \omega t}\psi(x,t)$ with $\omega \in \R$ into~\eqref{GP_periodic_potential} to obtain
\begin{align}\label{GP_time_harmonic}
    \iu \partial_t \psi = - \partial_x^2 \psi + \mu V(x) \psi + \omega \psi + \kappa |\psi|^2 \psi.
\end{align}
Real-valued stationary solutions of~\eqref{GP_time_harmonic} then satisfy the ordinary differential equation
\begin{align*}
     - \partial_x^2 \psi + \mu V(x) \psi + \omega \psi + \kappa \psi^3 = 0,
\end{align*}
which is of the from~\eqref{existence_problem}.

In order to study the stability of stationary solutions to~\eqref{GP_time_harmonic}, we write the equation as a system
\begin{align}\label{GP_system}
    \partial_t\boldsymbol{\psi} = J \left(-\partial_x^2 \boldsymbol{\psi} +\mu V(x) \boldsymbol{\psi} + \omega \boldsymbol{\psi} + \kappa |\boldsymbol{\psi}|^2 \boldsymbol{\psi} \right)
\end{align}
in $\boldsymbol{\psi} = (\Re(\psi), \Im(\psi))^\top$, where
\begin{align*}
    J= 
    \begin{pmatrix}
		0 & 1 \\ -1  & 0
	\end{pmatrix}
\end{align*}
is skew-symmetric. System~\eqref{GP_system} exhibits a rotational invariance, i.e., the map $\boldsymbol{\psi} \mapsto R(\gamma) \boldsymbol{\psi}$ with
\begin{align*}
    R(\gamma) =
    \begin{pmatrix}
        \cos(\gamma) & \sin(\gamma) \\
        -\sin(\gamma) & \cos(\gamma)
    \end{pmatrix}, \qquad \gamma \in \R
\end{align*}
maps solutions of~\eqref{GP_system} to solutions. The advantage of the formulation~\eqref{GP_system} over~\eqref{GP_time_harmonic} is that the nonlinearity is differentiable, which allows for linearization about a real-valued stationary solution $\boldsymbol{\psi} = (\psi,0)^\top$. The associated linearization operator $L_\mu(\underline{\psi}) \colon D(L_\mu(\underline{\psi})) \subset L^2(\R) \to L^2(\R)$ with dense domain $D(L_\mu(\underline{\psi})) = H^2(\R)$ is given by
\begin{align*}
	L_\mu(\underline{\psi}) =J 
    \begin{pmatrix}
        L_{+,\mu}(\underline{\psi}) & 0 \\ 0 & L_{-,\mu}(\underline{\psi})
    \end{pmatrix} 
\end{align*}
for $\underline{\psi} \in L^\infty(\R)$, where $L_{\pm,\mu}(\underline{\psi}) \colon D(L_{\pm,\mu}(\underline{\psi})) \subset L^2(\R) \to L^2(\R)$ with $D(L_{\pm,\mu}(\underline{\psi})) = H^2(\R)$ are defined by
\begin{align*}
     L_{-,\mu}(\underline{\psi}) \phi = - \phi'' + \mu V \phi + \omega\phi + \kappa\underline{\psi}^2 \phi, \qquad L_{+,\mu}(\underline{\psi})\phi = - \phi'' + \mu V \phi + \omega \phi + 3\kappa \underline{\psi}^2 \phi.
\end{align*}
One observes that the second-order operators $L_{\pm,\mu}(\underline{\psi})$ are self-adjoint and bounded from below.
Moreover, the spectrum of $L_\mu(\underline{\psi})$ possesses the Hamiltonian symmetry
\begin{align*}
    \lambda \in \sigma(L_\mu(\underline{\psi})) \quad\Rightarrow\quad
    - \lambda, \overline{\lambda} \in \sigma(L_\mu(\underline{\psi})).
\end{align*}
Therefore, a real-valued stationary solution $\psi$ to~\eqref{GP_time_harmonic} can only be spectrally stable if the spectrum of $L_\mu(\psi)$ is confined to the imaginary axis.

Let $n \in \N$. If $\psi \in H_\per^2(0,nT)$ is a real-valued $nT$-periodic stationary solution to~\eqref{GP_time_harmonic}, then $L_\mu(\psi)$ has $nT$-periodic coefficients and its action on the space $L^2_\per(0,nT)$ is well-defined. Thus, to analyze spectral stability against co-periodic perturbations, we introduce the differential operator $L_{\mu,\per}(\underline{\psi}) \colon D(L_{\mu,\per}(\underline{\psi})) \subset L_\per^2(0,nT) \to L_\per^2(0,nT)$ with dense domain $D(L_{\mu,\per}(\underline{\psi})) = H_\per^2(0,nT)$ by
\begin{align*}
    L_{\mu,\per}(\underline{\psi}) = J 
    \begin{pmatrix}
        L_{+,\mu,\per}(\underline{\psi}) & 0 \\
        0 & L_{-,\mu,\per}(\underline{\psi})
    \end{pmatrix}
\end{align*}
for $\underline{\psi} \in H^1(0,nT)$, where the operators $L_{\pm,\mu,\per}(\underline{\psi}) \colon D(L_{\pm,\mu,\per}(\underline{\psi})) \subset L_\per^2(0,nT) \to L_\per^2(0,nT)$ with $D(L_{\pm,\mu,\per}(\underline{\psi})) = H_\per^2(0,nT)$ are defined by 
$$
     L_{-,\mu,\per}(\underline{\psi}) \phi = - \phi'' + \mu V \phi + \omega\phi + \kappa\underline{\psi}^2 \phi, \qquad L_{+,\mu,\per}(\underline{\psi})\phi = - \phi'' + \mu V \phi + \omega \phi + 3\kappa \underline{\psi}^2 \phi.
$$
We recall from~\S\ref{sec:periodic_diff_operators} that the spectrum of $L_{\mu,\per}(\underline{\psi})$ consists of isolated eigenvalues of finite algebraic multiplicity only. Moreover, due to Hamiltonian symmetry, we find that a real-valued $nT$-periodic stationary solution $\psi$ to~\eqref{GP_time_harmonic} can only be spectrally stable against co-periodic perturbations if the spectrum of $L_{\mu,\per}(\psi)$ is confined to the imaginary axis.

In the following, we divide our analysis in two parts. In the first part, we consider the defocusing Gross-Pitaevskii equation~\eqref{GP_time_harmonic} with $\kappa=1$. Here, we prove the existence of nondegenerate stationary $1$-front solutions connecting periodic end states. These $1$-fronts correspond to so-called \emph{dark solitons}, established
in~\cite{Alfimov2013,Yan2015Dark}. We apply Theorem~\ref{t:existence_multifront} to obtain stationary multifront solutions to~\eqref{GP_time_harmonic} lying near formal concatenations of these primary $1$-fronts. Theorem~\ref{t:existence_periodic} then yields the existence of periodic solutions bifurcating from the formal periodic extension of these multifronts  in case the multifront connects to the same periodic end state at $\pm\infty$. Since $0$ lies in the essential spectrum of linearization operator, the spectral stability of these solutions is a subtle and unresolved issue, beyond the scope of this application section. We refer to~\cite{Pelinovky2008Dark} for a stability analysis of dark solitons in case of a localized potential.

In the second part, we consider the focusing case $\kappa=-1$. We first recall existence results from~\cite{Pelinovsky2011} and references therein, yielding nondegenerate stationary $1$-pulse solutions to~\eqref{GP_time_harmonic}. Taking these so-called \emph{gap solitons} as building blocks, we then employ Theorems~\ref{t:existence_multifront} and~\ref{t:existence_periodic} to construct multipulse solutions as well as periodic pulse solutions. Subsequently, we combine Theorem~\ref{thm:instability_periodic} with Krein index counting theory~\cite{Kapitula2004,AddendumKapitulaKevrekidis2004} to establish spectral stability of the periodic pulse solutions. Our spectral analysis yields orbital stability against co-periodic perturbations through the stability theorem of Grillakis, Shatah and Strauss~\cite{GrillakisShatahStraussII}. Finally, we combine Theorem~\ref{thm:instability_multifront} with Sturm-Liouville theory~\cite{Zettl2021Recent} and Krein index counting arguments to derive instability conditions for the multipulse solutions.

\subsubsection{Fronts in the defocusing Gross-Pitaevskii equation}

We consider the defocusing case $\kappa=1$. Moreover, we assume $\omega<0$. We search for real-valued stationary multifront solutions to~\eqref{GP_time_harmonic} connecting periodic states at $\pm\infty$. Real-valued stationary solutions to~\eqref{GP_time_harmonic} solve the ordinary differential equation
\begin{align}\label{defocusing_GP}
    -\psi'' + \mu V(x) \psi + \omega \psi + \psi^3 = 0.
\end{align}
We aim to employ Theorem~\ref{t:existence_multifront} to construct multifront solutions to~\eqref{defocusing_GP} lying near formal concatenations of front solutions with a single interface. These $1$-front solutions, known as dark solitons, have been rigorously constructed in~\cite{Yan2015Dark} by leveraging a comparison principle, see also~\cite{Torres2008} for the case of a cubic-quintic nonlinearity. However, dark solitons are degenerate solutions to~\eqref{GP_time_harmonic}, obstructing an application of Theorem~\ref{t:existence_multifront}. The reason is that a dark soliton $\psi(\mu)$ to~\eqref{defocusing_GP} necessarily connects to nonconstant periodic end states $v_\pm(\mu) \in H^2_\per(0,T)$, which obey $L_{-,\per,\mu}(v_{\pm}(\mu)) v_\pm(\mu) = 0$, since $v_\pm(\mu)$ must solve~\eqref{defocusing_GP}. Hence, we arrive at $0 \in \sigma(L_{-,\mu}(v_\pm(\mu))) \subset \sigma_{\mathrm{ess}}(L_{-,\mu}(\psi(\mu))) \subset \sigma(L_\mu(\psi(\mu)))$ by~\eqref{Bloch_inclusion} and  Proposition~\ref{prop:essential_spec1front}. 

Here, we circumvent this issue by regarding dark solitons as nondegenerate stationary $1$-front solutions to the real-valued reaction-diffusion problem
\begin{align}\label{defocusing_GP_real}
\partial_t \psi = -\partial_x^2 \psi + \mu V(x) \psi + \omega \psi + \psi^3, \qquad \psi(x,t) \in \R, \, x \in \R, \,  t \geq 0,
\end{align}
which obviously admits the same stationary solutions as~\eqref{GP_time_harmonic}. In the following result, we construct nondegenerate odd $1$-front solutions to~\eqref{defocusing_GP_real} in the case of a small potential by perturbing from the black NLS soliton
$$
    \psi_0(x) = \sqrt{-\omega} \tanh\left(\sqrt{\frac{-\omega}{2}}x\right),
$$
which solves~\eqref{defocusing_GP} at $\mu = 0$.

\begin{Theorem}\label{thm:existence_front_GP}
Let $\omega < 0$ and $T > 0$. Let $V \in C^1(\R)$ be even, $T$-periodic, and real-valued. Let $\chi_\pm \colon \R \to [0,1]$ be a smooth partition of unity such that $\chi_+$ is supported on $(-1,\infty)$, $\chi_-$ is supported on $(-\infty,1)$, and we have $\chi_+(x)= \chi_-(-x)$ for all $x\in \R$. Assume 
\begin{align} \label{effective_potential_GP_defocusing}
    \int_\R V'(x) \psi_0(x) \psi_0'(x) \de x \neq 0.
\end{align}
Then, there exist constants $C,\mu_0>0$ such that for all $\mu \in (0,\mu_0)$ there exists a nondegenerate stationary odd solution $\psi(\mu)$ to~\eqref{defocusing_GP_real} satisfying 
\begin{align} \label{GPfrontbound1}
    \|\psi(\mu)-\psi_0\|_{L^\infty} \leq C \mu, \qquad \chi_\pm(\psi(\mu)- v_{\pm}(\mu)) \in H^2(\R),
\end{align}
where $v_\pm (\mu) \in H_\per^2(0,T)$ are even periodic solutions to~\eqref{defocusing_GP} obeying the bound
\begin{align} \label{GPfrontbound2}
\|v_{\pm}(\mu) \mp \sqrt{-\omega}\|_{H_\per^2(0,T)} \leq C \mu.
\end{align}
\end{Theorem}

\begin{proof}
The proof proceeds along the lines of the proof of Theorem~\ref{thm:RDE_1front} and is divided into two steps. In the first step, we construct small-amplitude periodic solutions $v_{\pm}(\mu)$ to~\eqref{focusing_GP} by bifurcating from the equilibria $\pm \sqrt{-\omega}$ at $\mu = 0$. In the second step, we connect these periodic solutions by an interface, which arises as a localized perturbation of the black soliton $\psi_0$.

Define the smooth nonlinear operator $\mathcal{F}_\per \colon H_\per^2(0,T) \times \R \to L_\per^2(0,T)$ by
$$
    \mathcal{F}_\per(v,\mu) = - v'' + (v^2 + \omega) v + \mu V v.
$$
Since $\mathcal{F}_\per(\sqrt{-\omega},0) = 0$ and 
$$
    \partial_v \mathcal{F}_\per(\sqrt{-\omega},0)  = -\partial_x^2 - 2 \omega
$$
is invertible, an application of the implicit function theorem yields that there exist $\mu_1 > 0$ and a locally unique smooth function $v_+ \colon (-\mu_1,\mu_1) \to H^2_\per(0,T)$ with $v_+(0) = \sqrt{-\omega}$ satisfying
$$
    \mathcal{F}_\per(v_+(\mu),\mu) = 0
$$
for all $\mu \in (-\mu_1,\mu_1)$. Symmetry of the equation~\eqref{defocusing_GP} yields that $x \mapsto v_+(-x;\mu)$ is also a solution. So, by uniqueness we find $v_+(x;\mu) = v_+(-x;\mu)$ for $x \in \R$ and $\mu \in (-\mu_1,\mu_1)$, after shrinking $\mu_1 > 0$ if necessary. Setting $v_-(\mu) = - v_+(\mu)$, we deduce that $v_\pm(\mu) \in H^2_\per(0,T)$ are even periodic solutions to~\eqref{defocusing_GP} obeying~\eqref{GPfrontbound2} for $\mu \in (-\mu_1,\mu_1)$.

We proceed by constructing the interface connecting $v_-(\mu)$ and $v_+(\mu)$. For convenience, we abbreviate $A = - \partial_x^2 + \omega$ and $\mathcal{N}(\psi) = \psi^3$, and write~\eqref{defocusing_GP} in the abstract form
\begin{align} \label{GP_abstract}
    A \psi + \mathcal{N}(\psi) + \mu V \psi = 0. 
\end{align}
Inserting the ansatz
$$
    \psi = v_-(\mu) \chi_- + v_+(\mu) \chi_+ + \psi_0 + \sqrt{-\omega} \chi_- - \sqrt{-\omega} \chi_+ + \varphi
$$
with error term $\varphi \in H^2_\text{odd}(\R)$ into~\eqref{GP_abstract}, we arrive at the equation 
\begin{align*}
   \mathcal{F}(\varphi,\mu) = 0,
\end{align*}
where $\mathcal{F}\colon H_\text{odd}^2(\R)\times (-\mu_1,\mu_1) \to L_\text{odd}^2(\R)$ is the smooth nonlinear operator given by
\begin{align*}
\mathcal{F}(\varphi,\mu) = L_{+,\mu}(\chi_-\tilde v_-(\mu) +\chi_+\tilde  v_+(\mu) + \psi_0) \varphi
    + R(\mu) + \check{\mathcal{N}}(\varphi,\mu)
\end{align*}
with $\tilde{v}_\pm(\mu) = v_\pm(\mu) \mp \sqrt{-\omega}$ and
\begin{align*}
    R(\mu) &= A (\psi_0 + \chi_- \tilde v_-(\mu) + \chi_+ \tilde v_+(\mu)) + \mu V  (\psi_0 + \chi_- \tilde v_-(\mu) + \chi_+ \tilde v_+(\mu)) \\
    &\qquad + \, \mathcal{N} (\psi_0 + \chi_- \tilde v_-(\mu) + \chi_+ \tilde v_+(\mu)) , \\
    \check{\mathcal{N}}(\varphi,\mu) &= \mathcal{N}(\psi_0 + \chi_- \tilde v_-(\mu) + \chi_+ \tilde v_+(\mu)+\varphi)-\mathcal{N}(\psi_0 + \chi_- \tilde v_-(\mu) + \chi_+ \tilde v_+(\mu))\\
    & \qquad-\,\mathcal{N}'(\psi_0 + \chi_- \tilde v_-(\mu) + \chi_+ \tilde v_+(\mu)) \varphi.
\end{align*}
We emphasize that $\mathcal{F}$ is well-defined, because $\chi_+\tilde v_+(\mu) + \chi_- \tilde v_-(\mu)$ and $\psi_0$ are odd functions. 

By Sturm-Liouville theory, cf.~\cite[Theorem 2.3.3]{KapitulaPromislow2013}, the kernel of the second-order operator $L_{+,0}(\psi_0)$ is spanned by the even function $\psi_0' \in H^2(\R)$. Hence, using that $L_{+,0}(\psi_0)$ maps odd functions to odd functions, its restriction $L_{+,0}(\psi_0)|_{H^2_\mathrm{odd}}$ to the subspace of odd functions is well-defined and invertible. In addition, by employing analogous arguments as for the estimates~\eqref{resesttoy} and~\eqref{nonlesttoy} in the proof of Theorem~\ref{thm:RDE_1front}, we find a $\mu$-independent constant $C>0$ such that 
$$
    \|R(\mu)\|_{L^2} \leq C |\mu|, \qquad
    \|\check{\mathcal{N}}(\varphi,\mu)\|_{L^2} \leq C \|\varphi\|_{H^2}^2
$$
for all $\mu \in (-\mu_1,\mu_1)$ and $\varphi \in H_\text{odd}^2(\R)$ with $\|\varphi\|_{H^2}\leq 1$. We conclude that $\mathcal{F}(0,0) = 0$ and $\partial_\varphi \mathcal{F}(0,0) = L_{+,0}(\psi_0)|_{H^2_\mathrm{odd}}$ is invertible. Therefore, the implicit function theorem yields $\mu_2 \in (0,\mu_1)$ and a smooth function $\varphi \colon (-\mu_2,\mu_2) \to H_\text{odd}^2(\R)$ with $\varphi(0) = 0$ satisfying
$$
    \mathcal{F}(\varphi(\mu),\mu) = 0
$$
for $\mu \in (-\mu_2,\mu_2)$. The $1$-front
$$
    \psi(\mu) =v_-(\mu) \chi_- + v_+(\mu) \chi_+ + \psi_0 + \sqrt{-\omega} \chi_- - \sqrt{-\omega} \chi_+ + \varphi(\mu),
$$
is an odd stationary solution to~\eqref{defocusing_GP} for $\mu \in (-\mu_2,\mu_2)$, which obeys~\eqref{GPfrontbound1} by smoothness of $v_\pm$ and $\varphi$, and by exponential localization of $\chi_\pm (\psi_0 \mp \sqrt{-\omega})$ and its derivatives. The nondegeneracy of $\psi(\mu)$ as a stationary solution to~\eqref{defocusing_GP_real} follows from~\eqref{effective_potential_GP_defocusing} using analogous arguments as in the proof of Theorem~\ref{thm:RDE_1front}, and is therefore omitted here.
\end{proof}

\begin{Remark} \label{rem:symmetry_GP}
Let $\psi(\mu)$ be the nondegenerate stationary $1$-front solution to~\eqref{defocusing_GP_real}, established in Theorem~\ref{thm:existence_front_GP}. Thanks to the reflection symmetry of~\eqref{defocusing_GP_real}, we find that $-\psi(\mu)$ is also a nondegenerate $1$-front solution. It connects $v_+(\mu)$ to $v_-(\mu)$.
\end{Remark}

The nondegeneracy of the $1$-front solutions $\pm \psi(\mu)$ to~\eqref{defocusing_GP_real}, established in Theorem~\ref{thm:existence_front_GP} and Remark~\ref{rem:symmetry_GP}, permits the application of Theorems~\ref{t:existence_multifront} and~\ref{thm:instability_multifront}, yielding the existence of nondegenerate stationary multifront solutions to~\eqref{defocusing_GP_real}, see Figure~\ref{fig:GP_multifronts}.

\begin{Corollary}\label{cor:GP_defocus_multifront} 
Let $M \in \N$. Let $\{\psi_j\}_{j=1}^M$ be a sequence of front solutions to~\eqref{defocusing_GP}, established in Theorem~\ref{thm:existence_front_GP} and Remark~\ref{rem:symmetry_GP}, with end states $v_{j,\pm}(\mu)$. Assume that it holds $v_{j,+}(\mu) = v_{j+1,-}(\mu)$ for $j = 1,\ldots,M-1$. Then, there exists $N \in \N$ such that for any $n \in \N$ with $n \geq N$ there exists a nondegenerate stationary multifront solution $\tilde \psi_n$ to~\eqref{defocusing_GP_real} of the form
\begin{align*} \tilde{\psi}_n = a_n + \sum_{j = 1}^M \chi_{j,n} \psi_j(\cdot - jnT),
\end{align*}
where $\chi_{j,n}, j = 1,\ldots,M$ is the smooth partition of unity defined in~\S\ref{sec:m-front}, and $\{a_n\}_n$ is a sequence in $H^2(\R)$ converging to $0$ as $n \to \infty$. 
\end{Corollary}

If the nondegenerate multifronts, obtained in Corollary~\ref{cor:GP_defocus_multifront}, connect to the same end state at $\pm\infty$, then Theorem~\ref{t:existence_periodic} yields large wavelength periodic pulse solutions approximating a formal periodic extension of the multifront.

\begin{Corollary}\label{cor:GP_defocus_periodic} 
Let $\psi$ be a multifront solution to~\eqref{defocusing_GP}, as established Corollary~\ref{cor:GP_defocus_multifront}. Assume that $\psi$ connects to the same periodic end state $v \in H^2_\per(0,T)$ as $x \to \pm \infty$. Then, there exists $N \in \N$ such that for all $n \in \N$ with $n \geq N$ there exists a stationary $nT$-periodic solution $\psi_n$ to~\eqref{defocusing_GP} given by
\begin{align*}
\psi_n(x) = \chi_n(x) \psi(x) + (1 - \chi_n(x)) v(x) + a_n(x), \qquad x \in \left[-\tfrac{n}{2}T,\tfrac{n}{2} T\right),
\end{align*}
where $\chi_n$ is the cut-off function from Theorem~\ref{t:existence_periodic}, and $\{a_n\}_n$ is a sequence with $a_n \in H_\per^2(0,nT)$ satisfying $\|a_n\|_{H_\per^2(0,nT)} \to 0$ as $n \to \infty$.
\end{Corollary}
\begin{figure}[t]
    \centering
    \includegraphics[width=0.3\textwidth]{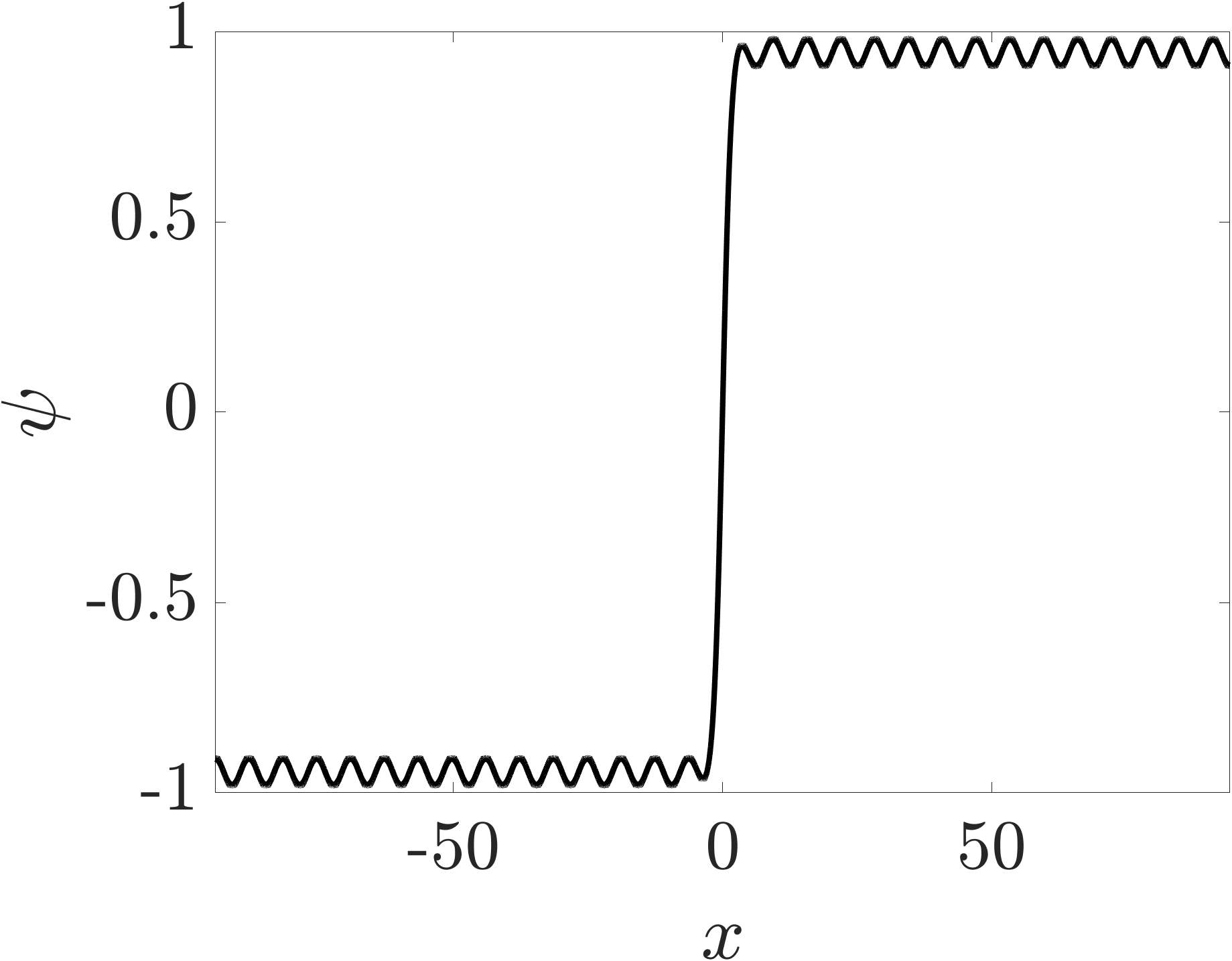} \hspace{1em}
    \includegraphics[width=0.3\textwidth]{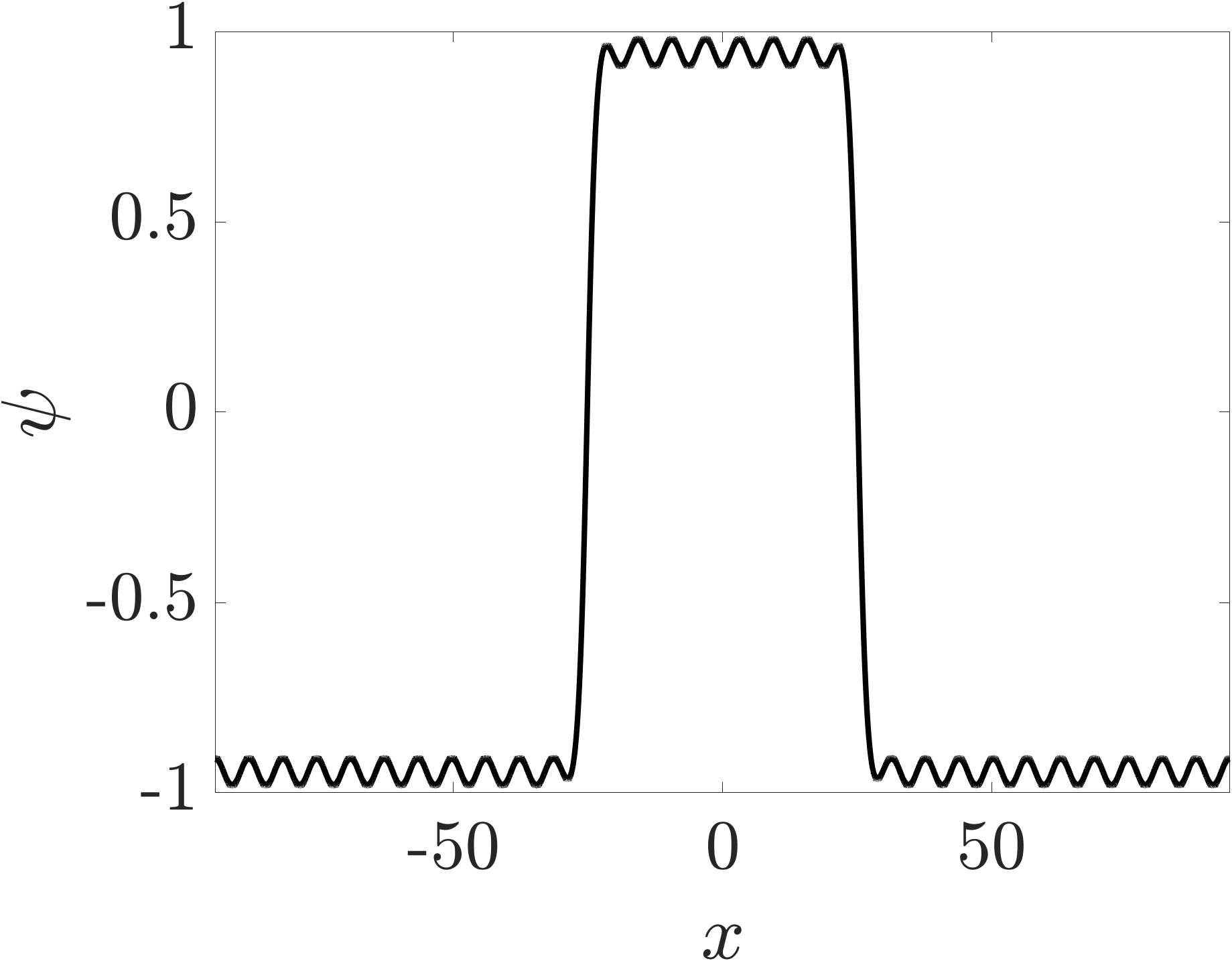}\hspace{1em}
    \includegraphics[width=0.3\textwidth]{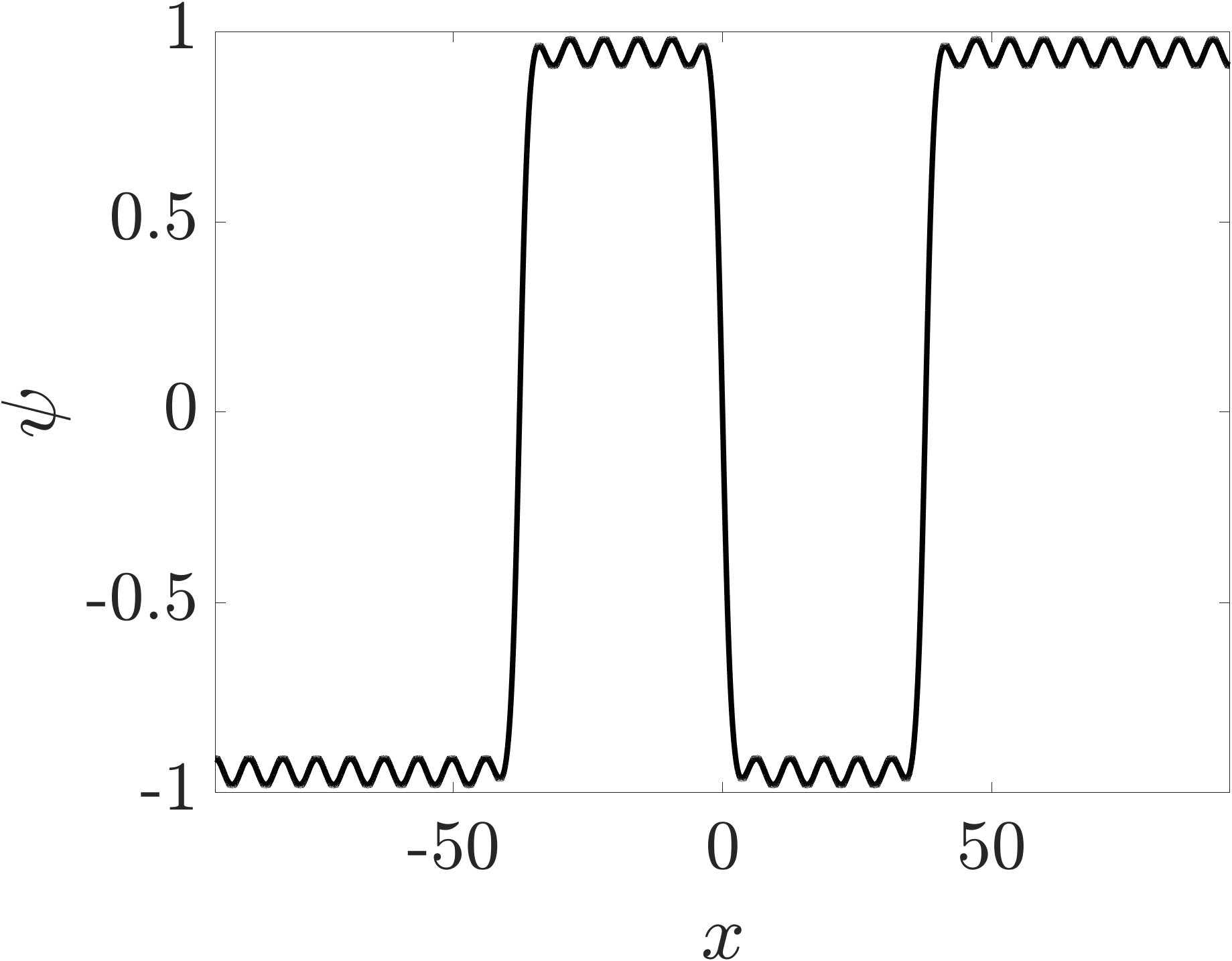}
    \caption{Approximations of stationary real-valued $1$-, $2$-, and $3$-front solutions to the Gross-Pitaevskii equation~\eqref{GP_time_harmonic} for system coefficients $\kappa = 1,\omega = -1, V(x) = 0.2\cos^2(x/2)$, and $\mu=1$. The solutions are obtained through numerical continuation by starting from a formal concatenation of shifted black solitons $\pm \psi_0(\cdot - \varsigma)$, $\varsigma \in \R$, which solve~\eqref{defocusing_GP} at $\mu = 0$.}
    \label{fig:GP_multifronts}
\end{figure}

\begin{Remark}
Under some Conley-Moser-type conditions, all bounded solutions to~\eqref{defocusing_GP} can be characterized using symbolic dynamics~\cite{Alfimov2013}. Specifically, there exists a homeomorphism between the set of all real bounded solutions of~\eqref{defocusing_GP} and the set of bi-infinite sequences of numbers $1,\ldots, N$ for some integer $N \in \N$. As explained in~\cite{Alfimov2013}, this symbolic identification yields the existence of multifronts in~\eqref{defocusing_GP}, as well as periodic solutions featuring multiple front interfaces on a single periodicity interval. The Conley-Moser-type conditions are verified numerically in~\cite{Alfimov2013} in case of the periodic potential $V(x) = \cos(2x)$ in~\eqref{defocusing_GP}. Notably, since we have 
\begin{align*}
\int_\R V'(x) \psi_0(x) \psi_0'(x) \de x = -\frac{8\pi}{\sinh\left(\sqrt{\frac{2}{-\omega}}\, \pi\right)} \neq 0
\end{align*}
for $\omega < 0$, Theorem~\ref{thm:existence_front_GP} and Corollaries~\ref{cor:GP_defocus_multifront} and~\ref{cor:GP_defocus_periodic} rigorously establish the existence of multifronts and periodic pulse solutions to the defocusing Gross-Pitaevskii equation~\eqref{defocusing_GP} with potential $V(x) = \cos(2x)$, provided $\mu > 0$ is sufficiently small.
\end{Remark}

\subsubsection{Pulses in the focusing Gross-Pitaevskii equation}

We now turn to the focusing case $\kappa=-1$. We are interested in the existence and stability of real-valued stationary multipulses and periodic pulse solutions to~\eqref{GP_time_harmonic}. Real-valued stationary solutions to~\eqref{GP_time_harmonic} obey the ordinary differential equation
\begin{align}\label{focusing_GP}
    -\psi'' + \mu V(x) \psi + \omega \psi - \psi^3 = 0.
\end{align}

We first discuss the existence of $1$-pulse solutions to~\eqref{focusing_GP}. These so-called gap solitons will serve as building blocks for the construction of (periodic) multipulse solutions. We emphasize that any real-valued stationary pulse solution $\psi \in H^2(\R) \setminus \{0\}$ to~\eqref{GP_time_harmonic} is degenerate, since $(0,\psi)^\top$ lies in the kernel of the operator $L_{-,\mu}(\psi)$. Therefore, we proceed as in the defocusing case and consider the associated real-valued reaction-diffusion problem
\begin{align} \label{focusing_GP_real}
  \partial_t \psi = -\partial_x^2 \psi + \mu V(x) \psi + \omega \psi - \psi^3, \qquad \psi(x,t) \in \R, \, x \in \R, \,  t \geq 0.
\end{align}

Existence of nondegenerate stationary $1$-pulse solutions to~\eqref{focusing_GP_real} has been shown in different regimes for the parameters $\omega,\mu$ and the potential $V$. For instance, Lyapunov-Schmidt reduction was employed in~\cite{Pelinovsky2011} to find bifurcating $1$-pulse solutions from the family of bright NLS solitons
\begin{align*}
\phi_0(x;\varsigma,\omega) = \sqrt{2\omega} \sech\left(\sqrt{\omega}(x-\varsigma)\right),
\end{align*}
which solve~\eqref{focusing_GP} at $\mu = 0$ and satisfy
\begin{align*} 
\left\langle \partial_\omega \phi_0(\varsigma,\omega),\phi_0(\varsigma,\omega)\right\rangle_{L^2} = \frac{1}{\sqrt{\omega}}\end{align*}
for each $\varsigma\in\R$ and $\omega > 0$. Specifically, the analysis in~\cite[Section~3.2.3]{Pelinovsky2011} yields the following result.

\begin{Theorem} \label{thm:GP_focusing_1_soliton}
Let $\omega_0,T > 0$. Let $V \in C^2(\R)$ be $T$-periodic and real-valued. Let $\varsigma_0 \in \R$ be a simple zero of the derivative of the effective potential $V_{\textup{eff}} \colon \R \to \R$ given by
\begin{align*}
    V_\textup{eff}(\varsigma) = \int_\R V(x+\varsigma) \phi_0(x;0,\omega_0)^2 \de x.
\end{align*}
Then, there exist $\mu_0 > 0$ and a smooth map $\phi \colon (-\mu_0,\mu_0) \times (\omega_0-\mu_0,\omega_0+\mu_0) \to H^2(\R)$ with $\phi(0,\omega_0) = \phi_0(\varsigma_0,\omega_0)$ such that for each $\mu \in (-\mu_0,\mu_0) \setminus \{0\}$ and $\omega \in (\omega_0-\mu_0,\omega_0+\mu_0)$ we have that $\phi(\mu,\omega)$ is a nondegenerate stationary solution to~\eqref{focusing_GP_real} satisfying
\begin{align*}
 \left\langle \partial_\omega \phi(\mu,\omega),\phi(\mu,\omega)\right\rangle_{L^2} > 0.
\end{align*}
In addition, $L_{+,\mu}(\phi(\mu,\omega))$ has precisely one negative eigenvalue in case $\mu V_\textup{eff}''(\varsigma_0) > 0$, whereas it has precisely two negative eigenvalues in case $\mu V_\textup{eff}''(\varsigma_0) < 0$ (counting algebraic multiplicies). Finally, $L_{-,\mu}(\phi(\mu,\omega))$ possesses no negative eigenvalues. 
\end{Theorem}

On the other hand, if $\omega,\mu \in \R$ and $V \in C^1(\R)$ are chosen such that $\omega$ lies in the so-called \emph{semi-infinite gap} $(s_0,\infty)$ where $s_0$ is the spectral bound of the periodic differential operator $\partial_x^2 - \mu V$ acting on $L^2(\R)$, one can use variational methods~\cite{Ackerman2019Unstable,Pelinovsky2011,Lions1984Concentration,Stuart1995Bifurcation} to prove the existence of nontrivial nondegenerate stationary $H^2$-solutions to~\eqref{focusing_GP_real} arising as critical points of the Hamiltonian $\mathcal{H} \colon H^1(\R) \to \R$ given by
\begin{align*}
\mathcal{H}(\psi) = \frac{1}{2}\int_\R \left(|\partial_x\psi(x)|^2 + \left(\omega + \mu V(x)\right) |\psi(x)|^2 - \frac{1}{2} |\psi(x)|^4 \right)\de x.
\end{align*}

We refer to~\cite{Pelinovsky2011} and references therein for further details on the existence of gap-soliton solutions to~\eqref{focusing_GP}. In the remaining part of this section, we assume that $\omega,\mu \in \R$ and $V \in C^1(\R)$ are such that a nondegenerate stationary $1$-pulse solution to~\eqref{focusing_GP_real} exists. Theorems~\ref{t:existence_multifront},~\ref{t:existence_periodic} and~\ref{thm:instability_multifront} then readily yield the existence of associated multifronts and periodic pulse solutions to~\eqref{focusing_GP}, see Figure~\ref{fig:GPpulses}.

\begin{Corollary}\label{cor:GP_existence}
Let $T > 0$ and $\omega,\mu \in \R$. Let $V \in C^1(\R)$ be $T$-periodic and real-valued. Let $M \in \N$ and $\theta \in \{\pm 1\}^M$. Let $\psi_0 \in H^2(\R)$ be a nondegenerate stationary solution to~\eqref{focusing_GP_real}. Then, there exists $N\in \N$ such that for each $n \in \N$ with $n \geq N$ the following assertions hold true.
\begin{enumerate}
    \item[1.] There exists a nondegenerate stationary $M$-pulse solution to~\eqref{focusing_GP_real} of the form
    \begin{align} \label{form_multipulse_GP}
        \psi_n = \sum_{j=1}^M \theta_j\psi_0(\cdot-jnT) + a_n,
    \end{align}
    where $\{a_n\}_n$ is a sequence in $H^2(\R)$ converging to $0$ as $n \to \infty$.
    \item[2.] There exists an $nT$-periodic solution $\psi_{\per,n}$ to~\eqref{focusing_GP} given by
    \begin{align*}
        \psi_{\per,n}(x) = \chi_n(x) \psi_0(x) + a_n(x), \qquad x \in [-\tfrac{n}{2}T,\tfrac{n}{2}T),
    \end{align*}
    where $\chi_n$ is the cut-off function from Theorem~\ref{t:existence_periodic}, and $\{a_n\}_n$ is a sequence with $a_n \in H_\per^2(0,nT)$ satisfying $\|a_n\|_{H_\per^2(0,nT)} \to 0$ as $n \to \infty$. 
\end{enumerate}
\end{Corollary}

\begin{Remark}
It is also possible to construct multipulse solutions lying near the formal concatenation of \emph{different} nondegenerate stationary pulse solutions $\psi_1, \psi_2 \in H^2(\R)$ to~\eqref{focusing_GP_real} using Theorem~\ref{t:existence_multifront}.
\end{Remark}

\begin{Remark}
The existence of multipulse solutions to~\eqref{focusing_GP} is also addressed in~\cite{Alama1992Multibump,Pelinovsky2011}. Corollary~\ref{cor:GP_existence} reveals that pulse solutions are accompanied by a family of periodic solutions with large spatial period. As far as the authors are aware, existence of these so-called \emph{soliton trains} has so far only been rigorously established in the case of the explicit periodic potential
\begin{align} \label{explicit_potential_cnoidal}
V(x) = \mathrm{cn}^2\left(\frac{1}{\sqrt{2}} x;k\right) - 1,
\end{align}
where $\mathrm{cn}(x;k)$ is the Jacobi cosine function (cnoidal wave) with elliptic modulus $k \in (0,1)$, cf.~\cite{Bronski2001}. 

For such a potential, periodic waves exist for specific values of $\omega$ and have the form
\begin{align*}
\psi(x) = \sqrt{\mu + k^2}\, \mathrm{cn}\left(\frac{1}{\sqrt{2}} x;k\right)
\end{align*}
for $\omega = \mu + k^2 - \frac{1}{2}$ and $\mu \geq -k^2$, and
\begin{align*}
\psi(x) = \frac{\sqrt{\mu + k^2}}{k}\, \mathrm{dn}\left(\frac{1}{\sqrt{2}} x;k\right)
\end{align*}
for $\omega = 1+ \mu k^{-2} - \frac{1}{2}k^2$ and $\mu \geq -k^2$. In the homoclinic limit $k \uparrow 1$, the period tends to infinity, the potential approaches $x \mapsto -\mu \tanh^2(x/\sqrt{2})$, and the periodic waves approximate the pulse $x \mapsto \smash{\sqrt{\mu^2 + 1}} \sech(x/\sqrt{2})$ on a single periodicity interval. Thus, these periodic waves resemble periodic pulse solutions, or soliton trains, for $0 \ll k < 1$.
\end{Remark}

\begin{figure}[t]
    \centering
    \includegraphics[width=0.3\textwidth]{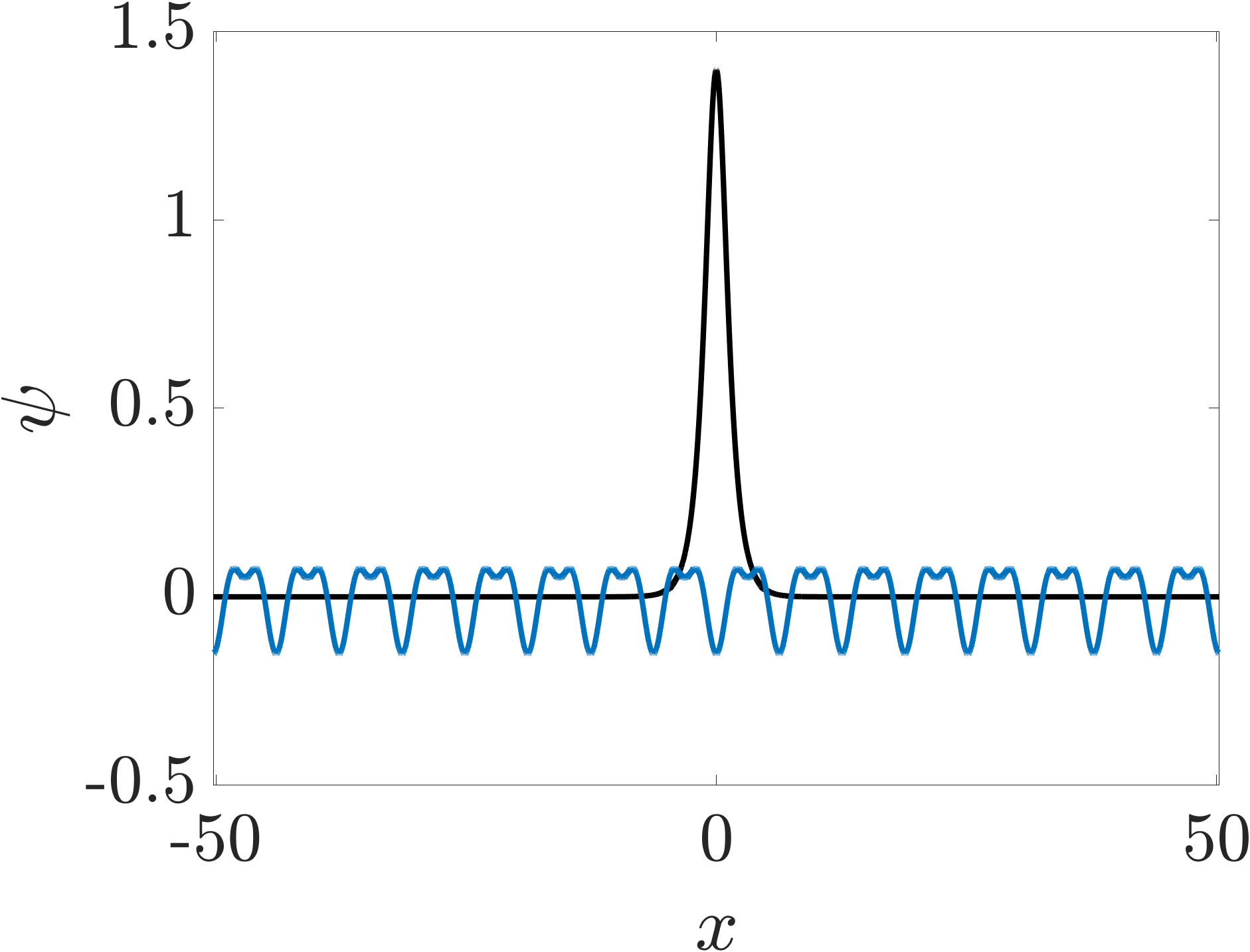} \hspace{1em}
    \includegraphics[width=0.3\textwidth]{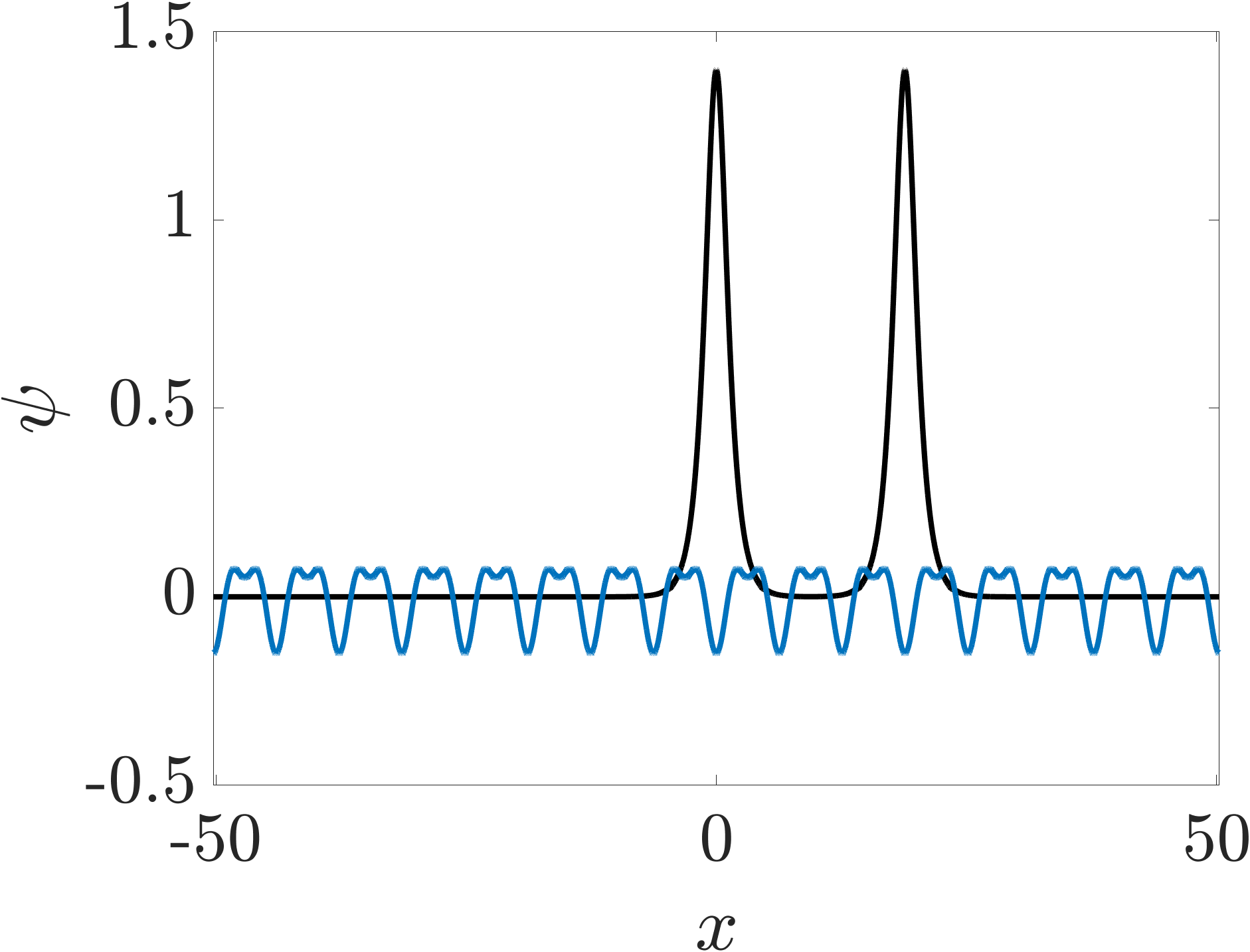} \hspace{1em}
    \includegraphics[width=0.2865\textwidth]{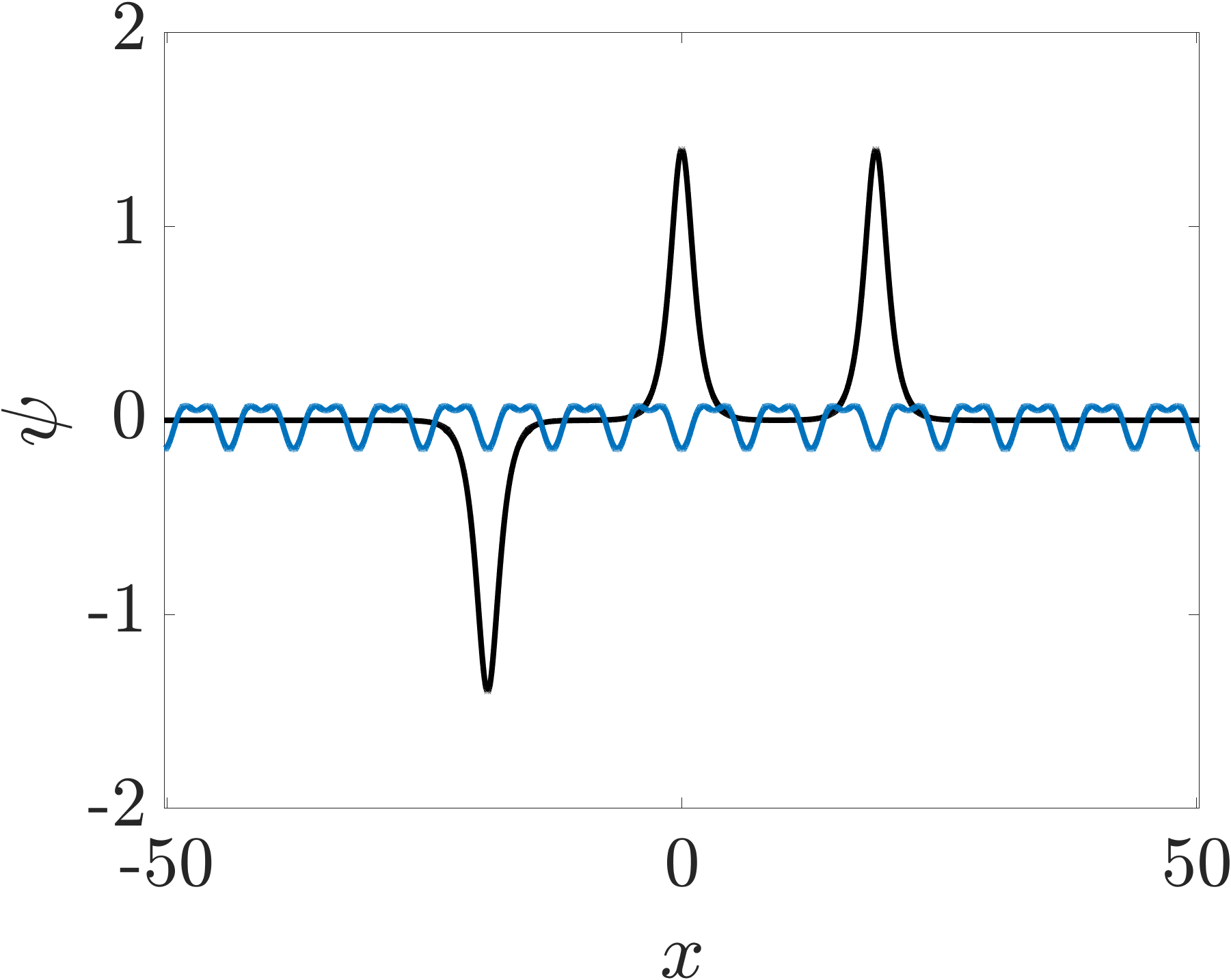}
    \caption{Approximations of $1$-, $2$-, and $3$-pulse solutions to~\eqref{focusing_GP} (black) pinned to minima of the periodic potential $V(x) = -0.1 \cos(x)-0.05\cos(2x)$ (blue) for the system coefficients $\omega = 1$ and $\mu = 0.5$. The solutions are obtained through numerical continuation by starting from a formal concatenation of bright NLS solitons $\pm \phi_0(\varsigma,\omega)$ for various $\varsigma \in \R$, which solve~\eqref{focusing_GP} at $\mu = 0$.}
    \label{fig:GPpulses}
\end{figure}

We proceed with analyzing the spectral stability of the stationary multipulses and periodic pulse solutions to~\eqref{GP_time_harmonic}, constructed in Corollary~\ref{cor:GP_existence}. To this end, we fix a nondegenerate primary pulse solution $\psi_0 \in H^2(\R)$ to~\eqref{focusing_GP_real} for a frequency $\omega = \omega_0$. Since $L_{+,\mu}(\psi_0)$ is invertible, it follows directly from the implicit function theorem that $\psi_0$ may be continued in $\omega$.

\begin{Lemma} \label{lem:cont_GP_omega}
Let $T>0$ and $\mu,\omega_0 \in \R$. Let $V \in C^1(\R)$ be $T$-periodic and real-valued. Let $\psi_0 \in H^2(\R)$ be a nondegenerate stationary solution to~\eqref{focusing_GP_real} at $\omega = \omega_0$. There exist $\nu > 0$ and a locally unique smooth map $\tilde\psi \colon (\omega_0-\nu,\omega_0 + \nu) \to H^2(\R)$ with $\tilde\psi(\omega_0) = \psi_0$ such that $\tilde\psi(\omega)$ is a solution to~\eqref{focusing_GP} for all $\omega \in (\omega_0 - \nu,\omega_0+\nu)$.
\end{Lemma}

We impose the following assumption on the existence and spectral properties of the primary pulse.

\begin{itemize}
    \item[\namedlabel{assGP}{\textbf{(GP)}}] There exist $T > 0$, $\mu,\omega_0 \in \R$, real-valued $T$-periodic $V \in C^1(\R)$, and a nondegenerate stationary solution $\psi_0 \in H^2(\R)$ to~\eqref{focusing_GP_real} at $\omega = \omega_0$ satisfying:
    \begin{itemize}
        \item[1.] $L_{+,\mu}(\psi_0)$ has precisely one negative eigenvalue (counting algebraic multiplicities).
        \item[2.] $L_{-,\mu}(\psi_0)$ has no negative eigenvalues.
        \item[3.] It holds $\langle \partial_\omega \tilde\psi(\omega_0),\psi_0 \rangle_{L^2} \neq 0$, where $\tilde\psi$ is the continuation of $\psi_0$ with respect to $\omega$, established in Lemma~\ref{lem:cont_GP_omega}.
    \end{itemize}
\end{itemize}

The spectral conditions in~\ref{assGP} are typically imposed in the stability analysis of stationary pulse solutions to the focusing Gross-Pitaevskii equation, see, for instance,~\cite[Section~4]{Pelinovsky2011} and references therein. We observe by Sturm-Liouville theory~\cite{Zettl2021Recent} that the second assertion in~\ref{assGP} holds if and only if $\psi_0$ has no zeros. Moreover, Theorem~\ref{thm:GP_focusing_1_soliton} shows that, as long as the derivative of the effective potential $V_\text{eff}$ has a simple zero, there exist nondegenerate pulse solutions $\psi_0 \in H^2(\R) \setminus \{0\}$ to~\eqref{focusing_GP_real}, obeying the spectral conditions in~\ref{assGP}, see also the forthcoming Remark~\ref{rem:GP_ass}.

The following result demonstrates that spectral stability of the periodic pulse solution $\psi_{\per,n}$, obtained in Corollary~\ref{cor:GP_existence}, is inherited from the constituting primary pulse $\psi_0$. Its proof employs Krein index counting theory~\cite{AddendumKapitulaKevrekidis2004,KapitulaKevrekidis2004} to show that, if the linearization $L_{\mu,\per}(\psi_{\per,n})$ has unstable eigenvalues, then they must be real and are separated from the imaginary axis by an $n$-independent spectral gap. We use Theorem~\ref{thm:instability_periodic} to relate the number of negative eigenvalues of the self-adjoint operators $L_{\pm,\mu,\per}(\psi_{\per,n})$ to those of $L_{\pm,\mu}(\psi_0)$, and to show the existence of an $n$-independent ball centered at the origin, in which $0$ is the only eigenvalue of $L_{\mu,\per}(\psi_n)$. Having established that any unstable eigenvalue of $L_{\mu,\per}(\psi_{\per,n})$ is real and bounded away from the imaginary axis, an application of Lemma~\ref{lem:invertibility_periodic}, together with standard spectral a-priori bounds, rules out the presence of unstable eigenvalues of $L_{\mu,\per}(\psi_{\per,n})$. 

\begin{Theorem}\label{thm:stability_periodic_GP}
Assume~{\upshape \ref{assGP}}. Suppose $\omega_0$ is larger than the spectral bound of the operator $\partial_x^2 - \mu V$ acting on $L^2(\R)$. Then, there exists $N \in \mathbb{N}$ such that for all $n \in \N$ with $n \geq N$ the following statements are equivalent:
\begin{itemize}
    \item[1.] The $1$-pulse $\psi_0 \in H^2(\R)$ is a spectrally stable solution to~\eqref{GP_time_harmonic}. 
    \item[2.] We have $\langle \partial_\omega \tilde\psi(\omega_0),\psi_0 \rangle_{L^2} > 0$.
    \item[3.] The periodic pulse solution $\psi_{\per,n} \in H_\per^2(0,nT)$ to~\eqref{GP_time_harmonic}, established in Corollary~\ref{cor:GP_existence}, is spectrally stable against co-periodic perturbations, i.e., the spectrum of $L_{\mu,\per}(\psi_{\per,n})$ is confined to the imaginary axis.
\end{itemize}
Moreover, if one of these statements holds, then we have the following:
\begin{itemize}
    \item[a.] $L_{-,\mu,\per}(\psi_n)$ has no negative eigenvalues and a simple eigenvalue at $0$.
    \item[b.] $L_{+,\mu,\per}(\psi_n)$ is invertible and has precisely one negative eigenvalue, which is simple.
    \item[c.] There exist $\nu_n > 0$ and a smooth map $\tilde{\psi}_{\per,n} \colon (\omega_0-\nu_n,\omega_0 + \nu_n) \to H^2_\per(0,nT)$ such that $\tilde{\psi}_{\per,n}(\omega_0) = \psi_{\per,n}$ and $\tilde{\psi}_{\per,n}(\omega)$ is a solution to~\eqref{focusing_GP} for all $\omega \in (\omega_0 - \nu_n,\omega_0 + \nu_n)$. It holds
    \begin{align} \label{negKrein_GP}
\left\langle \partial_\omega \tilde{\psi}_{\per,n}(\omega_0),\psi_{\per,n}\right\rangle_{L^2_\per(0,nT)} > 0.
    \end{align}
\end{itemize}
\end{Theorem}

\begin{proof}
The fact that the first two statements are equivalent follows from~\cite[Theorem 4.8]{Pelinovsky2011}. 

We prove that the third implies the second statement by contrapositon. If $\langle \partial_\omega \tilde\psi(\omega_0),\psi_0 \rangle_{L^2} < 0$, then $\psi_0$ is spectrally unstable by~\cite[Theorem 4.8]{Pelinovsky2011} with an element $\lambda \in \sigma(L_\mu(\psi))$ in the point spectrum with $\Re(\lambda)>0$. Upon applying Theorem~\ref{thm:instability_periodic}, we infer spectral instability of $\psi_{\per,n}$, provided $n \in \N$ is sufficiently large. 

Finally, we prove that the second implies the third statement. To this end, we assume that $\langle \partial_\omega \tilde\psi(\omega_0),\psi_0 \rangle_{L^2} > 0$. First, we note that the essential spectrum of $L_{\pm,\mu}(\psi_0)$ is by Proposition~\ref{prop:essential_spec1front} given by $\sigma(L_{\pm,\mu}(0)) = \sigma(-\partial_x^2+\mu V + \omega_0)$, which is, by assumption, confined to the positive half-line. Therefore, $L_{\pm,\mu}(0)$ is invertible and $0$ does not lie in the essential spectra of $L_{\pm,\mu}(\psi_0)$ and $L_\mu(\psi_0)$. On the other hand, $L_{-,\mu}(\psi_0)\psi_0 = 0$ and Sturm-Liouville theory, cf.~\cite[Theorem 2.3.3]{KapitulaPromislow2013}, imply that $0$ is a simple eigenvalue of $L_{-,\mu}(\psi_0)$. Differentiating $L_{-,\mu}(\tilde\psi(\mu)) \tilde\psi(\mu) = 0$ with respect to $\omega$ and setting $\omega = \omega_0$, we obtain $L_{+,\mu}(\psi_0) \partial_\omega \tilde\psi(\omega_0) = -\psi_0$, cf.~Lemma~\ref{lem:cont_GP_omega}. Therefore, $0$ is an eigenvalue of $L_\mu(\psi_0)$ whose algebraic multiplicity is at least $2$. Using that $L_{-\mu}(\psi_0)$ is self-adjoint and Fredholm of index $0$, observing that $0 \notin \sigma_{\mathrm{ess}}(L_\mu(\psi_0))$, and noting $\langle \partial_\omega \tilde\psi(\omega_0),\psi_0 \rangle_{L^2} > 0$, we find that $0$ is an isolated eigenvalue of $L_\mu(\psi_0)$ of algebraic multiplicity $2$. Hence, Theorem~\ref{thm:instability_periodic} yields $\eta_1 > 0$ and $N_1 \in \N$ such that for all $n \in \N$ with $n \geq N_1$ the total algebraic multiplicity of the eigenvalues of $L_{\mu,\per}(\psi_{\per,n})$ in the ball $B_0(\eta_1)$ equals $2$. 

Next, we employ Krein index counting theory~\cite{KapitulaKevrekidis2004,AddendumKapitulaKevrekidis2004} to prove the absence of eigenvalues of $L_{\mu,\per}(\psi_{\per,n})$ of positive real part. We start by counting negative eigenvalues of the self-adjoint operator $L_{+,\mu,\per}(\psi_{\per,n})$. Since $L_{+,\mu}(\psi_0)$ has precisely one negative eigenvalue $\lambda_1 < 0$ (counting algebraic multiplicities) and $\psi_0$ is a nondegenerate solution to~\eqref{focusing_GP_real}, there exists $\eta_2 > 0$ such that $\sigma(L_{+,\mu}(\psi_0)) \subset \{\lambda_1\} \cup (\eta_2,\infty)$. Moreover, since $\|\psi_{\per,n}\|_{L^\infty}$ can be bounded by an $n$-independent constant by Corollary~\ref{cor:GP_existence} and the continuous embedding $H^1(0,nT) \hookrightarrow L^\infty(\R)$ with $n$-independent constant, there exists by Lemma~\ref{lem:a_priori_toy} an $n$-independent constant $\eta_3 > 0$ such that the spectrum of the operator $L_{+,\mu,\per}(\psi_{\per,n})$ is confined to $[-\eta_3,\infty)$. Combining the last two sentences with Lemma~\ref{lem:invertibility_periodic} and Theorem~\ref{thm:instability_periodic}, we find that there exists $N_2 \in \N$ with $N_2 \geq N_1$ such that for all $n \in \N$ with $n \geq N_2$ the operator $L_{+,\mu,\per}(\psi_{\per,n})$ has precisely one eigenvalue in the set $[-\eta_3,\eta_2]$, which is simple and negative. In particular, this implies assertion b. Since $L_{+,\mu,\per}(\psi_{\per,n})$ is invertible, the implicit function theorem yields $\nu_n > 0$ and a smooth map $\smash{\tilde{\psi}_{\per,n}} \colon (\omega_0-\nu_n,\omega_0 + \nu_n) \to H^2_\per(0,nT)$ with $\smash{\tilde{\psi}_{\per,n}(\omega_0)} = \psi_{\per,n}$ such that $\smash{\tilde{\psi}_{\per,n}}(\omega)$ is a solution to~\eqref{focusing_GP} for all $\omega \in (\omega_0 - \nu_n,\omega_0 + \nu_n)$.

Our next step is to count eigenvalues of the operator $L_{-,\mu,\per}(\psi_{\per,n})$. Since $\psi_{\per,n} \in H^2_\per(0,nT)$ is a nontrivial solution to~\eqref{focusing_GP}, we find $L_{-,\mu,\per}(\psi_{\per,n})\psi_{\per,n} = 0$. So, by Sturm-Liouville theory, cf.~\cite[Theorem~6.3.1.(8)(3)]{Zettl2021Recent}, we deduce that $0$ is a simple isolated eigenvalue of $L_{-,\mu,\per}(\psi_{\per,n})$. Therefore, using analogous arguments as for the operator $L_{+,\mu,\per}(\psi_{\per,n})$, we infer that the facts that $L_{-,\mu}(\psi_0)$ has no negative eigenvalues and $0$ is an isolated simple eigenvalue of $L_{-,\mu}(\psi_0)$ imply that there exists $N_3\in \N$ with $N_3\geq N_2$ such that $L_{-,\mu,\per}(\psi_{\per,n})$ has no negative eigenvalues for all $n \in \N$ with $n \geq N_3$. This yields assertion a.

Next, we show that the eigenvalue $0$ of $L_{\mu,\per}(\psi_{\per,n})$ has geometric multiplicity one and algebraic multiplicity two. The fact that $0$ has geometric multiplicity one follows from the analysis of the operators $L_{\pm,\mu,\per}(\psi_{\per,n})$. An associated eigenfunction is given by $(0,\psi_{\per,n})^\top$. Differentiating the equation $L_{-,\mu,\per}(\tilde\psi_{\per,n}(\omega))\tilde\psi_{\per,n}(\omega) = 0$ with respect to $\omega$ and setting $\omega = \omega_0$, we obtain the identities
$$L_{+,\mu,\per}(\psi_{\per,n})\partial_\omega\tilde{\psi}_{\per,n}(\omega_0) = - \psi_{\per,n}, \qquad L_{\mu,\per}(\psi_{\per,n})\begin{pmatrix} \partial_\omega \tilde{\psi}_{\per,n}(\omega_0)\\ 0\end{pmatrix} = -\begin{pmatrix} 0 \\ \psi_{\per,n}\end{pmatrix}.$$
Using that the total algebraic multiplicity of the eigenvalues of $L_{\mu,\per}(\psi_{\per,n})$ in $B_0(\eta_1)$ is $2$, we conlcude that $0$ is an eigenvalue of $L_{\mu,\per}(\psi_{\per,n})$ of algebraic multiplicity $2$. Combining the latter with the Fredholm alternative, we arrive at $\smash{\langle \partial_\omega \tilde{\psi}_{\per,n}(\omega_0),\psi_{\per,n}\rangle_{L^2_\per(0,nT)}} \neq 0$. 

Therefore, we can apply the instability index formula from~\cite{AddendumKapitulaKevrekidis2004} to find $k_r \leq 1$ and $k_c = k_i^- = 0$, where $k_r$ is the number of real unstable eigenvalues of $L_{\mu,\per}(\psi_{\per,n})$, $k_c$ is the number of quadruplets of eigenvalues with non-vanishing real and imaginary parts, and $k_i^-$ is the number of purely imaginary eigenvalues of $L_{\mu,\per}(\psi_{\per,n})$ with negative Krein signature~\cite{AddendumKapitulaKevrekidis2004} (all counting algebraic multiplicities). In the following, we prove $k_r = 0$, which establishes~\eqref{negKrein_GP} by~\cite[Theorem~1]{AddendumKapitulaKevrekidis2004}. To this end, we first show that there exists an $n$-independent constant $\rho >0$ such that $\lambda \in \sigma(L_{\mu,\per}(\psi_{\per,n}))$ implies $|\Re(\lambda)| \leq \rho$. To establish this spectral a-priori bound, we consider the principal part of $L_{\mu,\per}(\psi_{\per,n})$, which is the skew-adjoint operator $A_0 \colon D(A_0) \subset L_\per^2(0,nT) \to L^2(0,nT)$ with dense domain $D(A_0) = H^2_\per(0,nT)$ given by $A_0 \boldsymbol{\psi} = -J\boldsymbol{\psi}''$. By Stone's Theorem, cf.~\cite[Theorem~3.24]{EngelNagel2000}, $A_0$ generates a unitary group on the Hilbert space $L_\per^2(0,nT)$. In particular,~\cite[Theorem~1.10]{EngelNagel2000} yields the resolvent bound
\begin{align} \label{resolventprincipal}
\left\|\left(A_0 - \lambda\right)^{-1} \boldsymbol{\psi}\right\|_{L_\per^2(0,nT)} \leq \frac{\|\boldsymbol{\psi}\|_{L^2_\per(0,nT)}}{|\mathrm{Re}(\lambda)|},
\end{align}
for $\boldsymbol{\psi} \in L^2_\per(0,nT)$ and $\lambda \in \C$ with $|\mathrm{Re}(\lambda)| > 0$. Using that $\|\psi_{\per,n}\|_{L^\infty}$ is bounded by an $n$-independent constant, we obtain an $n$-independent constant $C> 0$ such that the residual $L_{\mu,\per}(\psi_{\per,n}) - A_0$ enjoys the estimate
\begin{align*}
\left\|\left(L_{\mu,\per}(\psi_{\per,n}) - A_0\right) \boldsymbol{\psi} \right\|_{L_\per^2(0,nT)} &\leq C \|\boldsymbol{\psi}\|_{L_\per^2(0,nT)} 
\end{align*}
for $\boldsymbol{\psi} \in L_\per^2(0,nT)$. Combining this estimate with~\eqref{resolventprincipal} and~\cite[Theorem~IV.1.16]{Kato1995} yields an $n$-independent constant $\rho > 0$ such that $L_{\mu,\per}(\psi_{\per,n}) - \lambda = A_0 - \lambda + L_{\mu,\per}(\psi_{\per,n}) - A_0$ is bounded invertible for each $\lambda \in \C$ with $|\Re(\lambda)| \geq \rho$. On the other hand,  Lemma~\ref{lem:invertibility_periodic} and the spectral stability of $\psi_0$ imply that there exists $N_4 \in \N$ with $N_4 \geq N_3$ such that for all $n \in \N$ with $n \geq N_4$ the compact set $[-\rho,-\eta_1] \cup [\eta_1,\rho]$ lies in the resolvent set of $L_{\mu,\per}(\psi_{\per,n})$. Combining the latter with the spectral a-priori bound and the fact that $0$ is the only eigenvalue of $L_{\mu,\per}(\psi_{\per,n})$ in $(-\eta_1,\eta_1)$, we conclude that $k_r = 0$, which establishes the third statement and~\eqref{negKrein_GP}.
\end{proof}

Using that the Gross-Pitaevskii equation~\eqref{GP_system} may be expressed as the Hamiltonian system
\begin{align*}
\partial_t \boldsymbol{\psi} = J \nabla \mathcal{H}_\per(\boldsymbol{\psi})
\end{align*}
on $H^1_\per(0,nT)$ with Hamiltonian $\mathcal{H}_\per \colon H^1_\per(0,nT) \to \R$ given by
\begin{align*}
\mathcal{H}_\per(\boldsymbol{\psi}) = \frac12 \int_0^{nT} \left(\left|\boldsymbol{\psi}_x(x)\right|^2 + \left(\omega + \mu V(x)\right) |\boldsymbol{\psi}(x)|^2 -  \frac12 |\boldsymbol{\psi}|^4 \right) \de x,
\end{align*}
it immediately follows from Theorem~\ref{thm:stability_periodic_GP} and the stability theorem from Grillakis, Shatah, and Strauss~\cite{GrillakisShatahStraussII} that the periodic pulse solutions, established in Corollary~\ref{cor:GP_existence}, are orbitally stable if Assumption~\ref{assGP} holds and we have $\langle \partial_\omega \tilde\psi(\omega_0),\psi_0 \rangle_{L^2} > 0$.  

\begin{Corollary} \label{cor:GP_focus_periodic}
Assume~{\upshape \ref{assGP}} and $\langle \partial_\omega \tilde\psi(\omega_0),\psi_0 \rangle_{L^2} > 0$. Suppose $\omega_0$ is larger than the spectral bound of the operator $\partial_x^2 - \mu V$ acting on $L^2(\R)$.

Then, there exists $N \in \N$ such that for all $n \in \N$ with $n \geq N$ the periodic pulse solution $\boldsymbol{\psi}_{\per,n} = (\psi_{\per,n},0)^\top \in H_\per^2(0,nT)$ to~\eqref{GP_system}, where $\psi_{\per,n}$ is established in Corollary~\eqref{cor:GP_existence}, is orbitally stable. Specifically, for each $\eps > 0$ there exists $\delta > 0$ such that, whenever $v_0 \in H^1_\per(0,nT)$ satisfies $\|v_0\|_{H^1_\per(0,nT)} < \delta$, then there exists a global mild solution $\boldsymbol{\psi} \in C\big([0,\infty),H^1_\per(0,nT)\big)$ to~\eqref{GP_system} with initial condition $\boldsymbol{\psi}(0) = \boldsymbol{\psi}_{\per,n} + v_0$ obeying
\begin{align*}
\inf_{\gamma \in \R} \left\|\boldsymbol{\psi}(t) - R(\gamma)\boldsymbol{\psi}_{\per,n}\right\|_{H^1_\per(0,nT)} < \eps.
\end{align*}
\end{Corollary}

To the authors' best knowledge, Theorem~\ref{thm:stability_periodic_GP} and Corollary~\ref{cor:GP_focus_periodic} present the fist rigorous spectral and orbital stability result of any periodic wave in the Gross-Pitaevskii equation with periodic potential. In the case of the explicit periodic potential~\eqref{explicit_potential_cnoidal}, rigorous instability results, as well as numerical simulations indicating spectral stability, can be found in~\cite{Bronski2001}.

\begin{Remark} \label{rem:GP_ass}
For fixed $\omega_0, T > 0$, real-valued $T$-periodic $V \in C^2(\R)$, and simple zero $\varsigma_0 \in \R$ of $V_\textup{eff}'$, Lemma~\ref{lem:a_priori_toy} and Theorem~\ref{thm:GP_focusing_1_soliton} yield $\mu \in \R \setminus \{0\}$ such that the spectral bound of the operator $\partial_x^2 - \mu V$, acting on $L^2(\R)$, is smaller than $\omega_0$, and equation~\eqref{focusing_GP_real} possesses a nondegenerate stationary solution $\psi_0 \in H^2(\R)$ at $\omega = \omega_0$ satisfying the spectral conditions in~\ref{assGP}. Since it holds $\langle \partial_\omega \tilde\psi(\omega_0),\psi_0 \rangle_{L^2} > 0$, the associated stationary periodic pulse solution $\psi_{\per,n} \in H^2_\per(0,nT)$ to~\eqref{GP_time_harmonic}, established in Corollary~\ref{cor:GP_existence}, are spectrally and orbitally stable against co-periodic perturbations by Theorem~\ref{thm:stability_periodic_GP} and Corollary~\ref{cor:GP_focus_periodic}, provided $n \in \N$ is sufficiently large.
\end{Remark}

Next, we turn to the spectral analysis of the multipulse solutions, obtained in Corollary~\ref{cor:GP_existence}. For this, we first establish the following lemma.

\begin{Lemma}\label{lem:convergence_multipulse}
Let $T>0$ and $\mu,\omega_0 \in \R$. Let $V \in C^1(\R)$ be $T$-periodic and real-valued. Let $M \in \N$ and $\theta \in \{\pm1\}^M$. Let $\psi_0 \in H^2(\R)$ be a nondegenerate stationary solution to~\eqref{focusing_GP_real} at $\omega = \omega_0$. Then, there exists $N \in \N$ such that for all $n \in \N$ with $n \geq N$ there exist $\nu_n > 0$ and a smooth map $\tilde{\psi}_n \colon (\omega_0-\nu_n,\omega_0+\nu_n) \to H^2(\R)$ such that $\tilde{\psi}_n(\omega)$ is a solution to~\eqref{focusing_GP} for all $\omega \in (\omega_0-\nu_n,\omega_0+\nu_n)$, and we have $\tilde{\psi}_n(\omega_0) = \psi_n$, where $\psi_n \in H^2(\R)$ is the multipulse solution obtained in Corollary~\ref{cor:GP_existence}. Moreover, it holds
    \begin{align} \label{limitkrein_GP}
        \lim_{n \to \infty} \left\langle \partial_\omega \tilde\psi_n(\omega_0) , \psi_n \right\rangle_{L^2} = 
        M \left\langle \partial_\omega \tilde\psi(\omega_0), \psi_0 \right\rangle_{L^2},
    \end{align}
where $\tilde\psi$ is the continuation of $\psi_0$ with respect to $\omega$, established in Lemma~\ref{lem:cont_GP_omega}.
\end{Lemma}
\begin{proof}
Since $\psi_n$ is a nondegenerate stationary solution to~\eqref{focusing_GP_real} by Corollary~\ref{cor:GP_existence}, the existence of the map $\tilde{\psi}_n$ follows from the implicit function theorem. So, all that remains is to prove~\eqref{limitkrein_GP}. Differentiating the equations $L_{-,\mu}(\tilde\psi(\omega))\tilde\psi(\omega) = 0$ and $L_{-,\mu}(\tilde\psi_n(\omega))\tilde\psi_n(\omega) = 0$ with respect to $\omega$ and setting $\omega = \omega_0$, we arrive at
\begin{align}\label{eq:omega_derivative_multipulse}
   L_{+,\mu}(\psi_0)\left[ \partial_\omega\tilde\psi(\omega_0)\right] = - \psi_0, \qquad  L_{+,\mu}(\psi_n)\left[ \partial_\omega\tilde\psi_n(\omega_0)\right] = - \psi_n.
\end{align}
We claim that there exists a sequence $\{b_n\}_n$ in $H^2(\R)$, converging to $0$ as $n \to \infty$, with
\begin{align} \label{exp_omega_der_psi_n}
    \partial_\omega \tilde\psi_n(\omega_0) = b_n + \sum_{j=1}^M \theta_j \partial_\omega \tilde\psi(\omega_0)(\cdot-jnT).
\end{align}
Inserting the ansatz for $\partial_\omega \tilde\psi_n(\omega_0)$ into the linearization $L_{+,\mu}(\psi_n)$ and using~\eqref{eq:omega_derivative_multipulse}, we obtain
\begin{align*}
    &L_{+,\mu}(\psi_n) \left(b_n + \sum_{j=1}^M \theta_j \partial_\omega \tilde\psi(\omega_0)(\cdot-jnT)\right)
\\
    &\quad= L_{+,\mu}(\psi_n) b_n- \sum_{j=1}^M \theta_j \psi_0(\cdot-jnT) 
     +\sum_{j=1}^M  \theta_j \left(L_{+,\mu}(\psi_n) - L_{+,\mu}(\psi_0(\cdot - jnT)) \right) \partial_\omega \tilde\psi(\omega_0)(\cdot-jnT).
\end{align*}
Thus, we infer from~\eqref{form_multipulse_GP} and~\eqref{eq:omega_derivative_multipulse} that the correction $b_n$ has to solve the equation
\begin{align*}
    L_{+,\mu}(\psi_n) b_n &=-a_n-\sum_{j=1}^M  \theta_j \left(L_{+,\mu}(\psi_n) - L_{+,\mu}(\psi_0(\cdot - jnT)) \right) \partial_\omega \tilde\psi(\omega_0)(\cdot-jnT) .
\end{align*}
Using that $L_{+,\mu}(\psi_0)$ is invertible, we deduce from Lemma~\ref{lem:invertibility_multifront} that, provided $n \in \N$ is sufficiently large, $L_{+,\mu}(\psi_n)$ is also invertible and there exists an $n$-independent constant $C > 0$ such that
$$
    \left\|L_{+,\mu}(\psi_n)^{-1} \psi\right\|_{L^2} \leq C\|\psi\|_{L^2}
$$
for $\psi \in L^2(\R)$. In particular,
$$
    b_n = -L_{+,\mu}(\psi_n)^{-1} \left(a_n+\sum_{j=1}^M  \theta_j \left(L_{+,\mu}(\psi_n) - L_{+,\mu}(\psi_0(\cdot - jnT)) \right) \partial_\omega \tilde\psi(\omega_0)(\cdot-jnT)  \right)
$$
satisfies~\eqref{exp_omega_der_psi_n} and obeys
\begin{align*}
    &\|b_n\|_{L^2} \leq C \left(\| a_n\|_{L^2} + 3\sum_{j=1}^M
    \left\|\left(\psi_n^2 - \psi_0(\cdot - jnT)^2 \right) \partial_\omega \tilde\psi(\omega_0)(\cdot-jnT)\right\|_{L^2}\right)\to 0
\end{align*}
as $n \to \infty$ by Corollary~\ref{cor:GP_existence} and the continuous embedding $H^1(\R) \hookrightarrow L^\infty(\R)$. Therefore, using~\eqref{form_multipulse_GP} and~\eqref{exp_omega_der_psi_n}, we arrive at
\begin{align*}
    \left\langle \partial_\omega \tilde\psi_n(\omega_0) , \psi_n \right\rangle_{L^2} &= 
    \sum_{j,k=1}^M \left\langle \theta_j\partial_\omega \tilde\psi(\omega_0)(\cdot-jnT) , \theta_k\psi_0(\cdot-knT) \right\rangle_{L^2} +
    \left\langle b_n , a_n \right\rangle_{L^2}\\ 
    &\qquad + \, \sum_{j=1}^M \left\langle \theta_j\partial_\omega \tilde\psi(\omega_0)(\cdot-jnT) , a_n \right\rangle_{L^2} + \sum_{k=1}^M \left\langle b_n , \theta_k\psi_0(\cdot-knT) \right\rangle_{L^2} \\
    &\to  M\left\langle \partial_\omega \tilde\psi(\omega_0) , \psi_0 \right\rangle_{L^2}
\end{align*}
as $n \to \infty$.
\end{proof}

By combining Krein index counting theory with Theorem~\ref{thm:instability_multifront}, we establish spectral instability conditions for the multipulse solutions $\psi_n$, obtained in Corollary~\ref{cor:GP_existence}.  

\begin{Theorem}\label{thm:instability_GP_multipulses}
Let $M \in \N$ with $M \geq 2$. Take $\theta \in \{\pm 1\}^M$. Assume~{\upshape \ref{assGP}}. Suppose $\omega_0$ is larger than the spectral bound of the operator $\partial_x^2 - \mu V$ acting on $L^2(\R)$. Then, there exists $N \in \mathbb{N}$ such that for all $n \in \N$ with $n \geq N$ the following statements holds.
\begin{enumerate}
    \item[1.] If $\langle\partial_\omega \tilde\psi(\omega_0), \psi_0 \rangle_{L^2}$ is negative, then the $M$-pulse $\psi_n$, obtained in Corollary~\ref{cor:GP_existence}, is spectrally unstable.
    \item[2.] If $\langle\partial_\omega \tilde\psi(\omega_0), \psi_0 \rangle_{L^2}$ is positive, then the $M$-pulse $\psi_n$, obtained in Corollary~\ref{cor:GP_existence}, is spectrally unstable in each of the following cases:
    \begin{itemize}
        \item[(i)] $\psi_n$ has no zeros;
        \item[(ii)] There exist $m, \ell \in \N$ such that $M = 2m$ and $\psi_n$ has $2\ell$ zeros;
        \item[(iii)] $\psi_n$ is odd, $V$ is even, and there exist $m \in \N$ and $\ell\in \N_0$ such that $m + \ell$ is even, $M = 2m$, and $\psi_n$ has $2\ell+1$ zeros;
        \item[(iv)] $\psi_n$ and $V$ are even, and there exist $m, \ell \in \N$ such that $m+\ell$ is odd, $M = 2m+1$, and $\psi_n$ has $2\ell$ zeros;
        \item[(v)] There exist $m\in\N$ and $\ell \in \N_0$ such that $M = 2m+1$ and $\psi_n$ has $2\ell+1$ zeros.
    \end{itemize}
\end{enumerate}
\end{Theorem}

\begin{proof}
In order to prove the first assertion, we observe that Theorem~\ref{thm:stability_periodic_GP} yields that $\psi_0$ is spectrally unstable. Therefore, provided $n \in \N$ is sufficiently large, $\psi_n$ is also spectrally unstable by Proposition~\ref{prop:essential_spec1front} and Theorem~\ref{thm:instability_multifront}.

We proceed with proving the second assertion. For notational convenience, we denote by $n(L) \in \N_0$ the total algebraic multiplicity of all negative eigenvalues of an operator $L \colon D(L) \subset L^2(\R) \to L^2(\R)$ and by $z(L) \in \N_0$ the dimension of its kernel. Moreover, we denote by $k_r(L)$ the number of real unstable eigenvalues of $L$, by $k_c(L)$ the number of quadruplets of eigenvalues with non-vanishing real and imaginary parts, and by $k_i^-(L)$ the number of purely imaginary eigenvalues of $L$ with negative Krein signature (all counting algebraic multiplicities).

To prove instability in case (i), we notice, using similar arguments as in the proof of Theorem~\ref{thm:stability_periodic_GP}, that $n(L_{+,\mu}(\psi_0)) = 1$ and the nondegeneracy of $\psi_0$ imply that $n(L_{+,\mu}(\psi_n)) = M \geq 2$ and $z(L_{+,\mu}(\psi_n)) = 0$. Thus,~\cite[Theorem~4.8]{Pelinovsky2011} applies, which yields spectral instability. 

Next, we consider case~(ii). As in case (i), we find $n(L_{+,\mu}(\psi_n)) = 2m$ and $z(L_{+,\mu}(\psi_n)) = 0$. Moreover, combining $L_{-,\mu}(\psi_n)\psi_n = 0$ with Sturm-Liouville theory, cf.~\cite[Theorem~6.3.1.(8)(3)]{Zettl2021Recent}, we obtain $n(L_{-,\mu}(\psi_n)) = 2\ell$ as well as $z(L_{-,\mu}(\psi_n)) = 1$. Using Lemma~\ref{lem:convergence_multipulse}, we find that the Krein index formula in~\cite[Theorem~1]{AddendumKapitulaKevrekidis2004} yields
$$
    k_r(L_{\mu}(\psi_n)) + 2k_c(L_{\mu}(\psi_n)) + 2 k_i^-(L_{\mu}(\psi_n)) = 2 (m+\ell)-1
$$
Hence, we have $k_r(L_{\mu}(\psi_n)) \geq 1$, which means that $\psi_n$ is spectrally unstable.

We proceed with case (iii). Using similar arguments as before, we obtain $n(L_{+,\mu}(\psi_n)) = 2m$, $z(L_{+,\mu}(\psi_n)) = 0$ and $n(L_{-,\mu}(\psi_n)) = 2\ell+1$, and find that the odd function $\psi_n$ spans the kernel of $L_{-,\mu}(\psi_n)$. Since $\psi_n$ is odd and $V$ is even, the operators $L_{\pm,\mu}(\psi_n)$ and $L_\mu(\psi_n)$ leave the space of even functions invariant. Therefore, we can apply the index formula to these linear operators restricted to even functions, yielding 
\begin{align*}
 k_r(L_{\mu}(\psi_n)|_{\mathrm{even}}) + 2k_c(L_{\mu}(\psi_n)|_{\mathrm{even}}) + 2 k_i^-(L_{\mu}(\psi_n)|_{\mathrm{even}}) &= n(L_{+,\mu}(\psi_n)|_\mathrm{even}) + n(L_{-,\mu}(\psi_n)|_\mathrm{even})\\
 &= m + \ell +1,
\end{align*}
where we used that by~\cite[Theorem~6.3.1.(8)(3)]{Zettl2021Recent} the eigenfunctions of the Sturm-Liouville operators $L_{\pm,\mu}(\psi_n)$ are alternating between even and odd functions and that the principal eigenfunction is even. Since $m+\ell+1$ is odd, the number of real unstable eigenvalues of $L_\mu(\psi_n)$ is greater than $1$, yielding spectral instability.

We turn to case~(iv). Arguing as before, we have $n(L_{+,\mu}(\psi_n)|_\mathrm{even}) = m+1$, $z(L_{+,\mu}(\psi_n)|_\mathrm{even}) = 0$, $n(L_{-,\mu}(\psi_n)|_\mathrm{even}) = \ell$, and the even function $\psi_n$ spans the kernel of $L_{-,\mu}(\psi_n)|_\mathrm{even}$. Using Lemma~\ref{lem:convergence_multipulse} and applying the index formula, this amounts to
$$
    k_r(L_{\mu}(\psi_n)|_{\mathrm{even}}) + 2k_c(L_{\mu}(\psi_n)|_{\mathrm{even}}) + 2 k_i^-(L_{\mu}(\psi_n)|_{\mathrm{even}}) = m+\ell+1-1.
$$
Using that $m+\ell$ is odd, we infer $k_r(L_{\mu}(\psi_n)) \geq 1$, which implies spectral instability.

Finally, the proof of spectral instability in case (v) follows as in case~(ii).
\end{proof}

Although the first instability condition in Theorem~\ref{thm:instability_GP_multipulses} is well-known, cf.~\cite[Theorem~4.8]{Pelinovsky2011} and~\cite[Theorem~6.5]{Grillakis_Stability_1987}, the authors are not aware that the other instability conditions in Theorem~\ref{thm:instability_GP_multipulses} have been derived in the literature before.

\appendix

\section{Projections} \label{app:projections}

In this appendix, we prove some auxiliary results on finite-dimensional projections, which correspond to block matrices $P \in \C^{n \times n}$ satisfying $P^2 = P$.

The first result asserts that, if two projections $P,Q \in \C^{n \times n}$ are close in norm, then their range and kernel are complementary subspaces and the associated projection onto the range of $P$ along the kernel of $Q$ can be bounded in terms of $P$ and $Q$.

\begin{Lemma} \label{l:projest2}
Let $n \in \mathbb{N}$. Let $P,Q \in \C^{n \times n}$ be projections with
\begin{align*}\|P-Q\| < 1.\end{align*}
Then, $\mathrm{ran}(P)$ and $\ker(Q)$ are complementary subspaces and the projection $R$ onto $\mathrm{ran}(P)$ along $\ker(Q)$ obeys the bound
\begin{align} \label{projest}
    \|R\| \leq \frac{\|P\|}{1-\|P-Q\|}.
\end{align}
\end{Lemma}
\begin{proof}
Let $v \in \mathrm{ran}(P) \cap \ker(Q)$. Then,
\begin{align*} \|v\| = \|(P - Q)v\| \leq \|P-Q\|\|v\|\end{align*}
implies $v = 0$ as $\|P - Q\| < 1$. Hence, we infer $\mathrm{ran}(P) \cap \ker(Q) = \{0\}$ and, similarly, $\mathrm{ran}(Q) \cap \ker(P) = \{0\}$. So, we must have $\mathrm{rank}(Q) + \dim \ker(P) \leq n$ and, thus, we arrive at 
\begin{align*}
    \mathrm{rank}(Q) - \mathrm{rank}(P) = \mathrm{rank}(Q) + \dim \ker(P) - \dim \ker(P) - \mathrm{rank}(P) \leq n - n = 0.
\end{align*}
Similarly, $\mathrm{rank}(P) + \dim \ker(Q) \leq n$ yields $\mathrm{rank}(Q) - \mathrm{rank}(P) \geq 0$. We find 
\begin{align*}
    \dim \ker(Q) + \mathrm{rank}(P) = \dim \ker(Q) + \mathrm{rank}(Q) = n
\end{align*}
and conclude that $\mathrm{ran}(P)$ and $\ker(Q)$ are complementary subspaces. Now let $R$ be the projection onto $\mathrm{ran}(P)$ along $\ker(Q)$. Since we have $RP = P$ and $RQ = R$, it holds
\begin{align*} R = RQ - RP + P,\end{align*}
which readily yields the estimate~\eqref{projest}.
\end{proof}

Next, we show that, if two projections have the same kernel and there exist bases of their ranges which are close in norm, then the projections themselves are close in norm. 
\begin{Lemma} \label{l:projest3}
    Let $n \in \mathbb{N}$ and $k \in \{1,\ldots,n\}$. Let $A, B, C \in \C^{n \times k}$. Suppose $C^* A \in \C^{k \times k}$ is invertible and we have
\begin{align} \label{neumann}
\left\|A-B\right\| < \frac{1}{\left\|(C^* A)^{-1}\right\| \left\|C\right\|}.
\end{align}
Then, the projections  
\begin{align} \label{projform1}
P =  A(C^* A)^{-1} C^*
\end{align}
onto $\mathrm{ran}(A)$ along $\mathrm{ran}(C)^\perp$ and 
\begin{align} \label{projform2}
Q = B(C^* B)^{-1} C^*
\end{align}
onto $\mathrm{ran}(B)$ along $\mathrm{ran}(C)^\perp$ are well-defined and satisfy
\begin{align} \label{estprojcore2}
\begin{split}
\|P-Q\| &\leq \|A-B\| \left\|(C^* A)^{-1}\right\| \left\|C\right\|\left(1 + \frac{\left(\|A\| + \|A-B\|\right) \left\|(C^* A)^{-1}\right\|\|C\| }{1-\left\|(C^* A)^{-1}\right\| \left\|C\right\|\|A-B\|}\right), \\
\|P\| &\leq \|A\|\|\left\|(C^* A)^{-1}\right\| \left\|C\right\|.
\end{split}
\end{align}
\end{Lemma}
\begin{proof}
Since $C^* A$ is invertible, $\mathrm{ran}(A)$ and $\mathrm{ran}(C)^\perp$ are complementary subspaces and the projection $P$ given by~\eqref{projform1} is well-defined. Upon rewriting
\begin{align*}
C^* B = C^* A \left(I - (C^* A)^{-1} C^* (A - B)\right)
\end{align*}
and recalling~\eqref{neumann}, expansion as a Neumann series yields that $C^* B$ is invertible with
\begin{align} \label{estprojcore}
\left\|(C^* B)^{-1} - (C^* A)^{-1}\right\| \leq \frac{\left\|(C^* A)^{-1}\right\|^2 \|C\| \|A-B\| }{1 - \left\|(C^* A)^{-1}\right\| \|C\| \|A-B\|}.
\end{align}
Hence, the subspaces $\mathrm{ran}(B)$ and $\mathrm{ran}(C)^\perp$ are complementary and the projection $Q$ given by~\eqref{projform2} is well-defined. Finally, writing
\begin{align*}
P - Q = (A-B)(C^* A)^{-1} C^* + (A + B - A)\left( (C^* A)^{-1} - (C^* B)^{-1}\right) C^*
\end{align*}
and applying~\eqref{estprojcore} yields~\eqref{estprojcore2}.
\end{proof}

The following result shows that a projection matrix depends analytically on a parameter $\lambda \in \C$ if and only if its range and the real orthogonal complement of its kernel are analytic subspaces. 

\begin{Lemma} \label{l:proj_ana}
Let $n \in \N$ and $k \in \{1,\ldots,n\}$. Let $\Omega \subset \C$ be open. Let $V(\lambda), U(\lambda) \subset \C^n$ be complementary subspaces with $\dim(V(\lambda)) = k$ for each $\lambda \in \Omega$. Denote by $P(\lambda) \in \C^{n \times n}$ the projection onto $V(\lambda)$ along $U(\lambda)$. Then, the map $P \colon \Omega \to \C^{n \times n}$ is analytic if and only if there exist analytic maps $B_1,B_2 \colon \Omega \to \C^{n \times k}$ such that 
\begin{align} \label{basisids}
V(\lambda) = \mathrm{ran}(B_1(\lambda)), \qquad U(\lambda) = \left\{u \in \C^n : z^\top u = 0 \text{ for all } z \in \mathrm{ran}(B_2(\lambda))\right\}
\end{align}
for all $\lambda \in \Omega$. 
\end{Lemma}

The proof of Lemma~\ref{l:proj_ana} is partly based on~\cite[Lemma~3.3]{Beynetal2014} and requires the following technical lemma.

\begin{Lemma} \label{l:projlem}
Let $n \in \N$ and $k \in \{1,\ldots,n\}$. Let $P \in \C^{n \times n}$ be a projection of rank $k$. If $B_1 \in \C^{n \times k}$ is a basis of $\mathrm{ran}(P)$ and $B_2 \in \C^{n \times k}$ is a basis of $\mathrm{ran}(P^\top)$, then $B_2^\top B_1 \in \C^{k \times k}$ is invertible.
\end{Lemma}
\begin{proof}
Assume that $v \in \C^k$ satisfies $B_2^\top B_1 v = 0$. Then, $u := B_1 v \in \mathrm{ran}(P)$ satisfies $z^\top u = 0$ for all $z \in \mathrm{ran}(B_2) = \mathrm{ran}(P^\top)$. Hence, for all $w \in \C^n$ it holds $0 = (P^\top w)^\top u = w^\top P u$  implying $Pu = 0$. We conclude that $u \in \ker(P) \cap \mathrm{ran}(P) = \{0\}$. Since $B_1$ has full rank, we have $v = 0$. 
\end{proof}

\begin{proof}[Proof of Lemma~\ref{l:proj_ana}]
Assume that $P$ is analytic. Then, by~\cite[Section~II.4.2]{Kato1995} there exist analytic maps $B_1,B_2 \colon \Omega \to \C^{n \times k}$ such that $B_1(\lambda)$ is a basis of $V(\lambda)$ and $B_2(\lambda)$ is a basis of $\mathrm{ran}(P(\lambda)^\top)$ for all $\lambda \in \Omega$. Now, it follows
\begin{align*}
u \in \ker(P(\lambda)) = U(\lambda) &\qquad \Leftrightarrow \qquad \forall \, v \in \C^n : 0 = v^\top P(\lambda) u = \left(P(\lambda)^\top v\right)^\top u \\ &\qquad \Leftrightarrow \qquad \forall \, z \in \mathrm{ran}(B_2(\lambda)) : z^\top u = 0
\end{align*}
for all $\lambda \in \Omega$.

Conversely, assume that there exist analytic maps $B_1,B_2 \colon \Omega \to \C^{n \times k}$ such that~\eqref{basisids} holds for all $\lambda \in \Omega$. Let $\lambda \in \Omega$. We have $z \in \mathrm{ran}(B_2(\lambda))$ if and only if $z^\top u = 0$ for all $u \in U(\lambda) = \ker(P(\lambda))$. If $x \in \mathrm{ran}(P(\lambda)^\top)$ we find $v \in \C^n$ such that $x = P(\lambda)^\top v$. This yields $x^\top u = (P(\lambda)^\top v)^\top u = v^\top P(\lambda) u = 0$ for all $u \in U(\lambda)$, which proves the inclusion $\mathrm{ran}(B_2(\lambda)) \supset \mathrm{ran}(P(\lambda)^\top)$. Equality then follows from
\begin{align*}
    \dim \mathrm{ran}(B_2(\lambda)) &= n - \dim \ker(P(\lambda)) = \dim \ker(P(\lambda))^\perp \\
    &= \dim \mathrm{ran}(P(\lambda)^*) = \dim \mathrm{ran} (P(\lambda)^\top).
\end{align*}
Thus, by Lemma~\ref{l:projlem} the matrix $B_2(\lambda)^\top B_1(\lambda) \in \C^{k \times k}$ is invertible. Clearly, 
\begin{align} \label{anaproj}
 \Pi(\lambda) = B_1(\lambda) \left(B_2(\lambda)^\top B_1(\lambda)\right)^{-1} B_2(\lambda)^\top \in \C^{n \times n}
\end{align}
is a projection with $\ker(\Pi(\lambda)) = U(\lambda)$ and $\mathrm{ran}(\Pi(\lambda)) = V(\lambda)$. So, we must have $P(\lambda) = \Pi(\lambda)$. The formula~\eqref{anaproj} readily yields that $P$ is analytic.
\end{proof}

\section{Exponential dichotomies} \label{app:exp_dich}

Exponential dichotomies are powerful tools in the spectral analysis of linear differential operators. By reformulating the associated eigenvalue problem as a first-order nonautonomous system, they can be used to characterize invertibility as well as Fredholm properties~\cite{MasseraSchaefer1966,Palmer1984,Palmer1988}. 

A linear (nonautonomous) ordinary differential equation possesses an exponential dichotomy if it admits a fundamental set of solutions that decay exponentially in either forward or backward time.

\begin{Definition}
Let $n \in \mathbb{N}$, $\mathcal{J} \subset \R$ an interval and $A \in C(\mathcal{J},\C^{n \times n})$. Denote by $T(x,y)$ the evolution operator of the linear system
\begin{align} \phi' = A(x)\phi. \label{linsys}\end{align}
Equation~\eqref{linsys} has \emph{an exponential dichotomy on $\mathcal{J}$ with constants $K,\mu > 0$ and projections $P(x) \in \C^{n \times n}$} if for all $x,y \in \mathcal{J}$ it holds
\begin{itemize}
 \item $P(x)T(x,y) = T(x,y) P(y)$;
 \item $\|T(x,y)P(y)\| \leq K\eu^{-\mu(x-y)}$ for $x \geq y$;
 \item $\|T(x,y)(I-P(y))\| \leq K\eu^{-\mu(y-x)}$ for $y \geq x$.
\end{itemize}
\end{Definition}

In this appendix, we establish several results on exponential dichotomies that are relevant to our analysis. For a comprehensive introduction, we refer the reader to~\cite{Coppel1978}.

We start with an extension of the so-called ``pasting lemma'', which was first established in~\cite[Lemma~B.7]{RijkDoelmanRademacher2016}. The pasting lemma provides a tool for gluing together exponential dichotomies on two adjacent intervals.

\begin{Lemma} \label{l:pastingexpdi}
Let $n \in \mathbb{N}$, $a,b,c \in \R$ with $a < b < c$ and $A \in C([a,c],\C^{n \times n})$. Suppose that equation~\eqref{linsys} has exponential dichotomies on both $[a,b]$ and $[b,c]$ with constants $K, \mu > 0$ and projections $P_1(x), x \in [a,b]$ and $P_2(x), x \in [b,c]$, respectively. If $\ker(P_1(b))$ and $\mathrm{ran}(P_2(b))$ are complementary subspaces, then~\eqref{linsys} has an exponential dichotomy on $[a,c]$ with constants $C,\mu > 0$ and projections $\mathcal{Q}(x) = T(x,b) Q T(b,x), x \in [a,c]$, where $Q$ is the projection onto $\mathrm{ran}(P_2(b))$ along $\ker(P_1(b))$ and $T(x,y)$ denotes the evolution of system~\eqref{linsys}. Moreover, we have $C = K + K^2 \|Q\| + K^3$ and
\begin{align*}
\|P_1(a) - \mathcal{Q}(a)\| \leq CK \eu^{-2\mu (b-a)}, \qquad \|P_2(c) - \mathcal{Q}(c)\| \leq  CK \eu^{-2\mu (c-b)}.
\end{align*}
\end{Lemma}
\begin{proof}
Observe that $\mathcal{Q}(x)T(x,y) = T(x,y) \mathcal{Q}(y)$ for $x,y \in [a,c]$. The exposition in~\cite[pp.~16-17]{Coppel1978} shows that~\eqref{linsys} possesses an exponential dichotomy on the intervals $[a,b]$ and $[b,c]$ with constants $C,\mu > 0$ and projections $\mathcal{Q}(x)$. We need to show that the dichotomy estimates persist on the union $[a,c] = [a,b] \cup [b,c]$. Take $x \in [b,c]$ and $y \in [a,b]$. We estimate
\begin{align*} \|T(x,y)\mathcal{Q}(y)\| &\leq \|T(x,b)P_2(b)\| \|Q\| \|P_1(b)T(b,y)\| \leq K^2\|Q\|\eu^{-\mu(x-y)} \leq C \eu^{-\mu(x-y)},\end{align*}
where we use $P_2(b)Q = Q$ and $QP_1(b)=Q$. Similarly, one estimates $\|T(y,x)(I-\mathcal{Q}(x))\| \leq C\eu^{-\mu(x-y)}$ for $x \in [b,c]$ and $y \in [a,b]$. Finally, using $QP_1(b) = Q$ again, we infer
\begin{align*}
\|P_1(a) - \mathcal{Q}(a)\|
&\leq \left\|T(a,b)(I - \mathcal{Q}(b))\right\| \left\|P_1(b) T(b,a)\right\| \leq CK \eu^{-2\mu (b-a)}.
\end{align*}
Similarly, we derive $\|P_2(c) - \mathcal{Q}(c)\| \leq CK \eu^{-2\mu (c-b)}$.
\end{proof}

Next, we obtain an approximation result for the projections of two exponential dichotomies defined on the same interval.

\begin{Lemma}\label{l:projest}
Let $n \in \mathbb{N}$, $a,b \in \R$ with $a < b$ and $A \in C([a,b],\C^{n \times n})$. Suppose equation~\eqref{linsys} admits two exponential dichotomies on $[a,b]$ with constants $K_{1,2}, \mu_{1,2} > 0$ and projections $P_{1,2}(x)$. Then, we have
\begin{align*}
\left\|P_1(x) - P_2(x)\right\| \leq K_1K_2 \left(\eu^{-(\mu_1+\mu_2)(x-a)} + \eu^{-(\mu_1+\mu_2)(b-x)}\right),
\end{align*}
for all $x \in [a,b]$.
\end{Lemma}
\begin{proof}
Let $T(x,y)$ be the evolution of system~\eqref{linsys}. We estimate
\begin{align*}
\left\|P_1(x) - P_2(x)\right\| &\leq \left\|P_1(x)(I - P_2(x))\right\|  + \left\|(I - P_1(x))P_2(x)\right\|\\ 
&\leq \left\|P_1(x)T(x,a)\right\|\left\|T(a,x)(I - P_2(x))\right\| + \left\|(I - P_1(x))T(x,b)\right\|\left\|T(b,x)P_2(x)\right\| \\
&\leq K_1K_2 \left(\eu^{-(\mu_1+\mu_2)(x-a)} + \eu^{-(\mu_1+\mu_2)(b-x)}\right),
\end{align*}
for all $x \in [a,b]$.
\end{proof}

The following result establishes periodicity of the evolution operator, as well as the dichotomy projections, when the underlying system has periodic coefficients.

\begin{Lemma} \label{l:projper}
Let $n \in \mathbb{N}$ and $L > 0$. Let $A \in C(\R,\C^{n \times n})$ be $L$-periodic. Then, the evolution $T(x,y)$ of~\eqref{linsys} satisfies $T(x,y) = T(x-L,y-L)$ for each $x,y \in \R$. Moreover, if~\eqref{linsys} has an exponential dichotomy on $\R$ with projections $P(x)$, then $P(\cdot)$ is also $L$-periodic.
\end{Lemma}
\begin{proof}
By Floquet's theorem, cf.~\cite[Theorem~2.1.27]{KapitulaPromislow2013}, there exist an $L$-periodic function $Q \colon \R \to \C^{n \times n}$ and a matrix $B \in \C^{n \times n}$ such that $Q(x)$ is invertible for each $x \in \R$ and we have $T(x,y) = Q(x) \eu^{B(x-y)} Q(y)^{-1}$ for each $x,y \in \R$. Hence, we arrive at $T(x,y) = T(x-L,y-L)$ for each $x,y \in \R$. Now suppose that~\eqref{linsys} has an exponential dichotomy on $\R$. Then, the matrix $B$ must be hyperbolic. Let $\mathcal{P}$ be the spectral projection of $B$ onto its stable space. Then, by uniqueness of the exponential dichotomy, cf.~\cite[p.~19]{Coppel1978}, it holds $P(x) = Q(x) \mathcal{P} Q(x)^{-1}$ for $x \in \R$ and, thus, $P$ is $L$-periodic.
\end{proof}

We proceed by showing that the operator $\frac{d}{dx}-A(x)$ is invertible as a map from $H^1(\R)$ to $L^2(\R)$, and, in case of $L$-periodic coefficients of $A$, also as a map from $H_\per^1(0,L)$ to $L_\per^2(0,L)$, provided that~\eqref{linsys} has an exponential dichotomy on $\R$. Moreover, we prove that the inverse operator can be bounded in terms of the dichotomy constants and the supremum norm of $A$. 

\begin{Lemma} \label{l:inhomexpdi}
Let $n \in \mathbb{N}$ and $L > 0$. Let both $g \in C(\R,\C^n)$ and $A \in C(\R,\C^{n \times n})$ be bounded. Suppose that~\eqref{linsys} has an exponential dichotomy on $\R$ with constants $K,\mu > 0$. Then, the inhomogeneous problem
\begin{align} \label{inhomsys}
\phi' = A(x)\phi + g(x),
\end{align}
possesses a unique bounded solution $\phi \in C^1(\R)$. If $g \in L^2(\R)$, then $\phi$ lies in $H^1(\R)$ and obeys the estimate
\begin{align} \label{inhombound}
\|\phi\|_{H^1} \leq \left(\left(1 + \|A\|_{L^\infty}\right) \frac{2K}{\mu} + 1\right)\|g\|_{L^2}.
\end{align}
Moreover, if $A$ is $L$-periodic and $g \in L_\per^2(0,L)$, then $\phi$ lies in $H_\per^1(0,L)$ and satisfies 
\begin{align} \label{inhombound2}
\|\phi\|_{H_\per^1(0,L)} \leq \left(\left(1 + \|A\|_{L^\infty}\right) \frac{2K}{\mu} + 1\right)\|g\|_{L_\per^2(0,L)}.
\end{align}
\end{Lemma}
\begin{proof}
Let $T(x,y)$ be the evolution operator of~\eqref{linsys}. Denote by $P(x)$ the projections associated with the exponential dichotomy of~\eqref{linsys} on $\R$. Define $\phi \colon \R \to \C^n$ by
\begin{align*}
\phi(x) = \int_{-\infty}^x T(x,y) P(y) g(y) \de y - \int_x^\infty T(x,y) (I-P(y)) g(y) \de y.
\end{align*}
We note that the properties of the exponential dichotomy, the fundamental theorem of calculus and the fact that $g$ and $A$ are bounded and continuous, readily yield that $\phi$ is well-defined, bounded, continuously differentiable and solves~\eqref{inhomsys}. Due to the exponential dichotomy of system~\eqref{linsys} on $\R$ its only bounded solution is $0$. Therefore, the bounded solution $\phi$ of~\eqref{inhomsys} is unique. Finally, we estimate
\begin{align} \label{pointwisephi}
\|\phi(x)\| \leq K\int_{\R} \eu^{-\mu|x-y|} \|g(y)\| \de y = K\int_{\R} \eu^{-\mu|y|} \|g(x-y)\| \de y,
\end{align}
for $x \in \R$. 

Take $g \in L^2(\R)$. Applying Young's convolution inequality to~\eqref{pointwisephi} leads to the estimate
\begin{align*}
\|\phi\|_{L^2} \leq \frac{2K}{\mu} \|g\|_{L^2},\end{align*}
which implies $\phi \in L^2(\R)$. So, using the fact that $\phi$ solves~\eqref{inhomsys}, we establish
\begin{align*}
\|\phi'\|_{L^2} \leq \|A\|_{L^\infty} \|\phi\|_{L^2} + \|g\|_{L^2}
\leq \left(\|A\|_{L^\infty} \frac{2K}{\mu} + 1\right)\|g\|_{L^2},
\end{align*}
which proves $\phi \in H^1(\R)$ and establishes~\eqref{inhombound}. 

Suppose $A$ is $L$-periodic and $g \in L_\per^2(0,L)$. Then, by Lemma~\ref{l:projper} we deduce
\begin{align*}
\phi(x + L) &= \int_{-\infty}^{x+L} T(x+L,y) P(y) g(y) \de y - \int_{x+L}^\infty T(x+L,y) (I-P(y)) g(y) \de y\\
&= \int_{-\infty}^{x+L} T(x,y-L) P(y-L) g(y-L) \de y\\ 
&\qquad - \, \int_{x+L}^\infty T(x,y-L) (I-P(y-L)) g(y-L) \de y = \phi(x)
\end{align*}
for $x \in \R$. Hence, $\phi$ is also $L$-periodic. Using~\eqref{pointwisephi} and H\"older's inequality we estimate
\begin{align*}
\|\phi\|_{L_\per^2(0,L)}^2 &= \int_0^L \|\phi(x)\|^2 \de x \leq K^2\int_\R \int_\R \eu^{-\mu(|y| + |z|)} \int_0^L \|g(x-y)\|\|g(x-z)\| \de x \de y \de z\\ 
&\leq K^2\int_{\R} \int_{\R} \eu^{-\mu(|y| + |z|)} \|g\|^2_{L_\per^2(0,L)} \de y \de z \leq \frac{4K^2}{\mu^2} \|g\|^2_{L_\per^2(0,L)},
\end{align*}
which proves $\phi \in L_\per^2(0,L)$. Combining the later with the fact that $\phi$ solves~\eqref{inhomsys}, we obtain
\begin{align*}
\|\phi'\|_{L_\per^2(0,L)} \leq \|A\|_{L^\infty} \|\phi\|_{L_\per^2(0,L)} + \|g\|_{L_\per^2(0,L)}
\leq \left(\|A\|_{L^\infty} \frac{2K}{\mu} + 1\right)\|g\|_{L_\per^2(0,L)},
\end{align*}
which implies $\phi \in H_\per^1(0,L)$ and establishes~\eqref{inhombound2}. 
\end{proof}

Our next step is to prove that, if the linear differential operator $\El(\unu) - \lambda$, defined in~\S\ref{sec:Abstract_ex_stab}, is invertible, then the first-order formulation~\eqref{first-order_formulation} of the associated eigenvalue problem has an exponential dichotomy on $\R$. In conjunction with lemma~\ref{l:inhomexpdi}, this characterizes invertibility of the differential operator $\El(\unu)-\lambda$ in terms of exponential dichotomies.

\begin{Lemma} \label{lem:expdi1}
Let $\unu \in C(\R)$ be bounded and $\lambda \in \C$. If the linear operator $\El(\unu) - \lambda$ is invertible, then system~\eqref{first-order_formulation} has an exponential dichotomy on $\R$. 
\end{Lemma}
\begin{proof}
By assumption there exists for each $g \in L^2(\R)$ a unique solution $u \in H^k(\R)$ of the resolvent problem 
\begin{align} \label{resolventprob1pulse}
\left(\El(\unu) - \lambda\right) u = g.
\end{align}
This implies that for each $\psi \in H^{k-1}(\R)$ the inhomogeneous problem
\begin{align} \label{inhomsys1pulse}
U' = \A\left(x,\unu(x);\lambda \right) U + \psi
\end{align}
possesses a solution $U \in H^1(\R)$, which is given by $U = (u,\partial_x u - \psi_1,\ldots,\partial_x^{k-1} u- \sum_{i=1}^{k-1} \partial_x^{k-1-i} \psi_i)^\top$, where $u \in H^k(\R)$ is the solution of the resolvent problem~\eqref{resolventprob1pulse} with inhomogeneity 
$$g = \sum_{i = 1}^k \alpha_i \sum_{j = 1}^i \partial_x^{i-j} \psi_j \in L^2(\R).$$ 

Let $\mathcal{J} = (-\infty,0]$ or $\mathcal{J} = [0,\infty)$. For each $\psi \in H^{k-1}(\mathcal{J})$ we find a solution $U \in H^1(\mathcal{J})$ of~\eqref{inhomsys1pulse}. In the language of Massera and Sch\"affer, the pair $(H^{k-1}(\mathcal{J}),H^1(\mathcal{J}))$ is regularly admissible for equation~\eqref{inhomsys1pulse}, cf.~\cite[\S51]{MasseraSchaefer1966}. Moreover, in the ordered lattice $\mathrm{b}\mathcal{N}(\mathcal{J})$ of all Banach spaces, which are stronger than the Bochner space $L(\mathcal{J})$ of strongly measurable, locally Bochner integrable functions $h \colon \mathcal{J} \to \C^{km}$ endowed with the topology of convergence in mean, the space $H^{k-1}(\mathcal{J})$ is not weaker than $L^1(\mathcal{J})$, and $H^1(\mathcal{J})$ is not stronger than $L^\infty_0(\mathcal{J})$, where $L^\infty_0(\mathcal{J})$ is the $L^\infty$-closure of all functions with compact (essential) support in $L^\infty(\mathcal{J})$ endowed with the supremum norm, see~\cite[\S20 and~\S21]{MasseraSchaefer1966}. That is, the pair $(H^{k-1}(\mathcal{J}),H^1(\mathcal{J}))$ is not weaker than $(L^1(\mathcal{J}),L_0^\infty(\mathcal{J}))$, cf.~\cite[\S50]{MasseraSchaefer1966}. Now,~\cite[Theorem~64.B]{MasseraSchaefer1966} yields that system~\eqref{first-order_formulation} has exponential dichotomies on both half-lines, $\mathcal{J} = (-\infty,0]$ and $\mathcal{J} = [0,\infty)$. We denote by $P_{\pm}(\pm x), x \geq 0$ the associated projections. 

We show that the exponential dichotomies for~\eqref{first-order_formulation} on both half-lines can be pasted together to yield an exponential dichotomy for~\eqref{first-order_formulation} on $\R$. First, we observe that it must hold $\ker(P_{-}(0)) \cap \mathrm{ran}(P_{+}(0)) = \{0\}$, since any solution $U \in H^1(\R)$ of~\eqref{first-order_formulation} with $U(0) \in \ker(P_{-}(0)) \cap \mathrm{ran}(P_{+}(0))$ is exponentially localized and, thus, generates an element $u = U_1 \in H^k(\R)$ lying in the kernel of $\El(\unu) - \lambda$, which is invertible. On the other hand, the adjoint problem
\begin{align} \label{homsys1pulsadjoint}
U' = -\A\left(x,\unu(x);\lambda\right)^* U
\end{align}
has the evolution $\Phi_{\mathrm{ad}}(x,y) = \Phi(y,x)^*$, where $\Phi(x,y)$ is the evolution of~\eqref{first-order_formulation}. Therefore,~\eqref{homsys1pulsadjoint} also possesses exponential dichotomies on both half lines with projections $I-P_{\pm}(\pm x)^*, x \geq 0$. Any solution $U \in H^1(\R)$ of~\eqref{homsys1pulsadjoint} with $U(0) \in \ker(I-P_{-}(0)^*) \cap \mathrm{ran}(I-P_{+}(0)^*) =\ker(P_{-}(0))^\perp \cap \mathrm{ran}(P_{+}(0))^\perp$ yields an element $u = U_k \in H^k(\R)$ in the kernel of the adjoint operator $(\El(\unu) - \lambda)^*$, which is invertible since $\El(\unu) - \lambda$ is. Hence, we have $\ker(P_{-}(0))^\perp \cap \mathrm{ran}(P_{+}(0))^\perp = \{0\}$, which, in combination with $\ker(P_{-}(0)) \cap \mathrm{ran}(P_{+}(0)) = \{0\}$, implies that $\ker(P_{-}(0))$ and $\mathrm{ran}(P_{+}(0))$ are complementary subspaces. Hence, Lemma~\ref{l:pastingexpdi} implies that~\eqref{first-order_formulation} admits an exponential dichotomy on $\R$. 
\end{proof}

An important property of exponential dichotomies, often referred to as roughness or robustness, is their persistence under small perturbations, cf.~\cite[Section~4]{Coppel1978}. Additionally, exponential dichotomy projections can be chosen to depend analytically on parameters if the underlying system does, see~\cite[Appendix~A]{DasLatushkin2011} and references therein. Leveraging the results from~\cite{Coppel1978,DasLatushkin2011}, we  show that, if a system, depending analytically on a complex parameter $\lambda$, admits an exponential dichotomy on a half line with $\lambda$-uniform constants, then the exponential dichotomy persists under perturbations and an analytic choice of projection is always possible.

\begin{Lemma} \label{l:projanarough}
Let $n \in \N$ and $k \in \{1,\ldots,n\}$. Let $\Omega \subset \C$ be open and let $\lambda_0 \in \Omega$. Let $A \colon [0,\infty) \times \Omega \to \C^{n \times n}$ be such that $A(\cdot;\lambda)$ is continuous for each $\lambda \in \Omega$ and $A(x;\cdot)$ is analytic for each $x \geq 0$. Suppose that there exist $K,\mu > 0$ such that for each $\lambda \in \Omega$ system
\begin{align} \label{linana}
\phi' = A(x;\lambda) \phi 
\end{align}
has an exponential dichotomy on $[0,\infty)$ with constants $K,\mu > 0$ and projections $P(x;\lambda)$ of rank $k$. Then, there exist constants $C,\delta_0, {\varrho_0},\theta > 0$ such that the closed disk $\overline{B}_{\lambda_0}(\varrho_0)$ lies in $\Omega$ and for all $\delta \in (0,\delta_0)$ and $B \in C([0,\infty),\C^{n \times n})$ with $\|B\|_{L^\infty} \leq \delta$ the perturbed system
\begin{align} \label{linanaper}
    \phi' = \left(A(x;\lambda) + B(x)\right) \phi
\end{align}
has for each $\lambda \in \overline{B}_{\lambda_0}(\varrho_0)$ an exponential dichotomy on $[0,\infty)$ with $\lambda$- and $\delta$-independent constants and projections $Q(x;\lambda)$ satisfying the following properties:
\begin{itemize}
    \item[1.] The map $Q(x; \cdot) \colon B_{\lambda_0}(\varrho_0) \to \C^{n \times n}$ is analytic for each $x \geq 0$.
    \item[2.] We have
    \begin{align*}
\|Q(0;\lambda) - \tilde{P}(\lambda)\| \leq C\delta, \qquad \|\tilde{P}(\lambda)\| \leq C
\end{align*}
for each $\lambda \in \overline{B}_{\lambda_0}(\varrho_0)$, where $\tilde{P}(\lambda)$ is the projection onto $\mathrm{ran}(P(0;\lambda))$ along $\mathrm{ran}(P(0;\lambda_0))^\perp$.
\item[3.] The estimate
\begin{align} \label{projesti2}
\|Q(x;\lambda) - P(x;\lambda)\| \leq C\left(\delta + \eu^{-\theta x}\right)
\end{align}
holds for each $x \geq 0$ and $\lambda \in \overline{B}_{\lambda_0}(\varrho_0)$.
\end{itemize}
\end{Lemma}
\begin{proof}
Since system~\eqref{linana} depends analytically on $\lambda$ and the subspace $\mathrm{ran}(P(0;\lambda))$ is by~\cite[p.~19]{Coppel1978} uniquely determined,~\cite[Lemma~A.2]{DasLatushkin2011} and~\cite[Section~II.4.2]{Kato1995} yield that there exists an analytic map $B_{\mathrm{s}} \colon \Omega \to \C^{n \times k}$ such that $B_{\mathrm{s}}(\lambda)$ is a basis of $\mathrm{ran}(P(0;\lambda))$ for each $\lambda \in \Omega$. Since $B_{\mathrm{s}}$ is analytic, there exists a closed disk $\overline{B}_{\lambda_0}(\varrho_0) \subset \Omega$ of some radius $\varrho_0 > 0$ such that $\det(B_\mathrm{s}(\lambda_0)^* B_\mathrm{s}(\lambda)) \neq 0$ for all $\lambda \in \overline{B}_{\lambda_0}(\varrho_0)$. Thus, $\mathrm{ran}(P(0;\lambda_0))^\perp$ complements $\mathrm{ran}(P(0;\lambda))$ for all $\lambda \in \overline{B}_{\lambda_0}(\varrho_0)$.

By roughness of exponential dichotomies, cf.~\cite[Theorem~4.1]{Coppel1978}, there exists a constant $\delta_0 > 0$ such that, for each $\delta \in (0,\delta_0)$ and $B \in C([0,\infty),\C^{n \times n})$ with $\|B\|_{L^\infty} \leq \delta$, the perturbed system~\eqref{linanaper} admits an exponential dichotomy on $[0,\infty)$ with constants $K_1,\mu_1 > 0$, depending on $K$ and $\mu$ only, and projections $\mathcal{Q}(x;\lambda)$ satisfying
\begin{align} \label{projest33}
\left\|\mathcal{Q}(x;\lambda) - P(x;\lambda)\right\| \leq K_1\delta
\end{align}
for all $x \geq 0$ and $\lambda \in \overline{B}_{\lambda_0}(\varrho_0)$. 

Let $\delta \in (0,\delta_0)$ and $B \in C([0,\infty),\C^{n \times n})$ with $\|B\|_{L^\infty} \leq \delta$. Since~\eqref{linanaper} depends analytically on $\lambda$ and the subspace $\mathrm{ran}(\mathcal{Q}(x;\lambda))$ is by~\cite[p.~19]{Coppel1978} uniquely determined,~\cite[Lemma~A.2]{DasLatushkin2011} and~\cite[Section~II.4.2]{Kato1995} yield that $\mathrm{ran}(\mathcal{Q}(x;\lambda))$ possesses a basis, which is analytic in $\lambda$ on $B_{\lambda_0}(\varrho_0)$. Set $\check{B}_{\mathrm{s}}(\lambda) = \mathcal{Q}(0;\lambda) B_{\mathrm{s}}(\lambda)$. Estimate~\eqref{projest33}, analyticity of $B_{\mathrm{s}}$ on $\Omega$ and compactness of $\overline{B}_{\lambda_0}(\varrho_0)$ yield a $\lambda$- and $\delta$-independent constant $M_0 > 0$ such that it holds
\begin{align*}
\left\|B_{\mathrm{s}}(\lambda) - \check{B}_{\mathrm{s}}(\lambda)\right\| \leq M_0\delta, \qquad \left\|B_{\mathrm{s}}(\lambda)\right\| \leq M_0
\end{align*}
for each $\lambda \in \overline{B}_{\lambda_0}(\varrho_0)$. So, taking $\delta_0 > 0$ smaller if necessary, $\check{B}_{\mathrm{s}}(\lambda)$ is a basis of $\mathrm{ran}(\mathcal{Q}(0;\lambda))$ and we can arrange for
\begin{align*}
\left\|B_{\mathrm{s}}(\lambda) - \check{B}_{\mathrm{s}}(\lambda)\right\| \leq \frac{1}{2 \left\|\left(B_{\mathrm{s}}(\lambda_0)^* B_{\mathrm{s}}(\lambda)\right)^{-1}\right\| \|B_{\mathrm{s}}(\lambda_0)\|}
\end{align*}
for each $\lambda \in \overline{B}_{\lambda_0}(\varrho_0)$. Hence, taking $\delta_0 > 0$ smaller if necessary, Lemma~\ref{l:projest3} demonstrates that the projections $\tilde{P}(\lambda)$ onto $\mathrm{ran}(P(0;\lambda))$ along $\mathrm{ran}(P(0;\lambda_0))^\perp$ and $Q(0;\lambda)$ onto $\mathrm{ran}(\mathcal{Q}(0;\lambda))$ along $\mathrm{ran}(P(0;\lambda_0))^\perp$ are well-defined and there exists a $\lambda$- and $\delta$-independent constant $M_1 > 0$ such that
\begin{align*}
\|Q(0;\lambda) - \tilde{P}(\lambda)\| \leq M_1 \delta, \qquad \|\tilde{P}(\lambda)\|, \left\|Q(0;\lambda)\right\| \leq M_1
\end{align*}
for each $\lambda \in \overline{B}_{\lambda_0}(\varrho_0)$. Since $\mathrm{ran}(P(0;\lambda))$ and $\mathrm{ran}(\mathcal{Q}(0;\lambda))$ have bases which are analytic in $\lambda$ on $B_{\lambda_0}(\varrho_0)$ and the subspace $\mathrm{ran}(P(0;\lambda_0))^\perp$ is independent of $\lambda$, the projections $\tilde{P}(\lambda)$ and $Q(0;\lambda)$ are analytic in $\lambda$ on $B_{\lambda_0}(\varrho_0)$ by Lemma~\ref{l:proj_ana}. Hence, recalling that~\eqref{linanaper} has an exponential dichotomy on $[0,\infty)$ with constants $K_1,\mu_1 > 0$ and projections $\mathcal{Q}(x;\lambda)$, the exposition in~\cite[pp.~16-17]{Coppel1978} implies that~\eqref{linanaper} admits an exponential dichotomy on $[0,\infty)$ with constants $C_1,\mu_1 > 0$ with $C_1 = K_1 + K_1^2 M_1 + K_1^3$ and projections $Q(x;\lambda) = T(x,0;\lambda) Q(0;\lambda) T(0,x;\lambda)$ for each $\lambda \in \overline{B}_{\lambda_0}(\varrho_0)$, where $T(x,y;\lambda)$ is the evolution of system~\eqref{linanaper} which depends analytically on $\lambda$ by~\cite[Lemma~2.1.4]{KapitulaPromislow2013}. Therefore, $Q(x;\lambda)$ is for each $x \geq 0$ analytic in $\lambda$ on $B_{\lambda_0}(\varrho_0)$. Finally, Lemma~\ref{l:projest} yields
\begin{align} \label{projestanq}
\left\|Q(x;\lambda) - \mathcal{Q}(x;\lambda) \right\| \leq 2C_1 K_1  \eu^{-2\mu_1 x}
\end{align}
for each $\lambda \in \overline{B}_{\lambda_0}(\varrho_0)$ and $x \geq 0$. Combining~\eqref{projestanq} with~\eqref{projest33} we arrive at~\eqref{projesti2}, which completes the proof.
\end{proof}

\section{Compactness of a multiplication operator} \label{app:compact}

In this appendix, we prove an auxiliary compactness result for  multiplication operators mapping from $H^1(\R)$ into $L^2(\R)$.

\begin{Lemma} \label{l:compact}
Let $g \in H^1(\R)$. The multiplication operator $A \colon H^1(\R) \to L^2(\R)$ given by $Au = g u$ is well-defined and compact.
\end{Lemma}
\begin{proof}
The fact that $A$ is well-defined follows directly from the embedding $H^1(\R) \hookrightarrow L^\infty(\R)$. Let $V \subset H^1(\R)$ be a bounded subset and let $\eps > 0$. Since the embedding $H^1(\R) \hookrightarrow L^\infty(\R)$ is continuous, there exists $K > 0$ such that for each $u \in V$ we have $\|u\|_{L^\infty} \leq K$. There exists $R > 0$ such that
\begin{align*}
    \int_{\R \setminus [-R,R]} |g(x)|^2 \de x \leq \frac{\eps^2}{2K^2},
\end{align*}
implying
\begin{align*}
 \int_{\R \setminus [-R,R]} |(Au)(x)|^2 \de x \leq \|u\|_{L^\infty}^2 \int_{\R \setminus [-R,R]} |g(x)|^2 \de x \leq \frac{\eps^2}{2},
\end{align*}
for all $u \in V$. 

By the Rellich-Kondrachov theorem the embedding $H^1((-R,R)) \hookrightarrow L^2((-R,R))$ is compact. The set $W = \{(g u)|_{(-R,R)} : u \in V\}$ is bounded in $H^1((-R,R))$, because we have $\|(gu)|_{(-R,R)}\|_{H^1} \leq \|g\|_{H^1} \|u\|_{L^\infty} \leq K \|g\|_{H^1}$ for each $u \in V$, where $h|_{(-R,R)}$ denotes the restriction of a function $h \in H^1(\R)$ to the interval $(-R,R)$. So, by compactness of the embedding $H^1((-R,R)) \hookrightarrow L^2((-R,R))$, there exists a finite subset $G \subset L^2((-R,R))$ such that for each $h \in W$ there exists $f \in G$ with
\begin{align*}
\int_{-R}^R |h(x) - f(x)|^2 \de x \leq \frac{\eps^2}{2}. 
\end{align*}

We conclude that for each $u \in V$ there exists $f \in G$ such that
\begin{align*}
\int_\R |(Au)(x) - f(x) \mathbf{1}_{(-R,R)}(x)|^2 \de x =  \int_{\R \setminus [-R,R]} |(Au)(x)|^2 \de x + \int_{-R}^R |g(x) u(x) - f(x)|^2 \de x \leq \eps^2,
\end{align*}
where $\mathbf{1}_{(-R,R)}$ is the indicator function of the interval $(-R,R)$, whence $\|Au - f\mathbf{1}_{(-R,R)}\|_{L^2} \leq \eps$. We conclude that the set $A[V]$ is precompact in $L^2(\R)$, which concludes the proof.
\end{proof}

\bibliographystyle{abbrv}
\bibliography{bibliography}

\end{document}